\documentclass[12pt]{amsart}
\usepackage{amscd}
 \pdfoutput=1

\usepackage{amsmath}
\usepackage{amsfonts}
\usepackage{amssymb}
\usepackage{amsthm}
\usepackage{mathtools}

\usepackage[utf8]{inputenc}

	\usepackage{pinlabel}
	\usepackage{lmodern}
	\usepackage{textcomp}
	 \usepackage[hidelinks]{hyperref}
	 \usepackage{enumitem} 

	\usepackage{color}		

	\usepackage[all,cmtip]{xy}
	   \usepackage{tabu}
	   \usepackage{array}
    \usepackage{nicematrix}
    
    \usepackage{bbm} 

\newtheorem{thm}{Theorem}[section]

\newtheorem{prp}[thm]{Proposition}

\newtheorem{cor}[thm]{Corollary}

\newtheorem{lma}[thm]{Lemma}

\theoremstyle{definition}
\newtheorem{defi}[thm]{Definition}
 \theoremstyle{remark}
 \newtheorem{rmk}[thm]{Remark}

\numberwithin{equation}{section}

\newcommand{\gor}[1]{\textcolor{red}{#1}}
\newcommand{\myrmk}[1]{\textcolor{blue}{#1}}
\renewcommand{\gor}[1]{}
\renewcommand{\myrmk}[1]{}

\makeatletter
\newcommand{\labitem}[2]{%
\def\@itemlabel{{#1}}
\item
\def\@currentlabel{#1}\label{#2}}
\makeatother

\newcommand{\R}{{\mathbb R}}
\newcommand{\Z}{{\mathbb Z}}
\newcommand{\C}{{\mathbb C}}

\newcommand{\Or}{\mathcal{O}}
\newcommand{\Co}{\mathcal{C}}
\newcommand{\Hi}{\mathcal{H}}
\newcommand{\M}{\mathcal{M}}
\newcommand{\T}{\mathcal{T}}

\newcommand{\V}{\mathcal{V}}
\newcommand{\s}{\mathfrak{s}}

\newcommand{\F}{\mathcal{F}}
\newcommand{\I}{\mathcal{I}}

\newcommand{\Pu}{\mathcal{P}}

\newcommand{\Y}{\mathcal{Y}}

\newcommand{\dbar}{\bar{\partial}}

\newcommand{\id}{\operatorname{id}}

\newcommand{\ind}{\operatorname{index}}

\newcommand{\Ker}{\operatorname{Ker}}
\newcommand{\Coker}{\operatorname{Coker}}
\newcommand{\kd}{\operatorname{Ker}\dbar}
\newcommand{\cd}{\operatorname{Coker}\dbar}

\newcommand{\con}{\operatorname{con}}
\newcommand{\ev}{\operatorname{ev}}

\newcommand{\diag}{\operatorname{Diag}}
\newcommand{\sgn}{\operatorname{sign}}
\newcommand{\spn}{\operatorname{Span}}

\newcommand{\vk}{\operatorname{VirKer}}
\newcommand{\capp}{\operatorname{cap}}

\newcommand{\pre}{\operatorname{pre}}
\newcommand{\aux}{\operatorname{Aux}}

\newcommand{\im}{\operatorname{im}}
\newcommand{\ul}{\underline}

\newcommand{\en}{\text{end}}

\newcommand{\op}{\operatorname{op}}

\newcommand{\w}{{\bf w}}

\newcommand{\bm}{\operatorname{bm}}
\newcommand{\inter}{\operatorname{int}}
\newcommand{\Lag}{\operatorname{Lag}}

\newcommand{\negg}{\operatorname{ne}}
\newcommand{\pos}{\operatorname{po}}
\newcommand{\tp}{\operatorname{tp}}
\newcommand{\indu}{\operatorname{ind}}
\newcommand{\triv}{\operatorname{triv}}
\newcommand{\swi}{\operatorname{sw}}
\newcommand{\stab}{\operatorname{stab}}

\title[Orientations of Morse flow trees in Legendrian contact homology]
{Orientations of Morse flow trees in Legendrian contact homology}
\author{Cecilia Karlsson}
\address{University of Borås,
Department of Engineering\\
Högskolan i Borås,
S-501 90 Borås,
Sverige\\
cecilia.karlsson@hb.se}

\begin{document}

\begin{abstract}
Let $\Lambda$ be a closed, connected, spin Legendrian submanifold of the 1-jet space of a smooth $n$-dimensional manifold. We give a coherent orientation scheme for the moduli space of rigid Morse flow trees of $\Lambda$, implying that the Legendrian contact homology of $\Lambda$ with integer coefficients can be computed using Morse flow trees. If $n>1$ then this orientation scheme can be computed with an algorithm which uses intersections of oriented flow manifolds in $M$ together with combinatorial data coming from the trees.
\end{abstract}



\maketitle

\newpage

\tableofcontents

\newpage

\section{Introduction}
Let $M$ be a smooth manifold of dimension $n$, and let $J^1(M)=T^*M \times \R$ denote its 1-jet
space, with  local coordinates $(x_1,\dotsc,x_n,y_1,\dotsc,y_n,z)$. This space can be given the structure  of a contact manifold, with contact distribution $\xi = \Ker ( dz- \sum_i y_i dx_i)$.   
 A submanifold $\Lambda \subset J^1(M)$ is called \emph{Legendrian} if it is $n$-dimensional and everywhere tangent to $\xi$, 
and a \emph{Legendrian isotopy} is a smooth $1$-parameter family of Legendrian submanifolds. The problem of classifying all Legendrian submanifolds of a contact manifold, up to Legendrian isotopy, is a central problem in contact geometry. This motivates finding  Legendrian invariants, that is, invariants that are preserved under Legendrian isotopies. One such invariant is \emph{Legendrian contact homology},  which is a homology theory that fits into the package of Symplectic Field Theory (SFT), introduced by Eliashberg, Givental and Hofer in the paper \cite{sft}. 

More recently, it has been shown that Legendrian contact homology can be used not only to study properties of Legendrian submanifolds, but also to compute symplectic invariants of Weinstein manifolds obtained by surgery along Legendrians. See e.g.\ ~\cite{surg2,surg1}. It has also been shown, for example in ~\cite{el,rm}, that this can be generalized and used for computations in homological mirror symmetry. This indicates the importance of being able to explicitly understand the Legendrian contact homology complex with integer coefficients. This paper should be an important step in that direction.   

%

Briefly, Legendrian contact homology is the homology of a differential
graded algebra (DGA) associated to $\Lambda$.  One should note that
Legendrian contact homology has not been worked out in full detail for
all contact manifolds, but in the case when the contact manifold is
given by  the 1-jet space $J^1(M) = T^*M \times \R $ of a manifold
$M$, then the analytical details were established by Ekholm, Etnyre
and Sullivan in \cite{pr}. In the special case $M = \R$, this was also done by Eliashberg in \cite{eliinv} and independently by Chekanov in \cite{chekanov}. 

In the case of a 1-jet space, the DGA of $\Lambda$ can be defined by considering the \emph{Lagrangian projection} 
$\Pi_\C: J^1(M) \to T^*M$. The generators of the algebra are then given by the double points of $\Pi_\C(\Lambda)$, which correspond to \emph{Reeb chords} of $\Lambda$. These are flow segments of the \emph{Reeb vector field}
 $\partial_z$, having start and end point on $\Lambda$. We assume that all self-intersections of  $\Pi_\C(\Lambda)$ are transverse double points, which is a generic condition.
The differential $\partial$ of the algebra counts certain rigid pseudo-holomorphic disks in $T^*M$ with boundary on $\Pi_\C(\Lambda)$. With a clever choice of almost complex structure on $T^*M$, one gets that the homology of this complex is a Legendrian invariant. 

If $M = \R$, then the count of pseudo-holomorphic disks reduces to combinatorics, as described by Chekanov in \cite{chekanov},  but in higher dimensions the Cauchy-Riemann equations give rise to non-linear partial differential equations, which are hard to solve. 
To simplify an analogous problem in Lagrangian Floer homology, Fukaya
and Oh in \cite{fo} introduced \emph{Morse flow trees}, which are
trees with edges along gradient flow lines in $M$, and gave a
one-to-one correspondence between rigid Morse flow trees and rigid
pseudo-holomorphic disks in $T^*M$ with boundary on the
Lagrangians. This is built on an idea of Floer, see \cite{floertrees}.
In \cite{trees}, Ekholm generalized this method to also work in the Legendrian contact homology setting.

 In
the present paper we will combine the results in \cite{trees} with the results in \cite{orientbok}, where it is proved that Legendrian contact homology can be defined using $\Z$-coefficients, provided $\Lambda$ is spin. This corresponds to orienting moduli spaces of holomorphic disks in a coherent way, and builds on the work of Fukaya, Oh, Ohta and Ono   in \cite{fooo}, where they describe a way to orient the determinant line of the $\dbar$-operator over the space of trivialized Lagrangian boundary conditions for the unit disk in $\C$. Since the differential in Legendrian contact homology counts rigid pseudo-holomorphic disks, the orientation of the moduli space at a disk $u$ corresponds to a sign of $u$. That the orientation of the moduli space is coherent means that all these signs cancels in $\partial^2$, so that we get $\partial^2 =0$.

The main result of the present paper is the existence of an orientation scheme for the moduli space of rigid Morse flow trees of $\Lambda$, which uses 
orientations of stable and unstable manifolds associated to the Morse flow
trees together with combinatorial data from the trees. If $n>1$ then this scheme can be described by an algorithm which gives the signs of the trees. If $n=1$ we have partial results, but the author has not been able to compute some of the signs needed at the end of the algorithm. We refer to \cite{korta} for the explicit algorithm, in the present paper we focus on the analytical aspects of the proof that this algorithm indeed gives a coherent orientation scheme for moduli spaces of rigid Morse flow trees.  This is a generalization and an extension of the orientation scheme described in  \cite{kch} in the case when $\Lambda$ is the conormal lift of a knot in $\R^3$.

\subsection{Outline}\label{sec:outline}
Let $\Lambda \subset J^1(M)$ be a spin Legendrian submanifold.
 We will assume that $\Lambda$ is \emph{front generic}, meaning that
 $\Pi:\Lambda \to M$ is an immersion outside a co-dimension 1
 submanifold $\Sigma \subset \Lambda$. The points in $\Sigma$ we call
 \emph{cusp points}. We will also assume that $\Lambda$ has
 \emph{simple front singularities}, meaning that the points in
 $\Sigma$ are projected to standard cusp singularities under the front
 projection $\Pi_F: J^1(M) \to M \times \R$. If $n=2$ we also allow swallow-tail singularities. See \cite{trees}.  

 Locally, away from $\Sigma$, the Legendrian $\Lambda$ can be described as the multi-1-jet graph of locally defined functions $f_1,\dotsc,f_l: U \to \R$, $U \subset M$.
Fix a metric $g$ on $M$. The Morse flow trees of $\Lambda$ are defined using negative gradient flows of such function differences. I.e.\ the edges of the flow trees consist of flow lines of $- \nabla (f_i-f_j)$ in $M$, and these curves are patched together at $2$- and $3$-valent vertices, where the local defining functions for $\Lambda$ change. The tree is a directed graph with one root, and the direction of the edges are given by the flow direction of the defining vector fields. For example, at a $3$-valent vertex we can have an incoming flow line which solves $- \nabla(f_1-f_3)$, and two outgoing flow-lines solving $- \nabla (f_1-f_2)$ and $- \nabla (f_2 - f_3)$, respectively. See Figure \ref{fig:extree}. The $1$-valent vertices are either critical points of the flow lines ending at the vertex, or belong to $\Pi(\Sigma)$.  We give a more detailed description of flow trees in Section \ref{sec:tre}.

\begin{figure}[ht]
\labellist
\small\hair 2pt
\pinlabel $\R$ [Br] at 90 340
\pinlabel ${f_1}$ [Br] at  400 280 
\pinlabel ${f_2}$ [Br] at 400 215
\pinlabel $f_3$ [Br] at 390 145
\pinlabel $\Gamma$ [Br] at 220 22
\pinlabel $M$ [Br] at 67 15
\endlabellist
\centering
\includegraphics[height=5cm, width=6.5cm]{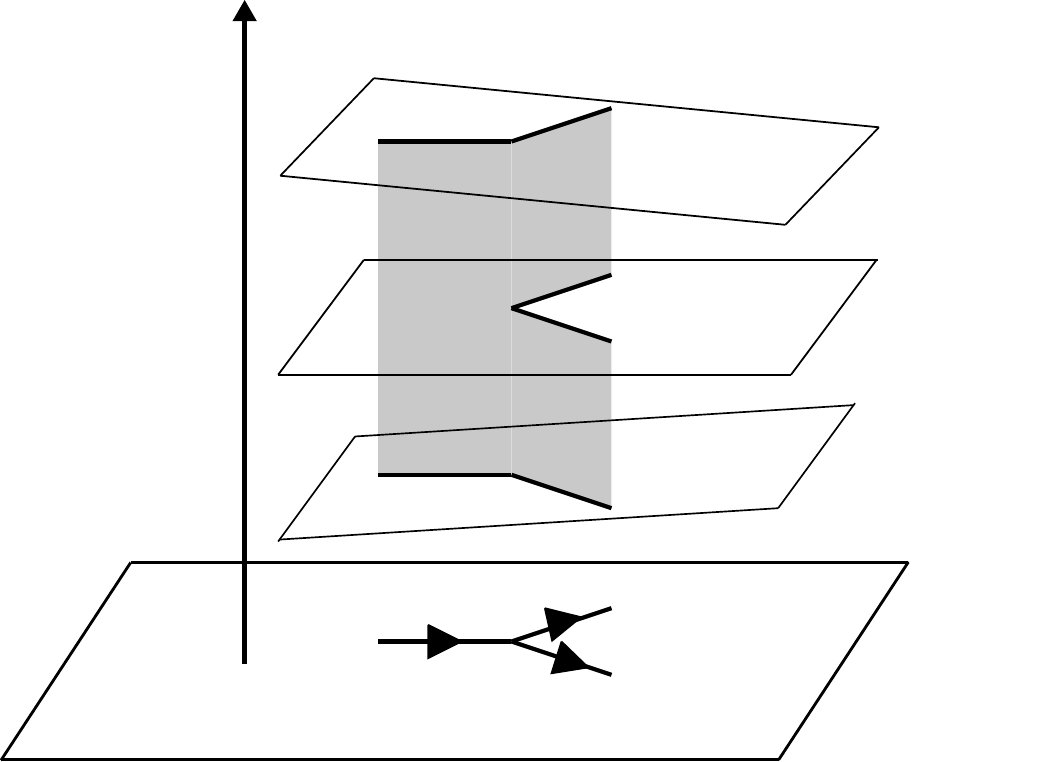}
\caption{The local picture of a Morse flow tree $\Gamma$ in a neighborhood of a $Y_0$-vertex. The graphs of the defining functions $f_1$, $f_2$ and $f_3$ are sketched, together with the lift of the tree to the sheets of $\Lambda$ determined by these functions. The shaded area indicates the corresponding pseudo-holomorphic disk.}
\label{fig:extree}
\end{figure}

%

The interior points of the edges of the tree has two lifts to $\Lambda$, one for each of the sheets corresponding to the defining functions for the flow line. The 1-valent vertices which do not lift to $\Sigma$ will lift to end points of Reeb chords, and the Lagrangian projection of the lift of a tree will give a closed curve in $\Pi_\C(\Lambda)$. 

By the results in \cite{trees} there is a one-to-one correspondence between flow trees and pseudo-holomorphic disks defined by $\Lambda$. Briefly this works as follows. Consider the fiber scaling $s_\lambda:(x,y,z) \mapsto (x, \lambda y, \lambda z)$ which pushes $\Lambda$ towards the zero section. Then there is a choice of metric on $M$ and almost complex structure of $T^*M$, so that after a perturbation of $\Lambda$ there is a one-to-one correspondence between rigid flow trees of $\Lambda$ and sequences of $J$-holomorphic disks $u_\lambda$ in $T^*M$ with boundary on $\Pi_\C(s_\lambda(\Lambda))$. The disks are characterized by the property that their boundaries are arbitrarily close to the cotangent lift of the corresponding tree $\Gamma$. 
Moreover, the linearized boundary conditions of the disks tend to constant $\R^n$-boundary conditions, except at cusp points, where we get a split boundary condition which is constantly $\R$ in the first $n-1$ directions, and given by a uniform $\pm \pi$-rotation in the last direction, where the sign depends on the type of vertex. In addition, when we let $\lambda \to 0$, then  the domain of the disk will split up into a finite union of strips and strips with one slit.  See Section \ref{sec:slit}.       

The orientation of the moduli space of rigid trees will be related to 
 orientations of the moduli spaces of the pseudo-holomorphic disks $u_\lambda$ converging to the trees. The latter space will be given orientation by orienting the determinant lines of the linearized $\dbar$-operator at the disks. This orientation, in turn, is defined by using methods from \cite{fooo} and \cite{orientbok}, and is called the \emph{capping orientation} of $u_\lambda$. 

Under the degeneration process, i.e.\ when $\lambda \to 0$, each rigid
flow tree $\Gamma$ (or more precisely, the disks $u_\lambda$) splits
up into \emph{elementary pieces}; pieces of the tree with at most $1$
vertex each. We use this to derive an algorithm for combinatorially
calculating the capping orientation of  $u_\lambda$ in the limit
$\lambda =0$, and this will be the \emph{capping orientation of}
$\Gamma$. Namely, we will give the elementary pieces of the tree
orientations coming from the stable and unstable manifolds of the
critical points contained in the tree, and from $\Pi(\Sigma)$. Then we glue the pieces back together again, and under this gluing we can understand the orientation of $\Gamma$ as oriented intersections of flow manifolds. Due to the rigidity of $\Gamma$, we will in the end get an oriented $0$-dimensional intersection, and by multiplying this by a sign calculated from combinatorial data from the tree, we obtain the capping orientation of the tree. This will also be called the \emph{sign of $\Gamma$}.

The capping orientation of $u_\lambda$ and of $\Gamma$ will depend on several choices, which we call \emph{initial orientation choices} and which we describe in Section \ref{sec:orientconv}.
The main result is the following.

\begin{thm}\label{thm:mainthm}
Let $\Lambda \subset J^1(M)$ be a closed, connected, spin Legendrian submanifold, and assume that we have fixed all initial orientation choices. Then there is a coherent orientation scheme for the moduli space of rigid pseudo-holomorphic disks of $\Lambda$ so that the following holds.
Assume that $\Gamma$ is a rigid flow tree of $\Lambda$, and let
 $u_\lambda$, $\lambda \to 0$, be a sequence of rigid pseudo-holomorphic disks converging to  $\Gamma$ in the sense of \cite{trees}. Then there is a $\lambda_0>0$ so that the orientation of $u_\lambda$ is independent of $\lambda$ for $0 < \lambda<\lambda_0$, and if $n>1$ then this orientation can be computed in terms of intersections of oriented flow manifolds in $M$ together with combinatorial data coming from $\Gamma$. 
\end{thm}

For a complete statement of the algorithm for computing the
orientation, we refer to \cite{korta}. We point out that for $n=1$ we get a similar algorithm for computing the orientation, but now the algorithm will depend on some auxiliarly variables. However, in this dimension one can as well use Chekanovs algorithm or the one given by Ng in \cite{ngcomp}, for example.  In Section
\ref{sec:glupf} we indicate how all the formulas in \cite{korta} are
derived, and in Section \ref{sec:examples} we compute some examples. In particular, the reader will see what combinatorial data from the trees that are involved. 



\begin{rmk}
 The capping orientation of $u_\lambda$ defined in this paper will
 differ from the capping orientation defined in \cite{orientbok}. The
 reason for this is that the orientation in the present paper is
 easier to work with in the tree-setting, while the orientation in
 \cite{orientbok} is more efficient when, for example, proving
 invariance of the Legendrian contact homology under Legendrian
 isotopies. However, in \cite{teckcob} it is proven that these
 different choices give rise to isomorphic DGA's. Thus
 it follows that we can use flow trees to compute the Legendrian contact homology of $\Lambda$ with $\Z$-coefficients. Also compare \cite{kch} where a similar issue is discussed. 
\end{rmk}


\subsection{Organization of the paper}
In Section \ref{sec:tre} we give an introduction to Morse flow trees, closely following \cite{trees}. In particular, we give a description of neighborhoods of partial flow trees, understood as immersed submanifolds of $M$. It is these neighborhoods that we will use when calculating the orientation of $\Gamma$ from intersecting submanifolds in $M$.


In Section \ref{sec:background} we review some background materials and also define the capping orientation of a $J$-holomorphic disk. We use an approach similar to the one in \cite{orientbok}, but we will use slightly different orientation conventions. 

In Section \ref{sec:cutstab} we give a brief outline of how to associate a family of pseudo-\linebreak
holomorphic disks $u_\lambda$ to a rigid flow tree $\Gamma$, again following \cite{trees}. We also define the capping orientation of a flow tree $\Gamma$, which will be the capping  orientation of a sequence of pre-glued disks associated to the tree. Indeed, in \cite{trees} it is shown that for each $\lambda$  sufficiently small, there is a pre-glued disk $w_\lambda$ arbitrary close to $u_\lambda$ so that the boundary conditions induced by $w_\lambda$ are very easy to describe explicitly. 

In Section \ref{sec:glupf} we prove that the capping orientation of
$\Gamma$ is independent of $\lambda$ for $\lambda$ sufficiently small,
and that the capping orientation of $u_\lambda$ equals the capping
orientation of $w_\lambda$ for such $\lambda$. This gives Theorem
\ref{thm:mainthm}. The main technical results are proven in Section
\ref{sec:stab} and Section \ref{sec:glupf}. In Section
\ref{sec:glupf} we also show how the explicit formulas in \cite{korta} are derived. In Section
\ref{sec:examples} we use the algorithm to compute orientations of some of the
trees that occurs in [\cite{trees}, Section 7]. In Section
\ref{sec:examples}, Section \ref{sec:der} and Section \ref{sec:analder} we derive expressions for the signs that are indicated but not explicitly stated in Section \ref{sec:glupf}. We also include
Appendix \ref{app:mini}, where we summarize some results concerning families of Fredholm operators.

  \section{Morse flow trees}\label{sec:tre}
In this section we give an introduction to Morse flow trees, following Ekholm in \cite{trees}. We will focus on the properties of the trees that will be of most importance for us, and refer to \cite{trees} for a deeper discussion of the subject.

We will first describe how Morse flow trees are built, see Subsection \ref{sec:buildtree}. In Subsection \ref{sec:parttree} we introduce partial flow trees, which are defined similar to Morse flow trees, but should be thought of as trees obtained from a Morse flow tree $\Gamma$ by cutting $\Gamma$ into smaller pieces. In Subsection \ref{sec:geomint} we describe neighborhoods of partial flow trees, given as submanifolds of $M$ via evaluation. It will be these submanifolds that we orient when we are deriving the sign of the  flow trees.   
In Subsection \ref{sec:slit} we describe a special type of domains associated to flow trees. 

We will assume that $\Lambda$ is closed and connected.


\subsection{Flow lines and vertices}\label{sec:buildtree}
As we briefly explained in Section \ref{sec:outline}, Morse flow trees
consist of gradient flow lines in $M$. These flow lines are induced by
pairs of locally defined functions related to $\Lambda$, and they are patched together to give a closed curve when lifted to $T^*M$. More precisely, the trees are constructed as follows. 

Given a Legendrian $\Lambda \subset J^1(M)$, it is a standard fact that $\Lambda\setminus \Sigma$ locally can be given as the \emph{$1$-jet lift} 
\begin{equation*}
J^1(f)=\{(x,df(x),f(x)), x \in U\subset M\}
\end{equation*}
of a locally defined function $f : U \to \R$. 
Fix a Riemannian metric $g$ on $M$ and let $\nabla$ denote the corresponding gradient operator. To define flow trees we will use the local functions $f_i$ which define $\Lambda$, or more precisely, the flow lines of the gradients of their pairwise function differences. Note that these vector fields  $-\nabla (f_i-f_j)$ are only locally defined on $M$. 

\begin{defi}
A \emph{local flow line} $\gamma$ of $\Lambda$ is a curve $\gamma: I \to U \subset M$ satisfying 
\begin{equation*}
\dot\gamma(t) = - \nabla(f_1-f_2)(\gamma(t)),
\end{equation*} 
where $f_1, f_2:U \to \R$ is a pair of locally defining functions for $\Lambda$. Here $I$ is some interval.
\end{defi}

\begin{rmk}
A local flow line is either an embedding or a constant map whose image is a critical point of $f_1-f_2$.
\end{rmk}

Assume that $\gamma$ is non-constant and that $f_1 > f_2$ along $\gamma$. Then we define the \emph{flow orientation} of $\gamma$ to be the orientation given by the vector field $- \nabla(f_1-f_2)(\gamma(t))$.

We associate to each flow line $\gamma$  its two $1$-jet lifts $(\tilde{\gamma}^1,\tilde{\gamma}^2)$, given by
\begin{equation*}
\tilde{\gamma}^i(t) = (\gamma(t), df_i(\gamma(t)), f_i(\gamma(t))) \in J^1(M), \qquad i=1,2,
\end{equation*}
and also its cotangent lifts $(\bar{\gamma}^1, \bar{\gamma}^2)$ given by
\begin{equation*}
\bar{\gamma}^i(t) = \Pi_\C(\tilde{\gamma}^i(t) ) \in T^*M, \qquad i=1,2.
\end{equation*}
The \emph{flow orientation} of $\tilde{\gamma}^1$ and $\tilde{\gamma}^2$ is given by the lift of the flow orientation of $\gamma$ and $- \gamma$, respectively. See Figure \ref{fig:local}.
The flow orientation of $\bar{\gamma}^i$ is the orientation induced by the orientation of $\tilde{\gamma}^i$ under the Lagrangian projection, $i=1,2$. 

\begin{figure}[hbt]
\labellist
\small\hair 2pt
\pinlabel ${f_1}$ [Br] at  80 260 
\pinlabel ${f_2}$ [Br] at 68 128
\pinlabel $\tilde \gamma_1$ [Br] at 150 235
\pinlabel $\tilde \gamma_2$ [Br] at 150 104
\pinlabel $\gamma$ [Br] at 140 14
\pinlabel $M$ [Br] at 72 40
\endlabellist
\centering
\includegraphics[height=5cm, width=6.5cm]{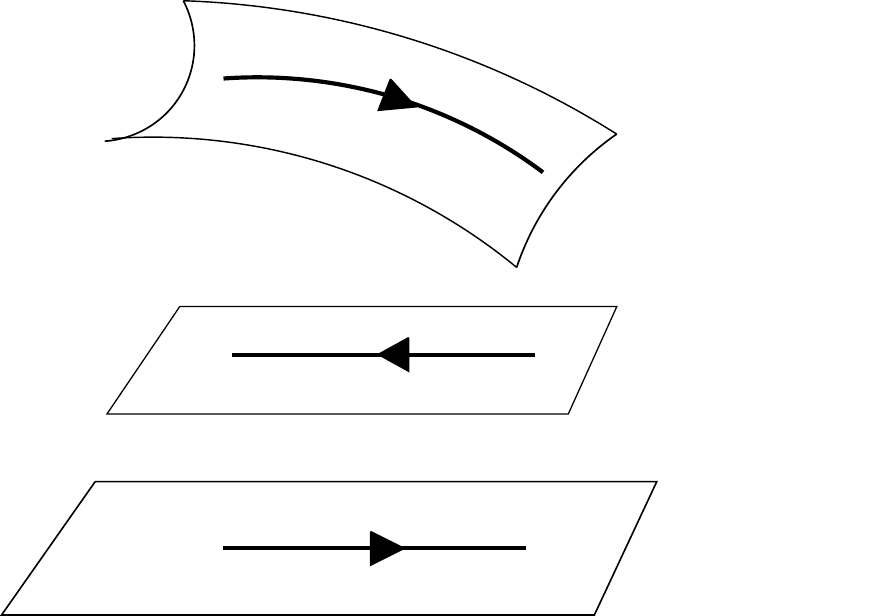}
\caption{The front projection of 1-jet lifts of a gradient flow line.}
\label{fig:local}
\end{figure}

Now assume that we have two pairs $f_1, f_2 : U \to \R$, $\tilde f_1, \tilde f_2 : \tilde U \to \R$ of locally defining functions of $\Lambda$ so that $U \cap \tilde U \neq \emptyset$ and  so that $J_1(f_i) =J_1(\tilde f_i)$ on $U \cap \tilde U$, $i=1,2$. Assume that $\gamma_1: I_1 \to U$ is a local flow line of $\Lambda$ satisfying $\dot \gamma_1(t) = - \nabla (f_1-f_2)(\gamma_1(t))$. Assume $\gamma_1$ can be extended over $\tilde U$ to a curve $\gamma:I\to U \cup \tilde U$, $I_1 \subset I$, which satisfies 
\begin{equation}\label{eq:patch}
\gamma(t) = 
\begin{cases}
&\gamma_1(t), \quad t \in I_1 \\
& \gamma_2(t), \quad t \in I\setminus I_1, \text{ where } \gamma_2(t) \in \tilde U, \quad \dot \gamma_2(t) = - \nabla(\tilde f_1 - \tilde f_2)(\gamma_2(t)).
\end{cases}
\end{equation}
\begin{defi}
We call $\gamma$ a \emph{flow line of $\Lambda$}, where we are allowed to make the extension in \eqref{eq:patch} over several subsets of $M$ (assuming inductively that $\gamma_1$ is a flow line of $\Lambda$ defined on $U$  and not necessarily a local flow line).
\end{defi}
\begin{defi}
If $\gamma:I \to M$ is a flow line of $\Lambda$ which cannot be extended to a larger interval as in \eqref{eq:patch} we call $\gamma$ a \emph{maximal flow line of $\Lambda$.}
\end{defi}

Thus, a flow line $\gamma$ of $\Lambda$ is locally a solution curve to $-\nabla(f_1-f_2)$ where $f_1$, $f_2$ are locally defining functions for $\Lambda$.  The \emph{1-jet lift of a flow line of $\Lambda$} is locally given by the 1-jet lift of the corresponding local flow line. The \emph{cotangent lift of a flow line of $\Lambda$} is defined in an analogous way, and  we define the \emph{flow orientation of $\gamma$} to locally be given by the flow orientation of the corresponding local flow line. The flow orientation of the 1-jet lift and the cotangent lift is defined similarly.

\begin{defi}\label{def:tree}
 A \emph{Morse flow tree} of $\Lambda \subset J^1(M)$ is an immersed tree $\Gamma$ in $M$ satisfying the following.
 \begin{enumerate}
 \item The edges are mapped to flow lines of $\Lambda$.
  \item\label{t1} The vertices of $\Gamma$ have valence of at most 3.
  \item\label{t2} The tree $\Gamma$ is rooted, oriented away from the root. The orientation of the edges are given by the flow orientation. The valence of the root is equal to $1$ or $2$.
  \item\label{t4} We get an oriented closed curve $\bar{\Gamma}\subset T^*M$ when concatenating the oriented cotangent lifts of the edges of $\Gamma$. 
  \item\label{t5}The vertices of $\Gamma$ are of the form listed in Table \ref{tab:trees}.
 \end{enumerate}
\end{defi}

 \begin{table}[ht]
 \begin{center}
   \caption{{\bf Complete list of vertices occurring in rigid flow trees.} 
   }\label{tab:trees}
 \tabulinesep=1.15mm
  \begin{tabu}{>{\centering\arraybackslash}m{0.4in} | >{\arraybackslash}m{2.8in}|>{\centering\arraybackslash}m{1.8in}}
  \hline Type  & Description & Local picture of the tree and its front lift \\
   \hline
   $P_{1,+}$ &
   \emph{Positive 1-valent punctures}, not contained in  $\Pi(\Sigma)$. &
\includegraphics[height=1.6cm]{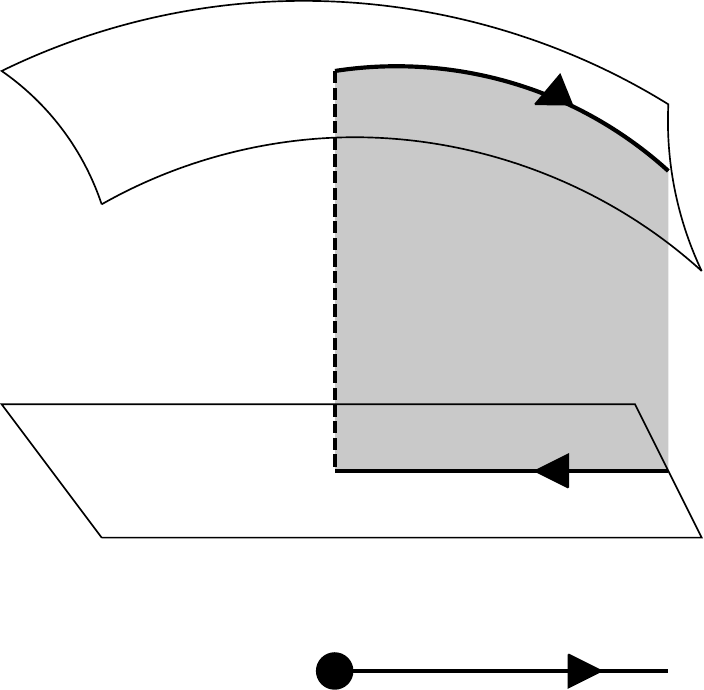}\\
   \hline
   $P_{1,-}$ &  \emph{Negative 1-valent punctures}, not contained in  $\Pi(\Sigma)$. &
\includegraphics[height=1.6cm]{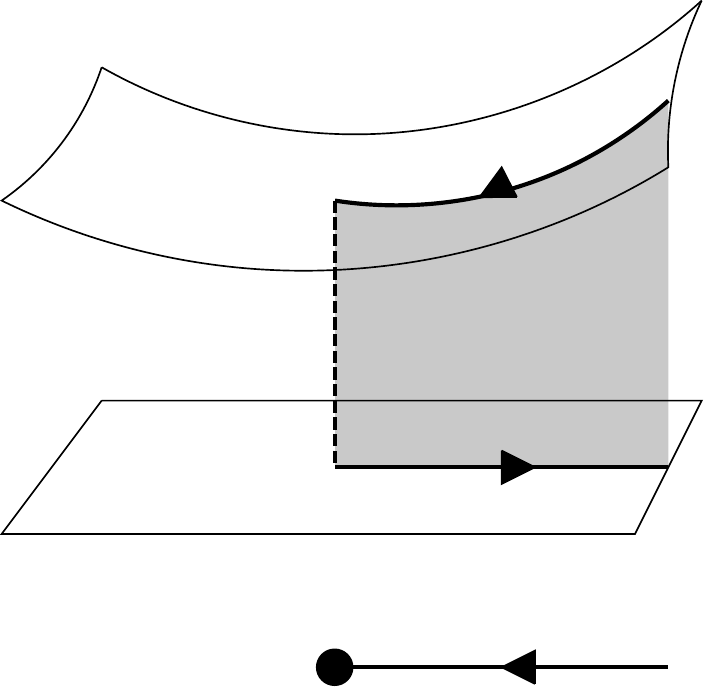}\\
   \hline
   $P_{2,+}$ & \emph{Positive 2-valent punctures $p$}, not contained in  $\Pi(\Sigma)$, with index
$I(p)=0$. &
\includegraphics[height=1.5cm]{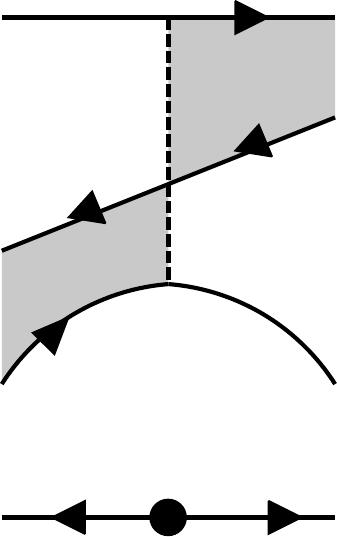}\\
   \hline
   $P_{2,-}$ & \emph{Negative 2-valent punctures $p$}, not contained in  $\Pi(\Sigma)$, with index
$I(p)=n$. & 
\includegraphics[height=1.5cm]{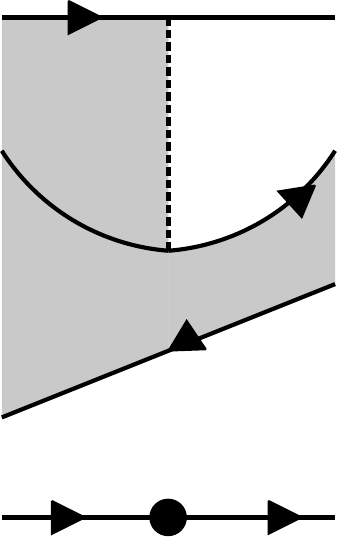}\\
   \hline
   $Y_0$ &  \emph{3-valent $Y_0$-vertices}, not contained in  $\Pi(\Sigma)$ and not containing any punctures. &
\includegraphics[height=2.1cm]{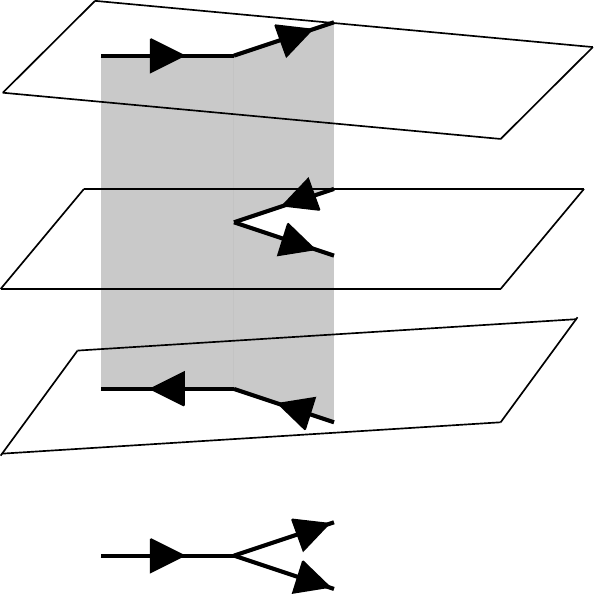}\\
   \hline
   end & \emph{1-valent end-vertices $e$},  contained in $\Pi(\Sigma)$, not containing any punctures, meeting $\Pi(\Sigma)$ transversely, the $1$-jet lift of $\Gamma$ through $e$ is traversing $\Sigma$ downwards.&
\includegraphics[height=2.2cm]{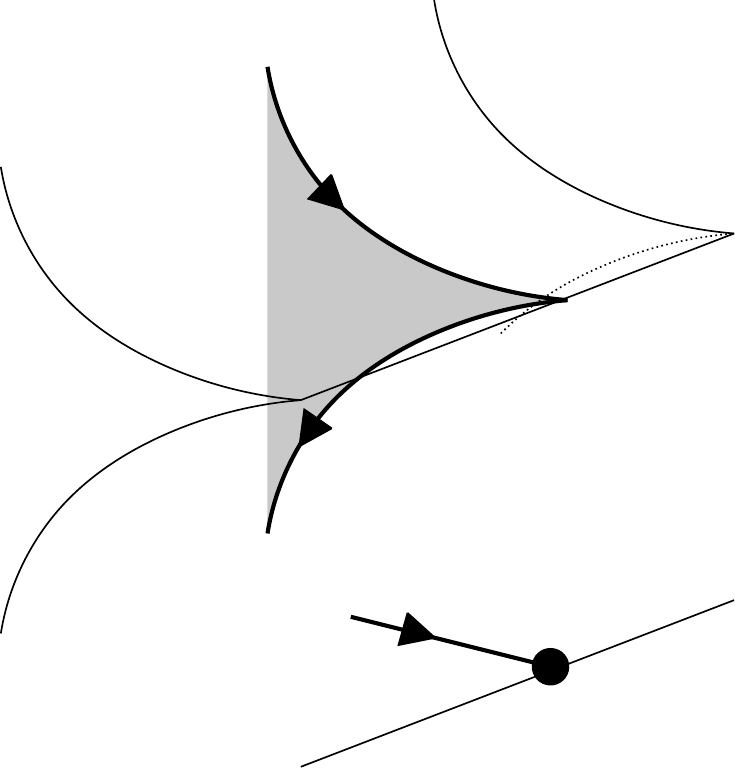}\\
   \hline
   switch & \emph{2-valent switch-vertices $s$}, contained in  $\Pi(\Sigma)$,  not containing any punctures, tangent to $\Pi(\Sigma)$,  one of the $1$-jet lifts of $\Gamma$ through $v$ is traversing $\Sigma$ upwards, the other 1-jet lift of $s$ is not contained in $\Sigma$. &
\includegraphics[height=2.7cm]{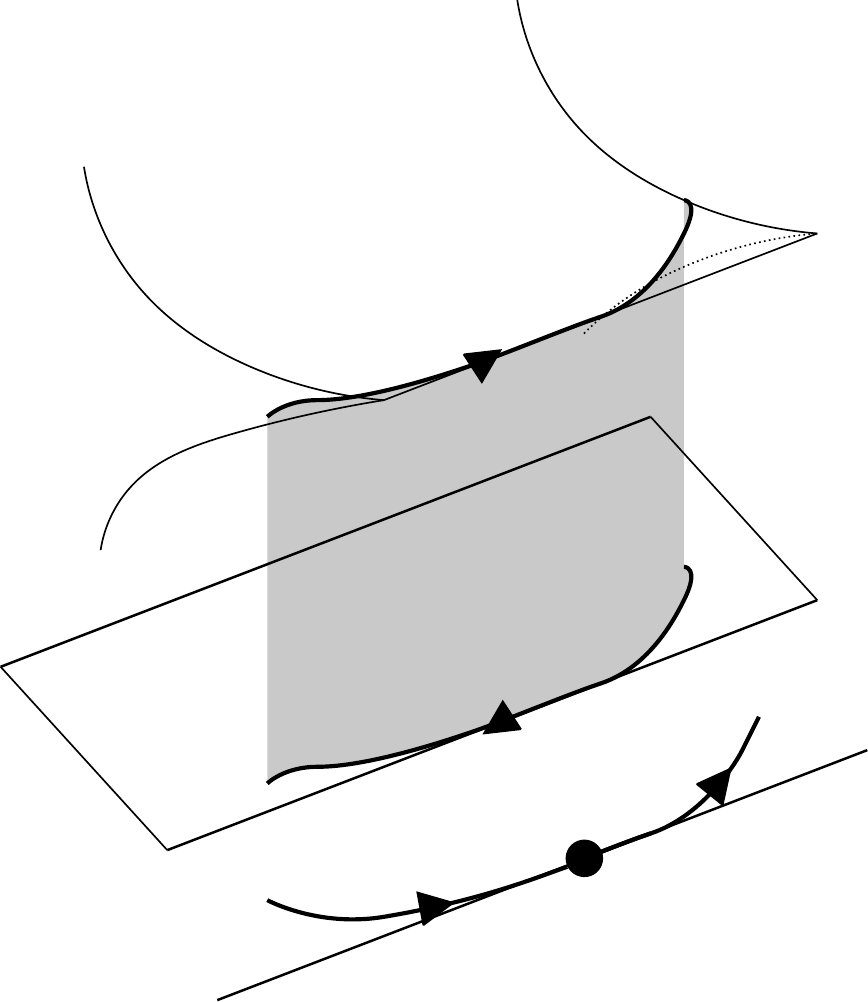}\\
   \hline
    $Y_1$ & \emph{3-valent $Y_1$-vertices  $v$},  contained in   $\Pi(\Sigma)$, not containing any punctures, meeting $\Pi(\Sigma)$ transversely, one of the $1$-jet lifts of $\Gamma$ through $v$ is traversing $\Sigma$ upwards, the other two 1-jet lifts of $v$ are not contained in $\Sigma$.&
\includegraphics[height=3.02cm]{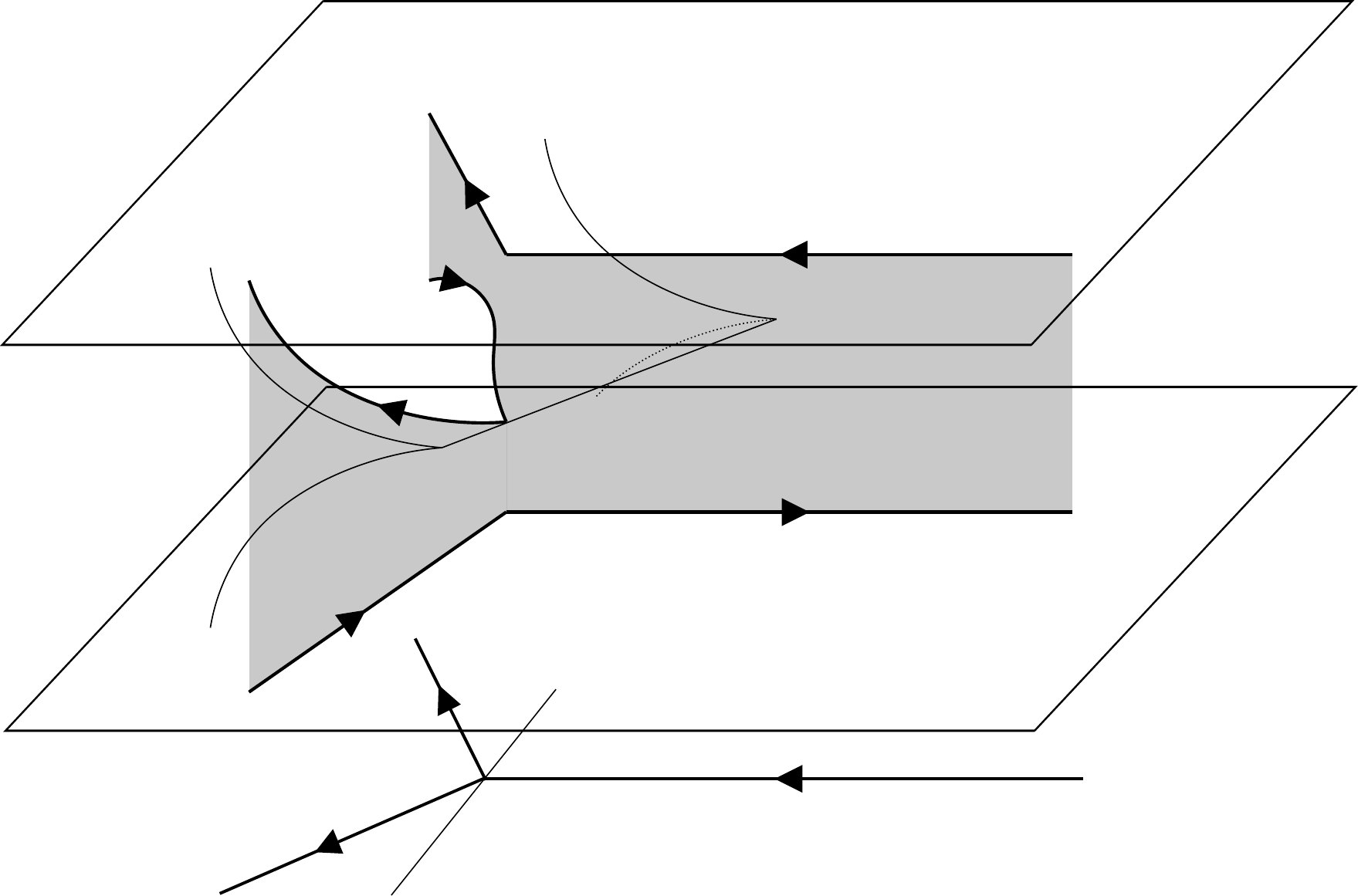}\\
   \hline
   \end{tabu}
   \end{center}
 \end{table}

\begin{defi}
 A vertex $v \in \Gamma$ is called a \emph{puncture} if the 1-jet lift
 of $\Gamma$ makes a jump at $v$. It is called a \emph{negative
   puncture} if we go from the upper sheet to the lower when following
 the lifted orientation of $\Gamma$,  and it is called a
 \emph{positive puncture} if we go from the lower sheet to the
 upper. 
\end{defi}

\begin{rmk}
We will only consider trees with exactly one positive puncture, and for such trees the positive puncture coincides with the root of the tree.
\end{rmk}

\begin{rmk}
When switching local defining functions along a flow line of $\Lambda$ as in \eqref{eq:patch} we do not introduce any vertex along the corresponding edge of $\Gamma$.
\end{rmk}

Next we fix some notations. Note that if $p$ is a puncture of a tree $\Gamma$, then it is a critical point of some local function difference  $f_1-f_2$, where we assume that $f_1>f_2$. Let $U \subset M$ be a maximal connected neighborhood of $p$ such that both $f_1$ and $f_2$ are defined on $U$. Denote by $W^u(p)$ and $W^s(p)$
the unstable and stable manifold, respectively, at $p$,  with respect to the negative gradient flow
of $f_1-f_2$. 
That is, if $\phi:U \times \R \to M$ is a solution of 
\begin{align*}
 &\frac{\partial}{\partial t}\phi(x,t)=-\nabla ((f_1-f_2)(\phi(x,t))), \quad \phi(x,0)=x,
\end{align*}
then 
\begin{align*}
 &W^u(p)=\{x\in U; \lim\limits_{t\to -\infty} \phi(x,t)=p\},  &W^s(p)=\{x\in U; \lim\limits_{t\to \infty} \phi(x,t)=p \}.
\end{align*}
 We define the \emph{index} $I(p)$ of a puncture $p$ to the be Morse index of $f_1-f_2$.
Because of our transversality assumptions we have 
\begin{align*}
 \dim(W^u(p)) = I (p), \quad \dim(W^s(p))= n-I (p).
\end{align*}

The dimension of a Morse flow tree is given as follows.
Let $e(\Gamma)$ be the number of end-vertices, $s(\Gamma)$ be the number of switches, and $Y_1(\Gamma)$ the number of $Y_1$-vertices of $\Gamma$.

\begin{defi}\label{def:dim}
If $\Gamma$ is a flow tree with positive puncture $a$ and negative punctures $b_1,\dotsc,b_m$, then the dimension of $\Gamma$ is given by 
\begin{equation*}
 \dim(\Gamma) = -2 + \dim W^u(a) + \sum_{j=1}^m (\dim W^s(b_j) - n + 1) + e(\Gamma) - s(\Gamma) - Y_1(\Gamma). 
\end{equation*}
\end{defi}

In this paper we will be interested in \emph{rigid flow trees}, i.e.\ trees of dimension 0 which are transversely cut out from the space of trees. 

In Legendrian contact homology, the \emph{grading} of a Reeb chord plays an important role. This is defined as follows. 
For each Reeb chord $c$ of $\Lambda$, let $c_+, c_- \in \Lambda \cap c$ denote the end points of $c$. We fix notation so that the $z$-coordinate of $c_+$ is larger than the $z$-coordinate of $c_-$. Pick a
 \emph{capping path} $\gamma_c$ of $c$, which is a smooth path in $\Lambda$  going from $c_+$ to
$c_-$, and we choose it so that it intersects $\Sigma$ transversely. Let  $D(\gamma_c)$  be the number of cusps that $\gamma_c$ traverses coming from the
upper sheet of the cusp  in direction to the lower sheet of the cusp, and let $U(\gamma_c)$ be the number
of cusps that $\gamma_c$ traverses coming from the lower sheet going to the upper
sheet (relative the $z$-axis).  We define the \emph{Maslov index} $\mu(\gamma_c)$ of $\gamma_c$ to be the quantity
\begin{equation*}
 \mu(\gamma_c) =  D(\gamma_c)-U(\gamma_c).
\end{equation*}
We also let $|\mu(c)| \in \{0,1\}$ denote the parity of the Maslov
index of $\gamma_c$,
and notice that it is independent of the chosen capping path. See \cite{legsub}. 

Now define the \emph{grading of $c$} to be given by 
\begin{equation}\label{eq:grading}
 |c|= \mu(\gamma_c) + I(c) -1.
\end{equation}

 \begin{defi}
  A Reeb chord $c$ is \emph{odd (even)} if $|c|$ is odd (even).
 \end{defi}

\begin{rmk}
Since we have a one-to-one correspondence between the Reeb chords of $\Lambda$ and the double points of $\Pi_\C(\Lambda)$, given by the projection, we will use these two notations interchangeably. Nor will we make any notational difference between these objects and punctures of Morse flow trees or punctures of pseudo-holomorphic disks. 
\end{rmk}

\subsection{Partial flow trees}\label{sec:parttree}
It will be useful to cut rigid flow trees into simpler parts, to obtain \emph{partial
flow trees}. These trees are defined in the same way as flow trees, except that we have
weakened condition \eqref{t4} in Definition \ref{def:tree}. Namely, we allow $1$-valent vertices where the Lagrangian lifts
of the tree do not match up. Such a vertex we call a \emph{special
puncture}. We denote the vertices that are not special punctures \emph{true vertices}.

We define the notion of \emph{positive} and \emph{negative special punctures} just
as in the case of ordinary punctures of a flow tree. That is, a
special puncture $p$ is positive if the oriented $1$-jet lift of
$\Gamma$ jumps from a lower sheet to a higher one at $p$, and
otherwise it is a negative puncture. We also define the Maslov index
of a special puncture just as it is done for ordinary punctures.

If $\Gamma$ is a true Morse flow tree, then the partial flow trees obtained by cutting
$\Gamma$ into smaller pieces are given the flow orientation from $\Gamma$. Thus, if we cut $\Gamma$ at only one point, then one of the partial flow trees will have the same positive puncture as $\Gamma$, and the other tree will have its (special) positive puncture at the cutting point.


We will  mainly be interested in a special kind of partial flow trees, which are  trees obtained by cutting a
rigid flow tree just once. These trees we denote \emph{sub flow trees}.
The dimension of a sub flow tree is computed similar to the case of Morse flow trees. Here $I(p)=n$ if $p$ is a special puncture.


Recall the fiber scaling  $s_\lambda:(x,y,z) \mapsto (x, \lambda y, \lambda z)$ as $\lambda \to 0$. In [\cite{trees}, Section 4.3] it is described how one can perturb $\Lambda$ so that there is a choice of metric of $M$ which is flat in neighborhoods of flow trees. In addition, in Section 4.4 of the same paper it is shown that this metric induces an almost complex structure such that we can identify $T^*M$ with $\C^n$ in neighborhoods of the cotangent lifts of the flow trees. Moreover, by these constructions we can find a partition of $\Gamma$ into \emph{elementary trees} and
\emph{edge point regions}. The elementary trees are characterized by the following. Along the cotangent lifts of the edges we have that $\Pi_\C(s_\lambda(\Lambda))$ tend to $\R^n \times ai \subset \C^n$ for some $a \in \R^n$ as $\lambda \to 0$, except in neighborhoods of cusp points, where $\Pi_\C(s_\lambda(\Lambda))$ is given by $\R^{n-1}$ in the directions tangent to $\Sigma$ times a bend in the remaining direction, as illustrated in Figure \ref{fig:bend}.
See Section \ref{sec:treetriv} for a more detailed explanation of this behavior. 

Moreover, we can assume
that the elementary trees have at most one true vertex and thus either
is of one of the forms drawn in Table \ref{tab:trees}, or are just
flow lines without true vertices. We classify them based on this data,
by denoting an elementary tree $\Gamma$ containing a true vertex of
type $v$ a \emph{v-piece}. 

The edge point regions are constructed to interpolate between elementary regions. They contain no true vertices and shrink to points under the degeneration process $\lambda \to 0$.

The sub flow trees together with the elementary trees will play an important role in the orientation algorithm. 

   \begin{figure}[ht]
      \labellist
\small\hair 2pt
\pinlabel $\times \R^{n-1}$ [Br] at 740 142
\endlabellist
 \vspace{-4cm}
 \hspace{-2.5cm}
\centering
 \includegraphics[width=5cm, height=7cm]{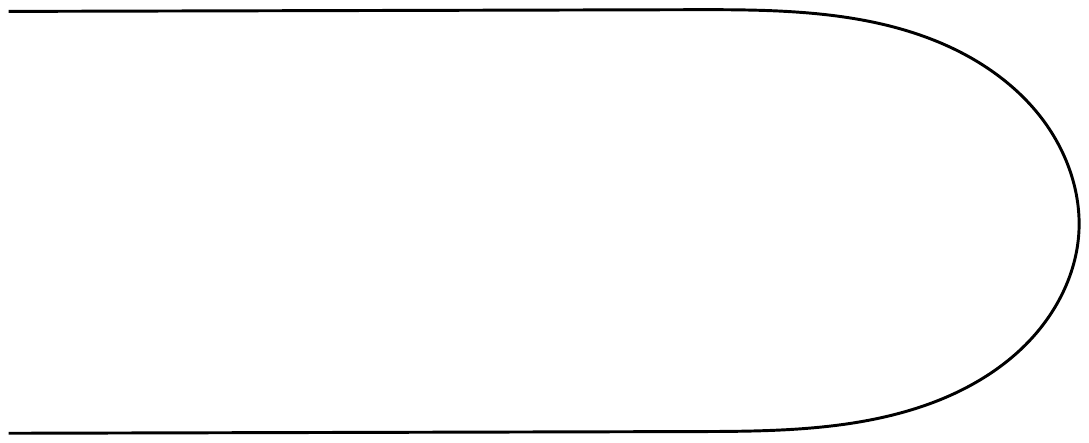} 
 \caption{Local picture in the Lagrangian projection close to an end- switch- or $Y_1$-vertex, consisting of a bended one-dimensional sheet $W$ times $\R^{n-1}$.}
 \label{fig:bend}
 \end{figure}

\subsection{Description of a neighborhood of a sub flow tree}\label{sec:geomint}
Given a sub flow tree $\Gamma$, let $\Omega=\Omega(\Gamma)$ be a
neighborhood of $\Gamma$ in the space of trees, and let $s$ be the point on the source tree mapping to the special puncture $q$ of
$\Gamma$. We define the \emph{natural evaluation map} $\ev_s:\Omega \to M$ to be given 
 by evaluation at $s$. In this section we describe the image of this map.

\begin{rmk}
To simplify notation we will often write  $\ev_q:\Omega \to M$ instead of $\ev_s$ when it is clear from the context that $q$ is the image of $s$ under evaluation.
\end{rmk}

\begin{defi}
Let $K \subset M$ be a subset and $f_1-f_2$ a local function difference associated to $\Lambda$. Then the \emph{flow-out of K in the direction of $\pm\nabla(f_1-f_2)$} is defined as a subset $\F^\pm(K) \subset K \times [0,\infty)$ together with a map $i^\pm: \F^\pm(K) \to M$ where 
\begin{align*}
\F^\pm(K)=\{(p,t) \in K \times [0,\infty) & | \text{ there is a maximal flow line } \gamma_p:I \to M   \text{ of } \Lambda\\ &  
\text{ such that } I\cap[0,t]=[0,t], \gamma_p(0)=p, \\
& \text{ and there is a } B>0 \text{ such that } \dot \gamma_p|_{[0,B]} =\pm \nabla(f_1-f_2)\}
\end{align*}
and where 
\begin{equation}\label{eq:inclusion}
i^\pm: \F^\pm(K) \to M, \qquad (p,t) \mapsto \gamma_p(t).
\end{equation}
\end{defi}

\subsubsection{Flow outs of some particular sub flow trees}\label{sec:flowouts}
Now assume that $\Gamma' \subset \Gamma$ is a sub flow tree of a rigid tree $\Gamma$. Assume that the special puncture of $\Gamma'$ is given by $q$, and let 
 $e$ be the edge of $\Gamma'$ ending at $q$ and let $p$ denote the other vertex of $e$. Let $f_1-f_2$ be the local function differece of $\Lambda$ defining $e$ in a neighborhood of $p$. We  will define something called the \emph{flow-out of $\Gamma'$ at $q$}, denoted by $\F_q(\Gamma')$.  In the case when $p$ is a $2$- or $3$-valent vertex, then this is the flow-out of the corresponding \emph{intersection manifold $\I_p(\Gamma')$} along $e$. All this is defined inductively  below, for each type of vertex $p$.  In all cases we let $D_p^k$ be a $k$-dimensional open disk centered at $p$ and contained in the connected subset of $M$ where $f_1-f_2$ is defined. We also let $S_p \subset M$ be a sphere centered at $p$ with radius so small that it intersects $e$ in precisely one point $s$, located between $p$ and $q$.  
 
 First we assume that $e$ is a local flow line of $-\nabla{(f_1-f_2)}$ and define the flow-outs and intersection manifolds as follows.

 \begin{description}
 \item[ \it $p$  positive $1$-valent puncture] \hfill \\
  Let $\I_p(\Gamma')$ be the connected component of $S_p \cap W^u(p)$ that intersects $e$, and let $\F_q(\Gamma')$ be the flow-out of $\I_p$ in the direction of $-\nabla{(f_1-f_2)}$. 
\item [\it $p$  negative $1$-valent puncture]\hfill \\
 Let $\I_p(\Gamma')$ be the connected component of $S_p \cap W^s(p)$ that intersects $e$, and let $\F_q(\Gamma')$ be the flow-out of $\I_p$ in the direction of $\nabla{(f_1-f_2)}$. 
 \item[\it $p$ an end]\hfill \\
  Choose $\I_p(\Gamma')=D_p^{n-1} \subset \Pi(\Sigma)$ and let $\F_q(\Gamma')$ be the flow-out of $\I_p$ in the direction of $\nabla{(f_1-f_2)}$. 
  \item[\it $p$ a switch]\hfill \\
Let $r$ be a point on the edge of $\Gamma'$ containing $p$ and not containing $q$, let $p_0\neq p$ be a true vertex of $\Gamma'$ on that edge and let $\Gamma'_1$ be the sub flow tree of $\Gamma$ with $r$ as a special positive (negative) puncture if $q$ is positive (negative). Assume, by induction, that $r$ is chosen so that $\F_r(\Gamma_1')$ is defined and let $t_0>0$ so that $(p_0,t_0)\in \F_r(\Gamma'_1)$ satisfies $i(p_0,t_0) = p$, where $i:\F_r(\Gamma'_1) \to M$ is the map in \eqref{eq:inclusion}. Let $D(p_0,t_0)\subset \F_r(\Gamma'_1)$ be a neighborhood of $(p_0,t_0)$ and let  $D_p^{n-1} \subset \Pi(\Sigma)$ so that $\I_p(\Gamma')=i(D(p_0,t_0)) \cap D_p^{n-1}$ is a disk. We define $\F_q(\Gamma')$ to be the flow-out of $\I_p(\Gamma')$ in the direction of $\nabla{(f_1-f_2)}$ ($-\nabla{(f_1-f_2)}$). 
  \item[\it$p$  $2$-valent negative puncture, $q$ is positive] \hfill \\
  Define $\F_q(\Gamma')$ to be the flow-out of $p$ in the direction of $\nabla{(f_1-f_2)}$. The intersection manifold $\I_p(\Gamma')$ equals $p$. 
 \item[\it$p$ is a $Y_0$-vertex, $q$ is positive]\hfill \\
Let $e_1,e_2$ be the other edges of $\Gamma$ containing $p$, let $r_1$ and $r_2$ be points on $e_1$ and $e_2$, respectively, and let $\Gamma'_j$, $j=1,2$, be the sub flow tree of $\Gamma$ with $r_j$ as special positive puncture. Assume, by induction, that  $r_j$ is chosen so that $\F_{r_j}(\Gamma_j')$ is defined for $j=1,2$, and let $p_j \neq r_j$ be the true vertex on $e_j\cap \Gamma_j'$. Let $t_1,t_2>0$ so that $(p_j,t_j)\in \F_{r_j}(\Gamma'_j)$ satisfies $i_j(p_j,t_j) = p$, $j=1,2$, where $i_j:\F_{r_j}(\Gamma'_j) \to M$ is the map from \eqref{eq:inclusion}. Let $D(p_j,t_j)\subset \F_{r_j}(\Gamma'_j)$ be a neighborhood of $(p_j,t_j)$, $j=1,2$, and define $\I_p(\Gamma')=i_1(D(p_1,t_1)) \cap i_2(D(p_2,t_2))$.
Let $D_p^{\dim \I_p(\Gamma')} \subset  \I_p(\Gamma')$.  We define the flow-out $\F_q(\Gamma')$ to be given by the flow-out of $D_p^{\dim \I_p(\Gamma')}$  in the direction of $\nabla{(f_1-f_2)}$.      
 \item[\it $p$ is a $Y_1$-vertex, $q$ is positive]\hfill \\
 Choose $D_p^{n-1} \subset \Pi(\Sigma)$ and define $\I_p(\Gamma')=i_1(D(p_1,t_1)) \cap i_2(D(p_2,t_2))\cap D_p^{n-1}$, where $i_1(D(p_1,t_1)),i_2(D(p_2,t_2))$ are defined analogous as in the case of $Y_0$-vertices. We define the flow-out $\F_q(\Gamma')$ to be given by the flow-out of $\I_p(\Gamma')$ in the direction of $\nabla{(f_1-f_2)}$. 
\end{description}

\noindent
Next we consider the general case when the edge $e$ above is allowed to consist of a general flow line of $\Lambda$. We assume that $q$ is not a self intersection point of $e$, and let $r$ be a point on $e$ so that $\F_r(\Gamma_1')$ is defined using the description above, where $\Gamma_1'$ is the sub flow tree of $\Gamma'$ having $r$ as special puncture. Let $t_0>0$ so that $(p,t_0)\in \F_r(\Gamma_1')$ satisfies $i(p,t_0)=q$, where $i : \F_r(\Gamma_1') \to M$ is the map in \eqref{eq:inclusion}, and let $D(p,t_0) \subset \F_r(\Gamma_1')$ be a neighborhood of $(p,t_0)$.

\begin{defi}\label{def:tangents}
We let $\F_q(\Gamma') = \F_r(\Gamma_1')$, $i_q(\F_q(\Gamma'))=i(D(p,t_0))$ and we define $T_q \F_q(\Gamma')$ to be the  tangent space of $i_q(\F_q(\Gamma'))$ at $q$.
\end{defi}

\begin{rmk}
 It might seem that the definitions of flow outs and intersection manifolds for 2--valent vertices made above are incomplete, but they are in fact sufficient for our algorithm. 
 \end{rmk}

\begin{rmk}\label{rmk:stablemfd}
If $p$ and $q$ belong to the same local flow line it follows that $T_q\F_q(\Gamma') \simeq T_q W^u(p)$ in the case $p$ is a positive 1-valent puncture and $T_q\F_q(\Gamma') \simeq T_q W^s(p)$ in the case $p$ is a negative 1-valent puncture .
\end{rmk}

\begin{figure}[ht]
\labellist
\small\hair 2pt
\pinlabel $\F_q(\Gamma)$ [Br] at 386 -30
\pinlabel $p$ [Br] at  653 273 
\pinlabel $\F_q(\Gamma)\cap D_q$ [Br] at 1398 -10
\pinlabel $p$ [Br] at  1586 273 
\endlabellist
\centering
\includegraphics[height=2.5cm]{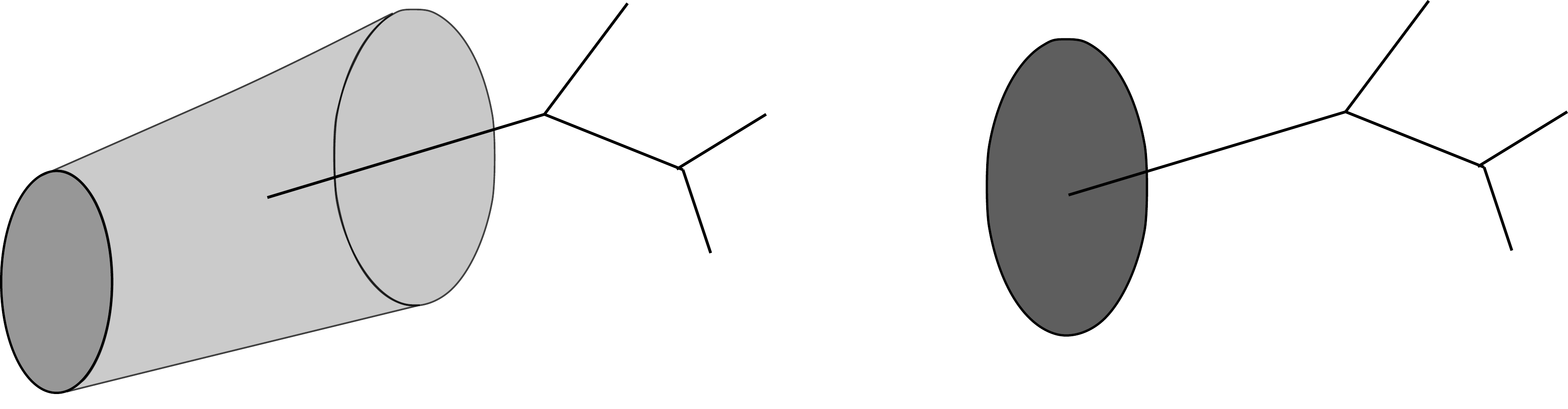}
\caption{The flow-out along $\Gamma$.}
\label{fig:ws}
\end{figure}


\begin{lma}\label{lma:nbhd}
Let $\Gamma$ be a sub flow tree with special puncture $q$. Then there
is a neighborhood 
$\Omega$ of $\Gamma$ in the space of flow
 trees of $\Lambda$, and an $\epsilon>0$ so that  
 \begin{equation*}
  \ev_q(\Omega)=(i_q(\F_q(\Gamma))\cap D_q)\times(-\epsilon,\epsilon).
 \end{equation*}
 Here
$(-\epsilon,\epsilon)$ represents the degree of freedom of $q$ to move along the flow
line itself.
\end{lma}

\begin{proof}
Follows from the proof of [\cite{trees}, Proposition 3.14].
\end{proof}

\begin{rmk}
When calculating the sign of a rigid flow tree $\Gamma$, we will first orient the sub flow trees $\Gamma'$ of $\Gamma$. This, in turn, is done by orientating the tangent space of the flow out of the tree $\Gamma'$ at its special puncture $q$.
\end{rmk}

\subsection{Standard domains}\label{sec:slit}
Recall the correspondence between rigid flow trees and  rigid
$J$-holomorphic disks with boundary on $\Lambda$, established in \cite{trees}. The domains of the disks are given by standard domains, which are defined as follows.
\begin{defi}
A \emph{standard domain} $\Delta_{m+1}(\bar{\tau})$, $\bar \tau = (\tau_1,\dotsc,\tau_{m-1})$, is a subset of $\R \times [0,m]\subset \R^2$ obtained by removing $m-1$ horizontal slits starting at $(\tau_j,j)$, $j=1,\dotsc,m-1$, and ending at $\infty$. All slits look the same, they are strips ending in a half-circle of width $\epsilon$, $0< \epsilon <<1$. A point $(\tau_j,j)$ where a slit ends will be called a \emph{boundary minimum}.  See Figure \ref{fig:1}. 
\end{defi}

\begin{figure}[ht]
\labellist
\small\hair 2pt
\pinlabel $p_0$ [Br] at 45 112
\pinlabel $p_1$ [Br] at 752 8
\pinlabel $p_2$ [Br] at 752 57
\pinlabel $p_3$ [Br] at 752 106
\pinlabel $p_4$ [Br] at 752 156
\pinlabel $p_5$ [Br] at 752 203
\pinlabel $(\tau_1,1)$ [Br] at  462 29 
\pinlabel $(\tau_2,2)$ [Br] at  367 78 
\pinlabel $(\tau_3,3)$ [Br] at  508 128 
\pinlabel $(\tau_4,4)$ [Br] at  430 175 
\endlabellist
\centering
\includegraphics[height=3.5cm]{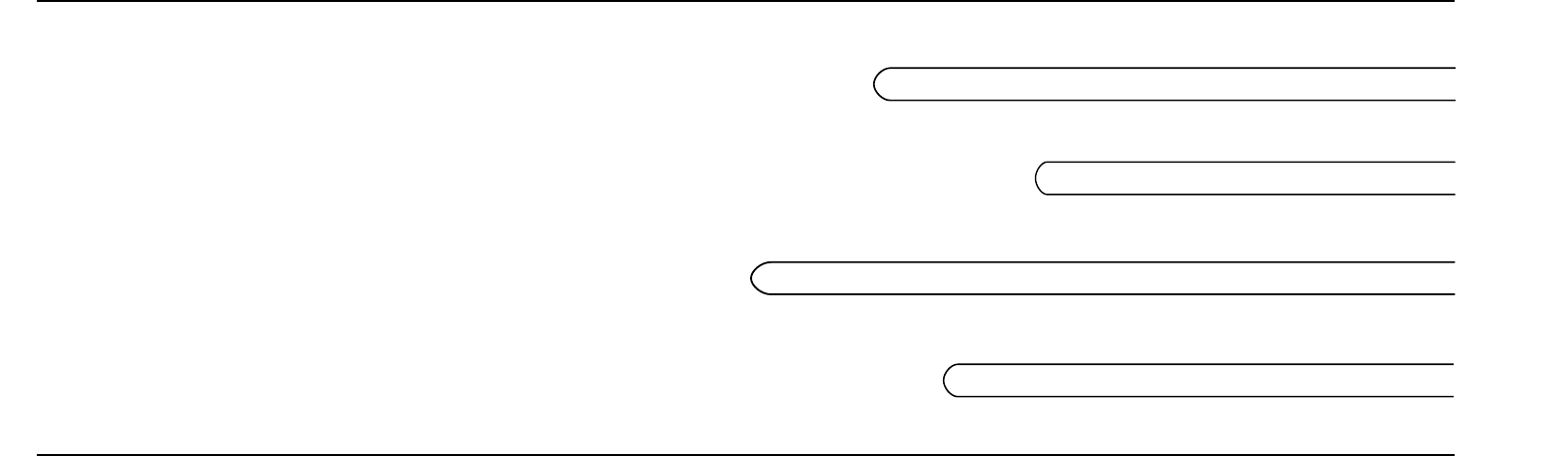}
\caption{A standard domain $\Delta_{m+1}$, $m = 5$.}
\label{fig:1}
\end{figure}

Given a (partial) flow tree $\Gamma$, we can associate a standard domain $\Delta(\Gamma)$ to $\Gamma$ by thinking of the boundary of the standard domain as being mapped to the $1$-jet lift of $\Gamma$. 
More precisely, start at the positive puncture of $\Gamma$. This represents the point $p_0$ at $-\infty$ of the corresponding domain. Go along the $1$-jet lift of $\Gamma$, in the direction of the lifted flow-orientation. This path corresponds to the bottom line of $\Delta(\Gamma)$, connecting $p_0$ to $p_1$, following the notation in Figure \ref{fig:1}. At the end of this flow line we find either a negative puncture or an end-vertex, which we represent by the point $p_1$ at the positive infinity. Jump to the line of $\Delta(\Gamma)$ over $p_1$, which represents the outgoing lifted flow line at $p_1$. Follow this path. 
To continue this description, it helps to think of the standard domain
as patched together by the (compactified) standard domains
corresponding to the elementary trees that the rigid tree is built out
of. These standard domains are given as follows, were we recall that
the elementary trees are given by the partial flow trees listed in
Table \ref{tab:trees} together with partial flow trees without true
punctures.

\begin{description}
 \item[\it Positive 1-valent puncture] A standard domain $\Delta_2$ with the point at $-\infty$ representing the positive puncture.
  \item[\it Negative 1-valent puncture] A standard domain $\Delta_2$ with the point at $+\infty$ representing the negative puncture.
  \item[\it Positive 2-valent puncture] A standard domain $\Delta_3$ with the point at $-\infty$ representing the positive puncture.
  \item[\it Negative 2-valent puncture] A standard domain $\Delta_3$ with one of the points at $+\infty$ representing the negative puncture. If the negative 2-valent puncture corresponds to $p_1$, then we say that the puncture is of \emph{Type 1}, and if it corresponds to $p_2$ we say it is of \emph{Type 2}. See Table \ref{tab:two}.
  \item[\it$Y_0$-vertex] A standard domain $\Delta_3$ with the boundary minimum representing the $Y_0$-vertex. 
    \item[\it End-vertex] A standard domain $\Delta_2$ with the point at $+\infty$ representing the end-vertex.
    \item[\it Switch] A standard domain $\Delta_2$ with a point on $\partial \Delta_2$ representing the point where the $1$-jet lift of the switch passes through $\Sigma$.
  \item[\it $Y_1$-vertex] A standard domain $\Delta_3$ with the
    boundary minimum representing the $Y_1$-vertex.
      \item[\it No true puncture] A standard domain $\Delta_2$.
\end{description}
Now, if we continue to follow the flow orientation of the $1$-jet lift of $\Gamma$, we get a recipe for how to glue these elementary standard domains together to obtain $\Delta(\Gamma)$.

   \begin{table}[ht] \caption{{\bf Different types of negative 2-valent punctures.} }\label{tab:two}
 \begin{center}
 \tabulinesep=1.15mm
  \begin{tabu}{>{\centering\arraybackslash}m{1.2in} | >{\centering\arraybackslash}m{2.2in}|>{\centering\arraybackslash}m{2.2in}}
   \hline
    Type &
  Front projection and tree &   Lagrangian projection and standard domain\\
 \hline
 negative, type 1 &
     \labellist
\small\hair 2pt
\pinlabel $1$ [Br] at 208 245
\pinlabel $2$ [Br] at  208 190 
\pinlabel $3$ [Br] at 208 130
\pinlabel $p_1$ [Br] at 125 40
\pinlabel $a$ [Br] at 15 8
\pinlabel $p_2$ [Br] at 225 8
\endlabellist
\centering
\includegraphics[height=3cm]{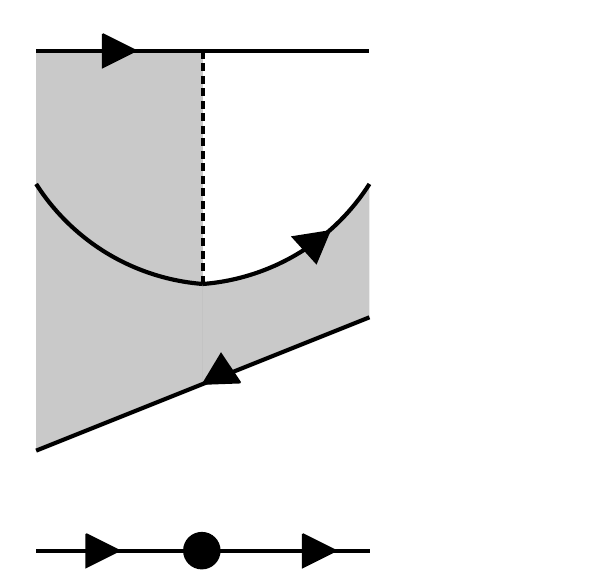} 
&
      \labellist
\small\hair 2pt
\pinlabel $1$ [Br] at 208 160
\pinlabel $2$ [Br] at  208 195 
\pinlabel $3$ [Br] at 208 235
\pinlabel $p_2$ [Br] at 218 53
\pinlabel $a$ [Br] at 30 30
\pinlabel $p_1$ [Br] at 218 10
\endlabellist
\centering
\includegraphics[height=3cm]{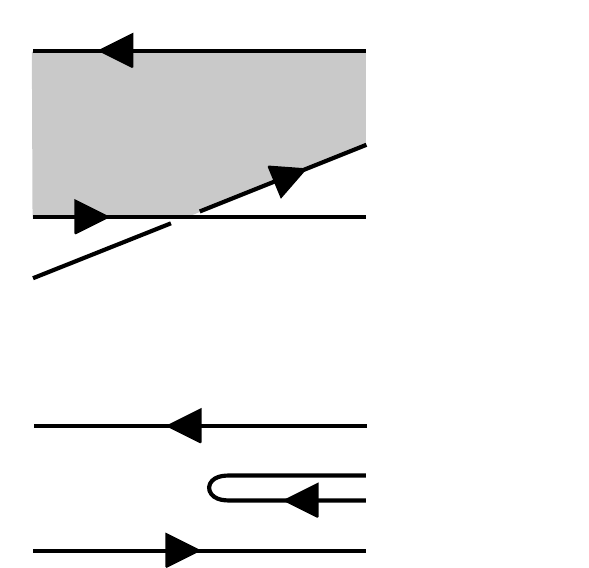} \\
\hline
 negative, type 2 & 
      \labellist
\small\hair 2pt
\pinlabel $1$ [Br] at 208 190
\pinlabel $2$ [Br] at  208 120 
\pinlabel $3$ [Br] at 208 70
\pinlabel $p_1$ [Br] at 225 8
\pinlabel $a$ [Br] at 15 8
\pinlabel $p_2$ [Br] at 125 40
\endlabellist
\centering
\includegraphics[height=3cm]{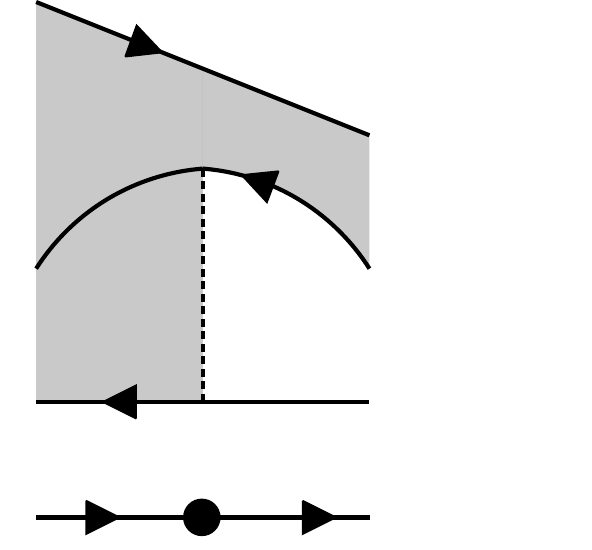} &
      \labellist
\small\hair 2pt
\pinlabel $1$ [Br] at 208 130
\pinlabel $2$ [Br] at  208 175 
\pinlabel $3$ [Br] at 208 210
\pinlabel $p_2$ [Br] at 218 53
\pinlabel $a$ [Br] at 30 30
\pinlabel $p_1$ [Br] at 218 10
\endlabellist
\centering
\includegraphics[height=3cm]{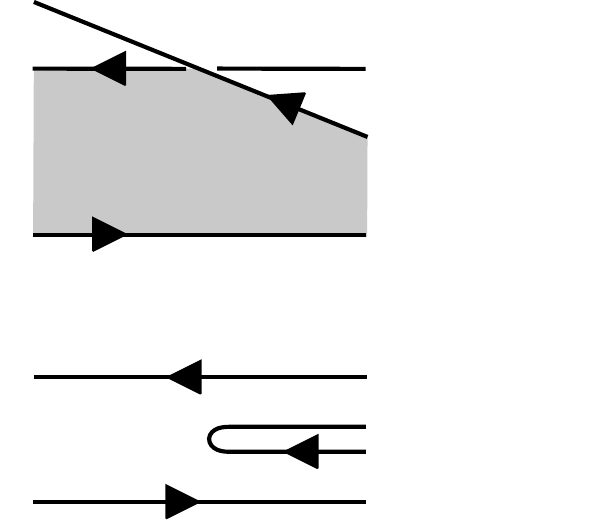} \\
   \hline
   \end{tabu}
   \end{center}

 \end{table}

\begin{rmk}
 Given a flow tree $\Gamma$, we can not read off how the standard domain $\Delta(\Gamma)$ looks like from the immersion of $\Gamma$ in $M$, we also need  data from the $1$-jet lift of $\Gamma$.
\end{rmk}

Besides working with standard domains, we will sometimes consider the domain of a pseudo-holomorphic disk to be given by the punctured unit disk in $\C$. To that end,
let $D_{m+1}$ be the unit disk in $\C$ with $m+1$ marked points on the boundary, where one of them is distinguished and fixed at $1$, say. Assume that $m>1$, and let  $\Co_{m+1}$ denote the space of conformal structures on $D_{m+1}$. A standard domain $\Delta_{m+1}(\bar{\tau})$ induces a conformal structure $\kappa(\bar{\tau})$ on $D_{m+1}$ via the Riemann mapping theorem. Here $\Delta_{m+1}(\bar{\tau})$ is given the flat metric. Moreover, since translations are biholomorphisms we get that $\kappa(\bar{\tau}) = \kappa(t+ \bar \tau)$ for all $t \in \R$, where $t+\bar\tau= (t+\tau_1,\dotsc,t+\tau_{m-1})$. The action $t + \bar\tau$ can be seen as $\R$ acting on $\R^{m-1}$, with  orbit space given by $\R^{m-2}$, and hence the mapping $\bar \tau  \mapsto \kappa(\bar \tau)$ gives an identification of $\Co_{m+1}$ with $\R^{m-2}$ via standard domains. In [\cite{trees}, Lemma 2.2] it is proven that this identification is a diffeomorphism. 
 We give the space $\Co_{m+1}$ the norm induced from the standard norm on $\R^{m-2}$ via this diffeomorphism.
 
 In Section \ref{sec:confvaror} we discuss the orientation of the
 space $\Co_{m+1}$, and the capping orientation of $\Gamma$ will depend on the orientation of this space.


%
%
 \section{Orientation of pseudo-holomorphic disks}\label{sec:background}
Here we review some background materials concerning orientations of pseudo-\linebreak
holomorphic disks. We follow  \cite{orientbok}, except that we will choose slightly different orientation conventions.

In Subsection \ref{sec:jholo} we give a definition of pseudo-holomorphic disks with boundary on $\Lambda$, and also discuss Sobolev spaces associated to the linearization of the $\dbar$-operator at such a disk. In Subsection \ref{sec:orientconv} we fix some orientation conventions. This is related to Subsection \ref{sec:gluew}, where we describe some exact sequences of finite-dimensional vector spaces, which occurs when we glue pseudo-holomorphic disks together at punctures. We then have the materials needed to define the capping orientation of a pseudo-holomorphic disk, which is done in Subsection \ref{sec:nonpunct}.

\subsection{\texorpdfstring{$J$}{J}-holomorphic disks}\label{sec:jholo}
Here we define punctured $J$-holomorphic disks with \linebreak  boundary on $\Lambda$. Then we discuss the Lagrangian boundary conditions that these disks induce, and define the linearization of the $\dbar$-operator  at such a disk. In particular, we specify certain weighted Sobolev spaces that will make this operator Fredholm. We start by describing the almost complex structure $J$.

\subsubsection{Complex structure}\label{sec:complex} 
In \cite{trees} it is proven that after a generic perturbation of $\Lambda$ and a  $C^1$-change of metric on $M$, there is a neighborhood of the rigid trees of $\Lambda$ where the metric is flat. Moreover, the almost complex structure on $T^*M$ induced by this metric can in local coordinates $U \subset M$ be identified with $i$ in $T^*U$.

Hereafter, we will let $J$ denote such an almost complex structure.

\subsubsection{Punctured \texorpdfstring{$J$}{J}-holomorphic disks and Lagrangian boundary conditions}\label{sec:jholodisks}
Let $\Lambda \subset J^1(M)$ be a Legendrian submanifold, and let $J$ be an almost complex structure on $T^*M$ as described in Section \ref{sec:complex}. Let $D_{m+1}$ be the unit disk in $\C$ with $m+1$ distinct punctures
 $q_0,q_1,\dotsc,q_m$ on $\partial D_{m+1}$, where $q_0=1$  and where the punctures are ordered
counterclockwise. Let $I_j$, $j=0,\dotsc,m$, be the part of $\partial D_{m+1}$ going
from $p_j$ to $p_{j+1}$ $(m+1=0)$. 
We think of the disk as having half-infinite
strips 
\begin{equation*}
 E_{p_j}=\{(\tau,t)\in [0,\infty)\times[0,1]\}, \quad (\tau,t)=\tau+it,
\end{equation*}
attached at each
puncture $p_j$, and assume that the coordinates are chosen so that 
$\tau+0i$ belongs to $I_{j-1}$ and $\tau+i$ belongs to $I_j$.

Let $a, b_1,\dotsc,b_m$ be Reeb chords of $\Lambda$, or equivalently, double points of $\Pi_\C(\Lambda)$. Let $c^\pm$ denote the end points of $c$, $c=a,b_1,\dotsc,b_m$, where $c$ is regarded as a Reeb chord. Here $c^+$ represents the point with the largest $z$-coordinate.

\begin{defi}
 A  \emph{$J$-holomorphic
disk of $\Lambda$} with positive puncture $a$ and negative
punctures $b_1,\dotsc,b_m$ is given by a map $u :D_{m+1}\to T^*M$ such that
\begin{itemize}
 \item $\dbar_J u := du + J \circ du \circ i =0$,
  \item $u|_{\partial D_{m+1}}$ admits a continuous lift
$\tilde{u}:\partial D_{m+1} \to \Lambda \subset J^1(M)$,
\item $\lim\limits_{\zeta \to p_0} u(\zeta)=a$, $\lim\limits_{\tau \to
\infty} \tilde{u}(\tau+0)=a^-$,  $\lim\limits_{\tau \to
\infty} \tilde{u}(\tau+i)=a^+$ in $E_{p_0}$,
\item $\lim\limits_{\zeta \to p_j} u(\zeta)=b_j$, $\lim\limits_{\tau \to
\infty} \tilde{u}(\tau+0)=b_j^+$,  $\lim\limits_{\tau \to
\infty} \tilde{u}(\tau+i)=b_j^-$ in $E_{p_j}$.
\end{itemize}
\end{defi}

Any  $J$-holomorphic disk (and Morse flow tree) of $\Lambda$ will induce a Lagrangian boundary condition for the corresponding linearized $\dbar$-operator. We fix some notation.

\begin{defi}
 A \emph{Lagrangian boundary condition} on $\partial D_{m+1}$ is a family of
maps 
\begin{equation*}
 \Theta=(\Theta_0, \Theta_1,\dotsc,\Theta_m), \qquad \Theta_j:I_j\to \Lag(n)
\end{equation*}
where $\Lag(n)$ is the space of Lagrangian subspaces of $\C^n$. 
\end{defi}

\begin{defi}
A \emph{trivialized Lagrangian boundary condition} for the $\dbar$-operator on $D_{m+1}$
is a collection of maps 
\begin{equation*}
 A=\{A_0,\dotsc,A_m\}, \qquad  A_j:I_j\to U(n).
\end{equation*}
\end{defi}

 \begin{rmk}
 A trivialized Lagrangian boundary condition $A$ induces a
 Lagrangian boundary condition $\Theta$ by
 $\Theta_i=A_i\R^n\subset\C^n$.  
 Also, given a Lagrangian boundary condition on the punctured disk we can always find
 a trivialization for it. This need not be the case for the non-punctured disk. In addition, given a \emph{trivialization} $(v_1(s), v_2(s), \cdots , v_n(s))\in \C^n$, $s \in \partial D_{m+1}$ of a Lagrangian boundary condition we will identify it with the trivialized Lagrangian boundary condition $A: \partial D_{m+1} \to U(n)$ given by the  matrices with column vectors  $(v_1(s), v_2(s), \cdots , v_n(s)), s \in \partial D_{m+1}$.
 \end{rmk}
 
If $u$ is a $J$-holomorphic disk with boundary on $\Lambda$, then the
tangent planes of the Lagrangian projection of $\Lambda$ along the
boundary of $u$ give a Lagrangian boundary condition defined on
$\partial D_{m+1}$. This induces a trivialized boundary condition $A_u$, after
some choices. We will use the spin structure of $\Lambda$ to perform
these choices in a controlled way. Similarly, the Lagrangian
projection of the 1-jet lift of a Morse flow tree $\Gamma$ gives also
a Lagrangian boundary condition, which together with the spin
structure induces a trivialized boundary condition on
$\partial\Delta(\Gamma)$. See Sections \ref{sec:triv} and \ref{sec:treetriv}, respectively. 


\begin{rmk}
Since we will use the fiber scaling  $s_\lambda:(x,y,z) \mapsto (x, \lambda y, \lambda z)$ which pushes $\Lambda$ to the zero section in $J^1(M)$, and in this way find sequences of disks converging to trees, we will have a $\lambda$-dependence on $\Lambda$. This we often suppress from the notation, but sometimes we indicate it by writing $ \Lambda_\lambda=s_\lambda(\Lambda)$. Thus, one should keep in mind that when we say that $u$ is a $J$-holomorphic disk of $\Lambda$, we mean that $u$ is a $J$-holomorphic disk of $\Lambda_\lambda$ for some $\lambda>0$.
\end{rmk}

\subsubsection{Weighted Sobolev spaces and Fredholm operators}\label{sec:sobolev}
Let $u: D_{m+1} \to T^*M$ be a $J$-holomorphic disk with boundary on $\Lambda$. We will associate weighted Sobolev spaces $\Hi_{k,\mu}[u]$ to $u$, $k=1,2$, and also a $\dbar$-operator $\dbar_u$ defined between these spaces. This $\dbar$-operator should be thought of as the linearization of $\dbar$ at $u$.

Let $\Hi_k(D_{m+1}, \C^n)$ be the closure of $C^\infty_0(D_{m+1},\C^n)$, with norm $\|f\|_k$  given by the standard Sobolev $\Hi_k$-norm. We will weight these spaces, as follows. 

For each puncture $q_i \in D_{m+1}$ we define a \emph{weight vector} 
\begin{equation*}
 \nu_i = (\nu_i^1,\dotsc,\nu_i^n) \in (-\pi/2, \pi/2)^n
\end{equation*}
and let $\nu= (\nu_0,\dotsc,\nu_m)$.

\begin{defi}\label{def:weight}
  A \emph{weight function $\w_{\nu}$ associated to the weight $\nu$} is a
  continuous function $\w_{\nu}: D_{m+1} \to GL(n)$ satisfying
  \begin{align*}
    \w_{\nu}(\tau,t) =
    \diag(e^{\nu_i^1 |\tau|},\dotsc,e^{\nu_i^n |\tau|})
  \end{align*}
  in $E_{p_i}(M)$, and  that $\w_{\nu}$ is close to the identity
  matrix in compact regions of the disk.
\end{defi}

Let the \emph{weighted Sobolev space $\Hi_{k,\nu}(D_{m+1},\C^n)$} be defined by 
\begin{equation*}
 \Hi_{k,\nu}(D_{m+1},\C^n) = \{f \in \Hi_k^{loc}(D_{m+1},\C^n); \w_{\nu}f \in \Hi_k(D_{m+1},\C^n)\}
\end{equation*}
with norm 
\begin{equation*}
\|f\|_{k,\nu} = \|\w_{\nu}f\|_{k}. 
\end{equation*}

Now assume that $u$ is a $J$-holomorphic disk of $\Lambda$ and let 
\begin{equation*}
\Hi_{2,\nu}[u]= \Hi_{2,\nu}[u](D_{m+1}, u^*TT^*M) 
 \end{equation*}
 denote the closed subspace of  $\Hi_{2,\nu}(D_{m+1}, u^*TT^*M)$, consisting of elements $s$ that are tangent to $\Pi_\C({\Lambda})$ along $\partial D_{m+1}$, and which satisfies 
 $\dbar s |_{\partial D_{m+1}} = 0$. 
Similarly, let 
\begin{equation*}
\Hi_{1,\nu}[0]   = \Hi_{1,\nu}[0](D_{m+1}, T^{*0,1}D_{m+1} \otimes  u^*TT^*M)                                                                         \end{equation*}
 be the closed subspace of  $\Hi_{1,\nu}(D_{m+1}, T^{*0,1}D_{m+1} \otimes  u^*TT^*M)$ consisting of elements $s$ satisfying $ s |_{\partial D_{m+1}} = 0$.

We define the \emph{linearized weighted $\dbar$-operator at $u$}  to be the $\dbar$-operator
\begin{align*}
 \dbar_{u,\nu} :\Hi_{2,\nu}[u] &\to \Hi_{1,\nu}[0], \\
 \dbar_{u,\nu}(v) &= dv + J \circ dv \circ i.  
\end{align*}

\subsubsection{Boundary conditions close to punctures}\label{sec:local}
Let $q$ be a positive puncture of a $J$-holomorphic disk $u$ of $\Lambda$ and let $v \in \Hi_{2,\nu}[u]$.  
After certain deformations of $\Lambda$ and the metric $g$ on $M$, as described in [\cite{trees}, Section 4],  we can find coordinates of $T^*M$ so that the boundary conditions induced by $u$ are given by
 \begin{align*} 
&v(\tau) \in \R^n\\
&v(\tau+i) \in \underbrace{e^{-i\theta_\lambda}\R\times\dotsm\times
^{-i\theta_\lambda}\R}_k\times \underbrace{e^{i\theta_\lambda}\R\times\dotsm\times
e^{i\theta_\lambda}\R}_{n-k}
\end{align*}
in a neighborhood of $q$. Here $\theta_\lambda \in (0, \pi/2)$ and $\theta_\lambda \to 0 $ as $\lambda \to 0$. 
Here the first $k$ directions correspond to $T_qW^u(q)$ and the last $n-k$ directions to $T_q W^s(q)$. Similarly, if $q$ is a negative puncture we might assume that the boundary conditions induced by $u$ are given by  
  \begin{align*} 
&v(\tau) \in \underbrace{e^{-i\theta_\lambda}\R\times\dotsm\times
^{-i\theta_\lambda}\R}_k\times \underbrace{e^{i\theta_\lambda}\R\times\dotsm\times
e^{i\theta_\lambda}\R}_{n-k}, \\
&v(\tau+i) \in \R^n
\end{align*}
 where agian the first $k$ directions correspond to $T_qW^u(q)$ and the last $n-k$ directions to $T_q W^s(q)$.

\subsection{Orientation conventions and initial choices}\label{sec:orientconv}
It is necessary for us to fix some orientation conventions.
First we discuss the \emph{initial choices}, which are the following. 
\begin{itemize}
 \item  choice of orientation of the base manifold $M$;
 \item   choice of orientations of the unstable manifolds associated to the Reeb chords of $\Lambda$;
  \item  choice of orientation of $\R^n$;
 \item choice of orientation of  $\C$;
  \item  choice of spin structure on $\Lambda$.
\end{itemize}

The choice of orientation of $M$ is in fact only a local choice. If we keep track of the local choices we can define the orientation algorithm also in the case when $M$ is non-orientable. However, we will from now on assume that $M$ is oriented, to simplify notation. We will also assume that the orientations of $T_pM$ and $\R^n$ coincide for all $p \in M$.

For each Reeb chord $p$ we let the chosen orientation of $W^u(p)$ induce an orientation of $W^s(p)$ via the identification
\begin{equation*}
 T_pW^u(p)\oplus T_pW^s(p) = T_pM
\end{equation*}
as oriented spaces. This chosen orientation of $W^u(p)$ will be the \emph{capping orientation of $W^u(p)$}, and the induced orientation of $W^s(p)$ is then the \emph{capping orientation of $W^s(p)$}. 

 That the orientation of $\C$ affects the sign of $\Gamma$ follows
 from the constructions in \cite{fooo}, which we will make use of when
 we define the capping orientation of a disk $u$. 
 Indeed, these constructions use determinant line bundles associated to complex
 $\dbar$-operators, which are given orientations induced by the
 orientation of $\C$. For a more detailed discussion, see
 e.g. [\cite{orientbok}, Section 4.5.6]. In Section \ref{sec:signfunctions} we will fix an orientation of $\C$ which makes it possible to describe the signs of rigid trees explicitly.

 In Subsection \ref{sec:triv} we describe how the choice of the spin structure of $\Lambda$ affect the capping orientations. 
In Subsection \ref{sec:ok} we discuss the orientation of finite dimensional vector spaces induced by exact sequences, and in Subsection \ref{sec:modspor} we consider a stabilization of  the linearized $\dbar$-operator, stabilized by the space of conformal variations. In this way we get a surjective operator and we can interpret the orientation of a $J$-holomorphic disk $u$ as a sign. In Subsection \ref{sec:confvaror} we fix an orientation of the space of conformal variations of the punctured disk.

 \subsubsection{Trivializations}\label{sec:triv}
If $u$ is a $J$-holomorphic disk of $\Lambda$, then the capping orientation will depend on the chosen trivialization of $T\Pi_\C(\Lambda)$ along the boundary of $u$, and this choice will depend on the spin structure of $\Lambda$. We describe this dependence here, following \cite{orientbok}.
 
Fix a spin structure $\s$ on $\Lambda$, and let $u:D_{m+1}\to T^*M$ be a
$J$-holomorphic disk of $\Lambda$. Then $\s$ induces a trivialization
$A_\s: \partial D_{m+1} \to U(n)$ of $T\Pi_\C(\Lambda)$ along the
boundary of $u$ in the following way. Assume that we have a
triangulation of $\Lambda$ such that all Reeb chord end points belong
to the $0$-skeleton, and so that all capping paths belong to the
$1$-skeleton. Using the spin structure we find a trivialization of $T\Lambda$ over the
$1$-skeleton of $\Lambda$ and extend it to the $2$-skeleton. If $\tilde \gamma \subset \Pi_\C(\Lambda)$ is a
boundary component of $u$, then we homotope $\tilde \gamma$ to the
$1$-skeleton, keeping the end points fixed, and give
$T\Pi_\C(\Lambda)|_{\tilde \gamma}$ the induced trivialization. The
spin structure $\s$ guarantees that this trivialization is unique, up
to homotopy. By doing this for all boundary components of $u$ we
obtain a trivialized Lagrangian boundary condition $A_\s$ for $u$. See [\cite{orientbok}, Section 3.4.2].

Next we stabilize the tangent bundle $T\Lambda$, by adding a trivial
bundle with fibers consisting of two auxiliary directions $\aux_1
\oplus \aux_2 \simeq \R^2$, so that we obtain the \emph{stabilized
  tangent bundle} $\tilde T\Lambda = T\Lambda \oplus \R^2$. Extend the
spin structure $\s$ of $T\Lambda$ trivially over the two extra
directions, continue to denote it by $\s$. We also extend $A_\s$ to be
equal to the identity in the auxiliary directions, giving us
\emph{the stabilized trivialized boundary condition of $u$ associated to $\s$}, 
 denoted by $\tilde
A_\s$.

We refer to [\cite{orientbok}, Section 4.4] for a discussion on how
the capping orientation changes if we change the spin structure. This is also
dealt with in Section  \ref{sec:treetriv}.

\subsubsection{Linear algebra}\label{sec:ok}
Let $V$ be a finite-dimensional vector space, and let
\begin{equation*}
  v_{\bar{n}} = (v_1,\dotsc,v_n)
\end{equation*}
be an oriented basis for $V$. We then represent the orientation $\Or(V)$ of $V$ by the wedge product $v_1\wedge\dotsm\wedge v_n \in \bigwedge^{\max}V$. Clearly $v_{\bar{n}}$ and $w_{\bar{n}}$ gives the same orientation of $V$  if and only if $v_1\wedge\dotsm\wedge v_n = \alpha \, w_1\wedge\dotsm\wedge w_n$ for some $\alpha \in (0,\infty)$. The oriented basis $v_{\bar{n}}$ of $V$ also induces an orientation of $V^*$ via the wedge $v_1^*\wedge \dotsm \wedge v_n^*$. Here $v_i^*$ is the dual vector of $v_i$.

If $V$ and $W$ are oriented vector spaces then the orientations $\Or(V)$ and $\Or(W)$, represented by $\phi_V \in \bigwedge^{\max}V$ and $\phi_W \in \bigwedge^{\max}W$, respectively, canonically induce an orientation of the $1$-dimensional space $\bigwedge^{\max}V \otimes \bigwedge^{\max}(W^*)$ via the  pair  $(\phi_V,\phi_W) \in \bigwedge^{\max}V \times \bigwedge^{\max} W$. Two pairs  $(\phi_V,\phi_W)$ and  $(\phi'_V,\phi'_W)$ are equivalent if and only if $\phi_V = \alpha \phi_V', \phi_W = \beta \phi_V'$, $\alpha\beta>0$. We call an equivalence class $(\phi_V,\phi_W)$ an \emph{orientation pair of  $\bigwedge^{\max}V \otimes \bigwedge^{\max}W$}.

We will often consider exact sequences involving at most four non-trivial
finite-dimensional vector spaces, as follows
\begin{equation}\label{eq:exact}
 0 \xrightarrow{} V_1 \xrightarrow{\alpha} W_1\xrightarrow{\beta} W_2
\xrightarrow{\gamma} V_2 \xrightarrow{} 0.
\end{equation}
A basic observation when working with orientations in Symplectic Field
Theory is that, after having fixed orientations of three of the spaces
in the sequence above, then the fourth one can be given an induced
orientation. This induced orientation is canonical, up to choices of
orientation conventions. We fix these conventions in what follows, and
this will differ slightly from the conventions in \cite{orientbok}. Also compare with the discussion in \cite{z}.

Pick a basis $(v_1,\dotsc,v_k)$ for $V_1$ and vectors $(u_1,\dotsc,u_l)\in W_2$ so that $(\gamma(u_1),\dotsc,\gamma(u_l))$ gives a basis for $V_2$. Then pick vectors $(w_1,..,w_m) \in W_1$ so that a basis for $W_1$ is given by $(\alpha(v_1),\dotsc,\alpha(v_k), w_1,\dotsc,w_m)$. By exactness of the sequence  \eqref{eq:exact} we then get that $(\beta(w_1),\dotsc,\beta(w_m),u_1,\dotsc,u_l)$ gives a basis for $W_2$. Now consider the map 
\begin{equation}\label{eq:o2}
    \Phi : \bigwedge^{\max}V_1\otimes \bigwedge^{\max}V_2 \to \bigwedge^{\max}W_1\otimes \bigwedge^{\max}W_2 
\end{equation}
given by
\begin{multline*}
   \Phi(v_1\wedge\dotsm \wedge v_k \otimes \gamma(u_1)\wedge\dotsm\wedge \gamma(u_l))\\
   \nonumber
    =v_1\wedge\dotsm \wedge v_k \wedge w_1 \wedge \dotsm \wedge w_m \otimes  u_1\wedge\dotsm\wedge u_l \wedge \beta(w_1)\wedge \dotsm \wedge \beta(w_m).
\end{multline*}

The following is a standard fact, see e.g.\ \cite{Floer-Hofer}.  

\begin{lma}\label{lma:o1o2}
The sequence \eqref{eq:exact} together with the map \eqref{eq:o2} induces an isomorphism 
\begin{equation}
 \phi:\bigwedge^{\max}V_1 \otimes  \bigwedge^{\max}V_2 \xrightarrow{\approx} 
\bigwedge^{\max}W_1 \otimes  \bigwedge^{\max}W_2, 
\end{equation}
and hence a one-to-one correspondence between orientation pairs of $(V_1,V_2)$ and of $(W_1, W_2)$.
\end{lma}


From now on, unless something else is explicitly stated, when we say that a space is given an orientation induced by an exact sequence, it is understood that we mean the orientation induced by the conventions in Lemma \ref{lma:o1o2}.

%

\subsubsection{The fully linearized \texorpdfstring{$\dbar$}{dbar}-operator and the sign of a rigid disk}\label{sec:modspor}

Recall that a \emph{Fredholm operator} $L$ between linear spaces $V$, $W$ is a bounded linear operator $L: V \to W$ satisfying that the kernel and cokernel of $L$ are finite-dimensional and that the image of $L$ is closed. The \emph{index} of $L$ is given by
\begin{equation*}
 \ind L = \dim \Ker L - \dim \Coker L
\end{equation*}
and the \emph{determinant line} of $L$ is given by 
\begin{equation*}
\det L = \bigwedge^{\max} \Ker L \otimes \bigwedge^{\max} (\Coker L)^*.           \end{equation*}
 The set $F(X,Y)$ of all Fredholm operators is open in the space of bounded linear operators from $V$ to $W$, and the collection of all determinant lines gives rise to a line bundle over $F(X,Y)$, see e.g.\ \cite{quillen}.

Let $u:D_{m+1} \to T^*M$ be a $J$-holomorphic disk of $\Lambda$, and
let $\kappa$ denote the conformal structure of $D_{m+1}$. Recall the
associated linearized operator $\dbar_u :\Hi_{2,\nu}[u] \to
\Hi_{1,\nu}[0]$ introduced in Section \ref{sec:sobolev}. 
In \cite{legsub} it is proven that we can choose the weight $\nu$  so
that  $\dbar_u$ is Fredholm. Also compare with Lemma
\ref{lma:wfred}. From now on we assume that such a choice has been made.

We define the
\emph{fully linearized $\dbar$-operator} $D\dbar_{u,\kappa}$ at
$(u,\kappa)$ to be given by
\begin{equation}\label{eq:conf1}
D\dbar_{u,\kappa} = \dbar_{u}\oplus \Psi_\kappa :  \Hi_{2,\nu}[u] \oplus T_\kappa \Co_{m+1} \to  \Hi_{1,\nu}[0],
\end{equation} 
where 
$\Psi_\kappa : T_\kappa \Co_{m+1} \to  \Hi_{1,\nu}[0]$ is a linear
map. For a more explicit description of $\Psi_\kappa$, see Section
\ref{sec:vectors}. For the moment it suffices to know that for generic
$J$ we have that $D \dbar_{u,\kappa}$ is surjective. See \cite{pr} for
a proof of this fact.

Let $\M(a, {\bf b}),$ ${\bf b} =( b_1,\dotsc.,b_m)$, denote the moduli
space of $J$-holomorphic disks of $\Lambda$ with positive puncture $a$
and negative punctures $b_1,\dotsc,b_m$. These spaces are said to be
\emph{transversely cut out} if $0$ is a regular value of the
corresponding fully linearized $\dbar$-operators, in which case the moduli
space is a manifold of (local)  dimension equal to the Fredholm index of the operators. It is a standard fact that an orientation of
$T_u\M(a,{\bf b})$ can be identified with an orientation of $\Ker D\dbar_{u,\kappa}$. This will be the
\emph{orientation of $u$}.

\begin{defi}
  We say that $u \in\M(a, {\bf b})$ is \emph{rigid} if
  $\M(a, {\bf b})$ is a transversely cut out manifold of dimension
  $0$.
\end{defi}
Thus, a rigid disk belongs to a zero-dimensional component of the moduli
space, and its orientation is therefore given in terms of a sign. This sign is the one that shows up in the signed count of rigid disks in the differential for Legendrian contact homology.

Now we relate the orientation of $u$ with the orientation
of $\det \dbar_{u}$, which we will fix in Section
\ref{sec:nonpunct} (and which will be the capping orientation of $u$).
From \eqref{eq:conf1} it follows that we have an exact sequence 
 \begin{equation}\label{eq:esdet}
  \begin{array}{ccccccccccc}
   0 & \to & \Ker \dbar_{u} & \to &  
   \Ker(\dbar_{u} \oplus \Psi_{\kappa}) & \to &
   T_\kappa \Co_{m+1} & \to 
   & \Coker \dbar_{u} & \to &0 \\
   & & v &\mapsto & (v,0)  & & & & & \\ 
   & &  & & (v,w)  & \mapsto & w & & & \\
      & &  & &  &  & w & \mapsto & \overline{\Psi_{\kappa}(w)}. & 
  \end{array}
 \end{equation}
Using Lemma \ref{lma:o1o2} we get an isomorphism 
\begin{equation}\label{eq:conf2}
\bigwedge^{\max}\Ker  D \dbar_{u,\kappa} \otimes  \bigwedge^{\max}T_\kappa \Co_{m+1}  \simeq
  \bigwedge^{\max}\Ker \dbar_{u} \otimes  \bigwedge^{\max} (\Coker   \dbar_{u})^*. 
\end{equation}
Thus, we see that an orientation of $\det \dbar_u$ together with an orientation of $T_\kappa \Co_{m+1}$ induces an orientation of  $\Ker D\dbar_{u,\kappa}$. In particular, if $D \dbar_{u,\kappa}$ is an isomorphism and $\dbar_u$ is injective (which is the case for rigid disks with more than one negative punctures), then the orientation is given by a sign  coming from comparing the orientation of $\Coker \dbar_u$ with the orientation of $T_\kappa \Co_{m+1}$ under the isomorphism in \eqref{eq:conf2}. 

In the next subsection we fix an orientation of the space of conformal structures.

\begin{rmk}\label{orientrule}
 If $m \leq 1$, then $\Co_{m+1}$ is empty and we have $ D \dbar_{u, \kappa} =  \dbar_{u}$.  The tangent space of the moduli space of rigid $J$-holomorphic disks at $u$ can then be identified with $\kd_{u}/T \T_{m+1}$, where $\T_{m+1}$ is the space of conformal automorphisms of $D_{m+1}$.  Here we embed $T \T_{m+1} \subset \kd_{u} $ via  $v \mapsto du \cdot v$, $v \in T \T_{m+1}$.
 We refer to [\cite{orientbok}, Section 3.4.1] for a discussion of how to orient $\T_{m+1}$, and we pick the same orientation as described there. The sign of $u$ then comes from comparing the orientation of $T \T_{m+1}$ with the one of $\kd_{u}$.
%
\end{rmk}

\subsubsection{Orientation of the space of conformal variations}\label{sec:confvaror}
In this subsection we will fix an orientation of the space $\Co_{m+1}$ of conformal structures on the $m+1$-punctured unit disk in $\C$. We will also carry it over to the setting of Section \ref{sec:slit} where we viewed conformal structures as represented by standard domains.

First consider the case when we view the domain of the punctured $J$-holomorphic disk as the unit disk in $\C$,
with $m+1$ punctures $p_0,\dotsc,p_m$. We assume that the 
positive puncture $p_0$ is located at $z=1$. If we fix the first two negative
punctures $p_1,p_2$ at $z=i$ and $z=-1$, respectively, then the position of the remaining
punctures $p_3,\dotsc,p_m$ uniquely determines the conformal structure $\kappa$ of the disk. Let  $v_i$ be the vector field supported
in a neighborhood of $p_i$ which agrees with the vector field $z \mapsto iz$ near
the puncture (making the puncture move counter-clockwise along the boundary), and which is
tangent to the boundary along the boundary. See Figure \ref{fig:confstrdisk}.  We then get a canonical identification of the
tangent space $T_\kappa\Co_{m+1}$ of $\Co_{m+1}$ at $\kappa$
 with the 
tuple of vector fields $(v_3,\dotsc,v_m)$.  We define the positive orientation of $T_\kappa \Co_{m+1}$ to be the orientation induced by the ordered tuple $(v_3,\dotsc,v_m)$. If $m=2$ we let
the one point space $\Co_3$ be positively oriented.

\begin{figure}[ht]
 \labellist
\small\hair 2pt
\pinlabel $p_0$ [Br] at 450 270
\pinlabel $p_1$ [Br] at 275 430
\pinlabel $p_2$ [Br] at 99 270
\pinlabel $p_3$ [Br] at 185 135 
\pinlabel $v_3$ [Br] at 200 7
\pinlabel $p_4$ [Br] at 309 108
\pinlabel $v_4$ [Br] at 445 40
\pinlabel $p_0$ [Br] at 870 250
\pinlabel $p_1$ [Br] at 1973 125
\pinlabel $p_2$ [Br] at 1973 215
\pinlabel $p_3$ [Br] at 1973 300
\pinlabel $p_4$ [Br] at 1973 380
\pinlabel $t_1$ [Br] at 1410 163
\pinlabel $t_2$ [Br] at 1348 252
\pinlabel $t_3$ [Br] at 1575 340
\endlabellist
\centering
\includegraphics[height=3.5cm]{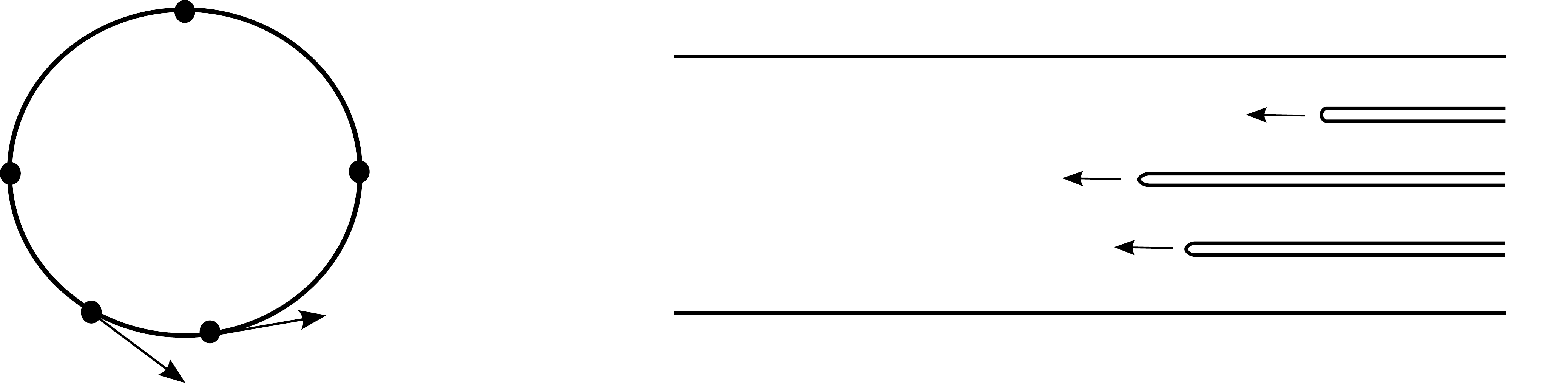}
\caption{To the left a punctured disk with the conformal variations $v_3, v_4$ indicated. To the right a standard domain with the conformal variations $t_1,t_2,t_{3}$ indicated.}
\label{fig:confstrdisk}
\end{figure}

We would like to translate this into the case when we view the disk
$D_{m+1}$ as a standard domain $\Delta_{m+1}$. Recall that in this setting a conformal structure $\kappa$ is
given by the location of $m-2$ out of $m-1$ slits, and that two standard domains are conformally equivalent if and only if their boundary minima differs by an overall translation.

Let $t_j$ be the conformal variation that corresponds to moving the $j$th boundary minimum
towards $-\infty$, and which is cut-off in a neighborhood of this boundary minimum.  See Figure \ref{fig:confstrdisk}. The tangent space $T_\kappa\Co_{m+1}$ is then spanned by $m-2$ of the 
$m-1$ tangent vectors $t_j$. We have the following. 
\begin{lma}[\cite{kch}, Lemma 6.4]\label{lma:confor}
 The orientation of $\Co_{m+1}$ given by $t_2\wedge\dotsm\wedge t_{m-1}$ agrees with the
 orientation of $\Co_{m+1}$ defined above.

 In addition we have
 \begin{equation*}
  t_1\wedge\dotsm\wedge t_{j-1}\wedge
\hat{t}_j\wedge
t_{j+1}\wedge\dotsm\wedge t_{m-1} = (-1)^{j-1} t_2\wedge\dotsm\wedge
t_{m-1}
 \end{equation*}
and in particular
\begin{equation*}
 t_1\wedge\dotsm\wedge t_{m-2} = (-1)^{m-2}t_2\wedge\dotsm\wedge
t_{m-1}.
\end{equation*}
 Here $\hat{v}$ means that the vector $v$ is absent.
\end{lma}

\subsection{Gluing of disks}\label{sec:gluew}
One of the most important tools in this paper is exact sequences of
kernels and cokernels of $\dbar$-operators. These sequences occur when
we glue $J$-holomorphic disks together at (special) punctures.

In [\cite{orientbok}, Section 3.2] these kind of sequences are described in detail, and proven to be exact. Since we will use these sequences repeatedly, we give an outline of the constructions found in that paper and also re-state the results that we will need. 

We mention that these sequences are needed in the inductive procedure of proving that the sign of a tree $\Gamma$ can be described as stated in Theorem \ref{thm:mainthm}. They are also used when we define the capping orientation of a $J$-holomorphic disk. 

\subsubsection{Dbar problems with trivialized Lagrangian boundary conditions}\label{sec:dbarnopunc}
Let $A: \partial D_{m+1} \to U(n)$ be a trivialized Lagrangian boundary condition, $\nu=(\nu_0, \dotsc, \nu_m)$ a weight vector and let
\begin{equation*}
\Hi_{2,\nu}[A]= \Hi_{2,\nu}[A](D_{m+1}, T^{*0,1}D_{m+1} \otimes \C^n) 
 \end{equation*}
 denote the closed subspace of  $\Hi_{2,\nu}(D_{m+1}, T^{*0,1}D_{m+1} \otimes \C^n)$, consisting of elements $s$ that are tangent to $A\R^n$ along $\partial D_{m+1}$, and which satisfies 
 $\dbar s |_{\partial D_{m+1}} = 0$. 
Similarly, let 
\begin{equation*}
\Hi_{1,\nu}[0]   = \Hi_{1,\nu}[0](D_{m+1}, T^{*0,1}D_{m+1} \otimes \C^n)                                                                         \end{equation*}
 be the closed subspace of  $\Hi_{1,\nu}(D_{m+1}, T^{*0,1}D_{m+1} \otimes  \C^n)$ consisting of elements $s$ satisfying $ s |_{\partial D_{m+1}} = 0$.

We define the \emph{linearized weighted $\dbar$-operator with trivialized boundary condition $A$}  to be the $\dbar$-operator
\begin{align*}
 \dbar_{A,\nu} :\Hi_{2,\nu}[A] &\to \Hi_{1,\nu}[0], \\
 \dbar_{A,\nu}(v) &= dv + J \circ dv \circ i.  
\end{align*}

\begin{defi}
   If $\dbar_{A,\nu}:\Hi_{2,\nu}[A] \to \Hi_{1,\nu}[0]$ is Fredholm, then the pair $(A,\nu)$ is called \emph{admissible}.
\end{defi}

We will also consider $\dbar$--problem on the closed (non-punctured) disk $D= D_0$, with trivialized Lagrangian boundary condition $A \colon \partial D \to U(n)$. This gives a Fredholm operator $\dbar_A \colon \Hi_2[A] \to \Hi_1[0]$ of Fredholm index equal to $n + \mu(A)$, where $\mu(A)$ is the Maslov index of $A$. See e.g. [\cite{jholo}, Theorem C.1.10]. Such a $\dbar$-problem we call a \emph{Maslov problem of dimension $n$ with trivialized boundary conditions given by $A$}.

\begin{rmk}
If we have a disk $u: D_{m+1} \to \C^n$ and a trivialization $A$ of $T \Pi_\C(\Lambda)$ along the boundary of $ u$ we often write $\dbar_{A,\nu} :\Hi_{2,\nu}[A] \to \Hi_{1,\nu}[0]$ instead of $\dbar_{u,\nu} :\Hi_{2,\nu}[u] \to \Hi_{1,\nu}[0]$ to indicate the chosen trivialization $A$. Also, if it is clear
 from the context which weight we are using we only write $\dbar_A$.
\end{rmk}

\subsubsection{Gluing weighted disks}\label{sec:glw}
Let $\Delta_A$ and $\Delta_B$ be two standard domains with corresponding
trivialized Lagrangian boundary conditions $A: \partial \Delta_A \to U(n)$ and  $B: \partial \Delta_B \to U(n)$. Let $E_0$ be the strip-like end of $\Delta_A$ at $\tau =- \infty$, and let $E$ be one of the strip-like ends of $\Delta_B$ at $\tau =+ \infty$. Pick coordinates $(\tau,t) \in (-\infty,0] \times [0,1]$ on $E_0$ and $(\tau,t) \in [0,\infty) \times [0,1]$ on $E$.  Assume that $A$ is
constant for $|\tau|$ large in $E_0$,
and that $B$ is constant for $\tau$ large in $E$, and that $A=B$ in these regions. Then we can glue $\Delta_A$ to $\Delta_B$, by joining $E_0$ and $E$, and get a trivialized
Lagrangian boundary condition $A\#B$ on the glued domain $\Delta_A\#\Delta_B$.

Let $\nu>0$ be so large so that $A$ and $B$ are constantly
equal  on $E_0$ and $E$ for $\tau<-\nu$ and $\tau>\nu$, respectively. Let
$\Delta(\rho) = \Delta_A \#_\rho \Delta_B$ be the standard domain obtained by identifying
$\{-\rho\}\times[0,1]$ on  $E_0$ with
$\{\rho\}\times[0,1]$ on $E$, where $\rho>>2\nu$. Let
$A\#_\rho B$ be the
induced boundary condition, which is equal to $A$ on the part of $\Delta(\rho)$ corresponding to $\Delta_A$, equal to $B$ on the part of $\Delta(\rho)$ corresponding to $\Delta_B$, and is constant in the gluing region. 

Let $\w_A$ and $\w_B$ be the weight functions associated to $\Delta_A$ and $\Delta_B$, respectively. Let $\nu_A$ denote the weight vector at the
end $E_0$  and let $\nu_B$ denote the weight vector at the 
end $E$, so that we have weights ${ \w}_A \simeq e^{\nu_A|\tau|}$ and ${ \w}_B \simeq e^{\nu_B|\tau|}$, respectively, at the actual ends.  
Let ${ \w}_\rho$ denote the glued weight function on
$\Delta(\rho)$, which equals ${ \w}_A$ and ${ \w}_B$, respectively, outside the
gluing region $\Delta_A\setminus (- \infty, \nu-\rho)\times[0,1] \cup \Delta_B\setminus (-\rho-\nu, \infty)\times[0,1]$,  and is a smooth interpolation of the weight functions in the
gluing region. 

Assume that we have chosen weights so that $(A, { \w_A} )$ and $(B, { \w_B})$ are both admissible pairs. Then the glued $\dbar$-problem will also be admissible, and the index can be computed as follows. Here we assume that the weight vectors take values in $(-\pi/2, \pi/2)^n$.

\begin{defi}
 If $\nu= (\nu_1,\dotsc,\nu_n)$ is a vector, we say that $ \sgn \nu < 0$ ($\sgn \nu > 0$) if $\sgn \nu_i < 0$ ($\sgn \nu_i > 0$) for all $i=1,\dotsc,n $. Also, if $\tilde\nu = (\tilde\nu_1,\dotsc,\tilde\nu_n)$ is another vector, we say that $\sgn \nu = \sgn \tilde\nu$ ($\sgn \nu = -\sgn \tilde\nu$) if $\sgn \nu_i = \sgn \tilde\nu_i$ ($\sgn \nu_i = -\sgn \tilde\nu_i$) for all $i=1,\dotsc,n$.
\end{defi}

\begin{lma}\label{lma:indexglu}
We have the following.
\begin{itemize}
\item[i)] If $\sgn \nu_A =\sgn \nu_B >0$ then $\ind
\dbar_{A\#B} = \ind \dbar_A +\ind \dbar_B +n$.
\item[ii)] If $\sgn \nu_A =\sgn \nu_B <0$ then  $\ind
\dbar_{A\#B} = \ind \dbar_A +\ind \dbar_B -n$.
\item[iii)] If $\sgn \nu_A =- \sgn \nu_B$ then $\ind \dbar_{A\#B} = \ind \dbar_A +\ind
\dbar_B $.
\end{itemize}
\end{lma}

\begin{proof}
The second case is stated in [\cite{orientbok}, Equation (3.6)]. The other two cases also  follow from this result, by noting that we can change all positive entries of ${ \nu}_A$  and ${ \nu}_B$ to negative, with the cost of raising the indices of $\dbar_A$ and $\dbar_B$ by the number of positive weights, see [\cite{legsub}, Section 6.6 -- Section 6.9]. The glued problem will have the same index as the one we get without changing the weights. Thus, in the first case we have to add $n$ to both the index of $\dbar_A$ and $\dbar_B$, and in the third case we have to add in total $n$ to $\ind \dbar_A + \ind \dbar_B$. 
\end{proof}

\begin{lma}\label{lma:gluweights}
Assume that we are in the situation of gluing as described above. Then, for sufficiently large $\rho$, there are exact gluing
sequences of the following form.
\begin{itemize}
 \item[a)] If $\sgn\nu_A =\sgn \nu_B>0$ then we get 
\begin{align}\label{eq:g1}
 0 \to
\kd_{A\#B} \xrightarrow{\alpha_\rho} 
\begin{bmatrix}
\kd_A \\
 \R^n \\ \kd_B
\end{bmatrix}
\xrightarrow{\beta_\rho}
\begin{bmatrix}
\cd_A \\
\cd_B 
\end{bmatrix}
\xrightarrow{\gamma_\rho}
\cd_{A\#B}
 \to 0.
\end{align}
 \item[b)] If $\sgn\nu_A =\sgn \nu_B<0$ then we get 
\begin{align}\label{eq:g2}
 0 \to
\kd_{A\#B} \xrightarrow{\alpha_\rho}
\begin{bmatrix}
 \kd_A \\
  \kd_B
\end{bmatrix}
\xrightarrow{\beta_\rho}
\begin{bmatrix}
\cd_A \\
 \R^n \\
  \cd_B
\end{bmatrix}
 \xrightarrow{\gamma_\rho}
\cd_{A\#B} 
 \to 0.
\end{align}
 \item[c)] If $\sgn\nu_A =-\sgn \nu_B$ then we get 
\begin{align}\label{eq:g3}
 0 \to
\kd_{A\#B} \xrightarrow{\alpha_\rho} 
\begin{bmatrix}
\kd_A \\
\kd_B
\end{bmatrix}
\xrightarrow{\beta_\rho}
\begin{bmatrix}
\cd_A \\
 \cd_B
\end{bmatrix}
 \xrightarrow{\gamma_\rho}
\cd_{A\#B} 
 \to 0.
\end{align}
\end{itemize}
 In the sequences $\eqref{eq:g1}$ and $\eqref{eq:g2}$ we have a canonical isomorphism $\R^n\simeq
A(-\{\rho\}\times\{1\})$ via evaluation.
\end{lma}
\begin{proof}
 Follows from [\cite{orientbok}, Lemma 3.1] 
\end{proof}

\begin{rmk}\label{rmk:maps}
The maps $\alpha_\rho$, $\beta_\rho$ and $\gamma_\rho$ are given as follows. First embed the kernels and cokernels of the $\dbar_A$ and $\dbar_B$-operators in the Sobolev spaces associated to $\dbar_{A \#B}$, by cutting off the elements with cut-off functions $\phi_A^\rho, \phi_B^\rho: \Delta(\rho) \to [0,1]$. The cut-off functions are required to satisfy
\begin{align}\label{eq:phiA}
 \phi_A^\rho &= 
 \begin{cases}
  1 &\text{ on } \Delta_A\setminus ((-\infty,-1/2\rho] \times [0,1]) \\
  0 &\text{ on } \Delta(\rho) \setminus(\Delta_A \setminus((-\infty,-\rho] \times [0,1]))
 \end{cases} \\ \nonumber & \\
\label{eq:phiB}
 \phi_B^\rho &= 
 \begin{cases}
  1 &\text{ on } \Delta_B\setminus ([1/2\rho, \infty) \times [0,1]) \\
  0 &\text{ on } \Delta(\rho) \setminus(\Delta_B \setminus([\rho,\infty] \times [0,1]))
 \end{cases}
\end{align}
and $|D^k\phi_U^\rho| = O(\rho^{-1})$, $k=1,2$, $U=A,B$.
Also embed the constant functions defined in the gluing region
representing $\R^n$ in a similar way, by using cut-off functions with
support in the gluing region. Then $\alpha_\rho$ is $L^2$-projection
onto the space spanned by the cut-off kernel elements, $\beta_\rho$ is
$\dbar$ composed with $L^2$-projection and $\gamma_\rho$ is projection
to the quotient. See [\cite{orientbok}, Section 3.2.2] for a more detailed description.
\end{rmk}

\begin{rmk}
Using this lemma together with the orientation conventions from Section \ref{sec:orientconv}, we see that orientations $o_A$ and $o_B$ on the determinant lines of $\dbar_A$ and $\dbar_B$, respectively, canonically induces an orientation $o_{A\#B}$ on the determinant line of $\dbar_{A\#B}$.
\end{rmk}

\begin{rmk}
 We have a similar situation concerning iterated gluings of $J$-holomorphic disks, see [\cite{orientbok}, Section 3.2.3], where it is proven that gluing is associative on the level of orientations. Notice that the same proofs go through even if we are using slightly different orientation conventions.
\end{rmk}

\subsection{Capping orientations of \texorpdfstring{$J$}{J}-holomorphic disks}\label{sec:nonpunct}
In this section we will define the capping orientation of a $J$-holomorphic disk $u$ of $\Lambda$. We follow the similar constructions in \cite{orientbok}. 

Briefly this works as follows. Let $u_b$ denote the restriction of $u$ to the boundary, and let $\hat{\Lambda}=\Pi_\C ({\Lambda})$. The choice
of spin structure on $\Lambda$ induces a trivialization of 
$u_b^*T\hat{\Lambda}$, which is well-defined up to homotopy. Closing up these paths of trivializations with something
called capping trivializations, we get a trivialized Lagrangian boundary
condition on the non-punctured disk. The determinant line of the
$\dbar$-operator with this boundary condition is oriented using evaluation, and
if we fix orientations of the $\dbar$-operators over the capping disks, then the sequences from Lemma \ref{lma:gluweights}  will induce an orientation of the determinant
line of the $\dbar$-operator at $u$. This will be the capping orientation
of $u$.

\subsubsection{Orientation of trivialized boundary conditions on the non-punctured disk}

Let $A$ be a trivialized Lagrangian boundary condition on the unit disk $D$ in $\C$, i.e., let $A$ be a loop $A : \partial D \to U(n)$. Let $\dbar_A$ denote the corresponding $\dbar$-problem as defined in Section \ref{sec:dbarnopunc}. 
The determinant lines of these type of operators give rise to a line bundle $\det \dbar$  over the space of trivialized Lagrangian boundary conditions on $D$, which is nothing but the space $\Omega(U(n))$ of free loops in $U(n)$. From \cite{orientbok} and \cite{fooo} we have the following result.

\begin{prp}[\cite{orientbok}, Lemma 3.8]\label{fukayaorient}
 The determinant line bundle $\det \dbar$ over $\Omega(U(n))$ is orientable, and a choice of
orientation of $\R^n$ and of $\C$ induces an orientation of this bundle. Moreover, the orientation is given via evaluation at a point on the boundary $\partial D$, where we identify the image of the evaluation with $\R^n$ using the given trivialization.
\end{prp}

Assume that we have fixed an orientation on $\R^n$ and on $\C$. We will call the orientation induced on $\det \dbar_A$ as in the lemma above the \emph{canonical orientation}. 


\subsubsection{Trivializations and weights associated to a punctured \texorpdfstring{$J$}{J}-holomorphic disk}\label{sec:spintriv}
Let  $u$ be a $J$-holomorphic disk of $\Lambda$, and let $u_b$ denote the restriction of $u$ to the boundary.
Fix a spin structure $\s$ on $\Lambda$. As described in Section
\ref{sec:triv}, this induces a stabilized trivialized Lagrangian boundary condition  $\tilde A_u : D_{m+1} \to U(n+2)$ of $u^*_b \tilde T \hat \Lambda$.

 Note that the $\dbar$-problem 
 \begin{equation*}
  \dbar_u:\Hi_2[\tilde A_u](D_{m+1}, u_b^*\tilde{T}\hat \Lambda) \to \Hi_1[0](D_{m+1},  T^{*0,1}D_{m+1} \otimes u_b^*\tilde{T}\hat \Lambda)
 \end{equation*}
is not Fredholm, since the boundary conditions in the auxiliary directions are trivial. To solve this issue, we will weight the involved Sobolev spaces. This we do not only do in the auxiliary directions, but also in the directions corresponding to $\Lambda$, since this will be needed in the degeneration process $\lambda \to 0$.

Recall the local picture in a neighborhood of a puncture $p=p_j$ of $u$, given in Section \ref{sec:local}. 
We choose coordinates at $p$ so that the first $k$ directions correspond to $T_pW^u(p)$  and the last $n-k$ directions corresponds to $T_pW^s(p)$.  We choose the weight vector $\nu_j$ at $p$ to be given by
\begin{equation*}
 \nu_j=(\underbrace{-\delta,\dotsc,-\delta}_{T_pW^u(p)},\underbrace{\delta,\dotsc,\delta}_{T_pW^s(p)},\underbrace{\delta,\delta}_{\aux})
\end{equation*}
if $p$ is positive, and to be given by
\begin{equation*}
 \nu_j=(\underbrace{\delta,\dotsc,\delta}_{T_pW^u(p)},\underbrace{-\delta,\dotsc,-\delta}_{T_pW^s(p)},\underbrace{-\delta,-\delta}_{\aux})
\end{equation*}
if $p$ is negative. Here $0<\delta << \pi$ is some fixed number. Now
define a weight function $\w_u =\w_\nu$ as in Section \ref{sec:sobolev}, using the weight vectors just described.

\begin{lma}\label{lma:wfred}
Let $u$ be a $J$-holomorphic disk of $\Lambda_\lambda$. If $\lambda >0$ is sufficiently small, then the $\dbar$-problem 
\begin{equation*}
 \dbar_u:\Hi_{2,\nu}[\tilde A_u](D_{m+1}, u_b^*\tilde{T}\hat \Lambda) \to \Hi_{1,\nu}[0](D_{m+1}, T^{*0,1}D_{m+1} \otimes u_b^*\tilde{T}\hat \Lambda)
\end{equation*}
is Fredholm, and is split in directions corresponding to $T\hat\Lambda$ and the auxiliary directions. Moreover, the kernel and cokernel of $\dbar_u$ equals the kernel and cokernel of the unweighted problem in the directions corresponding to $T\hat\Lambda$. In the auxiliary directions the problem is an isomorphism.
\end{lma}

\begin{proof}
 Follows from the choice of weights, the local description of the boundary conditions close to a puncture, together with the results in  [\cite{legsub}, Section 6].
\end{proof}

\subsubsection{Capping operators}\label{sec:capping}
To close up the trivialized Lagrangian boundary condition $\tilde A_u$ to give a trivialized Lagrangian boundary condition on the non-punctured disk, we use \emph{capping operators}.
These are $\dbar$-operators defined on $D_1$ with certain associated trivialized boundary conditions.

To get trivialized Lagrangian boundary conditions on $D_0$, we sometimes need to perform rotations in planes of $\C^n$. To that end, let $ R_{\pm \pi}(k,l)\subset SO(n)$, $k,l \in\{1,\dotsc,n\}$, be the path of matrices with entries
\begin{align*}
 &(R_{\pm \pi}(k,l))_{i,i} =1, \qquad i \neq k,l,  \\
 &(R_{\pm \pi}(k,l))_{k,k} =\cos(s\pi),     \\
 &(R_{\pm \pi}(k,l))_{k,l} =\mp\sin(s\pi),     \\
 &(R_{\pm \pi}(k,l))_{l,l} =\cos(s\pi),     \\
 &(R_{\pm \pi}(k,l))_{l,k} =\pm \sin(s\pi), \\    
&(R_{\pm \pi}(k,l))_{i,j} =0, \qquad  \text{ otherwise.}
\end{align*} 
$s \in [0,1]$.
This corresponds to a positive (negative) rotation by $\pi$ in the real subspace spanned by $\partial_{x_k}, \partial_{x_l}$. For example, $R_\pi(2,3) \in SO(3)$ is given by 
\begin{equation*}
 \begin{bmatrix}
  1 & 0 & 0 \\
  0 & \cos \theta & -\sin \theta \\
  0 & \sin \theta & \cos \theta
 \end{bmatrix},
\end{equation*}
$\theta \in [0,\pi]$.

If $A :[0,1] \to U(n)$ is a path of matrices, we define $A \hat * R_{\pm}(k,l)$ to be the concatenation $A *( R_{\pm}(k,l) \cdot A(1))$ (that is, we begin with the path $A$ and then we perform a $\pm\pi$ rotation in the plane spanned by $\partial_{x_k}, \partial_{x_l}$). Some examples:
\begin{align*}
A &\equiv
 \begin{bmatrix}
  -1 & 0 \\
  0 & -1
 \end{bmatrix}  &\Rightarrow A  \hat * R_{\pi}(1,2) &=
  \begin{bmatrix}
  -\cos \theta & \sin \theta \\
  -\sin \theta & -\cos \theta
 \end{bmatrix}, \\
 A &\equiv
 \begin{bmatrix}
  -1 & 0 \\
  0 & 1
 \end{bmatrix} &\Rightarrow A  \hat * R_{-\pi}(1,2) &=
  \begin{bmatrix}
  -\cos \theta & \sin \theta \\
  \sin \theta & \cos \theta
 \end{bmatrix}, \\
  A &\equiv
 \begin{bmatrix}
  1 & 0 \\
  0 & -1
 \end{bmatrix} & \Rightarrow  A \hat * R_{\pi}(1,2) &=
  \begin{bmatrix}
  \cos \theta & \sin \theta \\
  \sin \theta & -\cos \theta
 \end{bmatrix}, \\
\end{align*}
$\theta \in [0,\pi].$

Let $p$ be a Reeb chord of $\Lambda$. Recall that these chords have either odd or even Maslov index. We need to distinguish them even more, in the following way. Note that, close to a Reeb chord, the tangent map of the restriction $\tilde \Pi$ of the base projection  $\Pi: J^1(M) \to M$ to $\Lambda$ induces isomorphisms  $T\tilde \Pi: T_{p_\pm}\Lambda \to T_{\Pi(p)}M$. We define the \emph{type} of $p$ to be the tuple $(k_+,k_-) \in \{0,1\} \times \{0,1\}$ such that  
\begin{equation}\label{eq:sheetsign}
 \begin{cases}
  k_i=0, \text{ if } T \tilde \Pi: T_{p_i}\Lambda \to T_{\Pi(p)}M  \text{ is orientation-preserving}, \\
  k_i=1, \text{if } T\tilde \Pi: T_{p_i}\Lambda \to T_{\Pi(p)}M \text{ is orientation-reversing},
 \end{cases}
\end{equation}
$i \in \{+,-\}$. 
\begin{rmk}
 If $|\mu(p)| = 0$, then the type of $p$ is either $(0,0)$ or $(1,1)$, and if $|\mu(p)| = 1$, then the type of $p$ is either $(0,1)$ or $(1,0)$.
\end{rmk}

We will fix trivializations of the stabilizations of $T \Pi_\C(\Lambda)$ in neighborhoods of $\Pi_\C(p_\pm)$ as follows. Assume that  $(\partial_{x_1},\dotsc,\partial_{x_n})$ is a trivialization of $TM$ in a neighborhood $U$ of $\Pi(p_\pm)$, we extend it to a trivialization $(\partial_{x_1},\dotsc,\partial_{x_n}, \partial_{x_{n+1}}, \partial_{x_{n+2}})$ of $T \tilde M|_U=U \times (\R^n \oplus \aux_1 \oplus \aux_2)$. Now we lift this to trivializations of the two sheets $U_+$, $U_-$ of $\Pi_\C(\Lambda)$ containing $p_+$ and $p_-$, respectively, in the following way. Assume that $p$ is of type $(k_+, k_-)$. Then the lifted trivialization is given by  $(\partial_{x_1},\dotsc,\partial_{x_n}, \partial_{x_{n+1}},(-1)^{k_+}\partial_{x_{n+2}})$ along $U_-$, and by  $(e^{-i\theta_\lambda}\partial_{x_1},\dotsc,e^{-i\theta_\lambda}\partial_{x_k},e^{i\theta_\lambda}\partial_{x_{k+1}},\dotsc,e^{i\theta_\lambda}\partial_{x_n}, \partial_{x_{n+1}},(-1)^{k_-}\partial_{x_{n+2}})$ along $U_+$. Here the first $k$ directions correspond to $T_pW^u(p)$ and the last $n-k$ directions to $T_pW^s(p)$.

We will define two \emph{capping trivializations} $\tilde R_+(p), \tilde R_-(p):\partial D_1 \to U(n+2)$ associated to $p$, as follows. 
Let $\partial D_1$ be parametrized by $s \in [0,1]$, and use the trivialization of $T\Pi_\C(\Lambda)$ to identify $\R^n \subset \C^n$ with $T_p(\Pi_{\C}(\Lambda)) \subset T^*M$. Then the capping trivializations are defined as follows.
\begin{description}
 \item[\it $|\mu(p)|$ even, type $(0,0)$] 
 \begin{align*}
\tilde R_{+}(p) &= \diag
(\underbrace{e^{-i\theta_\lambda s},\dotsc,e^{-i\theta_\lambda s}}_{T_pW^u(p)},\underbrace{e^{-
i(2\pi -\theta_{\lambda})s},\dotsc, e^ {-i(2\pi -\theta_{\lambda})s}}_{T_pW^s(p)},\underbrace{e^{-i\pi s},e^{-i\pi s}}_{\aux})\hat* \\
&{}\qquad R_\pi(n+1, n+2), \\
 \tilde R_{-}(p) &=  R_{-\pi}(n+1, n+2) *\\
&{}\qquad 
\diag(\underbrace{e^{-i \theta_\lambda -i(2\pi -\theta_{\lambda})s},\dotsc, e^ {-i \theta_\lambda -i(2\pi -\theta_{\lambda})s}}_{T_pW^u(p)},\underbrace{e^{i \theta_\lambda -i\theta_\lambda s},\dotsc,e^{i \theta_\lambda -i\theta_\lambda s}}_{T_pW^s(p)},\underbrace{-e^{-i\pi s},-e^{i\pi s}}_{\aux}),
 \end{align*}
  \item[\it $|\mu(p)|$ even, type $(1,1)$]
 \begin{align*}
\tilde R_{+}(p) &= \diag
(\underbrace{e^{-i\theta_\lambda s},\dotsc,e^{-i\theta_\lambda s}}_{T_pW^u(p)},\underbrace{e^{-
i(2\pi -\theta_{\lambda})s},\dotsc, e^ {-i(2\pi -\theta_{\lambda})s}}_{T_pW^s(p)},\underbrace{e^{-i\pi s},-e^{-i\pi s}}_{\aux})\hat * \\
&{}\qquad R_{-\pi}(n+1, n+2), \\
 \tilde R_{-}(p) &= (R_{\pi}(n+1, n+2)\cdot\diag(\underbrace{e^{-i \theta_\lambda},\dotsc, e^{-i \theta_\lambda}}_{T_pW^u(p)},\underbrace{e^{i \theta_\lambda},\dotsc,e^{i \theta_\lambda}}_{T_pW^s(p)},\underbrace{1,-1}_{\aux})  )*\\
&{}\qquad 
\diag(\underbrace{e^{-i \theta_\lambda-i(2\pi -\theta_{\lambda})s},\dotsc, e^{-i \theta_\lambda-i(2\pi -\theta_{\lambda})s}}_{T_pW^u(p)},\underbrace{e^{i \theta_\lambda-i\theta_\lambda s},\dotsc,e^{i \theta_\lambda-i\theta_\lambda s}}_{T_pW^s(p)},\underbrace{-e^{-i\pi s},e^{i\pi s}}_{\aux}),
 \end{align*}
 \item[\it $|\mu(p)|$ odd, type $(1,0)$]
 \begin{align*}
 \tilde R_{+}(p) &= \diag
(\underbrace{e^{-i\theta_\lambda s},\dotsc,e^{-i\theta_\lambda s}}_{T_pW^u(p)},\underbrace{e^{-
i(2\pi -\theta_{\lambda})s},\dotsc, e^ {-i(2\pi -\theta_{\lambda})s}}_{T_pW^s(p)},\underbrace{e^{-2\pi i s},e^{-i\pi s}}_{\aux}), \\
 \tilde R_{-}(p) &= \diag
(\underbrace{e^{-i \theta_\lambda-i(2\pi -\theta_{\lambda})s},\dotsc, e^{-i \theta_\lambda-i(2\pi -\theta_{\lambda})s}}_{T_pW^u(p)}, \underbrace{e^{i \theta_\lambda-i\theta_\lambda s},\dotsc,e^{i \theta_\lambda-i\theta_\lambda s}}_{T_pW^s(p)},\underbrace{1,-e^{i\pi s}}_{\aux}),
 \end{align*}
  \item[\it $|\mu(p)|$ odd, type $(0,1)$, $\dim W^s(p) > 0$]
 \begin{align*}
 &\tilde R_{+}(p) = (R_{\pi}(n,n+2)\cdot\diag(\underbrace{1, \dotsc, 1}_{n+1},-1)   )* \\
 &{}\quad \diag
(\underbrace{e^{-i\theta_\lambda s},\dotsc,e^{-i\theta_\lambda s}}_{T_pW^u(p)},\underbrace{e^{-
i(2\pi -\theta_{\lambda})s},\dotsc, e^ {-i(2\pi -\theta_{\lambda})s}, -e^ {-i(2\pi -\theta_{\lambda})s}}_{T_pW^s(p)},\underbrace{e^{-2\pi i s},e^{-i\pi s}}_{\aux})\hat * \\
&{}\qquad R_{\pi}(n, n+2), \\
  &\tilde R_{-}(p) =  (R_{-\pi}(n, n+2) \cdot\diag
 (\underbrace{e^{-i \theta_\lambda},\dotsc, e^{-i \theta_\lambda}}_{T_pW^u(p)},\underbrace{ e^{i \theta_\lambda},\dotsc,e^{i \theta_\lambda}}_{T_pW^s(p)},\underbrace{1,1}_{\aux}) ) * \\
  &{}\quad \diag
 (\underbrace{e^{-i \theta_\lambda-i(2\pi -\theta_{\lambda})s},\dotsc, e^{-i \theta_\lambda-i(2\pi -\theta_{\lambda})s}}_{T_pW^u(p)},\underbrace{e^{i \theta_\lambda-i\theta_\lambda s},\dotsc,e^{i \theta_\lambda-i\theta_\lambda s},-e^{i \theta_\lambda-i\theta_\lambda s}}_{T_pW^s(p)},\underbrace{1,-e^{i\pi s}}_{\aux})\hat* \\
 &{}\qquad R_{-\pi}(n, n+2).
 \end{align*}
  \item[\it $|\mu(p)|$ odd, type $(0,1)$, $\dim W^s(p) = 0$]
 \begin{align*}
 &\tilde R_{+}(p) = ( R_{\pi}(n,n+2)\cdot\diag(\underbrace{1, \dotsc, 1}_{n+1},-1)  )* \\
 &{}\quad \diag
(\underbrace{e^{-i\theta_\lambda s},\dotsc,e^{-i\theta_\lambda s},-e^{-i\theta_\lambda s}}_{T_pW^u(p)},\underbrace{e^{-2\pi i s},e^{-i\pi s}}_{\aux})\hat * R_{\pi}(n, n+2), \\
  &\tilde R_{-}(p) =  (R_{-\pi}(n, n+2)\cdot \diag
 (\underbrace{e^{-i \theta_\lambda},\dotsc, e^{-i \theta_\lambda}}_{T_pW^u(p)},\underbrace{1,1}_{\aux})  ) * \\
  &{}\quad \diag
 (\underbrace{e^{-i \theta_\lambda-i(2\pi -\theta_{\lambda})s},\dotsc, e^{-i \theta_\lambda-i(2\pi -\theta_{\lambda})s}, -e^{-i \theta_\lambda-i(2\pi -\theta_{\lambda})s}}_{T_pW^u(p)},\,\underbrace{1,-e^{i\pi s}}_{\aux})\hat* \\
 &{}\qquad R_{-\pi}(n, n+2).
 \end{align*}
\end{description}

We also define two weight vectors $\nu_+(p), \nu_-(p)$ associated to $p$, as follows.
\begin{align*}
 &\nu_+(p)=(\underbrace{\delta,\dotsc,\delta}_{T_pW^u(p)},\underbrace{-\delta,\dotsc,-\delta}_{T_pW^s(p)},\underbrace{-\delta,-\delta}_{\aux}) \\
 & \nu_-(p) =(\underbrace{-\delta,\dotsc,-\delta}_{T_pW^u(p)}, \underbrace{\delta,\dotsc,\delta}_{T_pW^s(p)},\underbrace{\delta,\delta}_{\aux})
\end{align*}
and let $\w_{p,\pm} = \w_{\nu_\pm(p)}:D_1 \to \R^{n+2}$ be the associated
weight function.

Let $\dbar_{p,\pm}$ denote the $\dbar$-operator
\begin{equation*}
 \dbar_{p,\pm} : \Hi_{2,\nu_\pm(p)}[\tilde R_{\pm}(p)](D_1, \C^{n+2}) \to  \Hi_{1,\nu_\pm(p)}[0](D_1, \C^{n+2}).
\end{equation*}
\begin{defi}
The operator $\dbar_{p,+}$ ($\dbar_{p,-}$ ) is the \emph{positive (negative) capping operator at $p$}.
\end{defi}

\begin{rmk}\label{rmk:lamdep}
A priori the capping trivializations and the induced capping operators depend on $\lambda$, and we should write $\tilde R_{\pm,\lambda}(p)$ and $\dbar_{p,\pm,\lambda}$, respectively. But this dependence is continuous, and it turns out that there is a $\lambda_0>0$ so that the kernels and cokernels of the operators constitute a vector bundle over $[0,\lambda_0]$. For details, see Section \ref{sec:cappf}. 
\end{rmk}
 

\begin{lma}\label{lma:capbundle0}
For $\lambda$ sufficiently small we have that the kernels and cokernels are given as follows. 

\begin{tabular}{lll}\\
  {\it $p$ positive, $|\mu(p)|$ even:} &
$
 \kd_{p,+} =0, $ &$ \cd_{p,+} \simeq  T_pW^s(p),
 $\\ & \\
  {\it $p$ positive, $|\mu(p)|$ odd:} & 
$
 \kd_{p,+}=0, $ & $\cd_{p,+} \simeq T_pW^s(p)\oplus \aux_1, 
$\\ & \\
 {\it $p$ negative, $|\mu(p)|$ even:} & 
$
  \kd_{p,-} \simeq  \aux_2, $&$  \cd_{p,-} \simeq T_pW^u(p)\oplus \aux_1,   
$ \\ & \\
 {\it $p$ negative, $|\mu(p)|$ odd:} & 
$
  \kd_{p,-} \simeq \aux_2,  $&$  \cd_{p,-} \simeq T_pW^u(p),  
$ \\ & \\
\end{tabular}

\noindent
where the identifications are given by evaluation at $p$.
\end{lma}
\begin{proof}
 Since the problem is split this follows from the similar results in [\cite{legsub}, Section 6.9] and [\cite{orientbok}, Section 3.3.6].
 
 That is, we start by considering the corresponding 1-dimensional problems without putting weights on the Sobolev spaces. From [\cite{jholo}, Appendix C.4] we have that the problem with boundary conditions $e^{-i\pi s}$, $s \in [0,1]$ gives an isomorphism, the problem with boundary conditions $1$ is surjective with a 1-dimensional kernel spanned by a constant solution,   the problem with boundary conditions $e^{i\pi s}$, $s \in [0,1]$ is surjective with 2-dimensional kernel, and that  the problem with boundary conditions $e^{-i2\pi s}$, $s \in [0,1]$ is injective with 1-dimensional cokernel. From Lemma \ref{lma:genkerisoweight} it then follows that the problem with boundary conditions $e^{-i\theta_\lambda s}$, $s \in [0,1]$ is surjective with 1-dimensional kernel for $\theta_\lambda$ sufficiently small, which by [\cite{legsub}, Proposition 6.16] then becomes an isomorphism when we put a positive weight on the corresponding Sobolev space. Similarly, the problem with boundary conditions $e^{-i(2\pi-\theta_\lambda) s}$, $s \in [0,1]$ is injective with 1-dimensional cokernel, and when we put a negative weight on the Sobolev space this cokernel survives, see [\cite{legsub}, Proposition 6.15]. Similarly for the other problems listed above; by imposing a small negative weight we do not change the kernel nor the cokernel, and by imposing a small positive weight the kernel element spanned by the constant solution is killed and hence the dimension of the kernel decreases by one, and in the case we have an isomorphism the positive weight will allow for a cokernel element to exist, that is, we go from having an isomorphism to an injective problem with 1-dimensional cokernel.      
\end{proof}

\begin{cor}\label{cor:grad2}
Let $p$ be a puncture of a $J$-holomorphic disk with boundary on $\Lambda$. Then 
\begin{align*}
    \dim \kd_{p,+} &= 0, 
       &\dim \kd_{p,-} &=1 \\
    \dim \cd_{p,+} &\equiv n + |p|+1, 
       &\dim \cd_{p,-} &\equiv |p| & \pmod{2}.
\end{align*}
\end{cor}
\begin{proof}
 From the definition of the grading of $p$ we get that
 \begin{align*}
  &|p|\equiv I(p)+ \mu(\gamma_p) +1 \equiv \\
  &\begin{cases}
    \dim W^u(p) +1 \equiv n+ \dim W^s(p) +1 \equiv n+ \dim \cd_{p,+}+1, &  p \text{ pos, $|\mu(p)|$ even},\\
    \dim W^u(p) +1 \equiv  \dim \cd_{p,-}, & p \text{ neg, $|\mu(p)|$ even},\\
    \dim W^u(p)  \equiv n+ \dim W^s(p)  \equiv n+ \dim \cd_{p,+}+1, &  p \text{ pos, $|\mu(p)|$ odd},\\
    \dim W^u(p) \equiv  \dim \cd_{p,-}, & p \text{ neg, $|\mu(p)|$ odd,}\\
  \end{cases}
 \end{align*}
 everything modulo 2.
\end{proof}

\subsubsection{Orientations of capping operators} \label{sec:capopor}
We have to pick
orientations of the determinant lines of the capping operators in a controlled way. To do this, we use the constructions in \cite{orientbok}. This works as follows.

Pick an orientation of $\det \dbar_{p,+}$ for each Reeb chord 
$p$. Then use Lemma \ref{lma:gluweights} and the canonical orientation of $\det \dbar$ on the
non-punctured disk to orient $\det \dbar_{p,-}$.
More precisely, let $\dbar_{p}=\dbar_{\tilde R_- \#\tilde R_+}$ denote the
glued operator. From Lemma  \ref{lma:gluweights} we get the following exact sequence 
\begin{align}\label{eq:capseq}
 0 \to \Ker \dbar_{p} \to
\begin{bmatrix}
 \kd_{p,+} \\
\kd_{p,-}
\end{bmatrix}
\to 
\begin{bmatrix}
 \cd_{p,+} \\
\cd_{p,-}
\end{bmatrix}
\to 
\Coker \dbar_{p} \to 0.
\end{align}
Using that $\det \dbar_{p}$ and $\det \dbar_{p,+}$ already are oriented, we get an induced orientation of  $\det \dbar_{p,-}$. Such a choice of orientations of $\det \dbar_{p,+}$ and $\det \dbar_{p,-}$ will be called a \emph{coherent orientation choice}. 

\begin{rmk}
 In Section \ref{sec:elementor} we will describe a choice of orientations of $\det \dbar_{p,+}$ making it possible for us to compute the signs of the rigid trees of $\Lambda$ explicitly.
\end{rmk}

\begin{rmk}
  Lemma \ref{lma:caplim} implies that we can do this continuously in
  $\lambda$ if $\lambda$ is sufficiently small.
\end{rmk}

\subsubsection{Capping operators at special punctures}\label{sec:capspec}
Let $p$ be a special puncture of a partial flow tree $\Gamma$. To be able to close up the boundary conditions for disks associated to partial flow trees and in that way give them a capping orientation, we define capping operators $\dbar_{p,+}$ and $\dbar_{p,-}$. This is done similar to the case of true punctures, with some modifications in the auxiliary directions.

We define the type of a special puncture in the same way as for true punctures, and we also fix trivializations of the stabilized tangent space of $\Pi_\C(\Lambda)$ in neighbourhoods of these punctures in a completely analogous fasion. 
Then we define the trivialized capping boundary conditions associated to a special puncture by
\begin{description}
 \item[\it $|\mu(p)|$ even, type $(0,0)$] 
 \begin{align*}
\tilde R_{+}(p) &= \diag(\underbrace{e^{-i\theta_\lambda s},\dotsc,e^{-i\theta_\lambda s}}_{T_pM},\underbrace{e^{-i\pi s},e^{-i\pi s}}_{\aux}) \hat * R_\pi(n+1, n+2), \\
  \tilde R_{-}(p) &= \diag
 (\underbrace{e^{-i \theta_\lambda +i\theta_\lambda s},\dotsc,e^{-i \theta_\lambda +i\theta_\lambda s}}_{T_pM},\underbrace{1,1}_{\aux}) 
\end{align*}
  \item[\it $|\mu(p)|$ even, type $(1,1)$]
 \begin{align*}
\tilde R_{+}(p) &= \diag
(\underbrace{e^{-i\theta_\lambda s},\dotsc,e^{-i\theta_\lambda s}}_{T_pM},\underbrace{e^{-i\pi s},-e^{-i\pi s}}_{\aux}) \hat * R_{-\pi}(n+1, n+2), \\
  \tilde R_{-}(p) &= \diag
 (\underbrace{e^{-i \theta_\lambda +i\theta_\lambda s},\dotsc,e^{-i \theta_\lambda +i\theta_\lambda s}}_{T_pM},\underbrace{1,-1}_{\aux}) 
\end{align*}
 \item[\it $|\mu(p)|$ odd, type $(1,0)$]
 \begin{align*}
\tilde R_{+}(p) &= 
\diag(\underbrace{e^{-i\theta_\lambda s},\dotsc,e^{-i\theta_\lambda s}}_{T_pM},\underbrace{e^{-i\pi s},1}_{\aux}) \hat* R_{-\pi}(n+1, n+2), \\
  \tilde R_{-}(p) &= \diag
 (\underbrace{e^{-i \theta_\lambda +i\theta_\lambda s},\dotsc,e^{-i \theta_\lambda +i\theta_\lambda s}}_{T_pM},\underbrace{1,-e^{i\pi s}}_{\aux}) 
 \end{align*}
  \item[\it $|\mu(p)|$ odd, type $(0,1)$]
 \begin{align*}
 \tilde R_{+}(p) &= ( R_{\pi}(n, n+2)\cdot\diag(\underbrace{1,\dotsc,1}_{n+1},-1) ) *\\
 &{}\qquad \diag(\underbrace{e^{-i\theta_\lambda s},\dotsc,e^{-i\theta_\lambda s},-e^{-i\theta_\lambda s}}_{T_pM},\underbrace{e^{-i\pi s},1}_{\aux})\hat* \\
&{}\qquad R_{-\pi}(n+1, n+2)\hat*R_{\pi}(n,n+2), \\
  \tilde R_{-}(p) &=  ( R_{-\pi}(n, n+2)\cdot \diag(\underbrace{e^{-i \theta_\lambda },\dotsc,e^{-i \theta_\lambda}}_{T_pM},\underbrace{1,1}_{\aux}) ) * \\
  &{}\quad \diag
 (\underbrace{e^{-i \theta_\lambda +i\theta_\lambda s},\dotsc,e^{-i \theta_\lambda +i\theta_\lambda s},-e^{-i \theta_\lambda +i\theta_\lambda s}}_{T_pM},\underbrace{1,-e^{i\pi s}}_{\aux}) \hat* \\
 &{}\qquad R_{-\pi}(n, n+2).
 \end{align*}
\end{description}
%


We also define two weight vectors $\nu_+(p), \nu_-(p)$ associated to $p$, as follows.
\begin{align*}
 &\nu_+(p)=(\underbrace{\delta,\dotsc,\delta}_{T_pM},\underbrace{-\delta,-\delta}_{\aux}), \quad
 & \nu_-(p) =(\underbrace{\delta,\dotsc,\delta}_{T_pM},\underbrace{\delta,\delta}_{\aux})
\end{align*}
and let $\w_{p,\pm} = \w_{\nu_+(p)}:D_1 \to \R^{n+2}$ be the associated weight function.
Note that we in fact may assume that $\theta_\lambda = 0$.

\begin{lma}\label{lma:special}
If $p$ is a special puncture we get that 

\begin{tabular}{lll}\\
  {\it $|\mu(p)|$ even:} &
$
 \kd_{p,\pm} =0, $ &$ \cd_{p,\pm} \simeq  0
 $\\ & \\
  {\it $|\mu(p)|$ odd:} & 
$
 \kd_{p,\pm}=\aux_2, $ & $\cd_{p,+} \simeq 0 
$ \\ & \\
\end{tabular}

\noindent
where the identifications are given by evaluation at $p$.
\end{lma}
 
 The orientations of $\det \dbar_{p-}$ and $\det \dbar_{p+}$ are defined
 analogous to how it is done for true punctures. That is,
 we will fix an orientation of  $\det \dbar_{p+}$ and let it induce an orientation of $\det \dbar_{p-}$ via the exact sequence
  \begin{align}\label{eq:capseqspec}
 0 \to \Ker \dbar_{p} \to
\begin{bmatrix}
 \kd_{p,+} \\
 \R^n\\
\kd_{p,-}
\end{bmatrix}
\to 
\begin{bmatrix}
 \cd_{p,+} \\
\cd_{p,-}
\end{bmatrix}
\to 
\Coker \dbar_{p} \to 0.
\end{align}
The only difference is that we will use $(-1)^{n \cdot|\mu(p)|}$ times the canonical orientation of $\det \dbar_p$. 
 In Section \ref{sec:elementor} we will describe a choice of orientations of $\det \dbar_{p,+}$ making it possible for us to compute the signs of the rigid trees of $\Lambda$ explicitly.

\subsubsection{Capping orientation of a punctured $J$-holomorphic disk}\label{sec:cappingor}
Now we define the orientation of the linearized $\dbar$-operator at a punctured $J$-holomorphic disk $u$.

The orientation of $\dbar_u$ will be induced by the canonical
orientation of the $\dbar$-operator on the non-punctured disk,
together with the capping orientations of the capping operators of the
punctures of $u$. The  capping operators are glued to $u$ at the
punctures, to obtain the \emph{fully capped disk} $\hat u$ with
corresponding trivialized boundary condition $\hat{\tilde A}_u
$ and $\dbar$-problem 
\begin{equation*}
 \dbar_{\hat u}: \Hi_2[\hat{\tilde A}_u](D, \hat u^*\tilde TT^*M) \to \Hi_1[0](D,T^{*0,1}D \otimes  \hat u^*\tilde TT^*M). 
\end{equation*}
We denote this problem the \emph{fully capped problem corresponding to $u$}. 

The order in which we glue the capping operators to $u$ will be important, and we fix it as follows. Assume that the punctures of $u$ are given by $(q_0,\dotsc,q_m)$, where $q_0$ is the positive puncture and the negative punctures are ordered in the counterclockwise direction along $\partial D_{m+1}$. Then we first glue the capping trivialization $\tilde{R}_+(q_0)$ to $\tilde A_u$ at $q_0$, then $\tilde{R}_-(q_m)$ to $\tilde A_u$ at $q_m$, then $\tilde{R}_-(q_{m-1})$ to $\tilde A_u$ at $q_{m-1}$, and so on, until we get the trivialized boundary condition $\hat{\tilde A}_u$ on the closed disk. 

By the iterated version of Lemma \ref{lma:gluweights}, these gluings induce an exact sequence
\begin{align}
\label{eq:cappingseqtilde}
 0 \to \kd_{\hat{u}} \to
 \begin{bmatrix}
 \kd {q_{1,-}} \\
 \vdots\\
 \kd {q_{m,-}} \\
 \kd {q_{0,+}} \\
 \kd_{u}
 \end{bmatrix} 
 \to 
  \begin{bmatrix}
  \cd {q_{1,-}} \\
  \vdots \\
  \cd {q_{m,-}} \\
  \cd {q_{0,+}} \\
  \cd_{u}
  \end{bmatrix}
  \to  
  \cd_{\hat u} 
  \to 0,
\end{align}
and if we give $\det \dbar_{\hat u}$ the canonical orientation and fix
orientations of the capping operators, we then get an induced
orientation of $\det \dbar_{u}$. Using that the stabilized problem
$\dbar_u$ is split and is an isomorphism in the auxiliary directions,
this in turn induces an orientation of the linearized $\dbar$-operator
at $u$ in a canonical way. This is the \emph{capping orientation of
  $u$}. Using the isomorphism \eqref{eq:conf2} and that we
have fixed an orientation of the space of conformal variations we get
that this induces an orientation of $u$ which can be represented by a
sign $\sigma_{\capp}(u)$,  \emph{the capping sign of $u$}.

\begin{rmk}
 For $\lambda>0$ we can as well use operators defined on unweighted spaces in directions corresponding to $\Lambda$, and get an sequence isomorphic to the sequence \eqref{eq:cappingseqtilde}. 
\end{rmk}

\begin{defi}
 We will use the notation
 \begin{align*}
\Ker \capp(u) &= \bigoplus_{i=1}^m \kd_{q_{i,-}} \oplus \kd_{q_{0,+}}  \\                                                                                  
 \Coker \capp(u) &= \bigoplus_{i=1}^m \cd_{q_{i,-}} \oplus \cd_{q_{0,+}}. 
\end{align*}
\end{defi}
\begin{lma}\label{lma:kerisocap}
Let $u$ be a rigid $J$-holomorphic disk of $\Lambda_\lambda$.  If $\lambda$ is sufficiently small, then
\begin{align*}
\kd_{\hat{ u}} &\simeq \kd_u \oplus \Ker \capp(u) \\
\cd_{\hat{ u}} &\simeq \cd_u \oplus \Coker \capp(u).
\end{align*}
\end{lma}
\begin{proof}
 First of all, notice that all elements in $\Ker \capp(u)$ are contained in the $\aux_2$-direction, and that we do not have any cokernel elements in this direction. Similarly, we have no kernel elements in the $\aux_1$-direction. Also recall that the $\dbar_u$-operator is an isomorphism in these directions. Thus, from the exact sequence \eqref{eq:cappingseqtilde} it then follows that the statement holds for the operators restricted to the auxiliary directions,
and thus it suffices to consider the spaces restricted to the directions corresponding to $\Lambda$. 

If $u$ has more than 1 negative punctures the statement follows from the exact sequence \eqref{eq:cappingseqtilde} together with the fact that both $\kd_u$ and $\Ker \capp(u)$ is trivial in the directions corresponding to $\Lambda$.

If $u$ has 0 or 1 negative punctures, the statement follows from passing to the limit $\lambda =0$, using that the family of $\dbar$-problems we get are all surjective together with the explicit description of the capping operators and the trivialization at the punctures. 
\end{proof}




%
 \section{Capping orientation of flow trees}\label{sec:cutstab}
In this section we define the capping orientation of a (partial) flow tree $\Gamma$. 
We will make use of the constructions in the proof of [\cite{trees}, Theorem 1.3], where the existence of a family
of rigid $J$-holomorphic disks of ${\Lambda}$ converging to $\Gamma$  is established.  This existence is proven by considering a sequence of pre-glued disks $\{w_\lambda\}_{\lambda \to 0}$, which
has explicitly given boundary conditions that are close to the cotangent lift of $\Gamma$, and then using a Picard-Floer lemma to find a true $J$-holomorphic disk $u_\lambda$ arbitrary close to $w_\lambda$ for each $\lambda$ sufficiently small. For details we refer to \cite{trees}.


In Subsection \ref{sec:analys} we relate Sobolev spaces and
$\dbar$-operators to flow trees, using the Lagrangian  boundary
conditions coming from $w_\lambda$. In Subsection \ref{sec:treetriv}
we describe an explicit trivialization of the boundary conditions in neighborhoods of vertices, and then we describe a way to compare this with the trivialization induced by the chosen spin structure. Then we use this in Subsection \ref{sec:capptree}  where we define the capping orientation of a (partial) flow tree, which will be equal to the  capping orientation of the corresponding (piece of the) pre-glued disk $w_\lambda$. Here we consider $\lambda$ as small and fixed. In Section \ref{sec:glupf}
we will take care of the $\lambda$-dependence in this definition, by introducing vector bundles over $(0, \lambda_0]$ for $\lambda_0$ sufficiently small and prove that the capping orientation of $\Gamma$  is well-defined if $\lambda$ is sufficiently small.

There will be a slight difference between the definition of the capping orientation of $w_\lambda$ given in this section and the one given in Section \ref{sec:background}. That is, in this section we will consider the domain of $w_\lambda$ to be given by the standard domain associated to $\Gamma$, while in Section \ref{sec:background} we considered it as given by the punctured unit disk in $\C$. 
In Subsection \ref{sec:endmp} we describe how this affects  orientations.

\subsection{Banach spaces associated to trees}\label{sec:analys}
Let $\Gamma$ be a rigid flow tree of $\Lambda$, and let $w_\lambda$,
$\lambda \to 0$, denote the associated sequence of pre-glued
disks. Here we will consider the domain of $w_\lambda$ to be given by
a standard domain, which we will denote by $\Delta(\Gamma_\lambda)$ to
indicate the $\lambda$-dependence.
 Namely, from [\cite{trees}, Remark 5.34] it follows the
 $\R^{m-2}$-coordinates of the conformal structure of
 $\Delta_{m+1}(\Gamma_\lambda)$ tend to $\infty$ as $\lambda \to
 0$. However, it is proven that $\Gamma$ determines the coordinates of the conformal structures on  $\Delta_{m+1}(\lambda)$ within $O(\log(\lambda^{-1}))$. 

Let $\Theta (\Gamma_\lambda):\partial \Delta(\Gamma_\Lambda) \to
\Lag(n+2)$ be the Lagrangian boundary condition associated to $w_\lambda$.                                                                         
Note that we are considering the boundary condition as extended to the
auxiliary direction.
By \cite{trees} we may assume that $\Theta(\Gamma_\lambda)$ tends to $\R^{n+2}$-boundary conditions as $\lambda\to 0$, except in parts corresponding to ends, switches and $Y_1$-vertices, where it tends to split boundary conditions $U \times \R^{n+1}$, where $U\subset \C$ is a uniform $+ \pi$-rotation given by the matrix $e^{i\pi s}$ acting on $\R \subset \C$, $s\in [0,1]$, if the vertex is an end, and a  uniform $- \pi$-rotation given by the matrix $e^{-i\pi s}$ acting on $\R \subset \C$, $s\in [0,1]$, otherwise.


For  end-pieces we will make a slight adjustment. That is, notice that
if $\Gamma$ has $k>0$ end vertices,  then the punctured unit disk
corresponding to $\Delta(\Gamma_\lambda)$ under the Riemann mapping
theorem has $k$ more marked points than the punctured unit disk
corresponding  to $w_\lambda$ in the setting of Section
\ref{sec:background}. Due to this, we will define capping operators
for end vertices, and we will use these operators to catch the
$\pi$-rotation performed by the boundary condition corresponding to
$w_\lambda$ at the end point. That is, we assume that the stabilized boundary conditions for the end-piece close to the end-vertex is given by $\R^{n+2}$, and we
define the capping operator of the point corresponding to the end
vertex $e$ to be given by
\begin{equation*}
 \dbar_e : \Hi_{2,\nu_e}[\tilde R_e](D_1,\C^{n+2}) \to \Hi_{1,\nu_e}(D_1, \C^{n+2})
\end{equation*}
where, for $\partial D_1 = \{e^{i\pi s}, s \in [0,1)\} $, and where $p$ is the positive special puncture of the end-piece,
\begin{equation*}
 \tilde R_e = \left(R_{-\pi}(n, n+2) \cdot \diag( \underbrace{1,\dotsc ,1}_{n+1},-1) \right) * \diag( \underbrace{1,\dotsc ,1}_{n-1},-e^{i\pi s},1,1), \quad p \text{ of type } (1,0),
\end{equation*}
\begin{equation*}
 \tilde R_e =  \diag( \underbrace{1,\dotsc ,1}_{n-1},e^{i\pi s},1,1) \hat * R_{-\pi}(n, n+2), \quad p \text{ of type } (0,1),
\end{equation*}
and where 
\begin{equation*}
 \nu_e =(\delta,\dotsc,\delta).
\end{equation*}
We assume that $\Sigma$ is projected to $\{x_n=0\}$ in local
coordinates at the end vertex under consideration. See Figure \ref{fig:cusp}.


In Definition \ref{def:endcap} in Section \ref{sec:endmp} below we
will define an orientation of $\det \dbar_e$. Here we notice the following.
\begin{lma}
 We have
 \begin{equation*}
  \dim \kd_e =1, \qquad \dim \cd_e =0
 \end{equation*}
 and the kernel is spanned by $i(z-1)\partial_{x_n}$.
\end{lma}
\begin{proof}
This is similar to the proof of Lemma \ref{lma:capbundle0}. Namely, from [\cite{jholo}, Appendix C.4] it follows that the $\dbar$-problem on the closed disk with trivialized boundary conditions given by $\pm 1$ is surjective with a 1-dimensional kernel spanned by $1$, while the problem with trivialized boundary conditions given by $e^{i\pi s}$ is surjective with a 2-dimensional kernel spanned by $1+z$ and  $i(z-1)$. If we puncture the disk at $z=1$ and put a positive weight there we see that the constant solution in the first case and the solution $1+z$ in the second case is killed. Hence the statement follows.
\end{proof}


We also associate a weight vector $\nu_i$ to each strip-like end of $\Delta(\Gamma_\lambda)$. For regions corresponding to punctures, we use the same vectors as defined in Section \ref{sec:spintriv}.  For regions corresponding to special punctures $p$ we use the following
\begin{equation*}
 \nu_p=(\underbrace{-\delta,\dotsc,-\delta}_{T_pM}, \underbrace{\delta, \delta}_{\aux})
\end{equation*}
if $p$ is positive and 
\begin{equation*}
 \nu_p=(\underbrace{-\delta,\dotsc,-\delta}_{T_pM}, \underbrace{-\delta, -\delta}_{\aux})
\end{equation*}
if $p$ is negative.
For regions corresponding to end vertices we use the following
\begin{equation*}
 \nu_e=(-\delta,\dotsc,-\delta).
\end{equation*}
Thus, to each elementary piece we get a corresponding $\dbar$-problem
which is Fredholm, and these problems can be glued together just as in
the case for $\dbar$-problem associated to $J$-holomorphic disks. This
implies that for each (partial) flow tree $\Gamma$ of $\Lambda$ and
each $\lambda$ sufficiently small we get a corresponding $\dbar$--problem
\begin{align*}
 \dbar_{\Gamma_{\lambda}} :\Hi_{2,\nu}[\Gamma_{\lambda}] \to \Hi_{1,\nu}[\Gamma_{\lambda}], \qquad
 \dbar_{\Gamma_{\lambda}}(v) = dv + J \circ dv \circ i  
\end{align*}
which is Fredholm.
Here the Sobolev spaces $\Hi_{i,\nu}[\Gamma_{\lambda}]$ are defined as
\begin{align*}
 \Hi_{2,\nu}[\Gamma_{\lambda}] &= \Hi_{2,\nu_{\Gamma}}[\Theta(\Gamma_\lambda)](\Delta(\Gamma_\lambda),
                                 w_\lambda^*TT^*M\oplus \R^2),\\
 \Hi_{1,\nu}[\Gamma_{\lambda}] &=
                                 \Hi_{1,\nu_{\Gamma}}[0](\Delta(\Gamma_\lambda),
                                 T^{*0,1}\Delta(\Gamma_{\lambda})\otimes
                                 w_\lambda^*TT^*M\oplus \R^2),
\end{align*}
following the notation of Section \ref{sec:jholo}. Here we simplify notation and use $w_\lambda$ to
denote the part of the pre-glued disk corresponding to $\Gamma$ in the
case when $\Gamma$ is a partial flow tree, and the weight $\nu_\Gamma$ is defined to consist of the weight vectors associated to the punctures of $\Gamma$.

 \subsection{Trivializations associated to trees}\label{sec:treetriv}
Let $\s$ be the chosen spin structure of $\Lambda$ and consider the  stabilized tangent bundle of $\Pi_{\C}(\Lambda)$;
\begin{equation*}
  T \tilde \Lambda := T \Pi_{\C}(\Lambda) \oplus \aux_1 \oplus
  \aux_2.
\end{equation*}
Let $\tilde \Gamma$ be the cotangent lift of $\Gamma$, and let $\tilde
\Gamma^*T\tilde \Lambda \subset \partial \Delta(\Gamma)\times
\C^{n+2}$ denote the Lagrangian pullback bundle over $\partial
\Delta(\Gamma)$. Recall from Section \ref{sec:triv} that $\s$ induces a Lagrangian trivialization of $\tilde
\Gamma^*T\tilde \Lambda$, which we will denote by $A_\s(\Gamma)$. In this section we will define another, default trivialization $A_d(\Gamma)$ and then we will explain how we can compensate for choosing the default trivialization instead of the one induced by the spin structure by a sign. We start with defining explicit trivializations along the parts of $\Delta(\Gamma)$ that correspond to elementary trees, and then we use the edge point regions to glue together the trivializations to a trivialization $A_d(\Gamma)$.


To that end, consider the lift of an elementary tree $\Gamma' \subset \Gamma$, and let $\Lambda_1(\Gamma'), \dotsc,$ $\Lambda_k(\Gamma')$, $k=2,3$ or $4$ depending on the type of elementary tree, denote the corresponding sheets of $\Lambda$. Assume that the $z$-coordinate restricted to $\Lambda_i(\Gamma')$ is greater than or equal to the $z$-coordinate restricted to $\Lambda_j(\Gamma')$ for $i < j$. Let $\tilde \gamma_i$ be the lift of $\Gamma'$ contained in $\Lambda_i(\Gamma')$, and let $\Lambda_i^\C(\Gamma')$, $\gamma_i$ be the Lagrangian projection of $\Lambda_i(\Gamma')$ and $\tilde \gamma_i$, respectively, for $1\leq i \leq k$. 

Recall that in a neighborhood of the cotangent lift of $\Gamma'$ we can identify $T^*M$ with $\C^n$, where the zero section of $T^*M$ is identified with the real part $\R^n$. Also recall that, using this local identification we may assume that $\Lambda_i^\C(\Gamma')$ is close to being a parallel copy of $\R^n$, $1\leq i \leq k$,  in the case when $\Gamma'$ is an elementary tree where the vertex is not contained in $\Pi(\Sigma)$. In the case when the vertex of $\Gamma'$ is contained in $\Pi(\Sigma)$ we may assume that the sheets $\Lambda_i^\C(\Gamma')$, $1 \leq i \leq k$, are close to parallel copies of $\R^n$ away from a neighborhood of $\Pi_\C(\Sigma)$, and that a pair of the sheets meet along $\Pi_\C(\Sigma)$ as in Figure \ref{fig:bend}. That is, these sheets are given by the product of $\R^{n-1}$ and a curve consisting of a half-circle joining two half-lines parallel to the remaining $\R$-factor (smoothed out).
%

On the linearized level, this half-circle gives rise to the positive (in the end-case) or a negative (in the switch and $Y_1$-case) $\pi$-rotation in the $\C$-factor where the half-circle lives. 

Using the local identification of $T^*M$ with $\C^n$ we can use the projection of  $\Lambda_i^\C(\Gamma')$ to $M$ to lift a given trivialization of $TM$ along $\Gamma'$ to a trivialization of  $T\Lambda_i^\C(\Gamma')$ along $\gamma_i$, $1 \leq i \leq k$. We will do this in a certain way, and to explain this we first need to assign signs to the arcs $\gamma_i$, away from cusps. That is, let $\gamma \subset \gamma_i$ be a connected subarc so that  $\Lambda_i^\C(\Gamma')$ is close to be a parallel copy of $\R^n$ in a neighborhood of  $\gamma$, and let 
\begin{equation}\label{eq:sheetsign}
 \sigma(\gamma) =
 \begin{cases}
  1, \text{ if } T\tilde \Pi: T_p\Lambda_i^\C(\Gamma') \to T_{\Pi(p)}M  \text{ is orientation-preserving for } p \in \gamma \\
    -1, \text{if } T\tilde \Pi: T_p\Lambda_i^\C(\Gamma') \to T_{\Pi(p)}M \text{ is orientation-reversing for } p \in \gamma,
 \end{cases}
\end{equation}
where $\tilde \Pi : \Lambda_i^\C(\Gamma') \to M$ is the restriction of the base projection $\Pi: J^1(M) \to M$ to $\Lambda_i^\C(\Gamma')$.

If now  $(\partial_{x_1},\dotsc,\partial_{x_n})$ is a trivialization of $TM|_{\Gamma'} \simeq \Gamma' \times \R^n$, we extend it to a trivialization $(\partial_{x_1},\dotsc,\partial_{x_n}, \partial_{x_{n+1}}, \partial_{x_{n+2}})$ of $T \tilde M|_{\Gamma'}=\Gamma' \times (\R^n \oplus \aux_1 \oplus \aux_2) $ and lift this to a trivialization 
of $\tilde T\Lambda_i(\Gamma')$ along $\gamma_i$, $1 \leq i \leq k$, as follows. 

First of all, let $f_{-1}, f_0, f_{1} \colon [0,1] \to \R$ satisfy the following. There is an $0<\epsilon <1$ so that 
\begin{align*}
 f_{-1}(s) &= -1, s <\epsilon,& f_{-1}(s) &= 1, s >1-\epsilon,\\   
 f_{0}(s) &= 0, s <\epsilon, & f_{0}(s) &= 1, s >1-\epsilon,\\   
 f_{1}(s) &= 1, s <\epsilon, & f_{1}(s) &= 0, s >1-\epsilon,
\end{align*}
and we assume that $f_{-1}, f_0$ are increasing, $f_{1}$ is decreasing. 


%


 \begin{description}
 \item[\it 1-valent true positive punctures] \hfill \\
 In this case we have two cotangent lifts $\gamma_1, \gamma_2$ of $\Gamma'$. Assume that the true puncture is given by $p$.  Lift the trivialization of $T \tilde M |_{\Gamma'}$ described above to a trivialization of $T \tilde \Lambda$ along $\gamma_1$ given by
\begin{equation}\label{eq:triven1}
   (e^{-i \theta_\lambda}\partial_{x_1},\dotsc, e^{-i \theta_\lambda}\partial_{x_k}, e^{-i \theta_\lambda f_{-1}(s)}\partial_{x_{k+1}},\dotsc,e^{-i \theta_\lambda f_{-1}(s)}\partial_{x_n},\partial_{x_{n+1}},\sigma(\gamma_1)\cdot\partial_{x_{n+2}}),
\end{equation}
where we assume $\gamma_1$ to be parametrized by $[0,1]$, and that the parametrization starts at the positive puncture. Here the first $k$ directions correspond to $T_pW^u(p)$. Similarly, we lift this to a trivialization along $\gamma_2$ given by
\begin{equation}\label{eq:triven}
   (\partial_{x_1},\dotsc,\partial_{x_n},\partial_{x_{n+1}},\sigma(\gamma_i)\cdot\partial_{x_{n+2}}),
\end{equation}
$i=2$.

 \begin{rmk}
  Note that $\sigma(\gamma_1)=\sigma(\gamma_2)$ if and only if  $|\mu(p)|$ is even.
 \end{rmk}

 \item[\it 1-valent true negative punctures] \hfill \\
 In this case we have two cotangent lifts $\gamma_1, \gamma_2$ of $\Gamma'$. Assume that the true puncture is given by $p$.  Lift the trivialization of $T \tilde M |_{\Gamma'}$ described above to a trivialization of $T \tilde \Lambda$ along $\gamma_1$ given by
\begin{equation*}
   (e^{-i \theta_\lambda}\partial_{x_1},\dotsc, e^{-i \theta_\lambda}\partial_{x_k}, e^{i \theta_\lambda f_{-1}(s)}\partial_{x_{k+1}},\dotsc,e^{i \theta_\lambda f_{-1}(s)}\partial_{x_n},\partial_{x_{n+1}},\sigma(\gamma_1)\cdot\partial_{x_{n+2}}),
\end{equation*}
where we assume $\gamma_1$ to be parametrized by $[0,1]$, and that the parametrization starts at the special positive puncture. Here the first $k$ directions correspond to $T_pW^u(p)$. The trivialization along $\gamma_2$ is given by \eqref{eq:triven} with $i =2$.

  \item[\it 2-valent positive punctures] \hfill \\
 The trivialization is given similar to the case of 1-valent punctures. The only difference is that the cotangent lift of $\Gamma'$ now consists of three components $\gamma_1, \gamma_2, \gamma_3$. We define the trivialization along $\gamma_1$ to be given by \eqref{eq:triven1}, where now $k=0$, and to be given by  \eqref{eq:triven} with $i =3$ for $\gamma_3$. For $\gamma_2$, we assume the lifted trivialization is given by 
 \begin{equation}\label{eq:triven2}
   ( e^{-i \theta_\lambda f_{0}(s)}\partial_{x_{1}},\dotsc,e^{-i \theta_\lambda f_{0}(s)}\partial_{x_n},\partial_{x_{n+1}},\sigma(\gamma_2)\cdot\partial_{x_{n+2}}),
\end{equation}
where we assume that $\gamma_2$ is parametrized by $[0,1]$, starting at the puncture labeled by $p_1$ in Figure \ref{fig:1}.

  \item[\it 2-valent negative punctures] \hfill \\
 Again, the cotangent lift of $\Gamma'$ now consists of three components $\gamma_1, \gamma_2, \gamma_3$. We define the trivialization along $\gamma_1$ to be given by \eqref{eq:triven1}, where now $k=n$,  to be given by  \eqref{eq:triven2} for $\gamma_2$, and to be given by  \eqref{eq:triven} with $i =3$ for $\gamma_3$.
  
  \item[\it $Y_0$-vertices] \hfill \\
 The trivialization is given similar to the case of 2-valent negative punctures.
 
 \item[\it ends] \hfill \\
Let $e$ denote the end vertex. In this case we have two cotangent lifts $\gamma_1, \gamma_2$ of $\Gamma'$, which meet at $\Pi_\C(\Lambda)$. Recall that we assume $\gamma_1$ to correspond to the lift of $\Gamma'$ with greatest $z$-coordinate. This implies that if we consider $\gamma_1, \gamma_2$ as parameterized by the corresponding parts of $\partial \Delta(\Gamma)$ then we will pass through $\Pi_\C(\Sigma)$ coming from $\gamma_1$ and leaving along $\gamma_2$, if we follow the direction induced by the orientation of $\partial \Delta(\Gamma)$. 

Assume that $T_e\Pi(\Sigma) = \spn(\partial_{x_1},\dotsc,\partial_{x_{n-1}})$, and that the union $\Lambda^\C_1(\Gamma') \cup \Lambda^\C_2(\Gamma')$ splits as $\R^{n-1} \times W \subset \R^{n-1} \times \C$, where $W$ is as in Figure \ref{fig:bend}. Let $\tilde U(e)$ be a neighborhood of the cotangent lift of $e$ so that we can find a neighborhood of $\gamma_i \setminus \tilde U(e) \subset \Lambda^\C_i(\Gamma')$ in which $\Lambda^\C_i(\Gamma')$ is close to a parallel copy of $\R^n$, $i =1,2$, and let $U(e) = \tilde U(e) \cap(\gamma_1 \cup \gamma_2)$. Lift the trivialization of $T \tilde M |_{\Gamma'}$ to a trivialization of $T \tilde \Lambda|_{\gamma_1 \cup \gamma_2}$ as follows. 
  \begin{itemize}
    \item Along $\gamma_1 \setminus U(e)$, assuming that this path is parametrized by $[0,1]$:
  \begin{equation*}  
  (e^{-i \theta_\lambda f_{1}(s)}\partial_{x_1},\dotsc,e^{-i \theta_\lambda f_{1}(s)}\partial_{x_n},\partial_{x_{n+1}}, \sigma(\gamma_i\setminus U(e))\cdot\partial_{x_{n+2}})                                            \end{equation*}
  \item Along $\gamma_2 \setminus U(e)$:
  \begin{equation*}  
  (\partial_{x_1},\dotsc,\partial_{x_{n+1}},\sigma(\gamma_2\setminus U(e))\cdot\partial_{x_{n+2}})                                            \end{equation*}
\item Along  $U(e)$: 
When passing through the cotangent lift of $e$, coming in along $\gamma_1$ and leaving along $\gamma_2$, in the lifted trivialization induced by the chosen trivialization along $\gamma_1 \setminus U(e)$ the $\partial_{x_n}$-vector will perform a positive $\pi$-rotation in the $\C$-component corresponding to $\partial_{x_n}$. We get two different trivializations, depending on the sign of the sheets.

  \item[] $\sigma(\gamma_1 \setminus U(e)) = -1$:
  \begin{equation*}
   \left(R_{-\pi}(n,n+2) \cdot \diag(\underbrace{1, \dotsc,1}_{n+1}, -1) \right) * \diag(\underbrace{1, \dotsc,1}_{n-1}, -e^{i \pi s}, 1,1),
  \end{equation*}
  \item[] $\sigma(\gamma_1 \setminus U(e)) = 1$:
  \begin{equation*}
   \diag(\underbrace{1, \dotsc,1}_{n-1}, e^{i \pi s}, 1,1)  \hat * R_{-\pi}(n,n+2).
  \end{equation*}

%
 \end{itemize}
 
 \begin{rmk}
 Note that $\sigma(\gamma_1\setminus U(e)) \neq \sigma(\gamma_2\setminus U(e))$. 
 \end{rmk}

 \begin{figure}[ht]
\labellist
\small\hair 2pt
\pinlabel $\Pi^{-1}(v)$ [Br] at  45 230 
\pinlabel ${x_1}$ [Br] at 40 70
\pinlabel $x_2$ [Br] at 10 40
\pinlabel $\Pi(\Lambda_v)$ [Br] at 120 15
\pinlabel $M$ [Br] at -50 5
\pinlabel $\Lambda$ [Br] at 80 150
\endlabellist
\centering
\includegraphics[height=5cm]{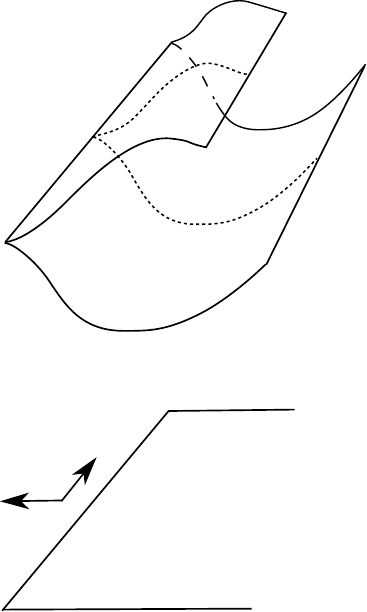}
\caption{Local coordinates at a cusp. Here $v$ represents a switch, end or $Y_1$-vertex, and $\Pi(\Lambda_v)$ is given the orientation from $\partial_{x_1}$.}
\label{fig:cusp}
\end{figure}

 \item[\it switches] \hfill \\
 Let $s$ be the switch-vertex and note that this vertex has two cotangent lifts. Let $\tilde s$ be the lift that is contained in $\Pi_\C(\Sigma)$. We have three components $\gamma_1, \gamma_2, \gamma_3$ of the cotangent lift of $\Gamma'$, where either $\gamma_1$ and $\gamma_2$ or $\gamma_2$ and $\gamma_3$ meet at $\Pi_\C(\Sigma)$. In the case when $\gamma_1$ and $\gamma_2$ meet we pass $\tilde s$  coming from $\gamma_2$ and leaving along $\gamma_1$, and in the other case we pass  $\tilde s$ coming from $\gamma_3$ and leaving along $\gamma_2$. In the first case, we say that we have an \emph{upper switch}, and in the second case we say we have a \emph{lower switch}.

 Consider the case of an upper (lower) switch. Assume that $T_s\Pi(\Sigma) = \spn(\partial_{x_1},\dotsc,\partial_{x_{n-1}})$, and that the union  $\Lambda^\C_1(\Gamma') \cup \Lambda^\C_2(\Gamma')$ ($\Lambda^\C_2(\Gamma') \cup \Lambda^\C_3(\Gamma')$) splits as $\R^{n-1} \times W \subset \R^{n-1} \times \C$, where $W$ is as in Figure \ref{fig:bend}. Note that $\Lambda^\C_3(\Gamma')$  ($\Lambda^\C_1(\Gamma')$) by assumption is close to a parallel copy of $\R^n$.
 
 Let $\tilde U(s) \subset \Lambda^\C_1(\Gamma') \cup \Lambda^\C_2(\Gamma')$ ($\tilde U(s) \subset \Lambda^\C_2(\Gamma') \cup \Lambda^\C_3(\Gamma')$) be a neighborhood of $\tilde s$ so that we can find a neighborhood of $\gamma_i \setminus \tilde U(s) \subset \Lambda^\C_i(\Gamma')$ in which $\Lambda^\C_i(\Gamma')$ is a parallel copy of $\R^n$,  $i =1,2$ ($i =2,3$), and let $U(s) = \tilde U(s) \cap(\gamma_1 \cup \gamma_2)$ ($U(s) =\tilde U(s) \cap(\gamma_2 \cup \gamma_3)$). Lift the trivialization of $T \tilde M |_{\Gamma'}$ to a trivialization of $T \tilde \Lambda|_{\gamma_1 \cup \gamma_2\cup \gamma_3}$ as follows. 
 
 {\bf Upper switch}
  \begin{itemize}
  \item Along $\gamma_1 \setminus U(s)$, parametrized by $s \in [0,1]$ going from the switch to the negative puncture:
  \begin{equation*}  
  (e^{-i \theta_\lambda f_0(s)}\partial_{x_1},\dotsc,e^{-i \theta_\lambda f_0(s)}\partial_{x_n},\partial_{x_{n+1}},\sigma(\gamma_1  \setminus U(s) )\cdot\partial_{x_{n+2}})                                            
  \end{equation*}
   \item Along $\gamma_2 \setminus U(s)$, parametrized by $s \in [0,1]$ going from the positive puncture to the switch:
  \begin{equation*}  
  (e^{-i \theta_\lambda f_1(s)}\partial_{x_1},\dotsc,e^{-i \theta_\lambda f_1(s)}\partial_{x_n},\partial_{x_{n+1}},\sigma(\gamma_2  \setminus U(s) )\cdot\partial_{x_{n+2}})                                            
  \end{equation*}
   \item Along $\gamma_3$:
     \begin{equation*}  
  (\partial_{x_1},\dotsc,\partial_{x_n},\partial_{x_{n+1}},\sigma(\gamma_3)\cdot\partial_{x_{n+2}})                                            
  \end{equation*}

\item Along  $U(s)$, the case when $\sigma(\gamma_2\setminus U(s)) =-1$:
\begin{equation*}
   \left(  R_{-\pi}(n,n+2)\cdot\diag(\underbrace{1, \dotsc,1}_{n+1}, -1)  \right) * \diag(\underbrace{1, \dotsc,1}_{n-1}, -e^{-i \pi s}, 1,1),
\end{equation*}
\item Along  $U(s)$, the case when  $\sigma(\gamma_2\setminus U(s)) =1$:
 \begin{equation*}
   \diag(\underbrace{1, \dotsc,1}_{n-1}, e^{-i \pi s}, 1,1)  \hat * R_{-\pi}(n,n+2).
  \end{equation*}
  \end{itemize}

   {\bf Lower switch}
  \begin{itemize}
   \item Along $\gamma_1$:
     \begin{equation*}  
  (e^{-i \theta_\lambda}\partial_{x_1},\dotsc,e^{-i \theta_\lambda}\partial_{x_n},\partial_{x_{n+1}},\sigma(\gamma_1)\cdot\partial_{x_{n+2}})                                            
  \end{equation*}
 \item Along $\gamma_i \setminus U(s)$, $i=2,3$:
  \begin{equation*}  
  (\partial_{x_1},\dotsc,\partial_{x_n},\partial_{x_{n+1}},\sigma(\gamma_i\setminus U(e))\cdot\partial_{x_{n+2}})                                            
  \end{equation*}
\item Along  $U(s)$, the case when $\sigma(\gamma_3 \setminus U(s)) =-1$:
\begin{equation*}
   \left(R_{-\pi}(n,n+2)\cdot \diag(\underbrace{1, \dotsc,1}_{n+1}, -1)  \right) * \diag(\underbrace{1, \dotsc,1}_{n-1}, -e^{-i \pi s}, 1,1),
\end{equation*}
\item Along  $U(s)$, the case when  $\sigma(\gamma_3 \setminus U(s)) =1$:
 \begin{equation*}
   \diag(\underbrace{1, \dotsc,1}_{n-1}, e^{-i \pi s}, 1,1)  \hat * R_{-\pi}(n,n+2).
  \end{equation*}

\end{itemize}

\item[\it $Y_1$-vertices] \hfill \\
This is similar to the case of switch-vertices, except that we now have two sheets $\Lambda^\C_1(\Gamma')$ and $\Lambda^\C_4(\Gamma')$ which are close to parallel copies of $\R^n$, and two sheets $\Lambda^\C_2(\Gamma')$ and $\Lambda^\C_3(\Gamma')$ meeting at $\Pi_\C(\Sigma)$ and whose union splits as $\R^{n-1} \times W \subset \R^{n-1} \times \C$, where $W$ is as in Figure \ref{fig:bend}. When passing through the cotangent lift of the $Y_1$-vertex $v$ that is contained $\Pi_\C(\Sigma)$, we do this by coming in along $\gamma_3$ and leaving along $\gamma_2$. Let $U(v)$ be a neighborhood of the lift of $v$, similar to how $U(s)$ was defined. The trivialization is now given as follows. 
  \begin{itemize}
   \item Along $\gamma_1$:
 \begin{equation*}  
  (e^{-i \theta_\lambda}\partial_{x_1},\dotsc,e^{-i \theta_\lambda}\partial_{x_n},\partial_{x_{n+1}},\sigma(\gamma_1)\cdot\partial_{x_{n+2}})                                            
  \end{equation*}
  \item Along $\gamma_2 \setminus U(v)$:
  \begin{equation*}  
  (e^{-i \theta_\lambda f_0(s)}\partial_{x_1},\dotsc,e^{-i \theta_\lambda f_0(s)}\partial_{x_n},\partial_{x_{n+1}},\sigma(\gamma_2  \setminus U(v) )\cdot\partial_{x_{n+2}})                                            
  \end{equation*}
   \item Along $\gamma_3 \setminus U(v)$:
  \begin{equation*}  
  (\partial_{x_1},\dotsc,\partial_{x_n},\partial_{x_{n+1}},\sigma(\gamma_3  \setminus U(v) )\cdot\partial_{x_{n+2}})                                            
  \end{equation*}
    \item Along $\gamma_4$:
      \begin{equation*}  
  (\partial_{x_1},\dotsc,\partial_{x_n},\partial_{x_{n+1}},\sigma(\gamma_4)\cdot\partial_{x_{n+2}})                                            
  \end{equation*}
\item Along  $U(v)$, the case when $\sigma(\gamma_3 \setminus U(v)) =-1$:
\begin{equation*}
   \left( R_{-\pi}(n,n+2)\cdot \diag(\underbrace{1, \dotsc,1}_{n+1}, -1) \right) * \diag(\underbrace{1, \dotsc,1}_{n-1}, -e^{-i \pi s}, 1,1),
\end{equation*}
\item Along  $U(s)$, the case when  $\sigma(\gamma_3 \setminus U(v)) =1$:
 \begin{equation*}
   \diag(\underbrace{1, \dotsc,1}_{n-1}, e^{-i \pi s}, 1,1)  \hat * R_{-\pi}(n,n+2).
  \end{equation*}
\end{itemize}
 \item[\it Regions without true vertices] \hfill \\
  The trivialization is given similar to the case of 1-valent positive punctures, where we let $k=n$.
\end{description}


%
%

The following lemma assures that the trivializations above can be glued together to a trivialized Lagrangian boundary condition $A_d(\Gamma)$ along $\partial \Delta(\Gamma)$.
\begin{lma}\label{lma:sumup}
Let $\Gamma$ be a rigid flow tree.
 Then we have that 
 \begin{equation*}
\sum_{q \text{ puncture of } \Gamma }|\mu(q)| \equiv s(\Gamma) + e(\Gamma) + Y_1(\Gamma)  \pmod 2.
 \end{equation*}
\end{lma}

\begin{proof}
Given the tree $\Gamma$, we fix capping paths at all punctures except at the positive puncture $p$.  Concatenating these paths with the $1$-jet lift of $\Gamma$ gives a capping path $\gamma$ for $p$, and from Table \ref{tab:trees} we see that each switch, end and $Y_1$-vertex contributes with $1 \pmod 2$ to the Maslov index of $\gamma$. Thus we get that 
 \begin{equation*}
  \mu(\gamma) \equiv \sum_{q \text{ negative puncture of } \Gamma }|\mu(q)| + s(\Gamma) + e(\Gamma) + Y_1(\Gamma)  \pmod 2.
 \end{equation*}
The result follows.
\end{proof}

%

Note that, in general the trivializations $A_d(\Gamma)$ and $A_\s(\Gamma)$ does not agree. Since the capping orientation of $\Gamma$ depends on the fact that we use the trivialization $A_\s(\Gamma)$ we must compensate with a sign if we use $A_d(\Gamma)$ instead. From [\cite{orientbok}, Theorem 4.29] it follows that the difference in orientation induced by $A_d(\Gamma)$ and $A_\s(\Gamma)$ can be measured by a sign
$\beta(\s,A_d)$, which is computed as follows. 

First homotope the trivialization $A_\s(\Gamma)$ so that it agrees with $A_d(\Gamma)$ at all Reeb chord end points.
 Then, for each maximal connected component $\gamma$ of $\partial\Delta(\Gamma)$, let $A_\s(\gamma)$ denote the trivialization of $\gamma^*T \tilde L$ induced by $A_\s(\Gamma)$, and let $A_d(\gamma)$ denote the trivialization induced by $A_d(\Gamma)$.  By considering the concatenation $ A_\s(\gamma)^{-1}* A_d(\gamma)$ we get a loop of trivializations associated to $\gamma$ (recall that we assume the trivializations to be equal at all infinities of $\Delta(\Gamma)$), and hence an element $\sigma(\gamma) \in \pi_1(SO(n+2)) = \Z_2$. 
 
 \begin{defi}\label{def:triv}
 Let 
\begin{equation}\label{eq:beta}
 \beta(\s, A_d) = \sum_\gamma \sigma(\gamma)
\end{equation}
where the sum is taken over all connected components $\gamma$ of $\partial \Delta(\Gamma)$. For a more detailed description, see [\cite{orientbok}, Section 4.4].  
\end{defi}

\subsection{Capping sequences associated to trees}\label{sec:capptree}
To define the capping orientation of a (partial) flow tree $\Gamma$ we
proceed similar to Section \ref{sec:nonpunct}. The main difference is
that we now have to take the capping operators at the end vertices into
account. This works as follows.

Assume that $\Gamma$ has a positive puncture $q_0$ and negative punctures $q_1,\dotsc,q_m$, ordered as explained
in Section \ref{sec:slit}. Notice that some of the punctures will be
special if $\Gamma$ is a partial flow tree. Also assume that $\Gamma$ has end-vertices
$e_1,\dotsc e_k$ ordered similar as the negative punctures. Let
$\dbar_{q \pm,\lambda}$, $\dbar_{e}$ denote the associated capping
operators as defined in Section \ref{sec:capping} and Section
\ref{sec:analys}, respectively, and let $\hat{\Gamma}_\lambda$ denote
the fully capped trivialized Lagrangian boundary condition associated to $\dbar_{\Gamma_\lambda}$, with the auxiliary directions included. Assume that we glue the capping operators to the $\dbar_{\Gamma_\lambda}$-problem in the order so that we get the induced exact gluing sequence 
\begin{equation}\label{eq:slcap}
 0 \to\Ker \dbar_{\hat{\Gamma}_\lambda} 
\to 
\begin{bmatrix}
 \bigoplus_{i=1}^m \Ker \dbar_{q_i-,\lambda}\\
  \bigoplus_{i=1}^k \Ker \dbar_{e_i}\\
 \Ker \dbar_{q_0+,\lambda}\\
\Ker \dbar_{\Gamma_\lambda}
\end{bmatrix}
\to
\begin{bmatrix}
 \bigoplus_{i=1}^m \cd_{q_i-,\lambda}\\
 \cd_{q_0+,\lambda}\\
\cd_{ \Gamma_\lambda}
\end{bmatrix}
\to
\cd_{\hat{\Gamma}_\lambda}
\to
0,
\end{equation}
which we denote the \emph{capping sequence for
  $\Gamma$}.  

\begin{defi}
The \emph{capping orientation} of $\Gamma$ is given by the orientation of $\det \dbar_{\Gamma_\lambda}$ induced by this sequence, where $\det \dbar_{\hat{\Gamma}_\lambda}$ is given the canonical orientation and the capping operators are given their fixed capping orientation. 
\end{defi}


Similar to the case of disks we will use the notation  
 \begin{align*}
\Ker \capp(\Gamma_\lambda) &= \Ker \capp(\Gamma_\lambda)^-\oplus \Ker \capp(\Gamma_\lambda)^e\oplus\kd_{q_{0,+}}  \\                                                                                  
 \Coker \capp(\Gamma_\lambda) &= \Coker \capp(\Gamma_\lambda)^-\oplus \cd_{q_{0,+}}, 
\end{align*}

where
\begin{align*}
 \Ker \capp(\Gamma_\lambda)^-&=\bigoplus_{i=1}^m \kd_{q_{i,-},\lambda}, \\
 \Ker \capp(\Gamma_\lambda)^e&= \bigoplus_{i=1}^k \Ker \dbar_{e_i},\\
 \Coker \capp(\Gamma_\lambda)^-&=\bigoplus_{i=1}^m \cd_{q_{i,-},\lambda}. 
\end{align*}

\begin{rmk}
 A priori, the capping orientation of $\Gamma$ depends on $\lambda$. In Section \ref{sec:glupf} we will prove that for $\lambda$ sufficiently small, this orientation is well-defined and independent of $\lambda$.
\end{rmk}

\subsection{Viewing ends as punctures, and its affect on orientations}\label{sec:endmp}
Let $\Gamma$ be a rigid flow tree of $\Lambda$, fix $\lambda>0$ and let $w_\lambda$ be the associated pre-glued disk. Recall that when we defined the capping orientation of $w_\lambda$ in Section \ref{sec:nonpunct} we considered the domain of the disk to be given by the unit disk in $\C$, with punctures on the boundary representing the punctures of $w_\lambda$. In Section \ref{sec:analys}, however, we considered the domain of $w_\lambda$ to be given by a standard domain, with punctures on  the boundary representing both the punctures of $w_\lambda$ and also the end-vertices of $\Gamma$. It turns out that the latter setting is more geometrically suitable when it comes to computing the sign of $\Gamma$ explicitly, and thus we need to discuss the affect on orientations of regarding the ends as punctures.

Let $\dbar_{w_\lambda}$ denote the $\dbar$-problem 
\begin{equation*}
 \dbar: \Hi_{2,\nu}[w_\lambda](D_{m+1}, w_\lambda^*TT^*M\oplus \R^2)
 \to \Hi_{1,\nu}[0](D_{m+1}, T^{*0,1}D_{m+1} \otimes w_\lambda^*TT^*M
 \oplus \R^2)
\end{equation*}
as defined in Section \ref{sec:sobolev}, and let $\dbar_{\Gamma_\lambda}$ denote the $\dbar$-problem defined in Section \ref{sec:analys} (where the punctured disk $D_{m+1}$ is replaced by the standard domain $\Delta(\Gamma_\lambda)$). Let $\dbar_{\hat w_\lambda}$ and $\dbar_{\hat \Gamma_\lambda}$ denote the corresponding fully capped problems.

\begin{lma}\label{lma:disktreeiso}
 We have
 \begin{equation*}
  \kd_{\hat w_\lambda} \simeq \kd_{\hat \Gamma_\lambda}, \qquad 
  \cd_{\hat w_\lambda} \simeq \cd_{\hat \Gamma_\lambda}.
 \end{equation*}
\end{lma}
\begin{proof}
 Follows from the definition of the capping operators at end vertices.
\end{proof}

%
%
%
To relate the capping orientation of $\det \dbar_{\Gamma_\lambda}$
with that of $\det \dbar_{w_\lambda}$ we need to fix some
conventions. So assume that $\Gamma$ has $m$ negative punctures and
let $\Co_{m+1}$ be the space of conformal structures for the disk
$D_{m+1}$, where we have not put marked points at the end
vertices. Give $\Co_{m+1}$ the orientation as described in Section
\ref{sec:confvaror}. Similarly, let $\Co_{m+1+e(\Gamma)}$ be the space
of conformal structures for the disk $D_{m+1+ e(\Gamma)}$ where we
have put marked points on $\partial D_{m+1}$ representing the end vertices of $\Gamma$. Note that this domain can be given a conformal structure so that it is conformally equivalent to $\Delta(\Gamma_\lambda)$. If $m <2$ then recall that $\Co_{m+1}$ is empty and that we let $\T_{m+1}$ denote the space of conformal automorphisms of $D_{m+1}$. 
 
 Let $e_i\in \partial D_{m+1+e(\Gamma)}$ correspond to the $i$th end of $\Gamma$, where we have numbered the ends in the order they occur along the boundary of $D_{m+1}$ when going along the boundary in the counter-clockwise direction. Let  $r_i$ be a vector field supported in a neighborhood of $e_i$, being tangent to the boundary of $D_{m+1+e(\Gamma)}$ along the boundary, pointing in the counter-clockwise direction of the boundary. Let $V({e(\Gamma)}) = \spn (r_1,\dotsc,r_{e(\Gamma)})$, oriented. If $m + e(\Gamma)\geq 2$, give $\Co_{m+1+e(\Gamma)}$ the orientation induced by the exact sequence
 \begin{equation}\label{eq:coor}
 0 \to  T\T_{m+1} \to   V({e(\Gamma)}) \to T \Co_{m+1 + e(\Gamma)} \to
 T \Co_{m+1} \to 0,
 \end{equation}
where $\Co_{m+1}$ and $\T_{m+1}$ are oriented as in Section \ref{sec:confvaror}. That means that if $ T\T_{m+1} = 0$ we have
\begin{equation}
 T \Co_{m+1 + e(\Gamma)}= T \Co_{m+1} \oplus  V({e(\Gamma)}),
\end{equation}
oriented, and if $T \Co_{m+1}=0$ that 
\begin{equation}
  V({e(\Gamma)})=   T\T_{m+1} \oplus T \Co_{m+1 + e(\Gamma)},
\end{equation}
oriented.

If $m + e(\Gamma)<2$ and $e(\Gamma)\neq 0$, then notice that we must have $m=0$, $e(\Gamma)=1$, and  $\T_{1+e(\Gamma)}=\T_{2}$ gets its orientation from the exact sequence
\begin{equation}\label{eq:enend}
0 \to V({e(\Gamma)})  \to T\T_{1} \to  T\T_{1+e(\Gamma)} \to 0.
\end{equation}
The reason for this is that the orientation of $\T_{2}$ and $\T_{1}$ by assumption is given as in [\cite{orientbok}, Section 3.4.1], and assuming that the orientation of $\T_{1}$ is represented by $v_1 \wedge v_2$, we see that putting a marked point on $D_1$ turns $v_1$ into a conformal variation and $v_2$ then representes the orientation of $\T_{2}$.

Recall that we in Section  \ref{sec:modspor} and Section
\ref{sec:cappingor} defined the orientation $\sigma_{\capp}(w_\lambda)$ of $w_\lambda$ to be given by comparing the capping orientation of $\cd_{w_\lambda}$ with the orientation of $T\Co_{m+1}$ via the isomorphism \eqref{eq:conf2} in the case when $\dbar_{w_\lambda}$ is injective, and by comparing the orientation of $\kd_{w_\lambda}$ with the orientation of $T\T_{m+1}$ in the case when $\dbar_{w_\lambda}$ is surjective. We define the sign $\sigma_{\pre}(\Gamma_\lambda)$ of $\Gamma_\lambda$ to be given completely analogous, with $\cd_{w_\lambda}$, $\kd_{w_\lambda}$ replaced by $\cd_{\Gamma_\lambda}$, $\kd_{\Gamma_\lambda}$ and $T\Co_{m+1}$, $T\T_{m+1}$ replaced by  $T \Co_{m+1 + e(\Gamma)}$, $T\T_{1+e(\Gamma)}$, respectively. 

Now notice the following.

\begin{lma}\label{lma:endvcon}
We have
\begin{equation}\label{eq:capendori}
\kd_{e_i} \simeq r_i
\end{equation}
with the isomorphism induced by gluing. 
\end{lma}
\begin{proof}
 Assume that the end under consideration is located at $z=-1$ on the boundary of the unit disk in $\C$, and w.l.o.g.\ assume $n=1$. Now we bubble off a disk in a neighborhood of $e_i$, so that the $\pi$-rotation occurring at $e_i$ is made uniformly around this disk. This gives us a standard Maslov problem on the bubbled disk, which we again identify with the unit disk in $\C^n$ with positive puncture at $z=1$. This problem has a solution space spanned by $z+1$ and $i(z-1)$. If we now identify this bubbled off disk with the capping disk at $e_i$, we see that the solution $z+1$ is killed by the positive weight at the positive puncture, and thus we can identify $\kd_e \simeq \spn(i(z-1))$. But the vector $i(z-1)$ is tangent to the boundary at $z=-1$, and thus can be identified with $r_i$ under gluing.
 \end{proof}

\begin{defi}\label{def:endcap}
We define the \emph{capping orientation of $\kd_{e_i}$} to be the orientation so that the identification \eqref{eq:capendori} is orientation-preserving. That is, the capping orientation is represented by $i(z-i)\partial_{x_n}$.
\end{defi}

Using this orientation in the capping sequence \eqref{eq:slcap} we get the following.

\begin{lma}\label{lma:signeq1}
If $\Gamma$ is a rigid flow tree and $\lambda$ is sufficiently small, then $\sigma_{\capp}(w_\lambda) = \sigma_{\pre}(\Gamma_\lambda)$.
\end{lma}
\begin{proof}
 If $m>1$ this follows from directly from the definition of the orientation of $\T_{1+e(\Gamma)}$ and $\Co_{m+1+e(\Gamma)}$ together with a comparison of the sequences \eqref{eq:slcap} and \eqref{eq:cappingseqtilde} for $\Gamma_\lambda$ and $w_\lambda$, respectively.
 
 If $m=0$, then $e(\Gamma)=1$. In this case $\Ker \dbar_{\hat w_\lambda} = \Ker \dbar_{\hat \Gamma_\lambda}$ is spanned by $(z+1)\partial_{x_n}$ and $i(z-1) \partial_{x_n}$, and $\Coker \dbar_{\hat w_\lambda} = \Coker \dbar_{\hat \Gamma_\lambda}$ can be identified with the span of $\partial_{x_1}, \dotsc, \partial_{x_{n-1}}, \partial_{x_{n+1}}$. Without loss of generality we can assume that the orientation of $\det \dbar_{\hat w_\lambda} = \det \dbar_{\hat \Gamma_\lambda}$ is given by 
 \begin{equation*}
  (z+1)\partial_{x_n}\wedge i(z-1) \partial_{x_n} \otimes \partial_{x_1}\wedge \dotsm \wedge \partial_{x_{n-1}} \wedge \partial_{x_{n+1}},
 \end{equation*}
and that the orientation of the determinant line of the positive capping operator of $\Gamma$ is given by 
 \begin{equation*}
 \partial_{x_1}\wedge \dotsm \wedge \partial_{x_{n-1}} \wedge \partial_{x_{n+1}}.
 \end{equation*}
 From the capping sequence for the $\dbar_{w_\lambda}$-problem it follows that the induced orientaton on $\det \dbar_{w_\lambda}$ is given by 
  \begin{equation*}
  (z+1)\partial_{x_n}\wedge i(z-1)\partial_{x_n},
 \end{equation*}
 while the capping sequence for the $\dbar_{\Gamma_\lambda}$-problem induces the orientation 
  \begin{equation*}
  (-1)^{1}\partial_{x_n}
 \end{equation*}
 on  $\det \dbar_{\Gamma_\lambda}$. 
 
 To get $\sigma_{\capp}(w_\lambda)$ we should compare the orientation of $\T_1$ with $ (z+1)\partial_{x_n}\wedge i(z-1)\partial_{x_n}$, and to get $\sigma_{\pre}(\Gamma_\lambda)$ it follows from \eqref{eq:enend} and Definition \ref{def:endcap} that we should compare the orientation of $\T_1$ with $ (-1)^1 i(z-1)\partial_{x_n}\wedge\partial_{x_n}$. Since $\partial_{x_n}\wedge i(z-1)\partial_{x_n}$ correspond to the orientation $(z+1)\partial_{x_n}\wedge i(z-1)\partial_{x_n}$  the result follows.
\end{proof}

\begin{rmk}
     In Section \ref{sec:glupf} we prove that $\sigma_{\pre}(\Gamma_\lambda)$ is independent of $\lambda$ for $\lambda$ sufficiently small.
     \end{rmk}


%

\section{Stabilization of capping sequences}\label{sec:stab}
 Let $\Gamma$ be a (sub) flow tree. We will stabilize the operator
 $\dbar_{\Gamma_\lambda}$  so that it becomes surjective. This will
 allow us to understand the capping orientation of $\Gamma$ as an orientation
 of the flow out $\F(\Gamma)$ of $\Gamma$ in $M$ in the case when
 $\Gamma$ is a sub flow tree, and by a sign if $\Gamma$ is rigid.

To perform this stabilization we use the similar constructions found in \cite{kch} and \cite{trees}. Namely, we define a vector space $V_{\con}(\Gamma_\lambda)$ spanned by vector fields along $\Gamma$ representing the conformal variations. We then add this space to the domain of $\dbar_{\Gamma_\lambda}$, and extend the operator over this space so that it becomes a surjective operator. See Subsection \ref{sec:vectors}. 

This stabilization will affect also the operator $\dbar_{\hat \Gamma_\lambda}$ and the orientations of the $\dbar$-operators in the capping- and gluing sequences. In Subsection \ref{sec:stabsign} and Subsection \ref{sec:glueback} we calculate the signs needed to compensate for the stabilization. It is these signs that show up in the formulas in \cite{korta}.

In this section we will throughout work with problems where we regard the end-vertices of $\Gamma$ as marked points on the domains of the disks.


\subsection{Stabilization by adding conformal variations}
\label{sec:vectors}
Let $\Gamma$ be a rigid tree.
We start by describing how to view representatives of cokernel
elements as vector fields along $\Gamma$.

Let $\{w_\lambda:\Delta(\Gamma_\lambda) \to T^*M\}$ be a  sequence of pre-glued disks converging to
$\Gamma$, and let $\kappa_\lambda$ be the conformal structure of
$\Delta(\Gamma_\lambda)$.  Then there is an $\lambda$-family
of fully linearized operators 
\begin{equation}
D\dbar_{(\Gamma_\lambda,\kappa_\lambda)} := \dbar_{\Gamma_\lambda}\oplus \Psi_{\kappa_\lambda} :  \Hi_{2,\nu}[\Gamma_\lambda] \oplus T_{\kappa_\lambda} \Co_{m+1+e(\Gamma)} \to  \Hi_{1,\nu}[\Gamma_\lambda].
\end{equation}
This stabilization was introduced in Section
\ref{sec:modspor} for the case of $J$-holomorphic disks, and can
be defined in completely the same way when we consider the domain to
be given by a standard domain. If the corresponding $\dbar$-problem is injective this induces an isomorphism 
\begin{equation}\label{eq:conf22}
 T_{\kappa_\lambda}\Co_{m+1+e(\Gamma)} \simeq\Psi_{\kappa_\lambda}(T_{\kappa_\lambda} \Co_{m+1+e(\Gamma)})\simeq \cd_{\Gamma_\lambda},
\end{equation}
where $\Psi_{\kappa_\lambda}$ is given by 
\begin{equation}\label{eq:conf3}
\Psi_{\kappa_\lambda}: {T_{\kappa_\lambda} \Co_{m+1+e(\Gamma)}} \to  \Hi_{1,\nu}[0], \qquad   
\gamma \mapsto J \circ dw_\lambda\circ \gamma.
\end{equation}
Here the conformal variation $\gamma$ is interpreted as  $\gamma= \partial_t
\gamma_t|_{t=0}$ where $\gamma_t$ is a path of endomorphisms on $T\Delta_{m+1+e(\Gamma)}$
satisfying $\gamma_0=i$, $\gamma_t^2=-1$.  

Moreover, let $\gamma_i$ denote the conformal variation corresponding to moving the $i$th boundary minimum of $\Delta(\Gamma_\lambda)$ towards $-\infty$. Following [\cite{trees}, Section 6.3.5], we get that the map \eqref{eq:conf3} can be given as 
\begin{equation}\label{eq:linconfvcon}
 \Psi_{\kappa_\lambda}(\gamma_i) = -\dbar_{\kappa_\lambda} (\beta_i \partial w_\lambda \cdot c \lambda^{-1} \partial_\tau)
\end{equation}
where  $\beta_i$ is a smooth cut-off function supported in a neighborhood of the $i$th boundary minimum, $c>0$ is a small constant, and  $\lambda^{-1}$ is needed to get good asymptotic behavior as $\lambda \to 0$. 

We let $t_i(\lambda) = - \beta_i \partial w_\lambda \cdot c \lambda^{-1} \partial_t$
and we define 
\begin{equation}\label{eq:vcon}
 V_{\con}(\Gamma_\lambda) := \spn(t_2(\lambda),\dotsc,t_{m+e(\Gamma)-1}(\lambda)).
\end{equation}
We give $V_{\con}(\Gamma_\lambda)$ the orientation from this identification.

\begin{lma}\label{lma:wlimit}
 The vector $t_i(\lambda)$ has a well-defined limit as $\lambda \to 0$, realized as a vector field in $M$ along $\Gamma$, pointing in the direction of $\Gamma$ against the flow direction (towards the positive puncture), and supported in a neighborhood of the vertex of $\Gamma$ corresponding to the $i$th boundary minimum. 
\end{lma}

\begin{proof}
 This follows from the fact that $w_\lambda$ is flow-like for small $\lambda$, that the order of the derivative of $w_\lambda$ is $\lambda$, and the scaling by $\lambda^{-1}$. See [\cite{trees}, Section 6].
\end{proof}

Using Lemma \ref{lma:wlimit} we define a vector bundle $V_{\con}(\ul{\Gamma}) \simeq [0, \lambda_0]\times \R^{m+e(\Gamma)-2}$ with fibers $V_{\con}(\Gamma_\lambda)$. We give  $V_{\con}(\ul{\Gamma})$ the topology as described in [\cite{trees}, Section 6]. 
Then the evaluation map from $V_{\con}(\ul{\Gamma})$ composed with the projection to $TM$ is well-defined and continuous, and gives  vector fields in $TM$  along $\Gamma$ tangent to $\Gamma$, supported in neighborhoods of $Y_0$- and $Y_1$-vertices and $2$-valent punctures.

Notice that if we consider the domain of $w_\lambda$ to be given by
the punctured unit disk $D_{m+1+e(\Gamma)}$, using the notation from
Section \ref{sec:endmp}, then $T_{\kappa_\lambda}\Co_{m+1+e(\Gamma)}$ is
  considered to be spanned by the vectors $(\tilde v_3,...,\tilde
  v_{m+e(\Gamma)}) = (v_3, \cdots, v_m, r_1, \cdots,
  r_{e(\Gamma)})$. Thus we get a  a natural vector bundle $ T_{\kappa_\lambda}
      \Co_{m+1}(\ul{\Gamma}) \simeq [0,\lambda_0] \times \R^{m-2}$
      with a trivialization given by $(\tilde v_3,\dotsc,\tilde v_{m+e(\Gamma)}) \mapsto(\partial_{x_1},\dotsc,\partial_{x_{m+e(\Gamma)-2}})$. 
It follows that for each pre-glued disk $w_\lambda$ we can define an isomorphism 
\begin{equation*}\label{eq:etaneed}
 \eta_\lambda: T_{\kappa_\lambda} \Co_{m+1+e(\Gamma)}(\Gamma_\lambda) \to V_{\con}(\Gamma_\lambda)
\end{equation*}
given by
\begin{equation}\label{eq:confvariso}
 \tilde v_j \mapsto \eta_\lambda(\tilde v_j) = t_{j-1}(\lambda), \quad j =
 3,\dotsc,m +e(\Gamma),
\end{equation}
and this map extends to a bundle map ${\eta}: T_{\kappa_\lambda} \Co_{m+1+e(\Gamma)}(\ul{\Gamma}) \to V_{\con}(\ul{\Gamma})$.

 \begin{rmk}
  Since we identify the space $V_{\con}(\Gamma)$ with
  $T_{\kappa_\lambda} \Co_{m+1+e(\Gamma)}$, we sometimes think of
  $V_{\con}(\Gamma_\lambda)$ as spanned by the vectors
  $t_1(\lambda),\dotsc,\hat t_{j}(\lambda),
  \dotsc,t_{m+e(\Gamma)-1}(\lambda)$. The map $\eta$ is extended to
  this setting by defining $\eta_\lambda(\tilde
  v_2)=-\eta_\lambda(\tilde v_1) =t_1(\lambda)$.
 \end{rmk}

  To calculate the determinant of the map \eqref{eq:confvariso}, let
  $q_1,\dotsc, q_m$ denote the negative punctures of $\Gamma$, and
  $e_1,\dotsc,e_{e(\Gamma)}$ denote the end vertices of $\Gamma$. Pick
    indices $k_1<\dotso<k_{e(\Gamma)}, l_1<\dotso<l_m$ so that
    $\{k_1,\dotsc,k_{e(\Gamma)},l_1,\dotsc ,l_m\} = \{1,\dotsc,m+e(\Gamma)\}$
    and so that the ends and punctures of $\Gamma$ are occurring along
    the boundary of $\Delta(\Gamma)$ in a way so that  $p_{l_i}=q_i$,
    $i=1,\dotsc,m$, and $p_{k_i} =e_i$, $i=1,\dotsc,e(\Gamma)$, with
    notation as in Figure \ref{fig:1}. Let $l =m+e(\Gamma)$.
 
For $l\geq 2$ we identify the ordered tuple $(p_1,\dotsc,p_l)$ with the oriented standard basis for $\R^{l}$, so that we can use wedge products of the points.

 \begin{defi}\label{def:musign}
 Let  $\mu(\Gamma)\in \{0,1\}$ satisfy 
   \begin{align*}
  (-1)^{\mu(\Gamma)}  p_1  \wedge p_2 \wedge \dotsm \wedge p_l 
    = p_{l_1} \wedge p_{l_2} \wedge \dotsm \wedge p_{l_m} \wedge p_{k_1} \wedge \dotsm \wedge p_{k_{e(\Gamma)}}.
   \end{align*}

 See Figure \ref{fig:endconf}.
 \end{defi}

 \begin{figure}[ht]
 \labellist
\small\hair 2pt
\pinlabel $q_0$ [Br] at 439 305
\pinlabel $-\tilde v_1$ [Br] at  610 390
\pinlabel $e_1$ [Br] at 392 410
\pinlabel $q_1$ [Br] at 278 459
\pinlabel $q_2$ [Br] at 99 304
\pinlabel $e_2$ [Br] at 177 170 
\pinlabel $\tilde v_4$ [Br] at 208 40
\pinlabel $q_3$ [Br] at 310 155
\pinlabel $\tilde v_5$ [Br] at 440 60
\pinlabel $q_0$ [Br] at 890 240
\pinlabel $e_1$ [Br] at 1940 70
\pinlabel $q_1$ [Br] at 1940 163
\pinlabel $q_2$ [Br] at 1940 252
\pinlabel $e_2$ [Br] at 1940 335
\pinlabel $q_3$ [Br] at 1940 420
\pinlabel $t_1$ [Br] at 1566 113
\pinlabel $t_2$ [Br] at 1394 200
\pinlabel $t_3$ [Br] at 1340 290
\pinlabel $t_4$ [Br] at 1240 371
\endlabellist
\centering
\includegraphics[width=13cm]{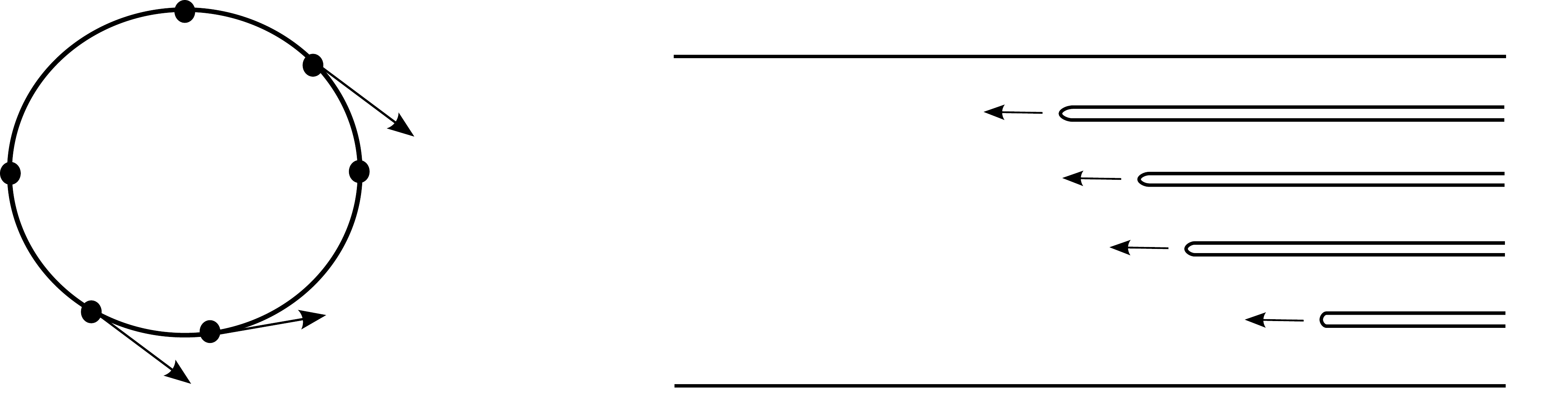}
\caption{In this example we have $l_1=2$, $l_2=3$, $l_3=5$, $k_1=1$, $k_2=4$ and $\mu(\Gamma)=0$, since $e_1 \wedge q_{1} \wedge q_{2} \wedge e_2 \wedge q_3 = q_1 \wedge q_2 \wedge q_3 \wedge e_1 \wedge e_2$. The vectors $t_1,\dotsc,t_4,
  \tilde v_1,\tilde v_4,\tilde v_5$ indicate the conformal variations,
  where $t_1 =-\tilde v_1$, $t_3=\tilde v_4$, $t_4 =\tilde v_5$.}
\label{fig:endconf}
\end{figure}

\begin{lma}\label{lma:end4}
 Let  $\Co_{m+1+e(\Gamma)}$ have the orientation given by \eqref{eq:coor}. Then the isomorphism \eqref{eq:confvariso} is orientation-preserving up to the sign $(-1)^{\mu(\Gamma)}$, 
where $\mu(\Gamma)$ is given in Definition \ref{def:musign}.
\end{lma}

\begin{proof}
We prove the formula in the case $p_1=e_1$, $p_2=e_2$. The other cases are similar.

We must replace $p_{l_1} =q_1,p_{l_2} =q_2 \in \{p_3,\dotsc,p_l\}$ with $p_1=e_1$ and $p_2=e_2$, respectively. From [\cite{orientbok}, Lemma 4.6], we get that
\begin{align*}
 &p_2 \wedge p_2\wedge \dotsm\wedge p_{l_1-1}\wedge p_{l_1+1} \wedge\dotsm\wedge p_{l_2-1}\wedge p_{l_2+1} \wedge\dotsm\wedge p_l \\
 &=(-1)^{l_2-1-l_1} p_3 \wedge\dotsm \wedge p_l.
\end{align*}
The result follows.
\end{proof}

 If $\Gamma' \subset \Gamma$ is a sub flow tree, then 
 let $V_{\con}(\Gamma'_\lambda)$  be the conformal variations of
 $\Gamma'$ inherited by those from ${\Gamma}$. That is,
 $V_{\con}(\Gamma'_\lambda)$ consists of the vectors in
 $V_{\con}({\Gamma}_\lambda)$ having support on the part of the
 standard domain of ${\Gamma}$ corresponding to $\Gamma'$. Let
 $\bm(\Gamma')$ be the number of boundary minima of
 $\Delta(\Gamma')$. Then notice that if $\bm(\Gamma') = 0$ we have
 $V_{\con}(\Gamma') = 0$, and if $0 <\bm(\Gamma') < \bm(\Gamma)$, $\bm(\Gamma)>1$ and
 $\Gamma'$ has a special positive puncture, then $\bm(\Gamma') =
 \dim V_{\con}(\Gamma')$.

 Moreover, assume that $\Gamma'$ is a partial flow tree obtained by
 gluing a sub flow tree $\Gamma_1$ to a $Y_0$- or $Y_1$-piece $\tilde \Gamma$. Then we let 
\begin{equation*}
 V_{\con}(\Gamma') =  V_{\con}(\Gamma_1).
\end{equation*}
Here we have omitted the $\lambda$-dependence to simplify notation.

Next assume that $\Gamma'$ is a sub flow tree obtained by gluing two sub flow
trees $\Gamma_1$ and $\Gamma_2$ together via a $Y_0$- or $Y_1$-piece
$\tilde \Gamma$. Then we let 
\begin{equation*}
 V_{\con,\pre}(\Gamma') =  V_{\con}(\Gamma_1) \oplus  V_{\con}(\Gamma_2).
\end{equation*}
Similarly, if $\Gamma'$ is obtained by gluing  a sub flow tree
$\Gamma_1$ to a $2$-piece $\tilde \Gamma$, we let 
\begin{equation*}
 V_{\con,\pre}(\Gamma') =  V_{\con}(\Gamma_1).
\end{equation*}
Notice that if $\Gamma$ is the true flow tree containing $\Gamma'$ and if
$\bm(\Gamma) = \bm(\Gamma')$, then $V_{\con,\pre}(\Gamma') = V_{\con}(\Gamma')$.

Now let $\Gamma'$ be a rigid flow tree, a sub flow tree, or a sub flow tree glued to a $Y_0$- or $Y_1$-vertex. We define two stabilized operators associated to $\Gamma'$, by 
\begin{align}
 \dbar_{\Gamma'_\lambda,s} &:  \Hi_{2,\nu}[\Gamma'_\lambda] \oplus V_{\con}(\Gamma'_\lambda)  \to \Hi_{1,\nu}[\Gamma'_\lambda], \\
 \dbar_{\Gamma'_\lambda,ps} &: \Hi_{2,\nu}[\Gamma'_\lambda] \oplus V_{\con,\pre}(\Gamma'_\lambda) \to \Hi_{1,\nu}[\Gamma'_\lambda],
\end{align}
acting as $\dbar$ on both spaces. If $\Gamma \supseteq \Gamma'$ is the
rigid flow tree from which $\Gamma'$ inherits its conformal
structures, and if $\bm(\Gamma) = \bm(\Gamma')$, then these two
operators coincide, and if $\Gamma = \Gamma'$ they give isomorphisms. 

 To simplify
notation, if $V_{\con}(\Gamma)$ is empty then we will use the notation
$\dbar_{\Gamma_\lambda,s}$ and $\dbar_{\Gamma_{\lambda}}$
    interchangeably, and similar for the pre-stabilized operator.

The following result is found in \cite{trees}.

\begin{lma}[\cite{trees}, Proposition 6.20]\label{lma:virker1}
Assume that $\Gamma$ is a sub flow tree or a rigid flow tree. If $\lambda$ sufficiently small, then
 the operators  $\dbar_{\Gamma_\lambda,s}$ and $\dbar_{\Gamma_\lambda,ps}$ are surjective Fredholm operators and
 \begin{align*}
  \ind(\dbar_{\Gamma_\lambda,s})&=\ind(\dbar_{\Gamma_\lambda}) + \dim V_{\con}(\Gamma_\Lambda), \\
  \ind(\dbar_{\Gamma_\lambda,ps})&= \ind(\dbar_{\Gamma_\lambda}) + \dim V_{\con,\pre}(\Gamma_\Lambda).
 \end{align*}

 Moreover, if $\Gamma$ is a sub flow tree with special puncture $q$ and $
\ev_q: \Hi_{2,\nu}[\Gamma_\lambda] \oplus V_{\con}(\Gamma_\lambda) \to T_qM, v \mapsto \lim_{z \to \pm \infty} v(z)$, where the limit is taken in the infinite strip associated
 to $q$ in  $\Delta(\Gamma)$ , then  $ev_q$ restricted to $\kd_{\Gamma_{\lambda,s}}$ is injective, and $\ev_q(\kd_{\Gamma_\lambda,s}) \to T_q\F_q(\Gamma)$ as $\lambda \to 0$.
\end{lma}

 \subsection{Stabilized capping orientation}\label{sec:stabsign}
 Let $\Gamma$ be a (sub) flow tree of $\Lambda$, and recall the capping sequence \eqref{eq:slcap} of $\Gamma$ defined in Section \ref{sec:cutstab}. In this section we will examine how the stabilization of the $\dbar_{\Gamma_\lambda}$-operators, as defined in Section \ref{sec:vectors}, affect this sequence. 
 
 First we have to understand how the stabilization affects the fully capped problem $\dbar_{\hat{\Gamma}_\lambda}$. So let $\dbar_{\hat{\Gamma}_{\lambda,s}}$ and $\dbar_{\hat{\Gamma}_{\lambda,ps}}$, respectively, denote the $\dbar$-problems
\begin{align*}
 \dbar_{\hat{\Gamma}_{\lambda,s}} &:  \Hi_{2,\nu}[\hat{\Gamma}_\lambda] \oplus V_{\con}(\Gamma_\lambda)  \to \Hi_{1,\nu}[\hat{\Gamma}_\lambda] \\
 \dbar_{\hat{\Gamma}_{\lambda,ps}} &: \Hi_{2,\nu}[\hat{\Gamma}_\lambda] \oplus V_{\con,\pre}(\Gamma_\lambda) \to \Hi_{1,\nu}[\hat{\Gamma}_\lambda],
\end{align*}
where these operators are obtained by gluing the capping operators to the stabilized and pre-stabilized $\dbar$-problem associated to $\Gamma$. 

To simplify notation we let 
\begin{align*}
  \Ker  {\hat{\Gamma}_{\lambda,s}} &\coloneqq  \kd_{\hat{\Gamma}_{\lambda,s}} &
  \Coker  {\hat{\Gamma}_{\lambda,s}} &\coloneqq  \cd_{\hat{\Gamma}_{\lambda,s}} \\
 \Ker  {\hat{\Gamma}_{\lambda,ps}} &\coloneqq  \kd_{\hat{\Gamma}_{\lambda,ps}} &
  \Coker  {\hat{\Gamma}_{\lambda,ps}} &\coloneqq  \cd_{\hat{\Gamma}_{\lambda,ps}}.
\end{align*}
Thus the stabilized version of the  capping sequence  \eqref{eq:slcap} is given by 
\begin{equation}\label{eq:sslcap}
 0 \to 
 \Ker  {\hat{\Gamma}_{\lambda,s}}  
 \xrightarrow{\alpha_\lambda}
 \begin{bmatrix}
  \Ker \capp(\Gamma_{\lambda}) \\
   \kd_{\Gamma_{\lambda,s}}
 \end{bmatrix}
 \xrightarrow{\beta_\lambda}
    \Coker \capp(\Gamma_{\lambda})
     \xrightarrow{\gamma_\lambda}
 \Coker  {\hat{\Gamma}_{\lambda,s}}  
 \to 0.
\end{equation}
To get a better understanding of this sequence we will make use of the following lemma, which is similar to Lemma \ref{lma:kerisocap}.

\begin{lma}\label{lma:kerisocaptree}
Let $\Gamma$ be either a rigid flow tree or a sub flow tree. Then 
\begin{align}\label{eq:gamma}
&\gamma_\lambda: \Coker \capp (\Gamma_\lambda) \to \Coker \hat \Gamma_{\lambda,s}\\
&\alpha_\lambda:\Ker \hat \Gamma_{\lambda,s} \to \Ker \capp (\Gamma_\lambda)\oplus\kd_{\Gamma_{\lambda,s}}
\end{align}
are isomorphisms if $\lambda$ is sufficiently small.
\end{lma}
\begin{proof}
 First notice that Lemma \ref{lma:kerisocap} holds for pre-glued disks $w_\lambda$ corresponding to sub flow trees $\Gamma_\lambda$, and that also Lemma \ref{lma:disktreeiso} holds in this situation. Thus we get that  
 \begin{align*}
\Ker \hat \Gamma_{\lambda} &\simeq \kd_{\hat w_\lambda} \simeq \Ker \capp (w_\lambda)  \oplus \kd_{w_\lambda}\\
\Coker \hat \Gamma_{\lambda} &\simeq \cd_{\hat w_\lambda} \simeq \Coker \capp (w_\lambda) \oplus \cd_{w_\lambda}.
\end{align*}
Consider first the case when $\Gamma$ is a rigid flow tree, and let
$V_{\con}(\Gamma_\lambda)=\spn(t_2,\dotsc,t_{m-1})$. By construction
of the stabilization there is an $2 \leq l\leq m-1$ such that we might assume that $t_2,\dotsc,t_l$ correspond to the negative punctures of $\Gamma$ in a way so that they kill the space $\cd_{w_\lambda}$. Moreover, the remaining vectors $t_{l+1},\dotsc,t_{m-1}$ can be chosen to correspond to the end-vertices of $\Gamma$ in such a way so that they kill the subspace of $\cd_{\Gamma_\lambda}$ which is mapped onto by $\Ker \capp (\Gamma_\lambda)^e$ as described in Lemma \ref{lma:endvcon}. This means that when we add $t_{l+1},\dotsc,t_{m-1}$ to the domain of the $\dbar_{\hat \Gamma_\lambda}$-problem, the kernel grows in dimension:
\begin{equation*}
 \dim \Ker \hat \Gamma_{\lambda,s} = \dim \kd_{\hat \Gamma_\lambda}
 +m-1-l = \dim \kd_{\hat \Gamma_\lambda} + \dim V_{\con}(\Gamma_\lambda) - \dim \cd_{w_\lambda}.
\end{equation*}
Hence the statement follows from the exact sequence \eqref{eq:sslcap}.

The case when $\Gamma$ is a sub flow tree is similar except that some of the elements in $V_{\con}(\Gamma_\lambda)$ now may contribute to $\kd_{\Gamma_{\lambda,s}}$.
\end{proof}

We use the map $\gamma_\lambda$ to stabilize the $\dbar_{\hat{\Gamma}_{\lambda,s}}$-problem so that it becomes surjective, which gives rise to an exact sequence 
 \begin{equation}\label{eq:exactstabs}
  0 \to  \kd_{\hat{\Gamma}_{\lambda}} \to  \Ker  \hat{\Gamma}_{\lambda,s} \to \Coker \capp(\Gamma_{\lambda}) \oplus V_{\con}(\Gamma_\lambda) \to  \cd_{\hat{\Gamma}_{\lambda}} \to 0.
 \end{equation}
 The canonical orientation of $\det \dbar_{\hat{\Gamma}_{\lambda}}$, together with the fixed orientation of the capping operators and on $V_{\con}(\Gamma_\lambda)$ induces an orientation of $\Ker {\hat{\Gamma}_{\lambda,s}}$ via this sequence. This, in turn, induces an orientation of  $\kd_{\Gamma_{\lambda,s}}$ via the stabilized version of 
the sequence \eqref{eq:sslcap}, which reads 
\begin{equation}\label{eq:ssslcap}
 0 \to 
 \Ker  {\hat{\Gamma}_{\lambda,s}}  
 \xrightarrow{\alpha_\lambda}
 \begin{bmatrix}
  \Ker \capp(\Gamma_{\lambda}) \\
   \kd_{\Gamma_{\lambda,s}}
 \end{bmatrix}
 \to 0.
\end{equation}
This will be called the \emph{stabilized capping orientation} of $\Gamma$, and induces in turn an orientation on $\det \dbar_{\Gamma_\lambda}$ via the exact sequence
 \begin{equation}\label{eq:exactstab}
  0 \to  \kd_{ \Gamma_{\lambda}} \to  \kd_{ \Gamma_{\lambda,s}} \to  V_{\con}(\Gamma_\lambda) \to  \cd_{\Gamma_\lambda} \to 0.
 \end{equation}
 
 In particular, if $\Gamma$ is a rigid flow tree with
 $\kd_{\Gamma_\lambda}=0$, then notice that
 $\kd_{\Gamma_{\lambda,s}}=0$, $\cd_{\Gamma_\lambda}\simeq
 V_{\con}(\Gamma)$ and the stabilized capping orientation is given by
 a sign  $\sigma_{\stab}(\Gamma)$. We denote this sign \emph{the
   stabilized sign  of $\Gamma$}.



\begin{prp}\label{prp:mainstabsign}
Let $w_\lambda$ be the pre-glued disk corresponding to the rigid flow tree $\Gamma$ for $\lambda$ sufficiently small. Then $\sigma_{\capp}(w_\lambda) = (-1)^{\mu(\Gamma)}\sigma_{\stab}(\Gamma)$, where $\mu(\Gamma)$ is given in Definition \ref{def:musign}.
\end{prp}
\begin{proof}
From Lemma \ref{lma:signeq1} we have that
$\sigma_{\capp}(w_\lambda)=\sigma_{\pre}(\Gamma)$ for sufficiently
small $\lambda$. Also notice that we have a commuting diagram
\begin{align}\label{eq:compareorient}
\xymatrix{
0 \ar[r]
&\Ker \dbar_{\Gamma_\lambda} 
\ar[r] 
\ar[d]^{\id}
&\Ker(\dbar_{\Gamma_\lambda} \oplus \Psi_{\kappa_\lambda})
\ar[r]
\ar[d]^{(\id \oplus\eta)}
&  T_{\kappa_\lambda} \Co_{m+1} 
\ar[r]
\ar[d]^{\eta}
& \cd_{\Gamma_\lambda}
\ar[r]
\ar[d]^{\id}
&0
\\
0 \ar[r]
 &\Ker \dbar_{\Gamma_\lambda}  
 \ar[r]
& \Ker \dbar_{\Gamma_{\lambda,s}} 
\ar[r]
&  V_{\con}(\Gamma_\lambda) 
\ar[r]
&\cd_{\Gamma_\lambda}
\ar[r]
&0,
} 
\end{align}
where the top exact sequence is given by \eqref{eq:esdet}  and the
lower one by \eqref{eq:exactstab}. The maps $\eta$ and $\Psi_{\kappa_\lambda}$ are defined in  Section \ref{sec:vectors}.

Now the result follows from Lemma \ref{lma:end4}.

\end{proof}

Next we consider the pre-stabilized case, that is, the case when we have operators $\dbar_{\Gamma_{\lambda,ps}}$ and $\dbar_{\hat{\Gamma}_{\lambda,ps}}$ and when we have a corresponding \emph{pre-stabilized capping sequence}
\begin{equation}\label{eq:pslcap}
 0 \to 
 \Ker  {\hat{\Gamma}_{\lambda,ps}}  
 \xrightarrow{\alpha_\lambda}
 \begin{bmatrix}
  \Ker \capp(\Gamma_{\lambda}) \\
   \kd_{\Gamma_{\lambda,ps}}
 \end{bmatrix}
 \xrightarrow{\beta_\lambda}
    \Coker \capp(\Gamma_{\lambda})
     \xrightarrow{\gamma_\lambda}
 \Coker  {\hat{\Gamma}_{\lambda,ps}}  
 \to 0.
\end{equation}

W.l.o.g. we may assume that 
\begin{equation*}
 V_{\con,\pre}(\Gamma_\lambda) = \spn(t_1(\lambda),\dotsc,t_{j-1}(\lambda),t_{j+1}(\lambda)\dotsc,t_{m-1}(\lambda))
\end{equation*}
following the notation in Section \ref{sec:vectors}. If now
$V_{\con,\pre}(\Gamma_\lambda) \neq V_{\con}(\Gamma_\lambda)$, it
follows from  \cite{trees} that we have
\begin{align}
\label{eq:hatkeriso}
  \kd_{\Gamma_{\lambda,s}} &\simeq \kd_{\Gamma_{\lambda,ps}} \oplus \langle t_j(\lambda) \rangle\\
  \label{eq:hatcokeriso}
  \Coker{\hat{\Gamma}_{\lambda,s}} &\simeq \Coker{\hat{\Gamma}_{\lambda,ps}}. 
\end{align}
We define the orientation of $V_{\con, \pre}(\Gamma_\lambda)$ to be given by
\begin{equation*}
 \Or(V_{\con, \pre}(\Gamma_\lambda)) = t_1(\lambda)\wedge \dotsm \wedge t_{j-1}(\lambda)\wedge t_{j+1}(\lambda)\wedge \dotsm \wedge t_{m-1}(\lambda).
\end{equation*}
Then this fixed orientation together with the canonical orientation of $\det \dbar_{\hat{\Gamma}_{\lambda}}$, the fixed orientations of the capping operators of $\Gamma$ and the isomorphisms \eqref{eq:gamma} and \eqref{eq:hatcokeriso} induces an orientation of $ \Ker  {\hat{\Gamma}_{\lambda,ps}} $ via the exact sequence
 \begin{equation}\label{eq:exactstabspp}
  0 \to  \kd_{\hat{\Gamma}_{\lambda}} \to  \Ker  \hat{\Gamma}_{\lambda,ps} \to \Coker \capp(\Gamma_{\lambda}) \oplus V_{\con,\pre}(\Gamma_\lambda) \to  \cd_{\hat{\Gamma}_{\lambda}} \to 0.
 \end{equation}
This orientation, in turn, induces an orientation of
$\kd_{\Gamma_{\lambda,ps}}$ via the exact sequence \eqref{eq:pslcap},
and this we call the \emph{pre-stabilized capping orientation} of $\Gamma_\lambda$. In addition, this induces an orientation on $\det \dbar_\Gamma$ via the exact sequence 
 \begin{equation}\label{eq:exactstabpre}
  0 \to  \kd_{ \Gamma_{\lambda}} \to  \kd_{ \Gamma_{\lambda,ps}} \to  V_{\con,\pre}(\Gamma_\lambda) \to  \cd_{\Gamma_\lambda} \to 0.
 \end{equation}

\begin{lma}\label{lma:prestab}
Assume that $\kd_{\Gamma_{\lambda,s}}$ is given its stabilized
orientation and that $\kd_{\Gamma_{\lambda,ps}}$ is given its
pre-stabilized orientation. If $V_{\con,\pre}(\Gamma_\lambda) \neq
V_{\con}(\Gamma_\lambda)$, then
\begin{equation*}
 \Or(\kd_{\Gamma_{\lambda,s}}) = (-1)^\sigma \Or(\kd_{\Gamma_{\lambda,ps}})\wedge t_j(\lambda)
\end{equation*}
where $\sigma$ satisfies
\begin{equation*}
 \Or(V_{\con}(\Gamma_\lambda)) = (-1)^\sigma \Or( V_{\con,\pre}(\Gamma_\lambda)) \wedge t_j(\lambda).
\end{equation*}
\end{lma}
\begin{proof}
 Follows by straightforward calculations similar to those in the proof
 of Proposition \ref{prp:mainstabsign}.

\end{proof}

\subsection{The stabilized gluing sign}\label{sec:glueback}
Now we analyze the situation when we glue together two partial sub flow trees $\Gamma_1$ and $\Gamma_2$, to get a tree $\Gamma = \Gamma_1 \# \Gamma_2$. 
Below we assume that the gluing gives the pre-stabilized problem of $\Gamma$. The case when we get the stabilized problem (i.e.\ when $V_{\con}(\Gamma)=V_{\con}(\Gamma_1)\oplus V_{\con}(\Gamma_2)$) is completely similar.

From the gluing we get an exact gluing sequence 
 \begin{equation}
 \begin{array}{ccccccccccc}\label{eq:glu10}
0 
&\to 
&\kd_{\Gamma_\lambda,ps}
&\to
& \begin{bmatrix}
\kd_{\Gamma_{\lambda,1,s}} \\
\kd_{\Gamma_{\lambda,2,s}}
\end{bmatrix}
&\to 
&\R^{n}
&\to 
& 0. 
\end{array}
\end{equation}
In what follows we will drop the $\lambda$-dependence when convenient, referring to Section \ref{sec:glupf} for the motivation of why we can do this if $\lambda$ is sufficiently small.

Let $\Or_{\indu}(\Gamma)$ denote the orientation of $\kd_{\Gamma_{ps}}
$ induced by \eqref{eq:glu10}, assuming that  $\kd_{\Gamma_{i,s}}$ is
given its stabilized capping orientation, $i=1,2$. We are about to
compare this orientation with the pre-stabilized capping orientation
of $\kd_{\Gamma_{ps}}$. The sign we get from this comparison shows up
in the combinatorial formula for the capping orientation of $\Gamma$. 

First let us 
assume that $\Gamma$ is a partial flow tree, that $\Gamma_1$ is either
an elementary tree or a  sub flow tree, and that $\Gamma_2$ is either
an elementary tree or a  sub flow tree glued to a $Y_0$- or
$Y_1$-piece. Also, we assume that the trees are numbered in such a way
so that the positive puncture of $\Gamma$ coincides with the positive
puncture of $\Gamma_2$, and so that $V_{\con,\pre}(\Gamma) =
V_{\con}(\Gamma_1) \oplus V_{\con}(\Gamma_2)$, oriented. This means
that the true punctures of $\Gamma$ are ordered by first taking the
ordered true punctures of $\Gamma_1$ and then the ordered true
punctures of $\Gamma_2$, except in the case when we are gluing a tree
$\Gamma_1$ to a  negative $2$-piece $\Gamma_2$ of type 1, where we
have that the order is given by first taking the true negative puncture of $\Gamma_2$ and then the ordered true punctures of $\Gamma_1$. 
To that end, we let $\tp_2(\Gamma_2)=1$ if $\Gamma_2$ is a $2$-piece
of type 1, and $\tp_2(\Gamma_2)=0$ otherwise. 
 
Let $p$ denote the special puncture at which we are performing the
gluing. 
The sign we get will depend on whether $p$ is the first negative
puncture of $\Gamma_2$ or not, with respect to the standard ordering
of the punctures. See Figure \ref{fig:gluediff}, and we will use the
notation from there.

\begin{figure}[ht]
\labellist
\small\hair 2pt
\pinlabel $p_0$ [Br] at 10 250
\pinlabel $\Gamma_2$ [Br] at 101 300
\pinlabel $\Gamma_1$ [Br] at 280 120
\pinlabel $p$ [Br] at  140 205 
\pinlabel $p_1$ [Br] at 160 290
\pinlabel ${\Gamma_2}$ [Br] at 456 215
\pinlabel $\Gamma_1$ [Br] at 656 305
\pinlabel $p$ [Br] at 570 282
\pinlabel $p_0$ [Br] at 420 250
\pinlabel $p$ [Br] at  160 42 
\pinlabel $p_1$ [Br] at 168 20
\pinlabel $p$ [Br] at  575 25 
\endlabellist
\centering
\includegraphics[height=7cm]{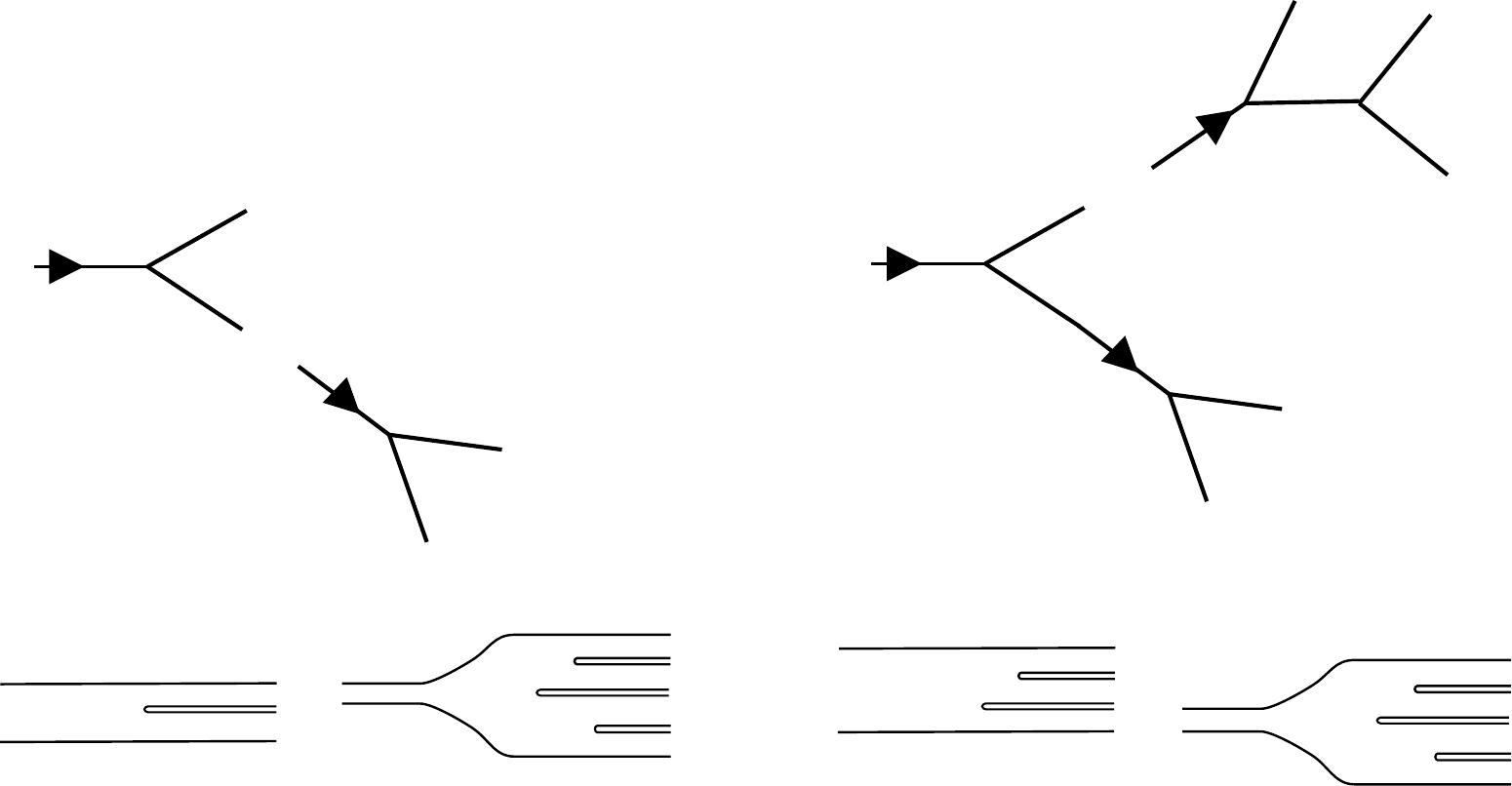}
\caption{The two different types of gluing, with Case 1 on the left hand side and Case 2 on the right hand side.}
\label{fig:gluediff}
\end{figure}

\begin{prp}\label{prp:megalemma4paux}
Let $\kd_{\Gamma_{\lambda,i,s}}$, $i=1,2$, be given its stabilized capping orientation, and assume that ${v_\lambda^1} \wedge \dotsm \wedge {v_\lambda^m}$ represents the orientation of $\kd_{\Gamma_{\lambda,ps}}$ induced by the gluing sequence \eqref{eq:glu10}.
Then the pre-stabilized capping orientation of $\Gamma$ is given by 
 $(-1)^{\sigma}v^1_\lambda \wedge\dotsm\wedge v^m_\lambda$, where $\sigma$ is given by:
 
 Case 1, $p$ is not the first negative puncture of $\Gamma_2$:
 \begin{align*}
 \sigma &= \dim \kd_{\Gamma_1}\cdot [|\mu(p_0)| +  |\mu(p)| +
   |\mu(p_1)|+  \tp_2(\Gamma_2)\cdot n] \\
 &\quad + \dim \Ker \capp (\Gamma_1)\cdot[ 
          |\mu(p_1)|+  \tp_2(\Gamma_2)\cdot n]  \\
 &\quad {}+ [\dim \kd_{\Gamma_1}+\dim \kd_{\Gamma_{1,s}}] \cdot [\dim \kd_{\Gamma_1} + \dim \kd_{\Gamma}]\\
 &\quad {}+[\dim \kd_{\Gamma}+\dim \kd_{\Gamma_{ps}}] \cdot[\dim \cd_{\Gamma} + \dim \cd_{\Gamma_1}  ].
\end{align*}

  Case 2, $p$ is the first negative puncture of $\Gamma_2$:
    \begin{align*}
 \sigma &=  e (\Gamma_1)\cdot[|\mu(p)| + |\mu(p_0)| + \dim \Ker \capp
          (\Gamma_2)+ e(\Gamma_2)]\\
      &\quad {} + \dim \kd_{\Gamma_1} \cdot \dim \Ker \capp (\Gamma_2)\\
  &\quad {}+  (\dim \cd_{\Gamma_1}+n)\cdot \dim \Coker \capp (\Gamma_2)+ n \cdot  \dim \cd_{\Gamma_2} \\
  &\quad {}+ [\dim \kd_{\Gamma_1}+\dim \kd_{\Gamma_{1,s}}] \cdot [\dim \cd_{\Gamma_2} +\dim \kd_{\Gamma_1} + \dim \kd_{\Gamma}]\\
 &\quad {}+ [\dim \kd_{\Gamma_2}+\dim \kd_{\Gamma_{2,s}}] \cdot [\dim \kd_{\Gamma_1} + \dim \kd_{\Gamma} + \dim \kd_{\Gamma_2}]\\
 &\quad {}+[\dim \kd_{\Gamma}+\dim \kd_{\Gamma_{ps}}] \cdot[\dim \cd_{\Gamma} + \dim \cd_{\Gamma_1} +\dim \cd_{\Gamma_2} ].
\end{align*}
\end{prp}

\begin{proof}
Consider the fully capped problem $\dbar_{{\hat \Gamma}}$ associated to $\Gamma$. Glue this problem to the $\dbar$-problem $\dbar_{p}$, where $p$ is the special puncture at which we are gluing $\Gamma_1$ to $\Gamma_2$. Denote the glued problem by $\dbar_G \coloneqq \dbar_{\hat{\Gamma} \#  p}$. Notice that this gluing induces an exact sequence  
\begin{equation}\label{eq:stabdone2}
0 \to
 \Ker \dbar_{G} 
\xrightarrow{\alpha_1}
  \begin{bmatrix}
\kd_{p}\\
\kd_{\hat{\Gamma}}
  \end{bmatrix}
 \xrightarrow{\beta_1}
  \begin{bmatrix}
\cd_{p}\\
\R^n\\
\aux_1\oplus \aux_2\\
\cd_{\hat{\Gamma}}
  \end{bmatrix}
\xrightarrow{\gamma_1}
 \Coker \dbar_{G}
\to 0.
\end{equation}
Also recall the definition of the boundary conditions of the $\dbar_p$-operator, which gives
\begin{align*}
 &\kd_p \simeq \R^n, & \cd_p \simeq 0, & \qquad\text{if } |\mu(p)|=0,\\
  &\kd_p \simeq \R^n\oplus \aux_2^2,  &\cd_p \simeq 0, &\qquad \text{if } |\mu(p)|=1,
\end{align*}
together with an exact sequence
\begin{equation}\label{eq:finnsredan}
 0 \to \kd_p \to
 \begin{bmatrix}
  \kd_{p,+}\\
  \R^n\\
  \kd_{p,-}
 \end{bmatrix}
\to 0 \to 0 \to 0.
\end{equation}
Here we use $\aux_2^2$ to denote a $2$-dimensional kernel in the $\aux_2$-direction. Also notice that we get 
\begin{align}\label{eq:sur}
 \Ker \dbar_{G} \simeq \kd_{\hat{\Gamma}}, \qquad 
 \Coker \dbar_{G} \simeq \cd_{\hat{\Gamma}}\oplus \aux_1 \oplus \aux_2, & \qquad\text{if } |\mu(p)|=0,\\\label{eq:sur2}
  \Ker \dbar_{G} \simeq \kd_{\hat{\Gamma}}\oplus \aux_2, \qquad 
 \Coker \dbar_{G} \simeq \cd_{\hat{\Gamma}}\oplus \aux_1 , & \qquad\text{if } |\mu(p)|=1,
\end{align}
unoriented.

Recall that the orientation of $\kd_{p,{\pm}}$ is chosen so that the orientation of $\det \dbar_p$ induced by the gluing sequence \eqref{eq:finnsredan} is given by $(-1)^{n \cdot |\mu(p)|}$ times the canonical orientation. From [\cite{orientbok}, Lemma 3.11] it then follows that this orientation together with the canonical orientation of $\det \dbar_{\hat \Gamma}$ induces $(-1)^{n \cdot |\mu(p)|}$ times the canonical orientation of $\det \dbar_G$ via the gluing sequence \eqref{eq:stabdone2}. (Notice that we use slightly different orientation conventions compared to \cite{orientbok}, so that we can keep the chosen orientation of $\R^n$).

By the associativity results from [\cite{orientbok}, Section 3.2.3] it follows that we can interpret the sequence \eqref{eq:stabdone2} as
\begin{equation}\label{eq:stabdone22}
0 \to
\kd_G
\xrightarrow{\alpha_1}
  \begin{bmatrix}
  \kd_{p,+}\\
\R^n\\
\kd_{p,-}\\
\Ker \capp (\Gamma)\\
\kd_\Gamma
  \end{bmatrix}
 \xrightarrow{\beta_1}
  \begin{bmatrix}
\R^n\\
\aux_1\oplus \aux_2\\
\Coker \capp (\Gamma)\\
\cd_\Gamma
  \end{bmatrix}
\xrightarrow{\gamma_1}
\cd_{G}
\to 0.
\end{equation}
From the identifications \eqref{eq:sur} and \eqref{eq:sur2}, together
with the nature of the maps $\alpha_1$, $\beta_1$ and $\gamma_1$, it
follows that the restriction and projection of $\beta_1$ to the
$\R^n$-summands may be assumed to be given by $v \mapsto -v$. 
Let us now remove $\R^n$ from the second and third columns by moving it to the bottom, compensating by a sign $(-1)^{\nu_1}$, where
\begin{equation*}
 \nu_1 \equiv n \cdot(|\mu(p)| + \ind(\dbar_{\hat \Gamma})) + n \equiv n \cdot |\mu(p)|  \pmod{2}.
\end{equation*}

It follows that the sequence 
\begin{equation}\label{eq:stabdone23}
0 \to
\kd_G
\xrightarrow{\alpha_1}
  \begin{bmatrix}
  \kd_{p,+}\\
\kd_{p,-}\\
\Ker \capp (\Gamma)\\
\kd_\Gamma
  \end{bmatrix}
 \xrightarrow{\beta_1}
  \begin{bmatrix}
\aux_1\oplus \aux_2\\
\Coker \capp (\Gamma)\\
\cd_\Gamma
  \end{bmatrix}
\xrightarrow{\gamma_1}
\cd_{G}
\to 0
\end{equation}
induces the canonical orientation on $\det\dbar_G$, given that $\aux_1 \oplus \aux_2$  and all the capping spaces are given their fixed orientation, and that $\det \dbar_\Gamma$ is given its capping orientation.

Now use associativity again, to replace $\kd_\Gamma$ and $\cd_\Gamma$ by $\kd_{\Gamma_1}\oplus \kd_{\Gamma_2}$ and $\cd_{\Gamma_1}\oplus \R^n \oplus \cd_{\Gamma_2}$, respectively, using that we have an exact sequence 
\begin{equation}\label{eq:st1}
 0 \to 
 \kd_{ \Gamma}
 \to
 \begin{bmatrix}
 \kd_{{\Gamma}_1}\\
 \kd_{{\Gamma}_2}
 \end{bmatrix}
\to  
  \begin{bmatrix}
 \cd_{{\Gamma}_1}\\
 \R^n\\
 \cd_{{\Gamma}_2}
 \end{bmatrix}
 \to 
\cd_{\Gamma}
\to 0.
\end{equation}

Thus, what we get is an exact sequence
\begin{equation}\label{eq:stabdone234}
0 \to
\kd_G
\xrightarrow{\alpha_1}
  \begin{bmatrix}
  \kd_{p,+}\\
\kd_{p,-}\\
\Ker \capp (\Gamma)\\
\kd_{\Gamma_1}\\
\kd_{\Gamma_2}
  \end{bmatrix}
 \xrightarrow{\beta_1}
  \begin{bmatrix}
\aux_1\oplus \aux_2\\
\Coker \capp (\Gamma)\\
\cd_{\Gamma_1}\\
\R^n\\
\cd_{\Gamma_2}
  \end{bmatrix}
\xrightarrow{\gamma_1}
\cd_{G}
\to 0.
\end{equation}

Let 
\begin{align*}
\Ker \capp(\Gamma_i)^s|_{\Gamma} =  \Ker \capp(\Gamma)^s\cap  \Ker \capp(\Gamma_i)^s \qquad s=-,e,                                                                                   \end{align*}
 and recall that we have to consider two different cases of gluings. In the first case we have
\begin{align*}
 \Ker \capp(\Gamma) &= \Ker\dbar_{p_{1},-} \oplus \Ker\capp (\Gamma_1)^{-}\oplus \Ker\capp (\Gamma_1)^{e} \oplus \kd_{ p_{0},+} , \\
  \Coker\capp(\Gamma) &= \cd_{ p_{1},-}\oplus \Coker\capp(\Gamma_1),\\
 \Ker \capp(\Gamma_1) &= \Ker \capp(\Gamma_1)^- \oplus \Ker\capp (\Gamma_1)^{e}\oplus \Ker \dbar_{ p,+},\\
 \Ker \capp(\Gamma_2) &= \Ker \dbar_{ p_1,-} \oplus \Ker \dbar_{ p,-}\oplus\Ker \dbar_{ p_0,+},\\
  \Coker \capp(\Gamma_2) &=  \Coker \dbar_{p_{1},-}, 
\end{align*}
oriented, and in the second case we have
\begin{align*}
  &\Ker \capp(\Gamma) \\
  &\qquad {} = \Ker\capp (\Gamma_1)^{-} \oplus \Ker\capp (\Gamma_2)^{-}|_\Gamma \oplus \Ker\capp (\Gamma_1)^{e}\oplus \Ker\capp (\Gamma_2)^{e}\oplus \kd_{ p_{0},+}, \\
  &\Coker\capp(\Gamma) =  \Coker\capp(\Gamma_1)\oplus \Coker\capp(\Gamma_2),\\
 &\Ker \capp(\Gamma_1) = \Ker \capp(\Gamma_1)^- \oplus \Ker\capp (\Gamma_1)^{e}\oplus \Ker \dbar_{ p,+},\\
 &\Ker \capp(\Gamma_2) = \Ker \dbar_{ p,-}\oplus \Ker\capp (\Gamma_2)^{-}|_\Gamma \oplus \Ker\capp (\Gamma_2)^{e}\oplus \Ker \dbar_{ p_0,+},
\end{align*}
oriented, using the notation from Figure \ref{fig:gluediff}. Note that $\Coker \dbar_{p_{1-}} $ equals zero if $\Gamma_2$ is not a negative 2-piece of type 1. 

Now define, in the first case,
\begin{equation*}
 \nu_2= \dim \kd_{\Gamma_1}\cdot (|\mu(p_0)| + |\mu(p_1)| + |\mu(p)|) 
 + \dim \Ker \capp (\Gamma_1)  \cdot  |\mu(p_1)|
\end{equation*}
if $\tp_2(\Gamma_2)=0$, and 
\begin{equation*}
 \nu_2= n \cdot\dim \kd_{\Gamma_1} 
 + (n+|\mu(p_1)|)\cdot  \dim \Ker \capp(\Gamma_1)
\end{equation*}
if $\tp_2(\Gamma_2)=1$, and
\begin{align*}
\nu_2&=  \dim \Ker \capp (\Gamma_2)^-|_{\Gamma} \cdot \dim \Ker \capp (\Gamma_1)^e + \dim \kd_{\Gamma_1} \cdot \dim \Ker \capp (\Gamma_2)\\
  &\quad {}+  (\dim \cd_{\Gamma_1}+n)\cdot \dim \Coker \capp (\Gamma_2)                       
\end{align*}
in the second case. It follows that if we compensate by $(-1)^{\nu_2 }$, then we can interpret the sequence \eqref{eq:stabdone234} as  
\begin{equation}\label{eq:stabdone2345}
0 \to
\kd_G
\xrightarrow{}
  \begin{bmatrix}
  \Ker \capp (\Gamma_1) \\
  \kd_{\Gamma_1}\\
\Ker \capp (\Gamma_2)\\
\kd_{\Gamma_2}
  \end{bmatrix}
\to 
  \begin{bmatrix}
  \Coker \capp (\Gamma_1)\\
  \cd_{\Gamma_1}\\
  \R^n\\
\aux_1\oplus \aux_2\\
\Coker \capp (\Gamma_2)\\
\cd_{\Gamma_2}
  \end{bmatrix}
\xrightarrow{}
\cd_{G}
\to 0.
\end{equation}
From [\cite{orientbok}, Lemma 3.11] it follows that this sequence induces the canonical orientation of  $\det \dbar_G$, given that the capping operators and $\R^n\oplus \aux_1 \oplus \aux_2$ have their fixed orientation and that $\det \dbar_{\Gamma_i}$ has its capping orientation, $i=1,2$.

Thus, what we have computed so far is that the difference between the capping orientation of $\det \dbar_{\Gamma}$ and the orientation $\Or_{\indu}(\Gamma)$ of $\det \dbar_\Gamma$ induced by the sequence \eqref{eq:st1} from the capping orientations of $\det \dbar_{\Gamma_i}$, $i=1,2$, is given by $(-1)^{\nu_2}$. 
It remains to compute the difference between the orientation of $\kd_{ \Gamma_{ps}}$ given by the sequence \eqref{eq:glu10} and the orientation of $\kd_{ \Gamma_{ps}}$ induced by the sequence \eqref{eq:exactstabpre} from the orientation $\Or_{\indu}(\Gamma)$. The sought sign will then be given by this difference times $(-1)^{\nu_2}$.


To that end, from \cite{quilt} it follows that we can rewrite the sequence \eqref{eq:st1} as 
\begin{equation*}
 0 \to 
 \kd_{ \Gamma}
\xrightarrow{\alpha}
 \begin{bmatrix}
 \kd_{{\Gamma}_1}\\
 \kd_{{\Gamma}_2}
 \end{bmatrix}
\xrightarrow{\beta} \R^n
\xrightarrow{\gamma}
\cd_{\Gamma}
\xrightarrow{\delta}
  \begin{bmatrix}
 \cd_{{\Gamma}_1}\\
 \cd_{{\Gamma}_2}
 \end{bmatrix}
\to 0.
\end{equation*}
This means that we can choose a splitting of $\R^n$ into 
\begin{equation}\label{eq:st4}
 \R^n = C \oplus K
\end{equation}
with corresponding projections $\pi_A: \R^n \to A$, $A=C,K$, so that 
\begin{equation}\label{eq:st3}
 0 \to 
 \kd_{ \Gamma}
\xrightarrow{\alpha}
 \begin{bmatrix}
 \kd_{{\Gamma}_1}\\
 \kd_{{\Gamma}_2}
 \end{bmatrix}
\xrightarrow{\pi_K \circ\beta} K
\to 0
\end{equation}
is exact, and so that we have an identification
\begin{equation}\label{eq:st2}
 \cd_\Gamma \simeq \cd_{\Gamma_1} \oplus \cd_{\Gamma_2} \oplus C,
\end{equation}
where we might assume that $\gamma|_C$ and $\delta$ both are given by the identity.

Moreover, we see that if we fix an orientation of $C$ so that \eqref{eq:st2} is orientation-preserving, assuming that we have fixed orientations of $\cd_{\Gamma_i}$, $i=1,2, \emptyset$, and if we give $K$ the orientation induced by \eqref{eq:st4}, then the sequence \eqref{eq:st3} induces the same orientation on $\kd_\Gamma$ as the sequence \eqref{eq:st1} does, times $(-1)^{\sigma_1}$, where 
\begin{equation*}
 \sigma_1 = n \cdot  \dim \cd_{\Gamma_2}.
\end{equation*}
This sign comes from moving $\R^n$ to the bottom in the sequence \eqref{eq:st1}.

Next consider the sequence \eqref{eq:exactstab} for $\Gamma_1$ and $\Gamma_2$, and the sequence \eqref{eq:exactstabpre} for $\Gamma$. It follows that we can find splittings 
\begin{align*}
 & \Ker \Gamma_{ps} = \kd_{\Gamma} \oplus \vk(\Gamma)_{ps}\\
 & \Ker \Gamma_{i,s} = \kd_{\Gamma_i} \oplus \vk(\Gamma_i), \quad i=1,2,
\end{align*}
such that 
\begin{align*}
 &V_{\con}(\Gamma_i) \simeq \cd_{\Gamma_i} \oplus \vk(\Gamma_i),\quad i=1,2,\\
 & V_{\con,\pre}(\Gamma) \simeq \cd_{\Gamma} \oplus \vk(\Gamma)_{ps} \simeq \cd_{\Gamma_1} \oplus \cd_{\Gamma_2} \oplus C \oplus \vk(\Gamma)_{ps}.
\end{align*}

Using the assumption that
\begin{equation*}
 V_{\con,\pre}(\Gamma) = V_{\con}(\Gamma_1) \oplus V_{\con}(\Gamma_2)
\end{equation*}
oriented, we see that 
\begin{equation}\label{eq:st5}
C \oplus \vk(\Gamma)_{ps} \simeq \vk(\Gamma_1) \oplus \vk(\Gamma_2),
\end{equation}
and this is an oriented identification up to a sign $(-1)^{\sigma_2}$, where 
\begin{equation*}
 \sigma_2 = \dim \cd_{\Gamma_2} \cdot \dim \vk(\Gamma_1).
\end{equation*}

Now it remains to compute the oriented identification of these spaces as induced by \eqref{eq:glu10}. That is, let us rewrite \eqref{eq:glu10} as 
\begin{equation*}
  0 \to 
  \begin{bmatrix}
   \kd_\Gamma\\
   \vk(\Gamma)_{ps}
  \end{bmatrix}
  \to
  \begin{bmatrix}
   \begin{pmatrix}
   \kd_{\Gamma_1}\\
   \vk(\Gamma_1)
   \end{pmatrix}\\
   \begin{pmatrix}
      \kd_{\Gamma_2}\\
   \vk(\Gamma_2)
    \end{pmatrix}
  \end{bmatrix}
  \to
  \begin{bmatrix}
   C \\K
  \end{bmatrix}
\to 0,
\end{equation*}
and then use \eqref{eq:st1} to see that there is a splitting $\kd_{\Gamma_1} =\kd_\Gamma \oplus C_1$ for some subspace $C_1$. Then we reorder the spaces to get the sequence 
\begin{equation*}
   0 \to 
  \begin{bmatrix}
   \kd_\Gamma\\
   \vk(\Gamma)_{ps}
  \end{bmatrix}
  \to
  \begin{bmatrix}
   \kd_{\Gamma}\\
   \vk(\Gamma_1)\\
   \vk(\Gamma_2)\\
   C_1\\
   \kd_{\Gamma_2}
  \end{bmatrix}
  \to
  \begin{bmatrix}
   C \\K
  \end{bmatrix}
\to 0,
\end{equation*}
which costs $(-1)^{\sigma_3}$,
\begin{multline*}
 \sigma_3 = \dim \vk(\Gamma_1) \cdot (\dim \kd_{\Gamma_1} + \dim \kd_{\Gamma})\\
 + \dim \vk(\Gamma_2) \cdot (\dim \kd_{\Gamma_1} + \dim \kd_{\Gamma} + \dim \kd_{\Gamma_2}).
\end{multline*}Now, this sequence induces the same orientation on $\kd_\Gamma$ as the sequence \eqref{eq:st3} does, in the same time as it induces the orientation 
\begin{equation*}
 \vk(\Gamma)_{ps} \oplus C =  \vk(\Gamma_1) \oplus \vk(\Gamma_2)
\end{equation*}
 on $\vk(\Gamma)_{ps} \oplus C $. Thus, if we add the sign $(-1)^{\sigma_4}$,
 \begin{equation*}
  \sigma_4 = \dim  \vk(\Gamma)_{ps} \cdot \dim C,
 \end{equation*}
we get that the orientation induced on $\kd_{\Gamma_{ps}}$ from the sequence \eqref{eq:glu10} differs from the orientation induced by \eqref{eq:st1} and \eqref{eq:exactstabpre} by a sign $(-1)^{\nu_3}$, where 
\begin{align*}
 \nu_3& = \sigma_1+\sigma_2+\sigma_3+\sigma_4  = n \cdot  \dim \cd_{\Gamma_2} \\
&\quad {}+ \dim \vk(\Gamma_1) \cdot [\dim \cd_{\Gamma_2} +\dim \kd_{\Gamma_1} + \dim \kd_{\Gamma}]\\
 &\quad {}+ \dim \vk(\Gamma_2) \cdot [\dim \kd_{\Gamma_1} + \dim \kd_{\Gamma} + \dim \kd_{\Gamma_2}]\\
 &\quad {}+\dim  \vk(\Gamma)_{ps} \cdot[\dim \cd_{\Gamma} + \dim \cd_{\Gamma_1} +\dim \cd_{\Gamma_2} ],
\end{align*}
using that 
\begin{equation*}
 \dim C \equiv \dim \cd_{\Gamma} + \dim \cd_{\Gamma_1} +\dim \cd_{\Gamma_2} 
\end{equation*}
modulo 2.

The result follows.
\end{proof}

We will also need the similar result when we glue a sub flow tree or
an elementary tree $\Gamma_1$ to a sub flow tree $\Gamma_2$ with true
positive puncture $a$ and special negative puncture $p$, say. We will
focus on the case when the true vertices of $\Gamma_2$ are given by
$k\geq 0$ switches and $a$. We have
the following.

\begin{prp}\label{prp:sista}
 Let $\Gamma_1$ be a sub flow tree or an elementary switch tree with
 special positive puncture $p$, and let $\Gamma_2$ be a sub flow tree
 with a true positive 1-valent puncture $a$, $k\geq 0$ switches,
 negative special puncture $p$, and no other vertices.

 Let $\kd_{\Gamma_{\lambda,i,s}}$, $i=1,2$, be given its stabilized capping orientation, and assume that ${v_\lambda^1} \wedge \dotsm \wedge {v_\lambda^m}$ represents the orientation of $\kd_{\Gamma_{\lambda,s}}$ induced by the gluing sequence \eqref{eq:glu10}.
Then the stabilized capping orientation of $\Gamma$ is given by 
 $(-1)^{\sigma}v^1_\lambda \wedge\dotsm\wedge v^m_\lambda$, where
 $\sigma$ is given by
\begin{align*}
 \sigma &= [|\mu(a)| +k+\dim \kd_{\Gamma_{1,s}}+1] \cdot \dim
 \kd_{\Gamma_1}\\
 &\quad{}+ [\dim\cd_{\Gamma_1} +n]\cdot[\dim W^u(a) + n + |\mu(a)|].
\end{align*}
\end{prp}
\begin{proof}
 This follows similar to the arguments in the proof of Proposition \ref{prp:megalemma4paux}.
\end{proof}

\begin{defi}
 We will denote the sign $\sigma$ from Proposition
 \ref{prp:megalemma4paux} and Proposition \ref{prp:sista} the \emph{gluing sign}. 
\end{defi}

To simplify the expression of the gluing sign, we will make use of the following observations.

\begin{lma}\label{lma:cokersimp}
If $\Gamma$ is a sub flow tree with a positive special puncture $p$
and with $\bm(\Gamma) = \dim V_{\con}(\Gamma)$, then
\begin{equation*}
 \dim \Coker \capp (\Gamma) \equiv \dim \kd_{\Gamma,s} +n+1 + |\mu(p)| \pmod{2}.
\end{equation*}
\end{lma}

\begin{proof}
By the dimension formula for trees we have
\begin{align*}
 \dim \kd_{\Gamma,s} \equiv \sum_{b \in \Pu^-(\Gamma)} \dim W^s(b) + (n+1)\cdot \bm(\Gamma) + n \cdot e(\Gamma) + Y_1(\Gamma) + s(\Gamma) 
\end{align*}
modulo 2, and 
\begin{multline*}
 \dim \Coker \capp (\Gamma) \equiv \sum_{b \in \Pu^-(\Gamma)} (\dim W^s(b) + n+1+ |\mu(b)|) \\
 \equiv \sum_{b \in \Pu^-(\Gamma)} \dim W^s(b) + (n+1)\cdot \bm(\Gamma)+ (n+1) \cdot (e(\Gamma)+1) + \sum_{b \in \Pu^-(\Gamma)}|\mu(b)| \pmod{2}.
\end{multline*}
Now the result follows since 
\begin{equation*}
 \sum_{b \in \Pu(\Gamma)}|\mu(b)| \equiv e(\Gamma) + Y_1(\Gamma) + s(\Gamma) \pmod{2}.
\end{equation*}
\end{proof}


%
 \section{Vector bundles and exact sequences}\label{sec:glupf}

In this section we prove Theorem \ref{thm:mainthm}. In particular, we prove that the capping orientation of $\Gamma$ is well-defined for $\lambda$ sufficiently small. We also prove that we can pass to the limit $\lambda =0$ in the gluing sequence \eqref{eq:glu10} and understand the orientation of the trees as oriented subspaces in $M$ via evaluation.

In Subsection \ref{sec:glulim} we prove that the exact gluing sequences introduced in Section \ref{sec:cutstab} and Section \ref{sec:stab} can be interpreted as sequences of vector bundles over the parameter space $\lambda \in (0,\lambda_0]$ for $\lambda_0$ sufficiently small. 
In Subsection \ref{sec:elementor} we prove
that the kernels of the $\dbar$-operators corresponding to elementary
trees give oriented vector bundles over $(0, \lambda_0]$ for
$\lambda_0$ sufficiently small, where the orientations of the fibers
coincide with the capping orientation. This uses results from Appendix
\ref{app:mini}. Together with the results in Subsection
\ref{sec:glulim} and Section \ref{sec:stab} this implies that we can
calculate the capping orientation of $\Gamma$ via a combinatorial
algorithm in combination with oriented intersections of flow manifolds
in $M$.

The remaining statements in Theorem \ref{thm:mainthm} are proven in
Subsection  \ref{sec:proof}. That is, we show that the capping orientation of the pre-glued disk $w_\lambda$ equals the capping orientation of the true $J$-holomorphic disk corresponding to $\Gamma$ for $\lambda$ sufficiently small.

We conclude this section by Subsection \ref{sec:stabstab}, where we indicate the procedure of calculating orientations by gluing the elementary pieces back together, and how the signs from Section \ref{sec:stab} come into play. 

\subsection{Stabilized gluing sequences in the limit}\label{sec:glulim}
Recall that if $p$ is a special puncture of a sub flow tree $\Gamma$, then we have a well-defined evaluation map $\ev_p:
\kd_{\Gamma_{\lambda},s} \to T_pM$ for  all $\lambda$ sufficiently small, so that $\lim_{\lambda \to 0}\ev_p(\kd_{\Gamma_{\lambda},s}) = T_p\F_p(\Gamma)$, where $\F_p(\Gamma)$ denotes the flow-out of $\Gamma$ at $p$. See Section \ref{sec:geomint} and Lemma \ref{lma:virker1}.

\begin{rmk}
Note that the evaluation map in fact takes value in planes
$\Theta_\lambda(p)\R^n$, but by the limiting process $\lambda \to 0$ we can
canonically identify it with $T_pM$.
\end{rmk}

Let $\Gamma_1$ and $\Gamma_2$ be two partial flow trees satisfying the
assumptions of Proposition \ref{prp:megalemma4paux} or Proposition \ref{prp:sista}, and assume that $\ev_p:\kd_{\Gamma_{i,s,\lambda}} \to T_p\F_p(\Gamma_i)$ is an isomorphism in the limit, $i=1,2$.  Here $p$ is the point at which $\Gamma_1$ and $\Gamma_2$ are glued. For example, this holds if the trees are sub flow trees or elementary trees. We will show that the gluing sequence \eqref{eq:glu10} can be understood as an exact sequence of subspaces in $T_pM$. We make the following definition.

\begin{defi}
 The \emph{stabilized gluing sequence in the limit} is given by
  \begin{equation}
 \begin{array}{ccccccccc}\label{eq:glu10l}
0 
&\to 
& T_p\F_p(\Gamma_1)\cap T_p \F_p(\Gamma_2)
&\to
& \begin{bmatrix}
T_p\F_p(\Gamma_1)  \\
T_p\F_p(\Gamma_2) 
\end{bmatrix}
&\to 
& T_pM 
&\to 
& 0 \to 0\\
& &v &\mapsto &(v,v) &   & & &\\
& & & & (u,w) &\mapsto & w-u.  & & 
\end{array}
\end{equation} 
\end{defi}

\begin{rmk}
  In some cases we will replace the tangent space of the flow-out of $\Gamma_i$ in
  \eqref{eq:glu10l} with the the tangent space of the intersection manifold, for $i=1$ or $i=2$.
\end{rmk}

Let now $\Gamma$ be either a rigid flow tree or a tree that satisfies
the assumptions of Proposition \ref{prp:megalemma4paux} or
Proposition \ref{prp:sista}. In the 2 latter
cases we let $p$ be a special puncture of $\Gamma$. To prove that the capping orientation of $\Gamma$ is well-defined, and that it can be computed via the stabilized gluing sequence in the limit, we construct certain vector bundles over the parameter space $\lambda$. 

First consider the capping sequence \eqref{eq:slcap}, which a priori seems to depend on $\lambda$. But, since the orientation of the fully capped problem is preserved under homotopies of the
trivialized boundary condition, and since $\Pi_\C(\Lambda_\lambda)$ is arbitrarily close to the
$0$-section in $T^*M$,  
we can homotope the boundary conditions to be
constantly equal to $\R^n$ for this problem, except for in uniformly small neighborhoods of ends, switches and $Y_1$-vertices where we homotope the boundary conditions to have a $\pm \pi$-rotation, like in Figure \ref{fig:bend}. This homotopy
changes the boundary conditions for the capping operators and for the
linearized $\dbar$-problem corresponding to $\Gamma$. We claim that, with our choices of weights on the Sobolev spaces corresponding to the $\dbar_{\Gamma_\lambda}$-problem and to the capping operators, we can use the determinant
lines over these perturbed problems to orient $\det \dbar_{\Gamma_\lambda}$ for $\lambda$ sufficiently small.

\begin{rmk}
In the following we will focus on what is happening in the directions corresponding to $\Lambda$ as $\lambda \to 0$, since the trivializations in the auxiliary directions are independent of $\lambda$. 
\end{rmk}

To make the above rigorous, we will define vector bundles 
\begin{align}
 &\pi_{\hat k}: \Ker  \ul{\hat\Gamma} \to (0, \lambda_0], & \pi_{\hat k}^{-1}(\lambda) &= \kd_{\hat \Gamma_\lambda}, \label{eq:hkb}\\  
  &\pi_{k}: \Ker \ul{ \Gamma} \to (0, \lambda_0], & \pi_{ k}^{-1}(\lambda) &= \kd_{ \Gamma_\lambda} \label{eq:kb},\\
   &\pi_{c}: \Coker \ul{ \Gamma} \to (0, \lambda_0], & \pi_{c}^{-1}(\lambda) &= \cd_{\Gamma_\lambda} \label{eq:cb},\\ 
   &\pi_{\hat c}: \Coker\ul{\hat \Gamma}  \to (0, \lambda_0], & \pi_{\hat c}^{-1}(\lambda) &= \cd_{\hat \Gamma_\lambda}, \label{eq:hcb}
\end{align}
 for some $\lambda_0>0$ sufficiently small. Note that the fact that
 $\dbar_{\Gamma_\lambda}$, $\dbar_{\hat \Gamma_\lambda}$ are Fredholm
 only assures that $\det \dbar_{\Gamma_\lambda}$, $\det \dbar_{\hat
   \Gamma_\lambda}$ give vector bundles over $(0, \lambda_0]$, but tells us nothing about whether the kernels and cokernels constitute bundles. We will prove that if $\lambda$ is sufficiently small, then in fact the dimension of the kernels and the cokernels of these operators are independent of $\lambda$, and we indeed get vector bundles as above. This will be proved via an inductive procedure where we glue together the elementary trees of the tree.

 To that end, we will interpret the $\dbar$-operators as bundle maps
\begin{align*}
 \dbar_\Gamma: \Hi_{2,\nu}[\ul{\Gamma}] \to \Hi_{1,\nu}[\ul{\Gamma}]
\end{align*}
\begin{align*}
 \dbar_{\hat\Gamma}: \Hi_{2,\nu}[\ul{\hat\Gamma}] \to \Hi_{1,\nu}[\ul{\hat\Gamma}].
\end{align*}
for bundles 
\begin{align*}
   &\pi_{i,\Gamma}: \Hi_{i,\nu}[\ul{\Gamma}] \to (0, \lambda_0], \quad \pi_{i,\Gamma}^{-1}(\lambda) = \Hi_{i,\nu}[{\Gamma_\lambda}], \qquad i=1,2, \label{eq:himu}
\end{align*}
and similar for the fully capped problems. To simplify notation we focus on the un-capped problem from now on, but we point out that all constructions goes through also in the capped case.

Now we define trivializations $\psi_{i,\Gamma}: (0,\lambda_0] \times
\Hi_{i,\nu}[\Gamma_{\lambda_0}] \to \Hi_{i,\nu}[\ul{\Gamma}]$,
$i=1,2$, given by identifying $\Delta(\Gamma_\lambda)$ with
$\Delta(\Gamma_{\lambda_0})$ for $\lambda$ sufficiently small, in the
same time as we also adjust for the change of the boundary conditions. That is, by the construction of the gluings that we use,  described below and coming from \cite{trees}, we get that $\Delta(\lambda_0) \subset \Delta(\lambda)$ if $\lambda_0 > \lambda$. 
Moreover, the standard domains may be assumed to coincide in neighborhoods of all infinities and of all boundary minima, while the inner pieces (where the gluings are performed) of $\Delta(\lambda)$ are longer than the corresponding ones for $\Delta(\lambda_0)$. More precisely, we have that the length of the inner pieces is of $O(\lambda^{-1})$. 
%
%

It follows that we can define a diffeomorphism $\psi_{\Gamma_\lambda}:\Delta(\Gamma_\lambda) \to \Delta(\Gamma_{\lambda_0})$ which equals the identity in neighborhoods of all infinities and of all boundary minima. Moreover, we might assume that the weight $\w_\lambda$ equals $1$ on the connected component containing the set $\{\psi_{\Gamma_\lambda} \neq \id\}$ for all $\lambda \in (0,\lambda_0]$ and that $\w_\lambda = \w_{\lambda_0}$ outside this set. 

In this way we can interpret the bundle trivializations as 
\begin{equation*}
 \psi_{i, \Gamma}(\lambda,u) = \tilde A(\Gamma_\lambda) {\tilde A(\Gamma_{\lambda_0})}^{-1}\circ \psi_{\Gamma_\lambda} \cdot u \circ \psi_{\Gamma_\lambda}
\end{equation*}
where $\tilde A(\Gamma_\lambda)$ are the  boundary conditions defined in Section
\ref{sec:treetriv}. By standard arguments, see e.g. \cite{husem}, if we can prove that $\dim \kd_{\Gamma_\lambda}$ is constant in $\lambda \in(0,\lambda_0]$, then this trivialization also gives trivializations for \eqref{eq:kb} and \eqref{eq:cb} (and similar arguments give trivializations for \eqref{eq:hkb} and \eqref{eq:hcb}). 

If $\Gamma$ is a sub flow tree, we also consider a bundle 
\begin{equation}\label{eq:kbs}
\pi_{\Gamma_s}: \Ker \ul{\Gamma}_s \to [0,\lambda_0]
\end{equation}
with fibers
\begin{equation*}
 \pi_{\Gamma_s}^{-1}(\lambda) =
 \begin{cases}
  \kd_{\Gamma_{s,\lambda}}, \quad \lambda > 0, \\
  T_p\F_p(\Gamma),  \quad \lambda =0.
 \end{cases}
\end{equation*}
We will prove that this gives a vector bundle with trivialization
given by the evaluation map at the special puncture $p$. A similar construction is made also for the pre-stabilized problem.

\begin{defi}
We call the bundle  $\Ker\ul{\Gamma}_{s}$ the \emph{stabilized kernel bundle} of $\Gamma$.
\end{defi}

Before we prove that \eqref{eq:hkb} -- \eqref{eq:hcb}, \eqref{eq:kbs} define vector bundles, we discuss some materials on exact sequences of bundles. Let $\pi_1:E \to B,\pi_2:F \to B $ be two vector bundles over the same base space $B$.
We say that 
 a vector bundle map $v: E \to F$ is of \emph{rank $k$} if $v_b : \pi^{-1}_1(b) \to \pi^{-1}_2(b)$ is of rank $k$ for each $b \in B$.  

\begin{defi}
 An \emph{exact sequence of vector bundles} is a sequence
 \begin{equation*}
  0 \to E_0 \xrightarrow{\alpha_1} E_1 \xrightarrow{\alpha_2} \dotsc \xrightarrow{\alpha_n} E_n \to 0
 \end{equation*}
such that the rank of $\alpha_i$ is constant for all $1\leq i \leq n$, and $\im \alpha_i = \Ker \alpha_{i+1}$ for all $1\leq i \leq n-1$.
\end{defi}

The following is straightforward after having fixed orientation conventions as in Section \ref{sec:orientconv}. Compare with the similar statement in \cite{hutch2}.
\begin{lma}
 Let $E_1, E_2, F_1, F_2$ be trivialized vector bundles over $[0,1]$, and let
 \begin{equation*}
  0 \to E_1 \xrightarrow{\alpha} F_1 \xrightarrow{\beta} F_2 \xrightarrow{\gamma} E_2 \to 0
 \end{equation*}
be an exact sequence. Then orientations of three of the spaces $E_1(0), E_2(0), F_1(0), F_2(0)$ canonically induce  orientations on $E_1(\lambda), E_2(\lambda), F_1(\lambda), F_2(\lambda)$ for all $\lambda \in[0,1]$. 
\end{lma}

Now we return to our particular cases. 
\begin{prp}\label{prp:comiso}
Let $\Gamma_1$ and $\Gamma_2$ be two partial flow trees satisfying the
assumptions of Proposition \ref{prp:megalemma4paux} or
Proposition \ref{prp:sista}. Also assume that there is a $\lambda_0>0$ so that they satisfy the following.
\begin{enumerate}
 \item\label{iih1} We have vector bundles  $\Ker \ul{\Gamma}_i$, $\Ker \ul{\Gamma}_{i,s}$,  $\Coker \ul{\Gamma}_i$ over $(0,\lambda_0]$, $i=1,2$,  as described above.
 \item\label{iih2} If $p_i$ is any special puncture of $\Gamma_i$, $i=1,2$, then  $\ev_{p_i}:\Ker \ul{\Gamma}_i \to \R^n $ is continuous.  
 \item\label{iih3} $\dim \ev_p(\kd_{\Gamma_{i,\lambda}}) = \dim \kd_{\Gamma_{i,\lambda}}$ for $\lambda \in (0, \lambda_0]$, $i=1,2$, where $p$ is the point where we glue $\Gamma_1$ to $\Gamma_2$.
\end{enumerate}
Then there is a $0 < \tilde \lambda_0 \leq \lambda_0$ so that \eqref{iih2} -- \eqref{iih3} hold for the glued tree $\Gamma= \Gamma_1 \#\Gamma_2$, and so that \eqref{iih1} hold for $\Gamma$ if $\dim V_{\con}(\Gamma_1)+\dim V_{\con}(\Gamma_2)= \dim V_{\con}(\Gamma)$, and so that \eqref{iih1} hold for $\Gamma$ if we replace the stabilized kernel by the pre-stabilized kernel in the case when $\dim V_{\con}(\Gamma_1)+\dim V_{\con}(\Gamma_2)< \dim V_{\con}(\Gamma)$.
\end{prp}

\begin{rmk}
 If $\Gamma_i$ has $\dim V_{\con}(\Gamma_i)=0$ then we let $ \Ker \ul{\Gamma}_{i,s}=\Ker \ul{\Gamma}_i $, and similar for the pre-stabilized bundle.
\end{rmk}

\begin{rmk}
 That the assumptions in the proposition hold for elementary trees is proved in Section \ref{sec:elementor}.
\end{rmk}

\begin{rmk}\label{rmk:hat}
 We have similar results for the fully capped problems, and also if we consider the $\dbar$-problems discussed in Section \ref{sec:background}, where we do not consider the end vertices as marked points.
\end{rmk}

\begin{proof}[Proof of Proposition \ref{prp:comiso}]
This follows from the proof of [\cite{trees}, Proposition 6.20].
\end{proof}

\begin{prp}\label{prp:megalemma2}
 Assume that $\Gamma_i$, $i=1,2$, satisfies the assumptions of Proposition \ref{prp:comiso}. Then the associated gluing sequences from Lemma \ref{lma:gluweights} lift to exact sequences of vector bundles. 
\end{prp}

\begin{proof}
 We prove the statement for the sequence \eqref{eq:g3}, the other cases are similar.
  First we modify the notation as follows. Let $\alpha_\lambda = (\alpha_\lambda^1, \alpha_\lambda^2)$ denote the  first  non-trivial map, $\beta_\lambda=(\beta_\lambda^1, \beta_\lambda^2)$ the second one and $\gamma_\lambda$ the third one, where  we assume the gluing parameter $\rho$ to be given by $\rho=1/\lambda$. Indeed, from \cite{trees} it follows that  we may assume the gluing region to be of size $O(\lambda^{-1})$, and also that the gluing region corresponding to $\lambda_1$ is contained in the gluing region corresponding to $\lambda_2$ if $\lambda_2<\lambda_1$. I.e., we may replace $\rho$ in Lemma \ref{lma:gluweights} with $\lambda^{-1}$. Note that the $\dbar_A$-problem from \eqref{eq:g3} is assumed to correspond to the $\dbar_{\Gamma_1}$-problem, the $\dbar_B$-problem correspond to the $\dbar_{\Gamma_2}$-problem, and the $\dbar_{A\#B}$-problem to the $\dbar_\Gamma$-problem.
 
 Let $\phi_1^\lambda$ and $\phi_2^\lambda$ be the cut-off functions corresponding to the cut-off functions $\phi_A^\rho=\phi_A^{\lambda^{-1}}$ and $\phi_B^\rho=\phi_B^{\lambda^{-1}}$ defined by the formulas \eqref{eq:phiA} and \eqref{eq:phiB}, respectively. We assume that  $||D^k\phi_i^\lambda||_\infty = O(\lambda)$, $k=1,2$,  $||\phi^{\lambda_1}_i-\phi^{\lambda_2}_i||_\infty = O(\frac{\lambda_1 - 2\lambda_2}{2\lambda_1})$, $||D(\phi^{\lambda_1}_i-\phi^{\lambda_2}_i)||_\infty = O(\lambda_1-\lambda_2)$,  $i=1,2$.

 We prove that the family $\alpha_\lambda$ gives rise to a vector bundle map, the proofs for $\beta_\lambda$ and $\gamma_\lambda$ are similar and left to the reader.
%
 Let $\bar{u}=(u_1,\dotsc,u_{k_1})$ be an orthonormal basis for $\kd_{\Gamma_{\lambda_0}}$, $\bar{v}=(v_1,\dotsc,v_{k_2})$ an orthonormal basis for $\kd_{\Gamma_{1,\lambda_{0}}}$, $\bar{w} = (w_1,\dotsc,w_{k_3})$ an orthonormal basis for $\kd_{\Gamma_{2,\lambda_{0}}}$. 
 Let 
 \begin{equation*}
  \bar{u}_\lambda = (u_1^\lambda,\dotsc,u_{k_1}^\lambda)=( \psi_{2, \Gamma}(\lambda,u_1),\dotsc, \psi_{2, \Gamma}(\lambda,u_{k_1}))
 \end{equation*}
be a basis for $\kd_{\Gamma_{\lambda}}$. We define bases for the other spaces in a completely analogous way.
 
 To prove that the family $\alpha_\lambda$ gives a vector bundle map it is enough to prove that the family of functions
 \begin{multline*}
   g_{1,\lambda}:=( \psi_{2, \Gamma_1}^{-1} \circ \alpha_\lambda^1 \circ \psi_{2, \Gamma},\psi_{2, \Gamma_2}^{-1} \circ \alpha_\lambda^2 \circ \psi_{2, \Gamma})\\ \in L(\kd_{\Gamma_{\lambda_0}}, \iota(\kd_{\Gamma_{1,\lambda_{0}}})\oplus  \iota(\kd_{\Gamma_{2,\lambda_{0}}})), 
    \end{multline*}
depends continuously on $\lambda$. Here $\iota$ denotes the inclusion of the kernels into $\Hi_{2,\nu}[\Gamma_{i,\lambda_0}]$, $i=1,2$, which is done by multiplying by the cut-off functions $\phi_1^{\lambda_0}$ and $\phi_2^{\lambda_0}$.

 We must prove that for each fixed $c \in(0,\lambda_0]$ we have that 
\begin{equation*}
 \|g_{1,\lambda}(u)-g_{1,c}(u)\|_{\nu,\lambda_0} \to 0
\end{equation*}
as $\lambda \to c$.

We will prove this component-wise, and thus it is enough to prove that for each $u \in \kd_{\Gamma_{\lambda_0}}$ we have that 
\begin{align*}
 &I(g_1) :=\\
 &\left\|\left(\int_{\Delta(\Gamma_{\lambda})}\langle \tilde A(\Gamma_{\lambda}) {\tilde A(\Gamma_{\lambda_0})}^{-1}\circ \psi_{\Gamma_{\lambda}} \cdot u \circ \psi_{\Gamma_{\lambda}}, \phi_i^{\lambda} x_{j,\lambda} \rangle \w^2 dm \right) \psi_{2,\Gamma_{i}}^{-1}(\phi_i^{\lambda} x_{j,\lambda} ) \right. \\
  &\qquad{}- \left.\left( \int_{\Delta(\Gamma_{c})}\langle \tilde A(\Gamma_{c}) {\tilde A(\Gamma_{\lambda_0})}^{-1}\circ \psi_{\Gamma_{c}} \cdot u \circ \psi_{\Gamma_{c}}, \phi_i^{c} x_{j,c} \rangle \w^2 dm \right) \psi_{2,\Gamma_{i}}^{-1}(\phi_i^{c} x_{j,c} )\right \|_{\nu,\lambda_0} \to 0
\end{align*}
as $\lambda \to c$, where $i=1,2 $, $x_{j,\lambda} =v_j^\lambda$ if
$i=1$, $j=1,\dotsc,k_2$ and $x_{j,\lambda} = w_j^\lambda$ if $i=2$, $j=1,\dotsc,k_3$, and similar for $c$. 

We estimate this from above, and recall that we assume the weight to be invariant under $\psi_{\Gamma_\lambda}$. We get
\begin{multline*}
 {I(g_1) \leq} I_1 \cdot \|  \phi_i^{\lambda} \circ \psi_{\Gamma_{\lambda}}^{-1}\cdot x_j  \|_{\nu,\lambda_0}  \\
+\left( \int_{\Delta(\Gamma_{c})}\langle \tilde A(\Gamma_{c}) {\tilde A(\Gamma_{\lambda_0})}^{-1}\circ \psi_{\Gamma_{c}} \cdot u \circ \psi_{\Gamma_{c}}, \phi_i^{c} x_{j,c} \rangle \w^2 dm \right)\cdot N_1
\end{multline*}
 where 
 \begin{align*}
  I_1 =
  &\left |\left(\int_{\Delta(\Gamma_{\lambda_0})}\langle \tilde A(\Gamma_{\lambda})\circ \psi_{\Gamma_{\lambda}}^{-1} \cdot{\tilde A(\Gamma_{\lambda_0})}^{-1} \cdot u, \phi_i^{\lambda} \circ \psi_{\Gamma_{\lambda}}^{-1} \cdot  x_{i} \rangle \w^2 |\det (D\psi_{\Gamma_{\lambda}}^{-1})| dm \right) \right. \\
 & \qquad{}- \left.\left(\int_{\Delta(\Gamma_{\lambda_0})}\langle \tilde A(\Gamma_{c})\circ \psi_{\Gamma_{c}}^{-1} \cdot{\tilde A(\Gamma_{\lambda_0})}^{-1} \cdot u, \phi_i^{c} \circ \psi_{\Gamma_{c}}^{-1} \cdot x_{i} \rangle \w^2 |\det (D\psi_{\Gamma_{c}}^{-1})| dm \right) \right |
 \end{align*}
and
 \begin{equation*}
  N_1 =  \| ( \phi_i^{\lambda} \circ \psi_{\Gamma_{\lambda}}^{-1} -  \phi_i^{c} \circ \psi_{\Gamma_{c}}^{-1})\cdot x_j)  \|_{\nu,\lambda_0}
 \end{equation*}
 
 Now the result follows from the fact that 
 \begin{align*}
 & \| \tilde A(\Gamma_{\lambda})\circ \psi_{\Gamma_{\lambda}}^{-1} -
   \tilde A(\Gamma_{c})\circ \psi_{\Gamma_{c}}^{-1}\|_{\infty} \to
   0,\\
& \|  \phi_i^{\lambda} \circ \psi_{\Gamma_{\lambda}}^{-1} -  \phi_i^{c} \circ \psi_{\Gamma_{c}}^{-1})  \|_\infty \to 0,\\
  &\||\det (D\psi_{\Gamma_{\lambda}}^{-1})|-|\det (D\psi_{\Gamma_{C}}^{-1})|\|_\infty \to 0,
 \end{align*}
as $\lambda \to c$. The first estimate we get from the constructions
in \cite{trees}, the second one we get from the assumption on
$\|\phi_i^\lambda - \phi_i^c\|_\infty$ together with the construction
of $ \psi_{\Gamma_{\lambda}}$, and the third estimate we get from the
fact that we might assume that $|\det (D\psi_{\Gamma_{\lambda}}^{-1})|
= O(\frac{\lambda_0}{\lambda})$ in the inner pieces, and equal to 1 otherwise.
%
%
%
%

We get similar estimates for $\beta_\lambda$ and $\gamma_\lambda$, but
here we in addition need to use that
$\|D(\phi_i^{\lambda_1}-\phi_i^{\lambda_2})\|_\infty= O(|\lambda_2 -
\lambda_1|)$, $i=1,2$.

That the ranks of $\alpha_\lambda$, $\beta_\lambda$ and
$\gamma_\lambda$ are constant for $\lambda$ sufficiently small follows
from Proposition \ref{prp:comiso} together with the nature of the mappings. 
 
\end{proof}

\begin{prp}\label{prp:glu}
 Assume that $\Gamma_i$, $i=1,2$, satisfies the assumptions of Proposition \ref{prp:comiso}. Also assume that we can define an orientation of $\Ker \ul{\Gamma}_{i,s}$ so that the orientation of each fiber gives the stabilized capping orientation of $\Gamma_i$, $i=1,2$.

 Then there is a $\lambda_0'>0$ so that this also hold for the glued tree $\Gamma$ if $\dim V_{\con}(\Gamma_1)+\dim V_{\con}(\Gamma_2)= \dim V_{\con}(\Gamma)$, and with the stabilized kernel replaced by the pre-stabilized kernel if $\dim V_{\con}(\Gamma_1)+\dim V_{\con}(\Gamma_2)< \dim V_{\con}(\Gamma)$.
 
 Moreover, suppose that $\F_p(\Gamma_i)$ is given the stabilized capping orientation of the tree $\Gamma_{i}$, $i=1,2$, and suppose that 
 $(v_1,\dotsc,v_k)$ gives an oriented basis for $T_p\F_p(\Gamma_1)\cap T_p \F_p(\Gamma_2)$, where the orientation is induced by the sequence \eqref{eq:glu10l}. Let 
 \begin{equation*}
  \xi: \Ker \ul{\Gamma}_{s} \to [0,\lambda_0'] \times T_p\F_p(\Gamma_1)\cap T_p \F_p(\Gamma_2), \text{ if }\dim V_{\con}(\Gamma_1)+\dim V_{\con}(\Gamma_2)= \dim V_{\con}(\Gamma),
 \end{equation*}
\begin{equation*}
  \xi: \Ker \ul{\Gamma}_{ps} \to [0,\lambda_0'] \times T_p\F_p(\Gamma_1)\cap T_p \F_p(\Gamma_2),  \text{ if }\dim V_{\con}(\Gamma_1)+\dim V_{\con}(\Gamma_2)< \dim V_{\con}(\Gamma),
 \end{equation*}
 be the bundle-trivialization from Lemma \ref{prp:comiso}, and let
 $v_i^\lambda=\xi_\lambda^{-1}(v_i)$. If  $\sigma$ is the gluing sign
 in Section \ref{sec:glueback}, then
 \begin{equation*}
 (-1)^\sigma v_1^\lambda \wedge\cdots \wedge v_k^\lambda 
 \end{equation*} 
gives the 
 stabilized capping orientation of $\Gamma$  
 if $\dim V_{\con}(\Gamma_1)+\dim V_{\con}(\Gamma_2)= \dim V_{\con}(\Gamma)$, and gives the 
 pre-stabilized capping orientation of $\Gamma$
 if $\dim V_{\con}(\Gamma_1)+\dim V_{\con}(\Gamma_2)< \dim V_{\con}(\Gamma)$.
\end{prp}

\begin{proof} 
We focus on the case when $\dim V_{\con}(\Gamma_1)+\dim V_{\con}(\Gamma_2)< \dim V_{\con}(\Gamma)$, the other case is similar.

From [\cite{trees}, Proposition 6.20] it follows that we can find a basis $\ul v^\lambda= (v_1^\lambda,\dotsc,v_k^\lambda)$ of $\kd_{\Gamma_{ps,\lambda}}$ satisfying
\begin{equation*}
 v_i^\lambda = \phi_1^\lambda w_i^\lambda + \phi_2^\lambda u_i^\lambda + r_i^\lambda
\end{equation*}
where $w_i^\lambda \in \kd_{\Gamma_{1,s,\lambda}}$, $u_i^\lambda \in \kd_{\Gamma_{2,s,\lambda}}$ are continuous families of vectors satisfying $\lim_{\lambda \to 0}\ev_{p}(w_i^\lambda) = \lim_{\lambda \to 0} \ev_{p}(u_i^\lambda)$, and $r_i^\lambda \in \Hi_{2,\nu}[\Gamma_\lambda]\oplus V_{\con}(\Gamma_\lambda)$ satisfies $\|r_i^\lambda\|_{W,\lambda} = O(\lambda)$. Here $\|\cdot\|_{W,\lambda}$ denotes the norm on $\Hi_{2,\nu}[\Gamma_\lambda]\oplus V_{\con}(\Gamma_\lambda)$, and $\phi_{k}^\lambda$, $k=1,2$, are the cut-off functions from the proof of Proposition \ref{prp:megalemma2}. Moreover, we assume that 
we have a  trivialization of $T_pM$ so that 
\begin{equation*}
 \lim_{\lambda \to 0} \ev_p(v_i^\lambda) = \partial_{x_i}, \quad i=1,
 \cdots,k, \qquad T_p(\F_p(\Gamma_1) \cap \F_p(\Gamma_2)) = \spn(\partial_{x_1},\dotsc,\partial_{x_k}).
\end{equation*}

To prove that the gluing sequence \eqref{eq:glu10} can be evaluated in
the limit, we complete the bases $(w_1^\lambda,\dotsc,w_k^\lambda)$
and $(u_1^\lambda,\dotsc,u_k^\lambda)$ to bases $\ul
w^\lambda=(w_1^\lambda,\dotsc,w_{k_1}^\lambda)$ and $\ul
u^\lambda=(u_1^\lambda,\dotsc,u_{k_2}^\lambda)$ of
$\kd_{\Gamma_{1,s,\lambda}}$ and $\kd_{\Gamma_{2,s,\lambda}}$,
respectively. We may moreover assume that the vectors are chosen so that 
\begin{align}
 &w_i^\lambda = \partial_{x_i} + O(\lambda), & i&=1,\dotsc,k_1, \label{eq:triv1}\\
 &u_i^\lambda = \partial_{x_i} + O(\lambda), &i&=1,\dotsc,k, \label{eq:triv2}\\
  &u_i^\lambda = \partial_{x_{k_1+i}} + O(\lambda), & i&=k+1,\dotsc,k_2, \label{eq:triv3}
\end{align}
in the gluing region. Here $O(\lambda)$ means a vector of norm $O(\lambda)$.

Following the arguments in the proof of Proposition \ref{prp:comiso}, we see that to prove that the left-most, non-trivial map in the gluing sequence \eqref{eq:glu10} behaves as stated in the limit  we must show that 
\begin{equation*}
 \left|\int_{\Delta(\Gamma_{\lambda})}\langle v_i^{\lambda} , \phi_1^{\lambda}w_j^{\lambda} \rangle \w^2 dm - \delta_i^j \right| \to 0
\end{equation*}
as $\lambda \to 0$, $i=1,\dotsc,k$, $j=1,\dotsc,k_1$,
and that
\begin{equation*}
 \left|\int_{\Delta(\Gamma_{\lambda})}\langle v_i^{\lambda} , \phi_2^{\lambda}u_j^{\lambda} \rangle \w^2 dm - \delta_i^j \right| \to 0
\end{equation*}
as $\lambda \to 0$, $i=1,\dotsc,k$, $j=1,\dotsc,k_2$. Here we have used bundle-trivializations given by the bases  $\ul{v}^\lambda$, $\ul{w}^\lambda$ and $\ul{u}^\lambda$, respectively. We prove that this holds for the first integral, the second one is completely similar. 

We get 
\begin{align*}
  \MoveEqLeft[36] {\left|\int_{\Delta(\Gamma_{\lambda})}\langle v_i^{\lambda} , \phi_1^{\lambda}w_j^{\lambda} \rangle \w^2 dm - \delta_i^j \right|} &  \\
   \leq    \left|\int_{\Delta(\Gamma_{1,\lambda})}\langle \phi_1^{\lambda} w_i^{\lambda}, \phi_1^{\lambda}w_j^{\lambda} \rangle \w^2 dm - \delta_i^j \right|  
    +    \left|\int_{\Delta(\Gamma_{\lambda})}\langle r_i^{\lambda} , \phi_1^{\lambda}w_j^{\lambda} \rangle \w^2 dm \right|\qquad{}&\\
  \quad {} +   \left|\int_{\Delta(\Gamma_{\lambda})}\langle \phi_2^{\lambda} u_i^{\lambda} , \phi_1^{\lambda}w_j^{\lambda} \rangle \w^2 dm \right|&.
\end{align*}
Now 
\begin{align*}
         \MoveEqLeft[36] \left|\int_{\Delta(\Gamma_{1,\lambda})}\langle \phi_1^{\lambda} w_i^{\lambda}, \phi_1^{\lambda}w_j^{\lambda} \rangle \w^2 dm - \delta_i^j \right| &  \\
     \MoveEqLeft[29] \leq      \left|\int_{\Delta(\Gamma_{1,\lambda})}\langle  w_i^{\lambda}, w_j^{\lambda} \rangle \w^2 dm - \delta_i^j \right|  +
          \left|\int_{\Delta(\Gamma_{1,\lambda})}((\phi_1^{\lambda})^2 -1)\langle w_i^{\lambda}, w_j^{\lambda} \rangle \w^2 \right|& \\  
         \leq  O(\lambda) + \left|\int_{\Delta(\Gamma_{1,\lambda})}((\phi_1^{\lambda})^2 -1)\langle \partial_{x_i}, \partial_{x_j} \rangle \w^2 \right| + O(\lambda) = O(\lambda).&
\end{align*}
Here we use \eqref{eq:triv1} and the fact that $ \w^2 \sim e^{-2\delta |\tau|}$ on $\{\phi_1^{\lambda}\neq1\}$ (where $|\tau| \to \infty$) and   $\{\phi_1^{\lambda}=1\} \to \Delta(\Gamma_{1,\lambda}) $ as $\lambda \to 0$.

We also have 
\begin{align*}
  \left|\int_{\Delta(\Gamma_{\lambda})}\langle r_i^{\lambda} , \phi_1^{\lambda}w_j^{\lambda} \rangle \w^2 dm \right| \leq
   \|r_i^{\lambda}\|_{\nu,\lambda}\|\phi_1^{\lambda}w_j^{\lambda}\|_{\nu,\lambda}  = O(\lambda)
\end{align*}
and 
\begin{align*}
  \left|\int_{\Delta(\Gamma_{\lambda})}\langle \phi_2^{\lambda} u_i^{\lambda} , \phi_1^{\lambda}w_j^{\lambda} \rangle \w^2 dm \right| =0
\end{align*}
since the cut-off functions $\phi_1^{\lambda}$ and $\phi_2^\lambda$ are assumed to have disjoint support for all $\lambda$. Hence the leftmost map behaves as claimed.

To prove that the second map also behaves as claimed in the limit, we must consider the integrals
\begin{equation*}
 \int_{\Delta(\Gamma_{\lambda})}\langle \dbar(\phi_1^{\lambda}w_i^{\lambda} ), \psi_{\lambda}\partial_{x_j} \rangle \w^2 dm \qquad \text{ and } \qquad \int_{\Delta(\Gamma_{\lambda})}\langle \dbar(\phi_2^{\lambda}u_i^{\lambda} ), \psi_{\lambda}\partial_{x_j} \rangle \w^2 dm, 
\end{equation*}
$j=1,\dotsc,n$, $i=1,\dotsc, k_1$ and $i=1,\dotsc, k_2$, respectively. Here $\psi_{\lambda}$ is a cut-off function having support in the gluing region, compare [\cite{orientbok}, Lemma 3.1]. But by the choice of the basis vectors the integrals can be understood, up to $O(\lambda)$, as 
\begin{align*}
  \int_{\Delta(\Gamma_{\lambda})}\langle \dbar(\phi_k^{\lambda})\partial_{x_i} , \psi_{\lambda}\partial_{x_j} \rangle \w^2 dm = 
  \begin{cases}
   -c_{\lambda}\delta_i^j, & k=1 \\
    c_{\lambda}\delta_i^j, & k=2
  \end{cases}
\end{align*}
where $c_{\lambda}$ is some positive constant uniformly bounded
away from zero. The sign comes from the derivative of the cut-off function, compare [\cite{orientbok}, Lemma 3.11].

Moreover, from Section \ref{sec:stab}, Proposition \ref{prp:comiso} and Proposition \ref{prp:megalemma2} it follows that all the constructions in the proof of Proposition \ref{prp:megalemma4paux} can be made continuously in $\lambda$.
This completes the proof.
\end{proof}

Now consider the situation when we only get the pre-stabilized problem from the gluing above. This is the case when we glue at a vertex that has a boundary minima associated to it (either a $Y_0$- or $Y_1$- vertex, or a $2$-valent puncture).   
To calculate the stabilized orientation of $\Gamma$, we must stabilize the $\dbar_{\Gamma_{ps,\lambda}}$-problem by adding the missing conformal variation $t^\lambda$ to the domain of the operator. Recall that this family of  vectors has a well-defined limit $\lim_{\lambda\to 0}t^\lambda =t$,  corresponding to the covariantly constant vector field tangent to the edge that ends at the special positive puncture of $\Gamma$, pointing in the direction opposite to the flow direction.

\begin{prp}\label{prp:stababund}
Let $\Gamma$ be the glued tree from Proposition \ref{prp:glu}, and assume that  $\dim V_{\con}(\Gamma_1)+\dim V_{\con}(\Gamma_2)< \dim V_{\con}(\Gamma)$. Then the stabilized problem $\dbar_{\Gamma_{s,\lambda}}$ satisfies the assumption of Proposition \ref{prp:glu}, and 
 \begin{equation*}
 (-1)^{\sigma+\dim V_{\con}(\Gamma_2)} v_1^\lambda \wedge\cdots \wedge v_k^\lambda \wedge t^\lambda 
 \end{equation*} 
 gives the stabilized capping orientation of $\Gamma$.   
\end{prp}

\begin{proof}
  Consider the case when we are gluing at a $Y_0$-vertex, the
  other cases are similar. With abuse of notation, let $\Gamma_v$ denote
  the $Y_0$-piece we are gluing at, and let $\Gamma_1$, $\Gamma_2$
  denote the sub flow trees that we glue to $\Gamma_v$ to get
  $\Gamma$. Let $p_i$ denote the special puncture of $\Gamma_i$,
  $i=1,2$, and let $u_i^\lambda \in \kd_{\Gamma_{i,s,\lambda}}$ such
  that $\ev_{p_i}(u_i^\lambda) \to f_i$ as $\lambda \to 0$, where $f_i
  \in T_{p_i}M$ represents the flow direction of $\Gamma_i$ at
  $p_i$. From the proof of [\cite{trees}, Proposition 6.20] it follows
  that we can find functions $r_i^\lambda$, $i=1,2$, so that $ t^\lambda + u_1^\lambda + u_2^\lambda+ r_1^\lambda + r_2^\lambda \in \kd_{\Gamma_{\lambda,s}}$ and so that $\|r_i^\lambda\| \to 0$ as $\lambda \to 0$. Moreover we can choose $u_i^\lambda$ continuously in $\lambda$, $i=1,2$. By the asymptotic behavior of $t^\lambda$ as $\lambda \to 0$ the claim now follows from Lemma \ref{lma:prestab}
\end{proof}

\subsection{Capping orientation of elementary trees}\label{sec:elementor}
In this section we prove the base case of the inductive assumptions in Proposition \ref{prp:comiso} and Proposition \ref{prp:glu}. That is, we prove that these assumptions hold for elementary trees. We also show that the capping orientation of the elementary trees can be given in terms of oriented submanifolds in $M$ times a sign which depends only on combinatorial data of the tree.  

We begin by proving that the kernels and cokernels of the capping operators constitute bundles over $(0,\lambda_0]$ for $\lambda_0$ sufficiently small, and that these bundles moreover can be compactified by adding a fiber at $\lambda =0$ to get the isomorphisms stated in Lemma \ref{lma:capbundle0}. 


%

\subsubsection{Capping operators}\label{sec:cappf}
Recall the capping operators $\dbar_{ p,\pm} = \dbar_{p,\pm, \lambda}$ defined in Section \ref{sec:capping} and Section \ref{sec:capspec}. More precisely, these are operators defined on the $1$-punctured disk with associated boundary conditions $\tilde R_{\pm} =\tilde  R_{\pm, \lambda}$, which are given by diagonal matrices composed with rotation matrices and which depend on $\lambda$ in such a way so that the diagonal elements tend to $e^{ik\pi  s}$, $k\in\{-2,-1,0,1\}$, as $\lambda \to 0$. We also defined operators $\dbar_{ p} = \dbar_{p, \lambda}$ on the non-punctured disk, with boundary conditions obtained  by gluing $\tilde R_{+,\lambda}$ to $\tilde R_{-,\lambda}$, and where this glued boundary condition tend to a diagonal matrix with entries $e^{ik\pi s}$, $k\in\{-2,0\}$, as $\lambda \to 0$, composed with rotation matrices. In addition we put weights on the Sobolev spaces associated to $\dbar_{ p,\pm,\lambda}$. These weights allow us to extend the definition of  $\dbar_{ p,\pm,\lambda}$ and $\dbar_{ p,\lambda}$ to also be Fredholm operators for $\lambda =0$. More precisely, we get that the $\dbar$-operators 
\begin{equation*}
 \dbar_{ p,\pm,0}: \Hi_{2,\nu_{\pm}(p)}[\tilde R_{\pm,0}] \to \Hi_{1,\nu_{\pm}(p)}[0], \qquad \dbar_{ p,0}: \Hi_{2}[\tilde R_{+,0}\# \tilde R_{-,0}] \to \Hi_{1}[0]
\end{equation*}
are Fredholm. Here $\tilde R_{\pm,0} = \lim_{\lambda \to 0} \tilde R_{\pm,\lambda}$ and the weights $\nu_{\pm}(p)$ are the ones defined in Section \ref{sec:capping} and Section \ref{sec:capspec}. We have the following.

\begin{lma}\label{lma:caplim}
Assume that $p$ is a true puncture. Then there is a $\lambda_0 >0$ so that we have oriented bundles $\Coker\ul{\dbar_{p,\pm}}$, $\Ker\ul{\dbar_{p,\pm}}$, $\Coker\ul{\dbar_{p}}$ and $\Ker\ul{\dbar_{p}}$ over $[0,\lambda_0]$, with fibers $\cd_{p,\pm,\lambda}$, $\kd_{p,\pm,\lambda}$, $\cd_{p,\lambda}$ and $\kd_{p,\lambda}$, respectively, satisfying
\begin{align*}
\Coker\ul{\dbar_{p,+}} \simeq [0,\lambda_0] \times T_pW^s(p) \qquad \Coker\ul{\dbar_{p ,-}} \simeq [0,\lambda_0] \times \left(T_pW^u(p)\oplus \aux_1\right),
\end{align*}
if $|\mu(p)|=0$, and 
\begin{align*}
\Coker\ul{\dbar_{p,+}} \simeq [0,\lambda_0] \times \left(T_pW^s(p)\oplus \aux_1 \right) \qquad \Coker\ul{\dbar_{p ,-}} \simeq [0,\lambda_0] \times T_pW^u(p),
\end{align*}
if $|\mu(p)|=1$, and
\begin{align*}
\Ker\ul{\dbar_{p,+}} & \simeq [0,\lambda_0] \times 0 \qquad \Ker\ul{\dbar_{p ,-}} \simeq \Ker\ul{\dbar_{p}} \simeq [0,\lambda_0] \times \aux_2, \quad |\mu(p)|=0,1,\\
\Coker\ul{\dbar_p} & \simeq [0,\lambda_0] \times \left(T_pM \oplus \aux_1\right),  \quad |\mu(p)|=0,1,
\end{align*}
with the trivialization given by evaluation at the puncture $p$. Moreover, the gluing sequence \eqref{eq:capseq} lifts to a well-defined exact sequence vector bundles
\begin{align}\label{eq:capglubun}
 0 \to \Ker \ul{\dbar_{p}} \to
\begin{bmatrix}
\Ker\ul{\dbar_{p ,-}} \\
\Ker\ul{\dbar_{p ,+}}
\end{bmatrix}
\to 
\begin{bmatrix}
 \Coker\ul{\dbar_{p ,-}} \\
\Coker\ul{\dbar_{p ,+}}
\end{bmatrix}
\to 
\Coker \ul{\dbar_{p}} \to 0.
\end{align}
We have an analogous result for the special punctures.
\end{lma}

\begin{proof}
The first statements follow from Corollary \ref{cor:canonbundle} and Corollary 
\ref{cor:canonbundleweight} together with Remark \ref{rmk:megarmk},
and the fact that the boundary conditions $\tilde R_{\pm,\lambda}$ are
split (we can neglect the rotation matrices here). The last statement follows similar to the proof of Proposition
\ref{prp:megalemma2} and Proposition \ref{prp:glu}.
\end{proof}

Now we discuss the orientations of the capping operators. Recall from Section \ref{sec:capopor} and Section \ref{sec:capspec} that if we pick an orientation of $\det \dbar_{p,+,\lambda}$ then the sequence \ref{eq:capseq}, respectively \ref{eq:capseqspec} in the case when $p$ is special, induces an orientation of $\det \dbar_{p,-,\lambda}$. By Lemma \ref{lma:caplim} we see that it is enough to choose an orientation of  $W^u(p)$, $W^s(p)$ and the auxiliary spaces to get induced orientations of the bundles  $\det\ul{\dbar_{p,\pm}}$.

That is, recall that we have fixed orientations of the auxiliary directions represented by the vectors $\partial_{x_{n+1}}, \partial_{x_{n+2}}$, say. Also recall that we have chosen a capping orientation $\Or_{\capp}(T_p W^u(p))$ of $W^u(p)$ and
that we have fixed an orientation of $T_pM$. Then we give $W^s(p)$ the orientation $\Or_{\capp}(T_pW^s(p))$ induced by the identification 
\begin{equation}\label{eq:os00}
 \Or(T_pM) = \Or_{\capp}(T_p W^u(p)) \wedge \Or_{\capp}(T_p W^s(p)).
\end{equation}
To relate these choices to the choice of orientation of $\det \dbar_{p,+,0}$ and the induced orientation of $\det \dbar_{p,-,\lambda}$ for $p$ a Reeb chord of $\Lambda$, let $\sigma_{\capp,+}(p) \in \{0,1\}$ so that 
\begin{align*}
 \Or(\kd_{p,+,0}) &= (-1)^{\sigma_{\capp,+}(p)} \mathbbm{1}, &  \Or(\cd_{p,+,0}) &= \Or_{\capp}(T_p W^s(p)), & |\mu(p)|&=0,\\
 \Or(\kd_{p,+,0}) &= (-1)^{\sigma_{\capp,+}(p)} \mathbbm{1}, &  \Or(\cd_{p,+,0}) &= \Or_{\capp}(T_p W^s(p)) \wedge \partial_{x_{n+1}}, & |\mu(p)|&=1
\end{align*}
represents this orientation. Here $\mathbbm{1}$ represents the standard orientation of the trivial vector space. Then let $\sigma_{\capp,-}(p) \in \{0,1\}$ so that
\begin{align*}
 \Or(\kd_{p,-,0}) &= (-1)^{\sigma_{\capp,-}(p)} \partial_{x_{n+2}}, &  \Or(\cd_{p,-,0}) &= \Or_{\capp}(T_p W^u(p))\wedge \partial_{x_{n+1}}, & |\mu(p)|&=0,\\
 \Or(\kd_{p,-,0}) &= (-1)^{\sigma_{\capp,-}(p)}  \partial_{x_{n+2}}, &  \Or(\cd_{p,-,0}) &= \Or_{\capp}(T_p W^u(p)), & |\mu(p)|&=1
\end{align*}
represents the induced capping orientation of $\det \dbar_{p,-,0}$. Then we use Lemma \ref{lma:caplim} to let these orientations induce orientations of the capping operators $\dbar_{p,\pm,\lambda}$ for $\lambda >0$. 

We make a similar construction for $p$ a special puncture. That is, assume that we have chosen an orientation of  $\det \dbar_{p,+,0}$ and  let $\sigma_{\capp,+,s}(p) \in \{0,1\}$ so that 
\begin{align*}
 \Or(\kd_{p,+,0}) &= (-1)^{\sigma_{\capp,+,s}(p)} \mathbbm{1}, &  \Or(\cd_{p,+,0}) &= \mathbbm{1}, & |\mu(p)|&=0,\\
 \Or(\kd_{p,+,0}) &= (-1)^{\sigma_{\capp,+,s}(p)} \partial_{x_{n+2}}, &  \Or(\cd_{p,+,0}) &= \mathbbm{1}, & |\mu(p)|&=1
\end{align*}
represents this orientation. Then let  $\sigma_{\capp,-,s}(p) \in \{0,1\}$ so that 
\begin{align*}
 \Or(\kd_{p,-,0}) &= (-1)^{\sigma_{\capp,-,s}(p)} \mathbbm{1}, &  \Or(\cd_{p,+,0}) &= \mathbbm{1}, & |\mu(p)|&=0,\\
 \Or(\kd_{p,-,0}) &= (-1)^{\sigma_{\capp,-,s}(p)} \partial_{x_{n+2}}, &  \Or(\cd_{p,+,0}) &= \mathbbm{1}, & |\mu(p)|&=1
\end{align*}
represents the induced orientation of  $\det \dbar_{p,-,0}$. Again we use the results from Lemma \ref{lma:caplim} to get induced orientations of $\dbar_{p,\pm,\lambda}$ for $\lambda >0$. 

In the following subsections we discuss a particular choice of the signs  $\sigma_{\capp,+}(p)$, $\sigma_{\capp,+,s}(p)$ making it possible to find explicit expressions of the signs of rigid flow trees.

\subsubsection{Orientations of end-pieces}\label{sec:endorient}
Let $\Gamma$ be a sub flow tree containing an end $e$ and one positive special puncture $q_1$ , and no other vertices.
Choose a trivialization of $TM$ along $\Gamma$ so that $T_e\Pi(\Sigma) = \spn(\partial_{x_1},\dotsc,\partial_{x_{n-1}})$, and recall that the induced boundary conditions of $\Gamma$ tend to constant $\R^n$-boundary conditions as $\lambda \to 0$, except for in a neighborhood of the end point where it splits as $\R^{n-1} \oplus U_+$, where $U_+$ denotes a uniform $+ \pi$ rotation in $\C$, as in Figure \ref{fig:bend}. Assume that $\partial_{x_n}$ points in the outward normal direction of $\Pi(\Lambda_e)$ at $e$, where $\Lambda_e$ is the union  of the two sheets  of $\Lambda$ containing $e$, and that $(\partial_{x_1},\dotsc,\partial_{x_n})$ gives an oriented basis for $T_eM$. We extend the trivialized boundary conditions from Section \ref{sec:treetriv} to the limit $\lambda =0$, using that we have control over the region where the $+\pi$-rotation is performed. Recall that we assume this rotation to be captured by the capping operator at $e$.  We also extend the family of operators   
  $\dbar_{\Gamma_\lambda}: \Hi_{2,\nu}[\Gamma_\lambda] \to \Hi_{1,\nu}[\Gamma_\lambda]$ to $\lambda =0$.

\begin{lma}\label{lma:kere}
For $\lambda$ sufficiently small the kernel of $\dbar_{\Gamma_\lambda}$ is spanned by functions $v_1^\lambda,\dotsc,v_{n}^\lambda$ satisfying
\begin{align*}
  \lim_{\lambda \to 0}\|v_i^\lambda - \partial_{x_i}\|_{2,\nu} +|\ev_{q_1}(v_i^\lambda) - \ev_{q_1}(\partial_{x_i})| &= 0, & i&=1,\dotsc,n.
\end{align*}
In particular, we have a vector bundle 
\begin{equation*}
\Ker \ul{{\Gamma}} =\{(\lambda, v \in \kd_{\Gamma_\lambda}) \} \subset [0,\lambda_0] \times\Hi_{2,\nu}(\Delta_2, \C^n) 
\end{equation*}
which is isomorphic to $[0, \lambda_0] \times \R^{n}$ and with
topology given as in Section \ref{sec:glulim}.  A trivialization $\xi$ is given by $\xi_\lambda(v_i^\lambda) = \partial_{x_i}$, $i=1,\dotsc,n$. 
\end{lma}

\begin{proof}
This follows from Lemma \ref{lma:genkerisoweight}, together with the explicit understanding of the behavior of the kernel elements in neighborhoods of the special punctures,  as in \cite{trees}.
\end{proof}

To give an orientation to this problem, recall that we have fixed an orientation of the capping operator corresponding to the end-vertex, and notice that the positive special puncture $q_1$ must have $|\mu(q_1)|=1$. For any special puncture $p$ with $|\mu(p)|=1$ (of any tree $\tilde \Gamma$), choose an orientation of $\det  \ul{\dbar_{p,+}}$ so that the capping sequence for $\dbar_{\Gamma_\lambda}$ gives the orientation 
\begin{equation*}
 \Or ( \kd_{\Gamma_\lambda}) =  v_1^\lambda \wedge \dotsm \wedge v_n^\lambda
\end{equation*}
for all $\Gamma$ being elementary end-pieces in the case when $n >1$, and so that it is given by 
\begin{equation*}
 \Or ( \kd_{\Gamma_\lambda}) =  -v^\lambda 
\end{equation*}
in the case when $n=1$, where $v^\lambda$ is the vector that evaluates to the flow direction of the elementary piece at the end vertex.

%
%

\subsubsection{Orientations of switch-pieces}\label{sec:switchorient}
Let now $\Gamma$ be an elementary tree with a switch-vertex $s$. 
 Choose a trivialization of $TM$ along $\Gamma$ so that $T_s\Pi(\Sigma) = \spn(\partial_{x_1},\dotsc,\partial_{x_{n-1}})$, and so that the induced boundary conditions tend to either $(\R^{n-1} \oplus U_-(\R),\R^n)$ or $(\R^n,\R^{n-1}\oplus U_-(\R))$ as $\lambda \to 0$, where $U_-: \R \to U(1)$ is equal to the identity except for in a neighborhood of the switch point where it performs a uniform $-\pi$-rotation in $\C$, as in Figure \ref{fig:bend}.
Assume that $\partial_{x_n}$ points in the outward normal direction of $\Pi(\Lambda_s)$ at $s$, where $\Lambda_s$ is the union of the two sheets of $\Lambda$ containing $s$, and that $(\partial_{x_1},\dotsc,\partial_{x_n})$ gives an oriented basis for $T_sM$. 
Again we extend the trivialized boundary conditions from Section \ref{sec:treetriv} to the limit $\lambda=0$ by using that we have control of the region where the $-\pi$-rotation is performed. Using this we extend the family of operators 
$\dbar_{\Gamma_\lambda}: \Hi_{2,\nu}[\Gamma_\lambda] \to \Hi_{1,\nu}[\Gamma_\lambda]$ to $\lambda =0 $. 

\begin{lma}\label{lma:kers}
For $\lambda$ sufficiently small the kernel of $\dbar_{\Gamma_\lambda}$ is spanned by functions $v_1^\lambda,\dotsc,v_{n-1}^\lambda$ satisfying
\begin{equation*}
 \|v_i^\lambda - \partial_{x_i}\|_{2,\nu} + |\ev_p(v_i^\lambda) - \ev_p(\partial_{x_i})| \to 0, \, \lambda \to 0, \, i=1,\dotsc,n-1,
\end{equation*}
where $p$ is a special puncture of $\Gamma$.

In particular, we have a vector bundle
\begin{equation*}
 \Ker \ul{\Gamma}=\{(\lambda, v \in \kd_{\Gamma_\lambda}) \} \subset [0,\lambda_0] \times\Hi_{2,\nu}(\Delta_2, \C^n)
\end{equation*}
which is isomorphic to $[0, \lambda_0] \times \R^{n-1}$ and with topology given as in Section \ref{sec:glulim}. A trivialization $\xi$ is given by $\xi_\lambda(v_i^\lambda) = \partial_{x_i}$, $i=1,\dotsc,n-1$. 
\end{lma}

\begin{proof}
For $\lambda =0$  it follows from [\cite{trees}, Lemma 6.6] that the $\dbar_{\Gamma_0}$-problem is surjective and that $\kd_{\Gamma_0}$ is $n-1$-dimensional, spanned by functions of the form $\partial_{x_1},\dotsc,\partial_{x_{n-1}}$ in our chosen trivialization. The statement now follows from Proposition \ref{prp:keriso}, together with the explicit understanding of the behavior of the kernel elements in neighborhoods of the special punctures,  as in \cite{trees}.
\end{proof}

To give an expression for the capping orientation of $\Gamma$, notice that
one of the special punctures of $\Gamma$ must be of odd Maslov index
while the other one is of even Maslov index. Assume first that it is
the negative puncture which  is of odd Maslov index, and denote it by
$p_1$. Recall that we already have fixed an orientation of $\det
\ul{\dbar_{p_1,+}}$, and that this induces one on $\det
\ul{\dbar_{p_1,-}}$ from the gluing sequence \eqref{eq:capglubun} and
the canonical orientation of $\det \ul{\dbar_{p_1}}$ times $(-1)^{n \cdot|\mu(p_1)|}$. 
Now if $q$ is any positive  special puncture of even Maslov index, choose an orientation of $\det \ul{\dbar_{q,+}}$ so that it gives the following capping orientation of elementary switch-trees with one positive special puncture of even Maslov index, one negative  special puncture of odd Maslov index and where the switch is located on the lower sheet:
\begin{equation*}
\Or(\kd_{\Gamma_\lambda}) =  v_1^\lambda \wedge \dotsm \wedge v_{n-1}^\lambda.
\end{equation*}

These choices of orientations of capping operators also induces an orientation of $\det \ul{\dbar_{p,-}}$ for $p$ any  special puncture of even Maslov index, and this together with the chosen orientation of $\det \ul{\dbar_{q,+}}$ for $q$ any special puncture of odd Maslov index induces an orientation 
\begin{equation*}
 \Or( \kd_{\Gamma_\lambda}) = (-1)^{\sigma_{\swi}^{(k,k')}(l)} v_1^\lambda \wedge \dotsm \wedge v_{n-1}^\lambda
\end{equation*}
in the case when $\Gamma$ is an elementary switch-tree with the switch on the lower sheet, and with positive puncture of odd Maslov index of type $(k,k')$.

Similarly, if $\Gamma$ is an elementary switch-tree with the switch on the upper sheet and a positive special puncture of type $(k,k')$ (where $k=k'$ if and only if the positive puncture has even Maslov index) then there is a $\sigma_{\swi}^{(k,k')}(u) \in \{0,1\}$ so that our choices of orientations of capping operators as above induces the orientation 
\begin{equation*}
 \Or( \kd_{\Gamma_\lambda}) = (-1)^{\sigma_{\swi}^{(k,k')}(u)} v_1^\lambda \wedge \dotsm \wedge v_{n-1}^\lambda
\end{equation*}
on $\Gamma$.

For completeness, we define $\sigma_{\swi}^{(0,0)}(l) = \sigma_{\swi}^{(1,1)}(l) = 0$, so that 
\begin{equation*}
 \Or( \kd_{\Gamma_\lambda}) = (-1)^{\sigma_{\swi}^{(k,k')}(i)} v_1^\lambda \wedge \dotsm \wedge v_{n-1}^\lambda
\end{equation*}
gives the orientation for $\Gamma$ an elementary switch-tree with the switch on the upper sheet if $i=u$, lower sheet if $i=l$, and where $(k,k')$ gives the type of the positive special puncture. We say that such a switch-piece is of \emph{type $\sigma_{\swi}^{(k,k')}(i)$}.

%
%
%

\subsubsection{Orientations of pieces associated to punctures}\label{sec:orp}
Now let $\Gamma$ be an elementary tree containing a true  puncture $p$. Here $p$ can be either positive or negative, $1$- or $2$-valent. 

Choose a trivialization of $TM$ along $\Gamma$ so that $TM= T W^u(p) \oplus W$ for some complementary space $W$ in the case when $p$ is positive, and so that $TM= T W^s(p) \oplus W$ for some complementary space $W$ in the case when $p$ is negative. Recall that we have an associated 
operator $\dbar_{\Gamma_\lambda}: \Hi_{2,\nu}[\Gamma_\lambda] \to \Hi_{1,\nu}[\Gamma_\lambda]$, as described in Section \ref{sec:analys}. By taking the limit $\lambda \to 0$, we see that the boundary conditions tend to constant boundary conditions, and due to our choice of weights $\nu_\Gamma$ we get a well-defined Fredholm problem in the limit $\lambda = 0$, given by
\begin{equation*}
 \dbar_{\Gamma_0}: \Hi_{2,\nu}[\id] \to \Hi_{1,\nu}[0].
\end{equation*}

 By Lemma \ref{lma:genkerisoweight} we have that $\dbar_{\Gamma_\lambda}$ is surjective for $\lambda$ sufficiently small, with a kernel of dimension $k=\dim W^u(p)$ if $p$ is positive, and of dimension $k=\dim W^s(p)$ if $p$ is negative, spanned by functions $u_1^\lambda,\dotsc,u_k^\lambda$ satisfying
 \begin{equation*}
  \|u_i^\lambda - \partial_{x_i} \|_{2,\nu} \to 0, \qquad \lambda \to
  0,\quad  i=1,\dotsc,k.
 \end{equation*}
  Moreover, the tuple $(\partial_{x_1},\dotsc,\partial_{x_k})$ gives a basis for $\kd_{\Gamma_0}$,
 and by the explicit descriptions of the vectors $u_1^\lambda,\dotsc,u_k^\lambda$ from \cite{trees} we see that we can choose them so that 
 \begin{equation*}
   |\ev_q(u_i^\lambda) - \ev_q(\partial_{x_i})| \to 0
 \end{equation*}
 as $\lambda \to 0$. Here $q$ is a special puncture of $\Gamma$. 


Hence we get the following. 
\begin{cor}\label{cor:osbundle}
 We have a vector bundle $\Ker\ul{{\Gamma}} \subset [0,\lambda_0]
 \times \Hi_{2,\nu}(\R\times [0,1], \C^n)$ given by $\Ker{\ul{\Gamma}}
 =\{(\lambda, v \in \kd_{\Gamma_\lambda}) \}$, which is  isomorphic to
 $[0, \lambda_0] \times \kd_{\Gamma_0}$ and with topology given as in
 Section \ref{sec:glulim}.  A trivialization $\xi$ is given by $\xi_\lambda(u_i^\lambda) = \partial_{x_i}$, $i=1,\dotsc,k$. 
\end{cor}

In the case of a positive (negative) puncture we can replace
$\kd_{\Gamma_0}$ by $T_pW^u(p)$ ($T_pW^s(p)$),
and an orientation of this space induces canonically an orientation on
$\kd_{\Gamma_\lambda}$ via the trivialization $\xi$, given the
orientations we already have fixed. To that end, assume that 
\begin{equation*}
\partial_{x_1}\wedge \cdots \wedge \partial_{x_k}=\Or_{\capp}(T_pW^s(p))
\end{equation*}
 in the
case $p$ is negative, and that
\begin{equation*}
\partial_{x_1}\wedge \cdots \wedge \partial_{x_k}=\Or_{\capp}(T_pW^u(p))
\end{equation*}
in the case $p$ is positive.  

\begin{lma}\label{lma:signtrup}
Let $p$ be a true puncture and $\Gamma$ the elementary piece containing it. Assume that we have chosen orientations of the capping operators corresponding to special punctures as in Section \ref{sec:endorient} and Section \ref{sec:switchorient}. Then there is a choice of orientations of the capping operators corresponding to true punctures so that the following holds. 

If $p$ is positive and 1-valent, then 
\begin{equation}\label{eq:pos1valor}
\Or( \kd_{\Gamma_\lambda})= u_1^\lambda\wedge \cdots \wedge u^\lambda_{k}.
\end{equation}

If $p$ is positive of type $(k,k')$ and 2-valent, and $q_1$, $q_2$ are the negative punctures of $\Gamma$ ordered using the notation in Section \ref{sec:slit}, then there is a $\sigma_{\pos,2}^{(k,k')}(|\mu(q_1)|, |\mu(q_2)|) \in \{0,1\}$ so that 
\begin{equation*}
\Or( \kd_{\Gamma_\lambda})= (-1)^{\sigma_{\pos,2}^{(k,k')}(|\mu(q_1)|, |\mu(q_2)|)},
\end{equation*}
where 
\begin{equation*}
 \sigma_{\pos,2}^{(0,0)}(0, 0) = \sigma_{\pos,2}^{(1,1)}(0, 0) = \sigma_{\pos,2}^{(0,1)}(1, 0) = \sigma_{\pos,2}^{(1,0)}(1, 0) = 0.
\end{equation*}

If $p$ is negative and 1-valent of type $(k,k')$ then there is a $\sigma_{\negg,1}^{(k,k')}(I(p)) \in \{0,1\}$
so that 
\begin{equation}\label{eq:pos1valor2}
\Or( \kd_{\Gamma_\lambda})= (-1)^{\sigma_{\negg,1}^{(k,k')}(I(p))} u_1^\lambda\wedge \cdots \wedge u^\lambda_{k}.
\end{equation}

If $p$ is negative and 2-valent and where the positive puncture of $\Gamma$ is of type $(k,k')$ and the other negative puncture of $\Gamma$ is denoted by $q$, then there is a $\sigma_{\negg,2}^{(k,k')}(|\mu(p)|, |\mu(q)|, \tp(p)) \in \{0,1\}$
so that 
\begin{equation*}
\Or( \kd_{\Gamma_\lambda})= (-1)^{\sigma_{\negg,2}^{(k,k')}(|\mu(p)|, |\mu(q)|, \tp(p))} u_1^\lambda\wedge \cdots \wedge u^\lambda_{k}.
\end{equation*}

\end{lma}

\begin{rmk}
 From this we get that the orientation of $\Gamma$ only depends on the chosen orientation of $T_pW^u(p)$ and the type of the positive puncture of $\Gamma$, together with the following data
 \begin{itemize}
  \item the index of $p$, if $p$ is 1-valent and negative,
  \item the parity of the punctures of $\Gamma$, in the case $p$ is 2-valent and positive,
  \item the parity of the punctures of $\Gamma$ and the type of $p$ as in Table \ref{tab:two}, in the case $p$ is 2-valent and negative.
 \end{itemize}
In Section \ref{sec:signs1val} and Section \ref{sec:2valentsigns} we will compute explicit values for  \linebreak $\sigma_{\pos,2}^{(k,k')}(|\mu(q_1)|, |\mu(q_2)|) , \sigma_{\negg,1}^{(k,k')}(I(p)), \sigma_{\negg,2}^{(k,k')}(|\mu(p)|, |\mu(q)|, \tp(p))$ in the case $n>1$.  
\end{rmk}

\begin{proof}[Proof of Lemma \ref{lma:signtrup}]
For $p$ positive and 1-valent, choose $\sigma_{\capp,+}(p)$ so that \eqref{eq:pos1valor} holds. Notice that this choice does not depend on the particular $p$, but only on its type and index. This can be seen by considering the capping sequence for $\Gamma$ in the limit $\lambda = 0$ and noting that the signs $\sigma_{\capp,-,s}(q)$ for $q$ a special negative puncture only depend on the type of $q$.

For $p$ positive and 2-valent of type $(k,k')$ choose $\sigma_{\capp,+}(p)$ so that 
\begin{equation*}
 \Or( \kd_{\Gamma_\lambda})= \mathbbm{1}
\end{equation*}
 in the case when $|\mu(q_1)| = |\mu(q_2)| = 0$ and $p$ is of type $(0,0)$ or $(1,1)$, and also in the case when $|\mu(q_1)| = 1, |\mu(q_2)| = 0$ and $p$ is of type $(0,1)$ or $(1,0)$. This fixes the choice of $\sigma_{\capp,+}$ for all positive 2-valent punctures ,and hence induce orientations 
 \begin{equation*}
\Or( \kd_{\Gamma_\lambda})= (-1)^{\sigma_{\pos,2}^{(k,k')}(|\mu(q_1)|, |\mu(q_2)|)},
\end{equation*}
on the remaining positive 2-valent pieces. The dependence on $|\mu(q_1)|, |\mu(q_2)|$ comes from the capping sequences for $\Gamma$, or more precisely, from $\sigma_{\capp,-,s}(q_1)$ and $\sigma_{\capp,-,s}(q_2)$ which in turn are determined by the type of $p$ and $|\mu(q_1)|, |\mu(q_2)|$.

For $p$ negative and 1-valent, assume that the positive special puncture of $\Gamma$ is given by $q$. Then the capping sequence for $\Gamma_\lambda$ is given by  
\begin{equation*}
0 \to 
\kd_{\hat \Gamma_\lambda}
\to
\begin{bmatrix}
\kd_{p-}\\
\kd_{q+}\\
\kd_{\Gamma_\lambda}
\end{bmatrix}
\to
\begin{bmatrix}
\cd_{p-}\\
\cd_{q+}
\end{bmatrix}
\to 
\cd_{\hat \Gamma_\lambda}
\to 0.
\end{equation*}
By passing to the limit we see that this sequence can be understood as
\begin{equation*}
0 \to 
\begin{bmatrix}
T_pW^s(p) \\
\aux_2
\end{bmatrix}
\to
\begin{bmatrix}
\aux_2 \\
0\\
T_pW^s(p) 
\end{bmatrix}
\to
\begin{bmatrix}
T_pW^u(p) \\
\aux_1 \\
0
\end{bmatrix}
\to 
\begin{bmatrix}
T_pW^u(p) \\
\aux_1
\end{bmatrix}
\to 
 0
\end{equation*}
if $|\mu(p)|$ is even, and as 
\begin{equation*}
0 \to 
\begin{bmatrix}
T_pW^s(p) \\
\aux_2^2
\end{bmatrix}
\to
\begin{bmatrix}
\aux_2 \\
\aux_2\\
T_pW^s(p) 
\end{bmatrix}
\to
\begin{bmatrix}
T_pW^u(p) \\
0\\
0
\end{bmatrix}
\to 
\begin{bmatrix}
T_pW^u(p)
\end{bmatrix}
\to 
 0
\end{equation*}
if $|\mu(p)|$ is odd. 

In both cases we see that the induced orientation on $\kd_{\Gamma_\lambda}$ only
depend on the parity of $|\mu(p)|$, the index of $p$,  and on
the chosen orientation of the capping operators. For the positive capping operator, this only depends on the type of $q$, which agrees with the type of $p$. For the negative capping operator, this is induced by the choice of orientation of the positive capping operator, which in turn only depends on the chosen orientation of $W^u(p)$, the dimension of $W^u(p)$, and the type of $p$. Since the type of $p$ determines the parity of $\mu(p)$, it follows that the orientation of $\Gamma$ can be given as in \eqref{eq:pos1valor2}.

For $p$ negative and 2-valent, the result follows in a similar way.
\end{proof}

\begin{rmk}
 From now on, we assume that we have fixed orientations of the capping operators as described above. To indicate that this only depend on the index and type $(k,k')$ of $p$ when $p$ is a true puncture, we write $\sigma^{t,(k,k')}_{\pm}(I(p))$ instead of $\sigma_{\capp,\pm}(p)$. 
\end{rmk}

\subsubsection{Orientations of \texorpdfstring{$Y_0$}{Y}-pieces}
We have a similar result for the case of elementary $Y_0$-trees $\Gamma$, but here we get that the problem associated to the tree is surjective with an $n$-dimensional kernel isomorphic to $\R^n$ via evaluation. We pick a trivialization of $TM$ along $\Gamma$ using the flat coordinates defined here, and we let  
 $\dbar_{\Gamma_\lambda}: \Hi_{2,\nu}[\Gamma_\lambda] \to \Hi_{1,\nu}[\Gamma_\lambda]$ denote the $\dbar$-operator associated to $\Gamma$. 
 
 The following is proven in a similar way as for the pieces containing true punctures.
 \begin{lma}\label{lma:orientpiece4}
  We have a vector bundle $\Ker \ul{\Gamma}\subset [0,\lambda_0] \times \Hi_{2,\nu}(\Delta_3, \C^n)$ given by $\Ker \ul{\Gamma} =\{(\lambda, v \in \kd_{\Gamma_\lambda}) \}$, which is isomorphic to $[0, \lambda_0] \times \R^n$ and with topology given  as in Section \ref{sec:glulim}. A trivialization $\xi:\Ker \ul{\Gamma}\to [0, \lambda_0] \times \kd_{\Gamma_0}$ is given by $\xi_\lambda(u_i^\lambda) =(\lambda,\partial_{x_i})$, where $(u_1^\lambda,\dotsc,u_n^\lambda)$ is a basis for $\kd_{\Gamma_\lambda}$ satisfying 
   \begin{equation*}
  \|u_i^\lambda - \partial_{x_i} \|_{2,\nu} + |\ev_p(u_i^\lambda) - \ev_p(\partial_{x_i})| \to 0, \qquad \lambda \to 0,
 \end{equation*}
 where $p$ is a special puncture of $\Gamma$.
  \end{lma}

The capping orientation of $\Gamma$ is given similar as the case of
2-pieces. That is, let $p_0, p_1, p_2$ be the punctures of $\Gamma$
ordered in the counterclockwise direction along the boundary of the
standard domain of $\Gamma$ so that $p_0$ represents the positive
puncture. Assume further that
\begin{equation*}
 \partial_{x_1}\wedge \cdots \wedge \partial_{x_n}
\end{equation*}
gives the positive orientation of $T_{p_0}M$.
Then we have the following.
\begin{lma}\label{lma:y0}
Let $(k,k')$ be the type of the positive puncture $p_0$ of $\Gamma$. Then there is a $\sigma_{Y_0}^{(k,k')}(|\mu(p_1)|,
  |\mu(p_2)|) \in \{0,1\}$ so that the capping orientation of $\Gamma$ is given by
\begin{equation}
\Or(\kd_{\Gamma_\lambda}) = (-1)^{\sigma_{Y_0}^{(k,k')}(|\mu(p_1)|,
  |\mu(p_2)|)} u_1^\lambda \wedge \cdots \wedge u_n^\lambda. 
\end{equation}
\end{lma}
\begin{proof}
Again this follows by studying the capping sequences for $\Gamma$ in the different cases.
 \end{proof}
%
%

\subsubsection{Orientations of \texorpdfstring{$Y_1$}{Y}-pieces}
We have a similar result for the case of elementary trees $\Gamma$ with a $Y_1$-vertex $v$.  Choose a trivialization of $TM$ along $\Gamma$ so that $T_v\Pi(\Sigma) = \spn(\partial_{x_1},\dotsc,\partial_{x_{n-1}})$, and recall that the induced boundary conditions tend to $(\R^n, \R^{n-1} \oplus U_-(\R), \R^n)$ as $\lambda \to 0$, where  $U_-: \R \to U(1)$ is equal to the identity except for in a neighborhood of the boundary minimum where it perform a uniform $-\pi$-rotation in $\C$, as in Figure \ref{fig:bend}. Assume that $\partial_{x_n}$ points in the outward normal direction of $\Pi(\Lambda_v)$ at $v$, where $\Lambda_v$ is the union of the two sheets of $\Lambda$ containing $v$, and that $(\partial_{x_1},\dotsc,\partial_{x_n})$ gives an oriented basis for $T_vM$.  Again we extend the trivialized boundary conditions from Section \ref{sec:treetriv} to the limit $\lambda=0$ by using that we have control of the region where the $-\pi$-rotation is performed. Using this we extend the family of operators  
 $\dbar_{\Gamma_\lambda}: \Hi_{2,\nu}[\Gamma_\lambda] \to \Hi_{1,\nu}[\Gamma_\lambda]$  to $\lambda =0$. 
 
 The following is proven in a similar way as for the switch-pieces.
 \begin{lma}\label{lma:orientpiece5}
  We have a vector bundle $\Ker \ul{\Gamma}\subset [0,\lambda_0] \times \Hi_{2,\nu}(\Delta_3, \C^n)$ given by $\Ker \ul{\Gamma}=\{(\lambda, v \in \kd_{\Gamma_\lambda}) \}$, which is isomorphic to $[0, \lambda_0] \times \R^{n-1}$, and with topology given as in Section \ref{sec:glulim}. A trivialization $\xi$ is given by $\xi_\lambda(u_i^\lambda) = \partial_{x_i}$, where $(u_1^\lambda,\dotsc,u_{n-1}^\lambda)$ is a basis for $\kd_{\Gamma_\lambda}$ satisfying 
   \begin{equation*}
  \|u_i^\lambda - \partial_{x_i} \|_{2,\nu} +|\ev_p(u_i^\lambda) - \ev_p(\partial_{x_i})| \to 0, \qquad \lambda \to 0.
 \end{equation*}
  \end{lma}

Let $p_0, p_1,p_2$ be the punctures of $\Gamma$. The capping orientation of $\Gamma$ is given similar as in the case of $Y_0$-pieces:
\begin{lma}
Let $(k,k')$ be the type of the positive puncture $p_0$ of $\Gamma$. Then there is a $\sigma_{Y_1}^{(k,k')}(|\mu(p_1)|,
  |\mu(p_2)|) \in \{0,1\}$ so that the capping orientation is given by
\begin{equation}
\Or(\kd_{\Gamma_\lambda}) = (-1)^{\sigma_{Y_1}^{(k,k')}( |\mu(p_1)|, |\mu(p_2)|)} u_1^\lambda \wedge \dotsm \wedge u_{n-1}^\lambda
\end{equation}
where $u_1^\lambda, \dotsc u_{n-1}^\lambda$ are the vectors from Lemma \ref{lma:orientpiece5}.
\end{lma}
\begin{proof}
Again this follows by studying the capping sequences for $\Gamma$ in the different cases.
 \end{proof}

\begin{rmk}\label{rmk:n1choice}
 If $n=1$, we do not have trees with switches, and hence the way of choosing orientations of the cappping operators as described above does not work. Instead, we choose orientation for the capping operators at special punctures of odd Maslov index as described above, and then we choose orientations of the capping operators at special punctures of even Maslov index so that the elementary trees of type $\sigma_{Y_1}^{(0,1)}( 0, 0)$ has sign equal to $(-1)^0$. The discussion and choices of signs for punctures and the remaining $Y_1$--vertices then remain the same. 
\end{rmk}

The signs $\sigma_{\pos,2}^{(k,k')}(|\mu(q_1)|, |\mu(q_2)|) , \sigma_{\negg,1}^{(k,k')}(I(p)), \sigma_{\negg,2}^{(k,k')}(|\mu(p)|, |\mu(q)|, \tp(p)),
\sigma_{Y_0}^{(k,k')}(|\mu(q_1)|$, $|\mu(q_2)|), \sigma_{Y_1}^{(k,k')}(|\mu(q_1)|, |\mu(q_2)|), \sigma_{\swi}^{(k,k')}(i)$ introduced above will be called the \emph{sign functions}, since it is natural to view them as functions of the type of the positive puncture of the elementary tree, the index of the puncture in the case when it is an 1-valent elementary tree, the parity of the negative punctures in the case of an elementary tree with a 2-valent puncture or a 3-valent tree, and also if the switch is upper or lower in the case of an elementary switch-tree. 

In Sections \ref{sec:examples}, \ref{sec:der} and \ref{sec:analder} we will derive explicit formulas for these signs, after some additional orientation choices.


\subsection{Proof of that the capping orientation of a tree is well-defined}
\label{sec:proof}
In this section we prove that the capping orientation of the pre-glued disk $w_\lambda$ associated to $\Gamma$ coincides with the capping orientation of the corresponding true $J$-holomorphic disk $u_\lambda$ for $\lambda$ sufficiently small. 

Again we will use the vector bundle set-up, and to that end we use the
constructions from [\cite{trees}, Section 6]. That is, for each
$\lambda$ sufficiently small we will connect $w_\lambda$ to $u_\lambda$ by a path $(u_\lambda(s),\kappa_\lambda(s))$, $s \in[0,1]$,  of  punctured disks $u_\lambda(s)$ with boundary on $\Pi_\C (\Lambda_\lambda)$ and with domain $\Delta(\kappa_\lambda(s))$. Here $\kappa_\lambda(s)$ indicates the conformal structure. We assume that $u_\lambda(0) = w_\lambda$ and $u_\lambda(1) = u_\lambda$. We give an outline of the construction of this path, and refer to \cite{trees} for the details. 

First let $\phi_s: \R^2 \to \R^2$, $s \in[0,1]$, be a family of diffeomorphisms so that $\phi_0 = \id$ and so that $\phi_1$ maps the standard domain of $u_\lambda$ to the standard domain of $w_\lambda$. An explicit description of $\phi_s$ is given in [\cite{trees}, Section 6.2], 
we summarize the most important features here: 
\begin{enumerate}
\item $\frac{d \phi_s}{d s} = b \lambda^{-1} \alpha(\tau,t) \partial_\tau,$
 where $\alpha: \C \to \C$ is a function satisfying $\alpha = 1$ in
 uniform neighborhoods of boundary minima, and has support in a larger
 neighborhood of the boundary minima, and where $b$ is a constant that
 is needed to compensate for $|D^k\alpha|$, $k=1,2$;
\item $\|d \phi_s - \id\|_{\infty} = O(s)$, $s \in [-1,1]$. 
\end{enumerate} 


By pre-composing $w_\lambda$ with $\phi_{-s}$ we get a disk $\tilde
u_\lambda(s)$ with boundary arbitrarily close to the boundary of
$w_\lambda$ and with standard domain
$\phi_{1-s}(\Delta(u_\lambda))$. Using the exponential map along the
path $\tilde u_\lambda(s)$, $s \in[0,1]$, we can define the path $u_\lambda(s)$. In addition, by the constructions in \cite{trees} we may assume that 
\begin{equation*}
\|A_{s_1,\lambda}\circ\phi_{-s_1} - A_{s_2,\lambda}\circ \phi_{-s_2}\|_{C^2} \to 0, \, \text{ as } s_1 \to s_2,
\end{equation*}
where $A_{s,\lambda}:\partial \Delta(\kappa_\lambda(s)) \to U(n)$ is
the  boundary condition induced by $u_\lambda(s)$. We may also assume
that $A_{s,\lambda}\circ\phi_{-s}$ is constant in $s$ in neighborhoods
of each puncture of $\Delta(w_\lambda)$, so that the same system of
capping operators can be used for all $u_\lambda(s)$, $s \in [0,1]$. 

By rigidity we may assume that the linearizations of $\dbar$ at
$u_\lambda$ and $w_\lambda$, restricted to $\Hi_{2,\nu}[u_\lambda]$
and $\Hi_{2,\nu}[w_\lambda]$, respectively, are either both injective
or both surjective, and we will assume the
former. The case of surjective problems are treated in a similar fashion. 

First note that $\phi_1$ induces an isomorphism
\begin{equation*}
  \phi_{1,*}:T_{\kappa_\lambda(1)}\Co_{m+1} \to T_{\kappa_\lambda(0)}\Co_{m+1},
\end{equation*}
and we may assume that this map is orientation-preserving.
From Section \ref{sec:modspor} it then follows that it is enough to prove the following result to get that the capping orientations of $u_\lambda$ and $w_\lambda$ coincide.

\begin{prp}\label{prp:iso}
 For each $\lambda>0$ sufficiently small, there exists an isomorphism 
 \begin{equation*}
  \chi_\lambda:\Coker \dbar_{u_\lambda} \to \Coker \dbar_{w_\lambda}
 \end{equation*}
so that
 \begin{align}\label{eq:chi1}
   \xymatrix{
&T_{\kappa_\lambda(1)}\Co_{m+1} 
 \ar[r]^{\phi_{1,*}}
\ar[d]^{\Psi_{\kappa_\lambda(1)}}_\simeq
   &T_{\kappa_\lambda(0)}\Co_{m+1}
   \ar[d]^{\Psi_{\kappa_\lambda(0)}}_\simeq\\
&\Coker \dbar_{u_\lambda} 
\ar[r]^{\chi_\lambda} 
&\Coker\dbar_{w_\lambda}
} 
\end{align}
and
\begin{align}\label{eq:chi2}
\xymatrix{
0 \ar[r]
&\Ker \dbar_{\hat u_\lambda} 
\ar[r] 
\ar[d]
&\capp(\Gamma)
\ar[r]
\ar[d]^{\id}
& \Coker \dbar_{u_\lambda} 
\ar[r]
\ar[d]^{\chi_\lambda}
& \Coker \dbar_{\hat u_\lambda}
\ar[r]
\ar[d]
&0
\\
0 \ar[r]
&\Ker \dbar_{\hat w_\lambda} 
\ar[r] 
&\capp(\Gamma)
\ar[r]
& \Coker \dbar_{w_\lambda} 
\ar[r]
& \Coker \dbar_{\hat w_\lambda}
\ar[r]
&0
} 
\end{align}
commute. Here the vertical maps in the first diagram are defined as in
Section \ref{sec:modspor}, the
leftmost and rightmost vertical maps in the second diagram are induced by
homotopy via the path $u_\lambda(s)$, $s \in [0,1]$, and the maps in
the horizontal rows are induced by the capping sequences for the disks.
\end{prp}
\begin{proof}
 Since the linearization of $\dbar$ is injective at both $w_\lambda$ and $u_\lambda$, and since we can choose $\lambda$ so that $w_\lambda$ and $u_\lambda$ are arbitrarily close to each other, we can find $\lambda_0>0$ and paths $u_\lambda(s)$ taking $w_\lambda$ to $u_\lambda$, as described above, so that $\dbar_{u_\lambda(s)}$ is injective for all $0 <\lambda\leq\lambda_0$ and all $s\in[0,1]$. 
 
 Now consider the vector bundles 
 \begin{align*}
  &\pi_i: \V_i(\Gamma,\lambda) \to [0,1] \quad i=1,2,
 \end{align*}
with fibers
\begin{equation*}
\pi_2^{-1}(s) := \V_{2,s,\lambda}[\Gamma] = \Hi_{2,\nu}[u_\lambda(s)](\Delta_{m+1}(\kappa_\lambda(s)), u_\lambda(s)^*TT^*M)
\end{equation*}
and 
\begin{equation*}
\pi_1^{-1}(s):=\V_{1,s,\lambda}[\Gamma]  = \Hi_{1,\nu}[0](\Delta_{m+1}(\kappa_\lambda(s)), T^{*0,1} \Delta_{m+1}(\kappa_\lambda(s)) \otimes u_\lambda(s)^*TT^*M),
\end{equation*}
respectively.

 We define trivializations 
 \begin{equation*}
  \hat{\Phi} :  \V_2(\Gamma,\lambda) \to  [0,1] \times \V_{2,0,\lambda}[\Gamma]
 \end{equation*}
 and 
 \begin{equation*}
  \hat{\Psi}: \V_1(\Gamma,\lambda) \to [0,1] \times \V_{1,0,\lambda}[\Gamma]
 \end{equation*}
 given by
 \begin{align*}
 \begin{array}{cccl}
  \hat{\Phi}_s:& \V_{2,s,\lambda}[\Gamma] &\to& [0,1] \times  \V_{2,0,\lambda}[\Gamma]\\ 
  &  g &\mapsto& (s,A_{0,\lambda}  (A_{s,\lambda}\circ \phi_{s})^{-1} \cdot g \circ \phi_{s})
 \end{array}
 \end{align*}
 and  
 \begin{align*}
 \begin{array}{cccl}
 \hat{\Psi}_s:& \V_{1,s,\lambda}[\Gamma]  &\to& [0,1] \times \V_{1,0,\lambda}[\Gamma] \\
   & \alpha \otimes g &\mapsto& (s, A_{0,\lambda}  (A_{s,\lambda}\circ \phi_{s})^{-1} \cdot\phi_s^*\alpha \otimes  g \circ \phi_{s}).
  \end{array}
 \end{align*}
 
 From [\cite{trees}, Section 6.3] it follows that the restrictions of
the fully linearized $\dbar$-operators  $D \dbar_{u_\lambda(s),
  \kappa_\lambda(s)}$ to  $\V_{2,s,\lambda}[\Gamma]$ give rise to a bundle map 
\begin{equation*}
 \dbar_J: \V_2(\Gamma,\lambda) \to \V_1(\Gamma, \lambda),
\end{equation*}
and since the dimension of the cokernel of $\dbar_{u_\lambda(s)}$ is
assumed to be constant in $s$ it follows from standard Fredholm
arguments that we have a vector bundle 
 $\Coker \dbar_J$ over $[0,1]$, with fibers $\Coker
 \dbar_{u_\lambda(s)}$. Choose a trivialization induced by precomposition with $\phi_s$ as above and let $\chi_1:
 \Coker \dbar_{u_\lambda(1)} \to \cd_{u_\lambda(0)}$ be the induced linear
 map. Then it follows that \eqref{eq:chi1} is commutative. Moreover,
 by Proposition \ref{prp:megalemma2}, it follows that also
 \eqref{eq:chi2} commutes.
\end{proof}

\begin{proof}[Proof of Theorem \ref{thm:mainthm}]
The results in Section \ref{sec:elementor} state that the assumptions of Proposition \ref{prp:comiso} and Proposition  \ref{prp:glu} hold for elementary trees, and that the capping orientation of these pieces can be expressed in terms of oriented submanifolds in $M$ together with combinatorial data coming from the tree. Thus, the results in Propositions \ref{prp:comiso} -- \ref{prp:stababund} imply that the stabilized capping orientation of a rigid tree $\Gamma$ can be computed in terms of oriented intersections of submanifolds in $M$ times signs coming from combinatorial data from the tree. By Proposition \ref{prp:mainstabsign}, this orientation gives the capping orientation of the corresponding pre-glued disk $w_\lambda$ if $\lambda$ is sufficiently small, up to a combinatorial sign. 

Since Proposition  \ref{prp:iso} implies that there is a $\lambda_0>0$ so that for all $\lambda < \lambda_0$ the capping orientation of $w_\lambda$ equals the capping orientation of the true $J$-holomorphic disk $u_\lambda$ corresponding to $\Gamma$, the result follows.
\end{proof}

\subsection{Computing the gluing algorithm formula}\label{sec:stabstab}
In this section we give a sketch of how the algorithm in \cite{korta} is derived from the earlier results in this paper. The algorithm is given in an inductive way, where the induction runs over the number of the true vertices of the sub flow trees of $\Gamma$. For technical reasons we only allow certain sub flow trees, described below. 

The induction is performed as follows. 
\begin{description}
\item[Base case]
The elementary trees are given orientations as in Section \ref{sec:elementor}.

\item[Inductive hypothesis]
After step $k \geq 1$ we assume that the following sub flow trees of
$\Gamma$ are given  orientations:
\begin{itemize}
\item  all sub flow trees with at most $k$ true vertices and a special
  positive puncture,
\item all sub flow trees with $k-1$ switch vertices, a true positive
  puncture, a negative special puncture, and no other vertices.
\end{itemize}
We also assume that these orientations can be realized as orientations of either the flow-outs of the trees at the special puncture, or as orientations of the intersection manifolds at the vertex adjacent to the special puncture.

\item[Inductive step]
In  step $k+1$ we create sub flow trees with $k+1$ true vertices. This
is done by gluing sub flow trees together at $3$-valent vertices, or
by joining a sub flow tree to a $2$-valent piece or a switch piece. The trees in the input of the gluing are given by one of the following collection of oriented partial flow trees.
\begin{enumerate}
 \labitem{{(y0g)}}{y0g} One $Y_0$-piece, oriented as in the base step, and two sub flow trees with positive special punctures and oriented flow-outs. The sub flow trees are glued to the negative special punctures of the $Y_0$-piece. See Figure \ref{fig:gluedy0flow}.
  \labitem{{(y1g)}}{y1g} One $Y_1$-piece, oriented as in the base step,  and two sub flow trees with positive special punctures and oriented flow-outs. The sub flow trees are glued to the negative special punctures of the $Y_1$-piece. See Figure \ref{fig:gluedy1flow}.
   \labitem{{(ng)}}{ng} One $2$-valent negative piece, oriented as in the base step, and one sub flow tree with a positive special puncture and oriented flow-out.  The sub flow tree is glued to the negative special puncture of the $2$-piece. See Figure \ref{fig:glue2whole}.
  \labitem{{(sgn)}}{sgn} One switch-piece, oriented as in the base step, and one sub flow tree with a positive special puncture and oriented flow-out. The sub flow tree is glued to the negative special puncture of the switch-piece. See Figure \ref{fig:glueswichp}.
  \labitem{{(sgp)}}{sgp} One switch-piece, oriented as in the base step, and one sub flow tree with a negative special puncture, a true $1$-valent positive puncture, $k \geq 0$ switches, and no other vertices. The latter tree has oriented flow-out at the negative special puncture.  The sub flow tree is glued to the positive special puncture of the switch-piece. See Figure \ref{fig:glueswichn}.
\end{enumerate}

We use Propositions \ref{prp:glu} and \ref{prp:stababund} to complete the induction step, in the following way. Let $\Gamma$ be a sub flow tree with special puncture $p$, true puncture $q$ adjacent to $p$ and let $e$ be the edge connecting them. Then, in all cases except for when $q$ is a switch we have that $\F_p(\Gamma) \simeq \I_q\times [0,\infty)$ (or $\F_p(\Gamma) \simeq \I_q\times [0,A]$ for some $A>0$). In the case when $q$ is a switch we have that $\F_p(\Gamma) \simeq I_q$. Thus, with Propositions \ref{prp:glu} and \ref{prp:stababund} together with our induction algorithm we can give an orientation to the intersection manifold of the glued tree from an oriented intersection of the trees in the input, times some additional gluing signs. This implies that also the flow-out of the glued tree can be given an orientation (by using the orientation of $e$ in the cases when it is needed), and the induction step is then completed. In Section \ref{sec:stabstabstab} we show in detail how this works in the case of gluing at a $Y_0$-vertex. The other cases are similar and are outlined in the following sections.

\item[The final step]
To finish the algorithm
we perform the \emph{final gluing}, in which we will recover $\Gamma$. This gluing is similar to the one in the induction step, but for technical reasons the inputs will be restricted to be of one of the following kind.
\begin{enumerate}
 \labitem{{(f0)}}{f11}  Two sub flow trees $\Gamma_1$, $\Gamma_2$. Here $\Gamma_1$ has one true $1$-valent vertex, one positive special puncture $p$, and no other vertices.
 The tree $\Gamma_2$ has a true $1$-valent positive puncture $a$, $k \geq 0$ switches, a negative special puncture $p$ and no other vertices. The orientations of $\Gamma_1$ and $\Gamma_2$ are given by  orientations of the tangent spaces of their flow outs at their common special puncture. 
\labitem{{(f1)}}{f1}  Two sub flow trees $\Gamma_1$, $\Gamma_2$. Here $\Gamma_1$ has one positive special puncture $p$, at least $2$ true vertices where the vertex $q$ adjacent to $p$ is not a switch, and the orientation of $\Gamma_1$ is given by the orientation of the intersection manifold at $q$. The tree $\Gamma_2$ has a true $1$-valent positive puncture, $k \geq 0$ switches, a negative special puncture and no other vertices. The orientation of $\Gamma_2$ is given by an orientation of the tangent space of the flow out at $q$. 
\labitem{{(f2)}}{f2} One $2$-valent positive piece,  oriented as in the base step, and $2$ sub flow trees $\Gamma_1$, $\Gamma_2$. The sub flow trees have positive special punctures, and their orientations are given in terms of orientations of the tangent spaces of the flow-outs at their positive punctures.
\end{enumerate}
The orientation of $\Gamma$ is given by the oriented intersection of the trees in the input, times some additional gluing signs. This we explain in Subsection \ref{sec:finish}. 
\end{description}

\begin{figure}[ht]
      \labellist
\small\hair 2pt
\pinlabel $p_0$ [Br] at 20 400
\pinlabel $p_1$ [Br] at  490 304 
\pinlabel $p_2$ [Br] at 490 520
\pinlabel $\Gamma_v$ [Br] at 260 480
\pinlabel $\Gamma_1$ [Br] at 850 201
\pinlabel $\Gamma_2$ [Br] at 850 726
\pinlabel $\F_{p_1}(\Gamma_1)\cap{D_{p_1}}$ [Br] at 820 0
\pinlabel $\F_{p_2}(\Gamma_2)\cap{D_{p_2}}$ [Br] at 1230 460 
\endlabellist
\centering
\includegraphics[height=4.cm]{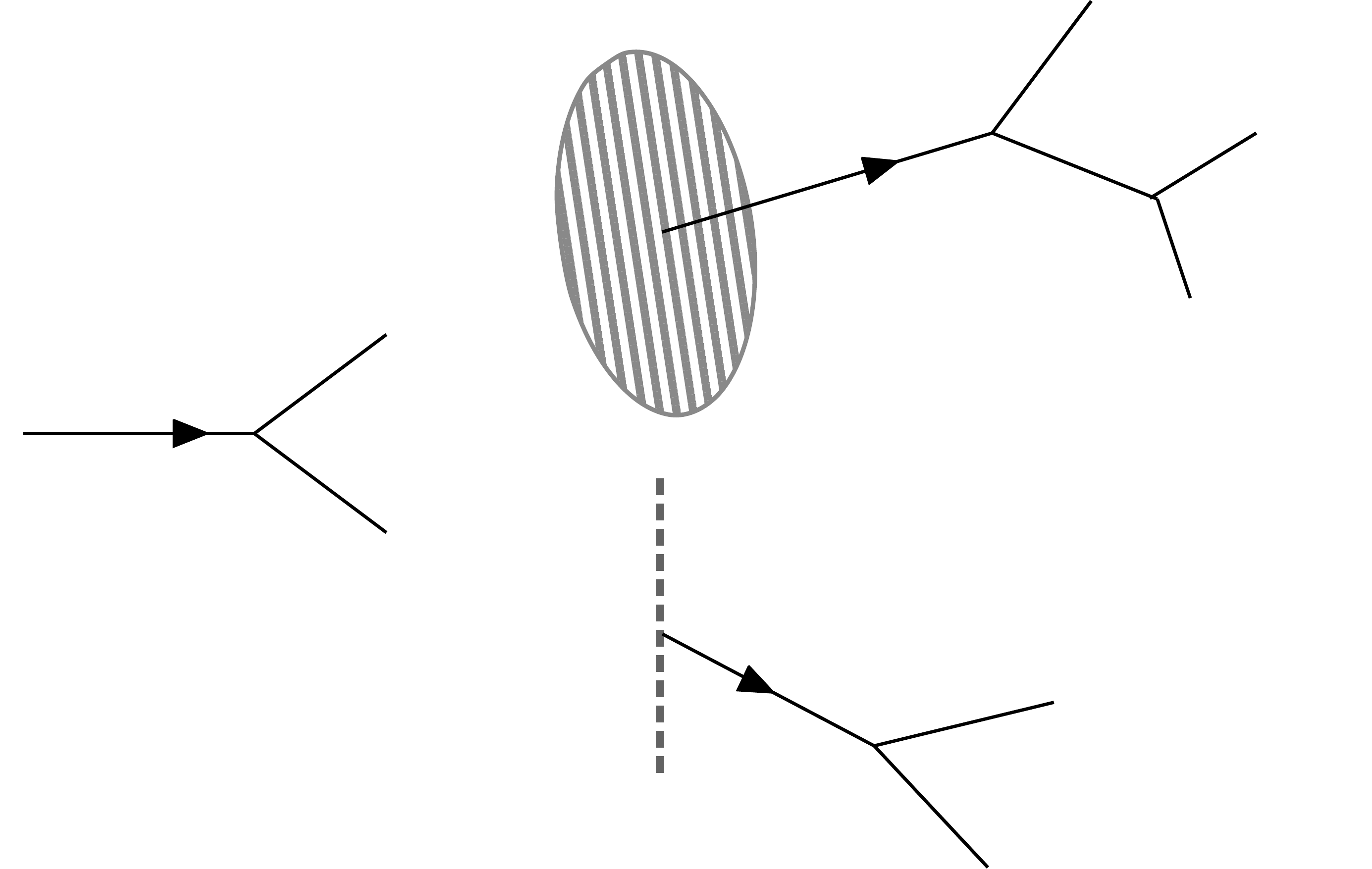} 
\hspace{1.8cm}
      \labellist
\small\hair 2pt
\pinlabel $t$ [Br] at 100 210
\pinlabel $x_1$ [Br] at  535 505
\pinlabel $x_2$ [Br] at 400 705
\pinlabel $x_3$ [Br] at 260 530
\pinlabel $\F_{p_0}(\Gamma)\cap{D_{p_0}}$ [Br] at 420 10
\endlabellist
\centering
\includegraphics[height=3.8cm, width=5.cm]{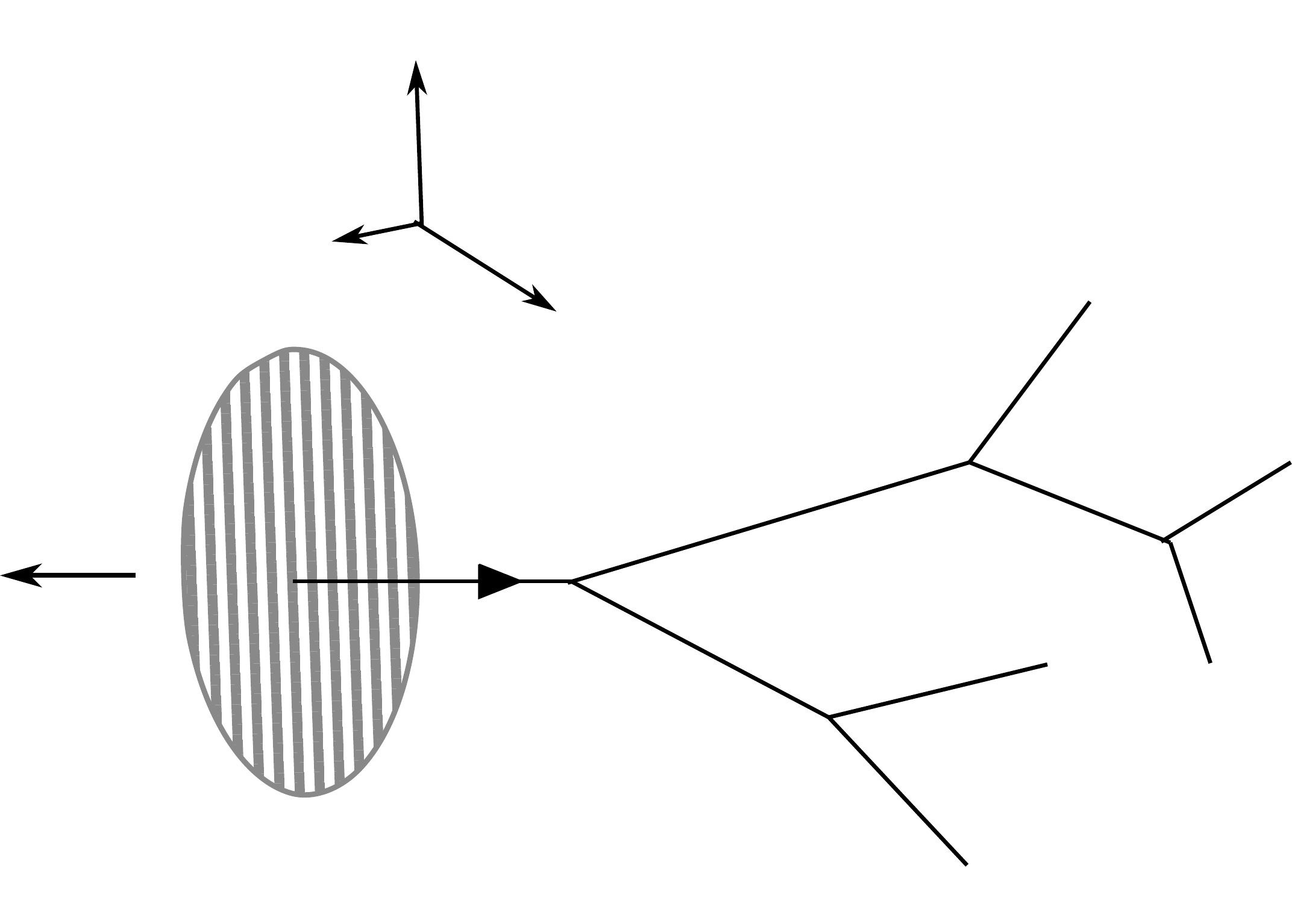} 
\caption{Gluing of two trees at a $Y_0$-vertex. The dashed disks are the flow-outs of the trees, intersected with a co-dimensional 1 disk transverse to the tree. The vector $t$ added when considering the flow-out of the glued tree is indicated in the second picture.}
\label{fig:gluedy0flow}
\end{figure}

\begin{figure}[ht]
      \labellist
\small\hair 2pt
\pinlabel $p_0$ [Br] at 20 400
\pinlabel $p_1$ [Br] at  490 304 
\pinlabel $p_2$ [Br] at 490 520
\pinlabel $\Sigma$ [Br] at 220 260
\pinlabel $\Gamma_v$ [Br] at 200 480
\pinlabel $\Gamma_1$ [Br] at 850 201
\pinlabel $\Gamma_2$ [Br] at 850 726
\pinlabel $\F_{p_1}(\Gamma_1)\cap{D_{p_1}}$ [Br] at 820 0
\pinlabel $\F_{p_2}(\Gamma_2)\cap{D_{p_2}}$ [Br] at 1230 460 
\endlabellist
\centering
\includegraphics[height=4.cm]{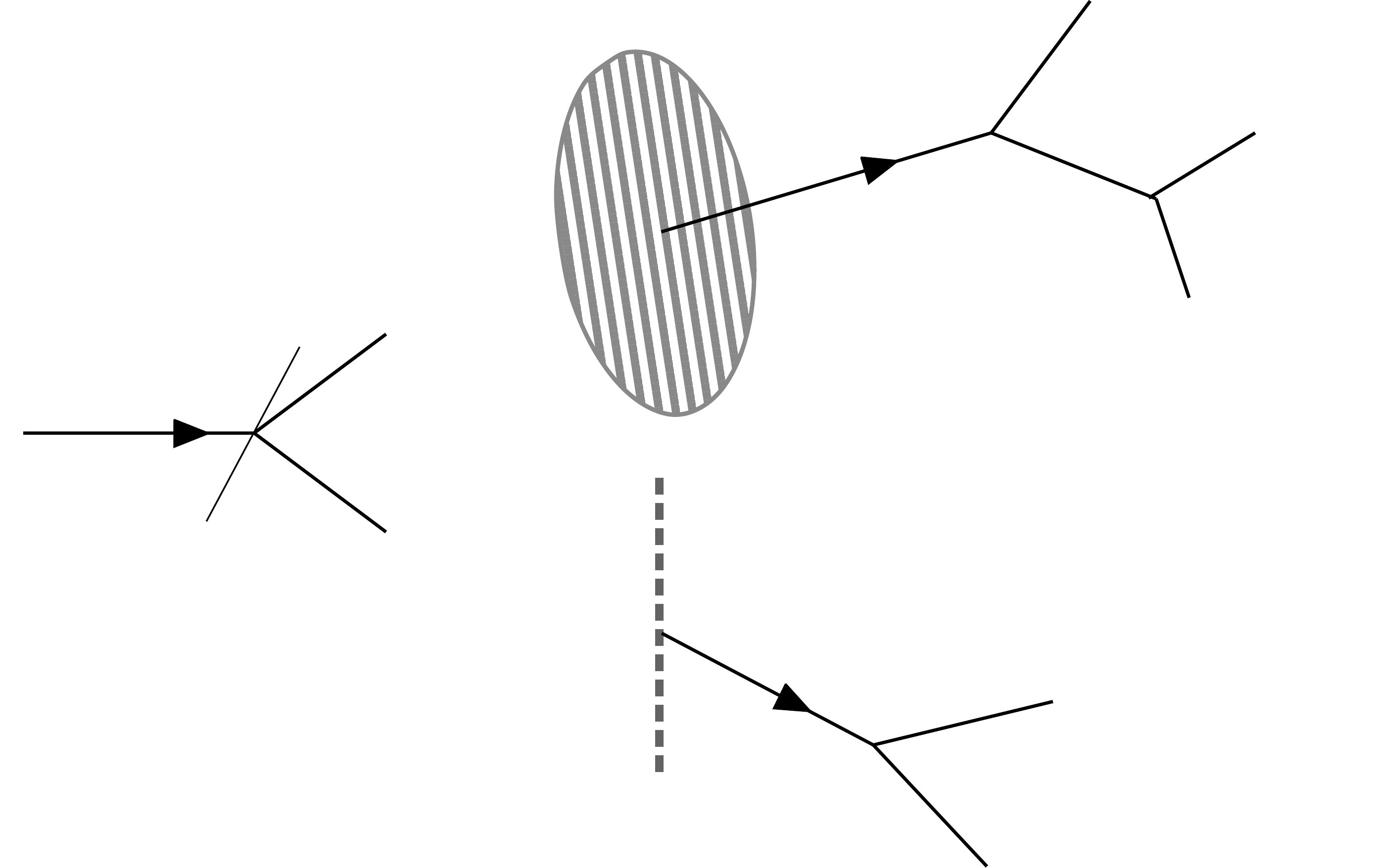} 
\hspace{1.8cm}
      \labellist
\small\hair 2pt
\pinlabel $t$ [Br] at 150 205
\pinlabel $x_1$ [Br] at  505 490
\pinlabel $x_2$ [Br] at 310 650
\pinlabel $x_3$ [Br] at 299 480
\pinlabel $\F_{p_0}(\Gamma)\cap{D_{p_0}}$ [Br] at 365 60
\endlabellist
\centering
\includegraphics[height=3.8cm]{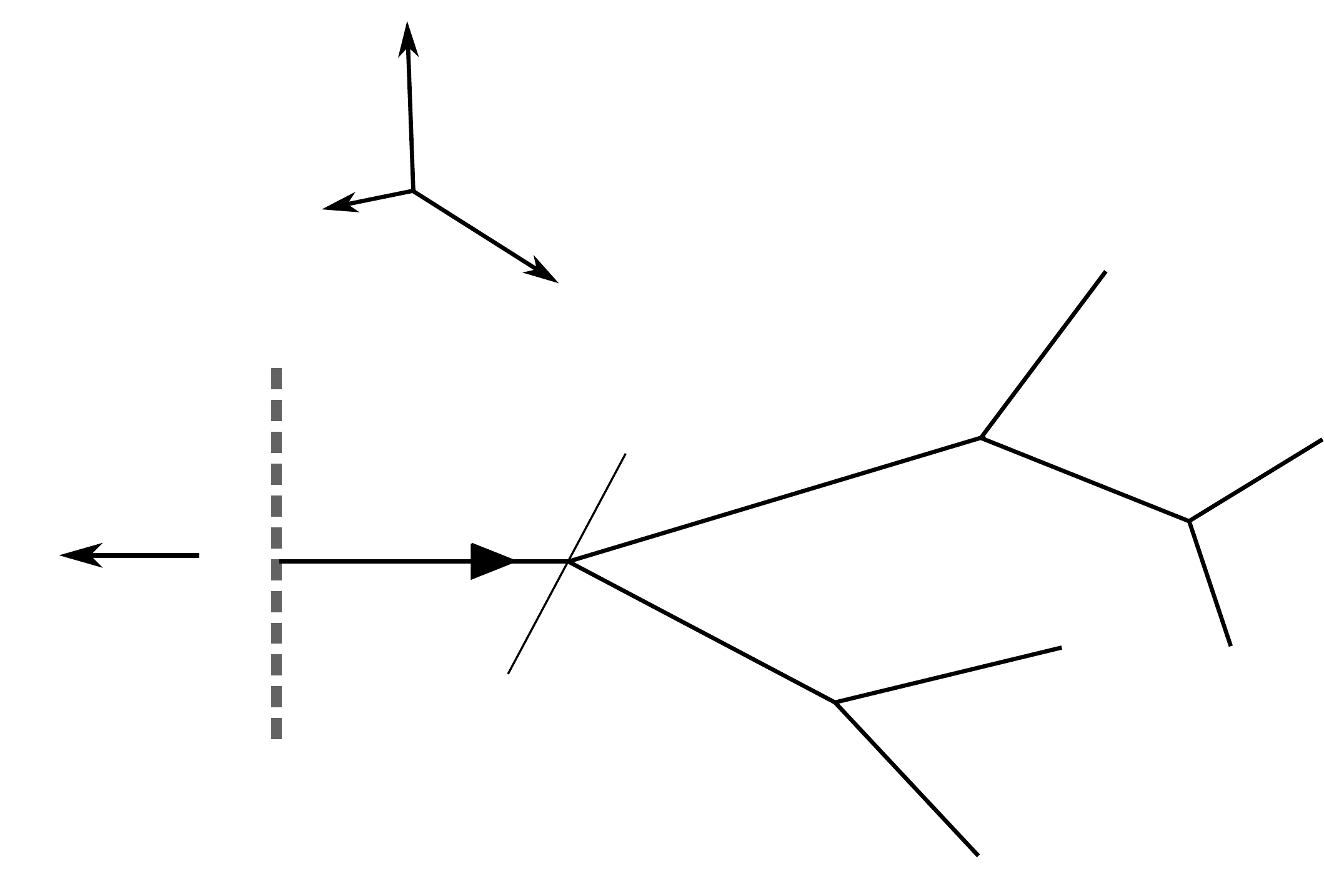} 
\caption{Gluing of two trees at a $Y_1$-vertex. The dashed disks are the flow-outs of the trees, intersected with a co-dimensional 1 disk transverse to the tree. The vector $t$ added when considering the flow-out of the glued tree is indicated in the second picture.}
\label{fig:gluedy1flow}
\end{figure}

\begin{figure}[ht]
\labellist
\small\hair 2pt
\pinlabel ${p_0}$ [Br] at  5 95 
\pinlabel ${p_1}$ [Br] at 225 115
\pinlabel $q$ [Br] at 115 125
\pinlabel $\Gamma_{q}$ [Br] at 80 50
\pinlabel $\Gamma_1$ [Br] at 310 50
\pinlabel $\Gamma$ [Br] at 800 50
\endlabellist
\centering
\includegraphics[height=2.cm]{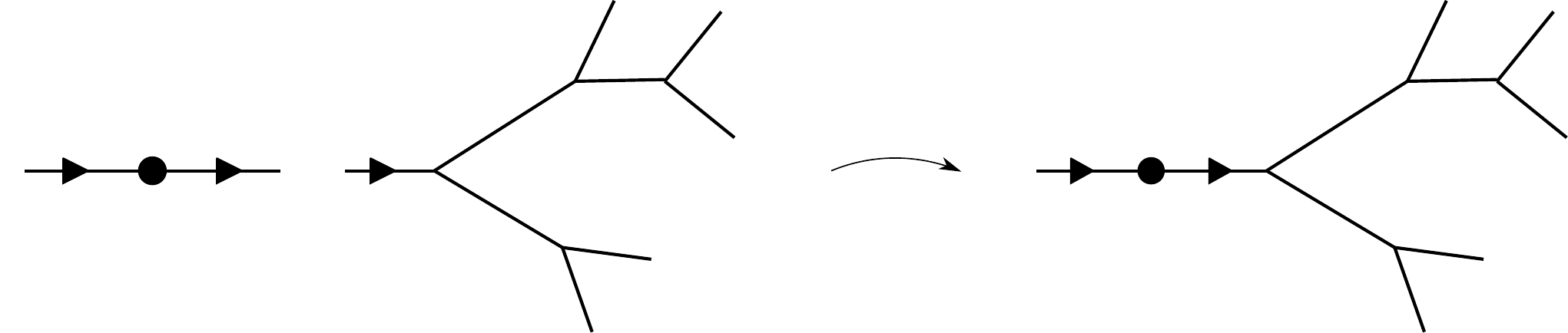}
\caption{Gluing a sub flow tree $\Gamma_1$ to a $2$-valent negative puncture.}
\label{fig:glue2whole}
\end{figure}

\begin{figure}[ht]
\labellist
\small\hair 2pt
\pinlabel ${p_0}$ [Br] at  35 125 
\pinlabel ${p_1}$ [Br] at 330 125
\pinlabel $s$ [Br] at 170  110
\pinlabel $\Gamma_{s}$ [Br] at 100 50
\pinlabel $\Sigma$ [Br] at 275 60
\pinlabel $\Gamma_1$ [Br] at 450 50
\pinlabel $\Gamma$ [Br] at 1140 50
\endlabellist
\centering
\includegraphics[height=2.cm]{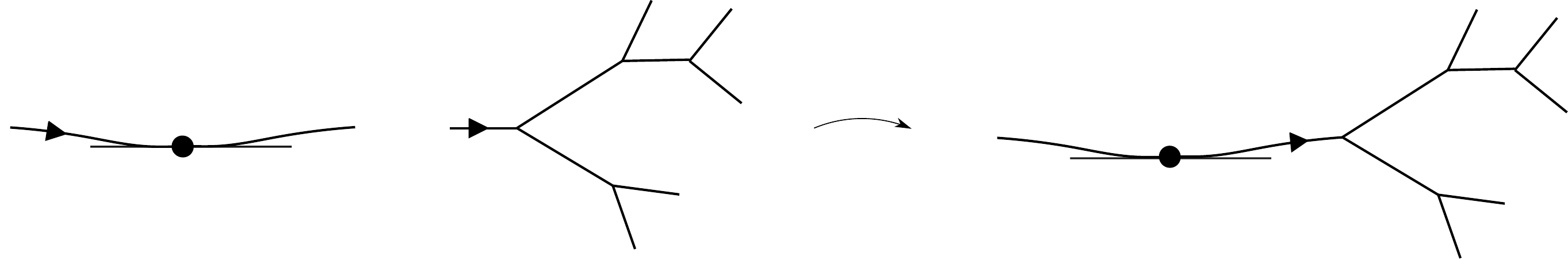}
\caption{Gluing a sub flow tree $\Gamma_1$ to a switch-piece
  $\Gamma_s$ with positive puncture $p_0$, negative puncture $p_1$ and
  switch-vertex $s$.}
\label{fig:glueswichp}
\end{figure}

\begin{figure}[ht]
\labellist
\small\hair 2pt
\pinlabel $a$ [Br] at  2 262 
\pinlabel ${p_0}$ [Br] at  715 290 
\pinlabel ${p_1}$ [Br] at 1010 290
\pinlabel $s$ [Br] at 845  275
\pinlabel $\Gamma_{s}$ [Br] at 860 210
\pinlabel $\Gamma$ [Br] at 500 50
\pinlabel $\Gamma_2$ [Br] at 360 210
\endlabellist
\centering
\includegraphics[height=3.cm]{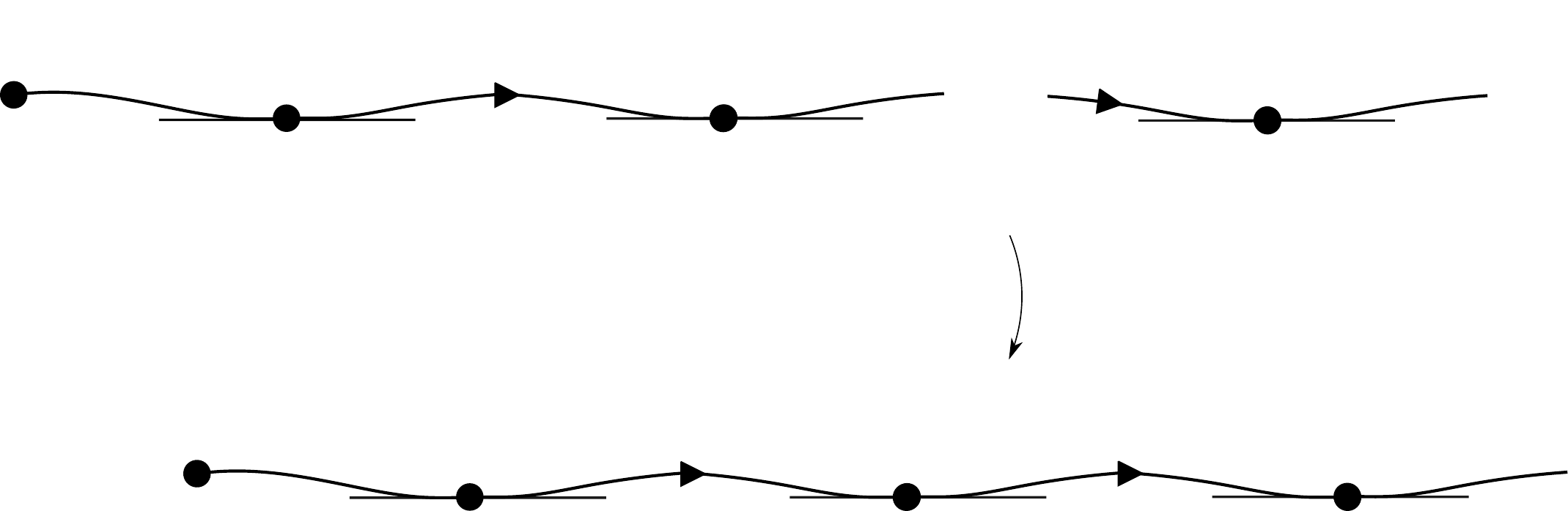}
\caption{Gluing a sub flow tree $\Gamma_2$  with $k$ switches and one true positive puncture $a$ to a switch-piece
  $\Gamma_s$ with positive puncture $p_0$, negative puncture $p_1$ and
  switch-vertex $s$.}
\label{fig:glueswichn}
\end{figure}

%
%
 \begin{rmk}
 In the cases \ref{y0g}, \ref{y1g}, \ref{ng}, the pre-stabilized capping orientation of $\Gamma$ in the limit equals the orientation of the intersection manifold. 
\end{rmk}

%

\subsubsection{Gluing at a $Y_0$-vertex}\label{sec:stabstabstab}
Let $\Gamma_v$ be the $Y_0$-piece we are gluing at, denote its
$Y_0$-vertex by $v$, and let $\Gamma_1$ and $\Gamma_2$ denote the two
sub flow trees that are to be glued to $\Gamma_v$ to give the sub flow
tree $\Gamma$. Let  $p_0,p_1, p_2$ denote the punctures of $\Gamma_v$,
ordered using the notation in Section \ref{sec:slit}.
  Assume that $\Gamma_1$, $\Gamma_2$ are numbered so that $\Gamma_i$ is attached to $\Gamma_v$ at $p_i$, $i=1,2$.  

By applying Proposition \ref{prp:glu} twice we get, for $\lambda$ sufficiently small, that the pre-stabilized kernel of $\Gamma$ is spanned by vectors $v_1^\lambda,\dotsc,v_m^\lambda$ satisfying 
\begin{align*}
 & \| v_i^{\lambda_1} - v_i^{\lambda_2}\|_{2,\nu} \to 0, &\lambda_1 &\to \lambda_2, \\
 & | \ev_v(v_i^\lambda) - \partial_{x_i}| \to 0 , & \lambda &\to 0,
\end{align*}
for $i=1,\dotsc,m$ and for a suitable choice of trivialization of $TM$.
We also get that the pre-stabilized capping orientation of $\Gamma$ can be computed using the sequence \eqref{eq:glu10l} twice, if we for each gluing multiply the result by the gluing sign coming from Proposition \ref{prp:megalemma4paux}.

 To perform this computation, we start by gluing $\Gamma_2$ to
 $\Gamma_v$, which will give us an intermediate tree
 $\Gamma_{2v}$. From the gluing sequence \eqref{eq:glu10l} we get that
 $\kd_{\Gamma_{2,s}} \simeq \kd_{\Gamma_{2v,s}}$, and the gluing sign
 from Proposition \ref{prp:megalemma4paux}, Case 1, equals 
 \begin{equation*}
  \nu_1\equiv |\mu(p_1)| \cdot(\bm(\Gamma_2) +1 + |\mu(p_2)|), \pmod{2}.
 \end{equation*}
Here we have used that 
\begin{equation}\label{eq:mu0f}
 |\mu(p_0)|+|\mu(p_1)|+|\mu(p_2)|\equiv 0, \pmod{2}
\end{equation}
for $Y_0$-pieces,
and that
\begin{equation}\label{eq:kerdim}
 \dim \Ker \capp (\Gamma') \equiv \bm(\Gamma') +1 + |\mu(p)|, \pmod{2}
\end{equation}
for any sub flow tree $\Gamma'$ with positive special puncture
$p$. Recall that $\bm(\Gamma') $ is the number of boundary minima of $\Delta(\Gamma')$. Moreover, from Lemma \ref{lma:orientpiece4} we get
\begin{equation}\label{eq:kercoker}
 \dim \kd_{\Gamma_{2v}} = \dim \kd_{\Gamma_2}, \quad
 \dim \cd_{\Gamma_{2v}} = \dim \cd_{\Gamma_2}.
\end{equation}

 Now we glue $\Gamma_1$ to $\Gamma_{2v}$. 
From Proposition \ref{prp:glu} we get that the orientation of 
\begin{equation*}
T_v \I_v(\Gamma) = T_v \F_{v}(\Gamma_1) \cap T_{v} \F_v(\Gamma_{2})
\end{equation*}
induced by the gluing sequence \eqref{eq:glu10l} equals the pre-stabilized orientation of $\Gamma$, times $(-1)^{\nu_1+\nu_2}$, where $\nu_2$ comes from gluing $\Gamma_1$ to $\Gamma_{2v}$ and is given by Proposition \ref{prp:megalemma4paux}, Case 2. Using that 
\begin{align}\label{eq:dimkercap}
 \dim \Ker \capp \Gamma_{2v} &\equiv 1 + \bm(\Gamma_2) + |\mu(p_0)| + |\mu(p_1)|,\\
 \label{eq:dimcokercap}
 \dim \Coker \capp \Gamma_{2v}  &\equiv \dim \Coker \capp \Gamma_{2} \equiv \dim \kd_{\Gamma_{2,s}} + n + 1 + |\mu(p_2)|, 
\end{align}
modulo 2, 
which follows from Lemma \ref{lma:capbundle0} and Lemma \ref{lma:special} together with Lemma \ref{lma:cokersimp}, and that 
\begin{align}\label{eq:sistacoker}
 \dim \cd_{\Gamma_i} &\equiv \bm(\Gamma_i) + \dim \vk( \Gamma_i), \pmod{2}, \quad i =1,2, \\ \nonumber
 \dim \cd_{\Gamma} &\equiv 1+\bm(\Gamma) + \dim \vk( \Gamma)_{ps}, \pmod{2}, \\ \nonumber
   \bm(\Gamma) &= \bm(\Gamma_1) +\bm(\Gamma_2) +1,
\end{align}
and also  \eqref{eq:kercoker} and \eqref{eq:mu0f}, we get that $\nu_2$ is given by 
\begin{align*}
  \nu_2 &\equiv e(\Gamma_1)\cdot[\bm(\Gamma_2) + e(\Gamma_2) +1] +\dim \kd_{\Gamma_1}\cdot[\bm(\Gamma_2) + |\mu(p_2)| +1] \\
  &\quad {}+ [\bm(\Gamma_1) + \dim \vk(\Gamma_1) +n]\cdot[|\mu(p_2)| + \dim \kd_{\Gamma_{2,s}} + n +1]\\
  &\quad {}+ n \cdot [\bm(\Gamma_2) + \dim \vk(\Gamma_2)]\\
  &\quad {} + \dim \vk(\Gamma_1) \cdot [\bm(\Gamma_2) + \dim \vk(\Gamma_2) + \dim \kd_{\Gamma_1}
 + \dim \kd_{\Gamma}] \\
   &\quad {}+ \dim \vk(\Gamma_2)\cdot [  \dim \kd_{\Gamma_1} + \dim \kd_{\Gamma} + \dim \kd_{\Gamma_2}] \\
   &\quad {}+  \dim \vk(\Gamma)_{ps} \cdot [ \dim \vk(\Gamma)_{ps} + \dim \vk(\Gamma_1) + \dim \vk(\Gamma_2)],
\end{align*}
modulo 2. 
This can be rearranged to 
\begin{align*}
  \nu_2 &\equiv e(\Gamma_1)\cdot[\bm(\Gamma_2) + e(\Gamma_2) +1] +[n + \dim \kd_{\Gamma_{1,s}}]\cdot[\bm(\Gamma_2) + |\mu(p_2)| +1] + n\\
  &\quad {}+ \bm(\Gamma_1) \cdot[|\mu(p_2)| + \dim \kd_{\Gamma_{2,s}} + n +1] + n \cdot \dim \kd_{\Gamma_{2,s}}\\
  &\quad {} + \dim \vk(\Gamma_1) \cdot [\dim \kd_{\Gamma_{2,s}} + n + \dim \kd_{\Gamma_1}
 + \dim \kd_{\Gamma_{ps}}] \\
   &\quad {}+ \dim \vk(\Gamma_2)\cdot [  n+\dim \kd_{\Gamma_{1,s}} + \dim \kd_{\Gamma_{ps}} + \dim \kd_{\Gamma_2}] \\
   &\quad {}+  \dim \vk(\Gamma)_{ps}, \pmod{2}. 
\end{align*}


%
%
%
%
  Now if we use that  
 \begin{equation*}
  \dim \kd_{\Gamma_{ps}}  + n =  \dim \kd_{\Gamma_{1,s}} + \dim \kd_{\Gamma_{2,s}}
 \end{equation*}
we see that
 \begin{align}\nonumber
 &\dim \vk(\Gamma_1) \cdot [\dim \kd_{\Gamma_{2,s}} + n + \dim \kd_{\Gamma_1}
 + \dim \kd_{\Gamma_{ps}}] \\ \nonumber
   &\quad {} + \dim \vk(\Gamma_2)\cdot [  n+\dim \kd_{\Gamma_{1,s}} + \dim \kd_{\Gamma_{ps}} + \dim \kd_{\Gamma_2}] \\ \nonumber 
   &\quad {} +  \dim \vk(\Gamma)_{ps}  \\ \nonumber 
   &\equiv \dim \vk(\Gamma_1) \cdot \dim \vk(\Gamma_1) + \dim \vk(\Gamma_2)\cdot \dim \vk(\Gamma_2)\\ \nonumber 
    &\quad {} +\dim \vk(\Gamma)_{ps} \\ \label{eq:bra}
    &\equiv n +  \dim \kd_{\Gamma_1} + \dim \kd_{\Gamma_2} + \dim \kd_{\Gamma}, \pmod{2}.
\end{align}
 
If we now sum everything up we get  
\begin{align*}
\nu_1 + \nu_2 &\equiv 
 e(\Gamma_1)\cdot[\bm(\Gamma_2) + e(\Gamma_2) +1] + \bm(\Gamma_1)\cdot[|\mu(p_2)| + \dim \kd_{\Gamma_{2,s}} + n +1] \\
 &\quad {}+[n+\dim \kd_{\Gamma_{1,s}}+|\mu(p_1)|]\cdot[\bm(\Gamma_2) + |\mu(p_2)| +1] \\
  &\quad {}+ \dim \kd_{\Gamma_1} + \dim \kd_{\Gamma_2} + \dim \kd_{\Gamma} + n \cdot \dim \kd_{\Gamma_{2,s}}, \pmod{2}.
\end{align*}

This is exactly the sign we add at each $Y_0$-vertex in [\cite{korta},
Section 6], except that in that case we also add the sign in Lemma \ref{lma:prestab}, which comes from going from the pre-stabilized problems to the stabilized problems of the incoming trees $\Gamma_1$ and $\Gamma_2$, and also the sign related to the capping orientation of the piece $\Gamma_v$ from Lemma \ref{lma:y0}.

\subsubsection{Gluing at a $Y_1$-vertex}\label{sec:y1glu}
 This is similar to the $Y_0$-case. The only difference is that the $Y_1$-piece does not have kernel isomorphic to $\R^n$, but to $\R^{n-1}$. Thus, in the first gluing sequence, where we glue $\Gamma_2$ to $\Gamma_v$ to get an intermediate tree $\Gamma_{2v}$ with pre-stabilized kernel isomorphic to $T_v\F_v(\Gamma_2) \cap T_v \Pi(\Sigma)$, we get that the gluing sign from Proposition \ref{prp:megalemma4paux} is given by 
\begin{align*}
  \nu_1&\equiv \dim \kd_{\Gamma_2} +  |\mu(p_1)| \cdot(\bm(\Gamma_2) +1 + |\mu(p_2)|)\\
   &\quad {} + \dim \vk(\Gamma_2) \cdot [\dim \kd_{\Gamma_2} + \dim \kd_{\Gamma_{2v}}] \\
   &\quad {} + \dim \vk(\Gamma_{2v}) \cdot [ \dim \cd_{\Gamma_{2}} + \dim \cd_{\Gamma_{2v}}], \pmod{2}.
 \end{align*}
Here we have used \eqref{eq:kerdim} and that  
\begin{equation}\label{eq:mu1}
 |\mu(p_0)|+|\mu(p_1)|+|\mu(p_2)|\equiv 1, \pmod{2}
\end{equation}
for $Y_1$-pieces. Now notice that 
\begin{align*}
 \dim \kd_{\Gamma_2} + \dim \kd_{\Gamma_{2v}} &\equiv \dim \kd_{\Gamma_{2,s}} + \dim \kd_{\Gamma_{2v,s}}\\
  &\quad {}+    \dim \vk(\Gamma_{2})+ \dim \vk(\Gamma_{2v}) \\
 & \equiv 1 + \dim \vk(\Gamma_{2})+ \dim \vk(\Gamma_{2v}), \pmod{2},
\end{align*}
and that 
\begin{align*}
 \dim \cd_{\Gamma_2} + \dim \cd_{\Gamma_{2v}} &\equiv \bm(\Gamma_2) + \bm(\Gamma_{2v}) + 1 \\
   &\quad {}+   
 \dim \vk(\Gamma_{2})+ \dim \vk(\Gamma_{2v}) \\
 & \equiv \dim \vk(\Gamma_{2})+ \dim \vk(\Gamma_{2v}), \pmod{2}.
\end{align*}
Thus we get 
\begin{align*}
  \nu_1&\equiv \dim \kd_{\Gamma_2} +  |\mu(p_1)| \cdot(\bm(\Gamma_2) +1 + |\mu(p_2)|) + \dim \vk(\Gamma_2) \cdot \dim \vk(\Gamma_{2v}) \\
   &\quad {} + \dim \vk(\Gamma_{2v}) \cdot [ \dim \vk(\Gamma_{2}) + \dim \vk(\Gamma_{2v})]\\
   &\equiv \dim \kd_{\Gamma_2} +  |\mu(p_1)| \cdot(\bm(\Gamma_2) +1 + |\mu(p_2)|) + \dim \vk(\Gamma_{2v}), \pmod{2}. 
\end{align*}

For the second gluing, the only difference compared to the $Y_0$-case is that we should use \eqref{eq:mu1} instead of \eqref{eq:mu0f}, and that we should replace $\vk(\Gamma_2)$ with $\vk(\Gamma_{2v})$ and $\kd_{\Gamma_2}$ with $\kd_{\Gamma_{2v}}$. Note that we should not replace $\kd_{\Gamma_{2,s}}$ with $\kd_{\Gamma_{2v,s}}$, since this summand comes from \eqref{eq:dimcokercap}. 
Thus we  get that $\nu_2$ is given by 
\begin{align*}
  \nu_2 &\equiv e(\Gamma_1)\cdot[\bm(\Gamma_2) + e(\Gamma_2) +1] +\dim \kd_{\Gamma_1}\cdot[\bm(\Gamma_2) + |\mu(p_2)| ] \\
  &\quad {}+ [\bm(\Gamma_1) + \dim \vk(\Gamma_1) +n]\cdot[|\mu(p_2)| + \dim \kd_{\Gamma_{2,s}} + n +1]\\
  &\quad {}+ n \cdot [\bm(\Gamma_2) + \dim \vk(\Gamma_{2v})]\\
  &\quad {} + \dim \vk(\Gamma_1) \cdot [\bm(\Gamma_2) + \dim \vk(\Gamma_{2v}) + \dim \kd_{\Gamma_1}
 + \dim \kd_{\Gamma}] \\
   &\quad {}+ \dim \vk(\Gamma_{2v})\cdot [  \dim \kd_{\Gamma_1} + \dim \kd_{\Gamma} + \dim \kd_{\Gamma_{2v}}] \\
   &\quad {}+  \dim \vk(\Gamma)_{ps} \cdot [ \dim \vk(\Gamma)_{ps} + \dim \vk(\Gamma_1) + \dim \vk(\Gamma_{2v})],
\end{align*}
modulo 2.
This can be rearranged to 
\begin{align*}
  \nu_2 &\equiv e(\Gamma_1)\cdot[\bm(\Gamma_2) + e(\Gamma_2) +1] +[n + \dim \kd_{\Gamma_{1,s}}]\cdot[\bm(\Gamma_2) + |\mu(p_2)|] \\
  &\quad {}+ \bm(\Gamma_1) \cdot[|\mu(p_2)| + \dim \kd_{\Gamma_{2,s}} + n +1] + n \cdot \dim \kd_{\Gamma_{2,s}}\\
  &\quad {} + \dim \vk(\Gamma_1) \cdot [\dim \kd_{\Gamma_{2,s}} + n +1+ \dim \kd_{\Gamma_1}
 + \dim \kd_{\Gamma_{ps}}] \\
   &\quad {}+ \dim \vk(\Gamma_{2v})\cdot [  n+\dim \kd_{\Gamma_{1,s}} + \dim \kd_{\Gamma_{ps}} + \dim \kd_{\Gamma_{2v}}] \\
   &\quad {}+  \dim \vk(\Gamma)_{ps}, \pmod{2}. 
\end{align*}
 
Now use that  $\dim \kd_{\Gamma_{2,s}} \equiv  \dim \kd_{\Gamma_{2v,s}} +1, \pmod{2},$ and that 
\begin{equation*}
 \dim \kd_{\Gamma_{2v,s}} + n   + \dim \kd_{\Gamma_{1,s}}
 + \dim \kd_{\Gamma_{ps}} \equiv 0, \pmod{2},
\end{equation*}
to get 
 \begin{align*}
  \nu_2 &\equiv e(\Gamma_1)\cdot[\bm(\Gamma_2) + e(\Gamma_2) +1] +[n + \dim \kd_{\Gamma_{1,s}}]\cdot[\bm(\Gamma_2) + |\mu(p_2)|] \\
  &\quad {}+ \bm(\Gamma_1) \cdot[|\mu(p_2)| + \dim \kd_{\Gamma_{2,s}} + n +1] + n \cdot \dim \kd_{\Gamma_{2,s}}\\
  &\quad {} + \dim \vk(\Gamma_1)+ \dim \vk(\Gamma_{2v}) + \dim \vk(\Gamma)_{ps} \\
  &\equiv e(\Gamma_1)\cdot[\bm(\Gamma_2) + e(\Gamma_2) +1] +[n + \dim \kd_{\Gamma_{1,s}}]\cdot[\bm(\Gamma_2) + |\mu(p_2)|] \\
  &\quad {}+ \bm(\Gamma_1) \cdot[|\mu(p_2)| + \dim \kd_{\Gamma_{2,s}} + n +1] + n \cdot \dim \kd_{\Gamma_{2,s}}\\
  &\quad {} + \dim \kd_{\Gamma_1}+ \dim \kd_{\Gamma_{2v}} + \dim \kd_\Gamma +n, \pmod{2}.  
\end{align*}
Summing up we get 
\begin{align*}
 \nu_1 + \nu_2 &\equiv e(\Gamma_1)\cdot[\bm(\Gamma_2) + e(\Gamma_2) +1] +[n + \dim \kd_{\Gamma_{1,s}}]\cdot[\bm(\Gamma_2) + |\mu(p_2)|] \\
  &\quad {}+ \bm(\Gamma_1) \cdot[|\mu(p_2)| + \dim \kd_{\Gamma_{2,s}} + n +1] + n \cdot \dim \kd_{\Gamma_{2,s}}\\
  &\quad {}+ \dim \kd_{\Gamma_1}+ \dim \kd_{\Gamma_{2v}} + \dim \kd_\Gamma +n \\
   &\quad {}+\dim \kd_{\Gamma_2} +  |\mu(p_1)| \cdot[\bm(\Gamma_2) +1 + |\mu(p_2)|] + \dim \vk(\Gamma_{2v}) \\
   &\equiv e(\Gamma_1)\cdot[\bm(\Gamma_2) + e(\Gamma_2) +1] +[n + \dim \kd_{\Gamma_{1,s}}]\cdot[\bm(\Gamma_2) + |\mu(p_2)|+1] \\
  &\quad {}+ \bm(\Gamma_1) \cdot[|\mu(p_2)| + \dim \kd_{\Gamma_{2,s}} + n +1] + n \cdot \dim \kd_{\Gamma_{2,s}}\\
  &\quad {}+ \dim \kd_{\Gamma_1}+\dim \kd_{\Gamma_2} +  \dim \kd_\Gamma + \dim \kd_{\Gamma_{1,s}} + \dim \kd_{\Gamma_{2,s}}+1  \\
   &\quad {} +  |\mu(p_1)| \cdot(\bm(\Gamma_2) +1 + |\mu(p_2)|), \pmod{2}.
\end{align*}

\subsubsection{Gluing at a negative 2-valent puncture}
This is again similar to the case of gluing at a $Y_0$-vertex. 
 The only difference is that now we only have to perform one gluing, and that  $\kd_{\Gamma_{1,s}} \simeq \R^n$ while $\kd_{\Gamma_q}$ is trivial. Here we use the notation from Figure \ref{fig:glue2whole}. 

 Recall that we have two different types of negative 2-pieces. For pieces of Type 1, the gluing sign in Proposition \ref{prp:megalemma4paux} is given by Case 1 and equals
 \begin{align*}
  \nu &\equiv \dim \kd_{\Gamma_1} \cdot n + [\bm(\Gamma_1) + 1 + |\mu(p_1)|] \cdot[|\mu(q)| + n] \\
   &\quad {} +[\dim \kd_{\Gamma_1} + n] \cdot \dim \kd_{\Gamma_1}\\
  &\equiv \dim \kd_{\Gamma_1}  + [\bm(\Gamma_1) + 1 + |\mu(p_1)|] \cdot[|\mu(q)| + n], \pmod{2}. 
 \end{align*}

 For pieces of Type 2, the gluing sign in Proposition \ref{prp:megalemma4paux} is given by Case 2 and equals 
 \begin{align*}
  \nu &\equiv e(\Gamma_1) \cdot[|\mu(p_1)| + |\mu(p_0)| + 1 + |\mu(p_1)| + |\mu(p_0)|] +
  \dim \kd_{\Gamma_1} \cdot [1 + |\mu(p_1)| + |\mu(p_0)|] \\
   &\quad {} + [\bm(\Gamma_1) + \dim \vk(\Gamma_1) + n] \cdot [n +  |\mu(q)| + 1] +[\dim \kd_{\Gamma_1} + n] \cdot \dim \kd_{\Gamma_1}\\
  &\equiv e(\Gamma_1) \cdot 1 +
  \dim \kd_{\Gamma_1} \cdot [n + |\mu(q)|] + \dim \vk(\Gamma_1) \cdot [n +  |\mu(q)| + 1]\\
   &\quad {} + \bm(\Gamma_1) \cdot [n +  |\mu(q)| + 1] +n \cdot |\mu(q)| \\
     &\equiv e(\Gamma_1)  + n \cdot [n + |\mu(q)|] + \dim \vk(\Gamma_1)
    + \bm(\Gamma_1) \cdot [n +  |\mu(q)| + 1] +n \cdot |\mu(q)|\\
    &\equiv e(\Gamma_1)  +\dim \kd_{\Gamma_1}
    + \bm(\Gamma_1) \cdot [n +  |\mu(q)| + 1], \pmod{2}. 
   \end{align*}
   Here we have used \eqref{eq:mu0f}, \eqref{eq:kerdim}  and Lemma \ref{lma:cokersimp}. We have also used that 
   \begin{equation*}
    \dim \vk(\Gamma_1) \equiv  \dim \kd_{\Gamma_1} +  \dim \kd_{\Gamma_{1,s}} \equiv \dim \kd_{\Gamma_1} + n, \pmod 2.
   \end{equation*}

 %
%
%
%
%
%
%

\subsubsection{Gluing at a switch-piece, gluing at the negative special puncture}\label{sec:switchglu}
  This is similar to the case of $Y_1$-pieces, but now we have only one gluing and the sign from Proposition \ref{prp:megalemma4paux} is given by Case 2. Here we use the notation that $\Gamma_2$ is the switch-piece with positive puncture $p_0$ and negative puncture $p_1$, and notice that 
  \begin{equation*}
 \dim \Ker \capp (\Gamma_2) \equiv |\mu(p_0)| + |\mu(p_1)| \equiv 1, \pmod{2},   
  \end{equation*}
and that  
\begin{equation*}
 \dim \Coker \capp (\Gamma_2) \equiv 0, \pmod{2}.   
  \end{equation*}
  
  It follows that we get a gluing sign
\begin{align*}
   \nu&\equiv  e(\Gamma_1)\cdot[|\mu(p_0)| + |\mu(p_1)|+|\mu(p_0)| + |\mu(p_1)|] +\dim \kd_{\Gamma_1}\cdot[|\mu(p_0)| + |\mu(p_1)|] \\
  &\quad {} + \dim \vk(\Gamma_1) \cdot [ \dim \kd_{\Gamma_1}
 + \dim \kd_{\Gamma}] \\
   &\quad {}+  \dim \vk(\Gamma) \cdot [ \dim \vk(\Gamma) + \dim \vk(\Gamma_1)]\\
   &\equiv \dim \kd_{\Gamma_1} +  \dim \vk(\Gamma_1) \cdot [ \dim \kd_{\Gamma_{1,s}} +1 
 + \dim \kd_{\Gamma_s}] + \dim \vk(\Gamma)\\
    &\equiv \dim \kd_{\Gamma_1} +  \dim \vk(\Gamma_1) \cdot 0 + \dim \vk(\Gamma)\\
    &\equiv \dim \kd_{\Gamma_1} +   \dim \kd_{\Gamma_s} + \dim \kd_{\Gamma}, \mod{2}.
\end{align*}

%
\subsubsection{Gluing at a switch-piece, gluing at the positive special puncture}
 Again this is similar to the first step in the case of $\Y_1$-vertices. But now the gluing sign is given by Proposition \ref{prp:sista} and equals 
 \begin{align*}
 \sigma &\equiv [|\mu(a)| +k+n+1+1] \cdot [n+1]+ [0 +n]\cdot[\dim W^u(a) + n + |\mu(a)|]\\
 &\equiv |\mu(a)| + k\cdot[n+1] + n\cdot[1 + \dim W^u(a)], \pmod{2}, 
\end{align*}
 where we use the notation from Figure \ref{fig:glueswichn}.


\subsubsection{The final gluing}\label{sec:finish}
Here we explain the last step in the algorithm, where we finally get the whole expression for the capping orientation of the rigid flow tree. Recall from the introduction of this section that this step is divided into three different cases \ref{f11}--\ref{f2}. We will use the notation from there.

\begin{description}[style=unboxed,leftmargin=0cm]

\item[\ref{f11} -- \it  Trees with nontrivial automorphism group]
In this case the kernel of the associated $\dbar$-operator will be $1$-dimensional, and $\dim V_{\con}(\Gamma)=0$ (so that the stabilized problem coincides with the un-stabilized).

 From Remark \ref{orientrule} it follows that the sign of $\Gamma$ is given by comparing the capping orientation of $\Gamma$ with the negative flow orientation of $\Gamma$. 
Indeed, in the case when $\Gamma$ has one negative puncture, let $w_\lambda$ be the the pre-glued disk associated to $\Gamma$ for $\lambda$ sufficiently small. The capping sign of $\Gamma$ comes from comparing the capping orientation of $\kd_{w_\lambda}$ with the orientation of $\T_2$ under the isomorphism $T\T_2 \ni v \mapsto dw_\lambda \cdot v \in \kd_{w_\lambda}$. Remark \ref{orientrule} implies that we use the orientation on $\T_2$ as in [\cite{orientbok}, Section 3.4.1], where this orientation is induced by evaluation at a point in the lower hemisphere of $\partial D_2$. Since this corresponds to $\R + 1i$ on the strip $\Delta_2 = \Delta(\Gamma)$, and since this part of the strip is mapped to the lower part of the Lagrangian lift of the tree $\Gamma$ which is oriented by the negative flow direction of $\Gamma$, it follows that the capping
orientation of $w_\lambda$ coincides with the positive
orientation of $T\T_2$ if and only if it agrees with the negative flow direction of $\Gamma$ in $M$. In the case when $\Gamma$ has no negative puncture but an end it follows from the orientation conventions of $\T_1$ in [\cite{orientbok}, Section 3.4.1], together with the conventions in \eqref{eq:enend}, that the sign of $\Gamma$ again is given as described above. Compare with Lemma \ref{lma:signeq1}.

Let $p$ be the point where we cut $\Gamma$ to get $\Gamma_1$ and $\Gamma_2$, and assume that the flow orientation of $\Gamma$ at $p$ is given by $\partial_{x_1}$. Also assume that the orientation of $T_p\F_p(\Gamma_1)\cap T_p\F_p(\Gamma_2)$ induced by the sequence \eqref{eq:glu10l} is given by
\begin{equation*}
 (-1)^{\sigma_1} \partial_{x_1},
\end{equation*}
where $T_{p}\F_p(\Gamma_i)$ is given the inductively defined capping orientation of $\Gamma_i$, $i=1,2$.


Now it follows from Proposition \ref{prp:glu} together with
Proposition \ref{prp:mainstabsign} that the capping orientation of $\Gamma$ is given by $(-1)^\nu$,
\begin{align*}
 \nu= 1+\sigma_1 + \sigma +  \beta(\s, A_d) +\mu(\Gamma),
\end{align*}
where $\sigma$ is the sign from Proposition \ref{prp:sista} and equals
\begin{align*}
 \sigma &\equiv [|\mu(a)| +k+\dim \kd_{\Gamma_{1}}+1] \cdot \dim
 \kd_{\Gamma_1} + [0 +n]\cdot[\dim W^u(a) + n + |\mu(a)|]\\
 &\equiv [|\mu(a)| +k] \cdot [\dim W^u(a) + k + n + 1]  + n\cdot[\dim W^u(a) + n + |\mu(a)|]\\
 &\equiv k \cdot [\dim W^u(a) +|\mu(a)| + n ] + \dim W^u(a)\cdot[n + |\mu(a)|] + |\mu(a)| + n, \pmod{2},
\end{align*}
and where $ \beta(\s, A_d)$ is the sign from Definition \ref{def:triv} which is needed to compensate for the fact that we have been using a trivialization which may differ from the one induced by the spin structure $\s$ in the calculations above.
Moreover, notice that in this case $\mu(\Gamma)=0$ (recall that $\mu$ was defined in Definition \ref{def:musign}).

Summing up we get that  
\begin{align*}
 \nu&\equiv 1+\sigma_1 +k \cdot [\dim W^u(a) +|\mu(a)| + n ] + \dim W^u(a)\cdot[n + |\mu(a)|] + |\mu(a)| + n\\
 &\quad {}+   \beta(\s, A_d), \pmod{2}.
\end{align*}

\begin{rmk}\label{rmk:sigmastab}
 That this equals the sign from \cite{korta} follows from the fact that in that paper, we collect the $k$ gluings of switches to the sub flow tree with $a$ as a positive puncture in the sign called $\nu_{\stab}$. In the formula above, however, we assume that the signs from these gluings  already are captured by the orientation of $\F_p(\Gamma_2)$.
\end{rmk}

\item[\ref{f1} -- \it Trees with trivial automorphism group, the positive puncture is $1$-valent]
If the automorphism group of the domain corresponding to the rigid
flow tree $\Gamma$ is trivial, then we have that $\kd_{\Gamma_s}$ also
is trivial, and by Proposition \ref{prp:mainstabsign} we have that the
stabilized sign of $\Gamma$ gives the capping sign of $w_\lambda$ times $(-1)^{\mu(\Gamma)}$. 

To calculate the stabilized sign of $\Gamma$, we use the arguments
from Proposition \ref{prp:glu}, with the pre-stabilized tree
$\Gamma_1$ and the (un-stabilized) tree $\Gamma_2$ as inputs. See
Figure \ref{fig:finish1val}. Let $q$ denote the vertex of $\Gamma_1$
adjacent to $p$, and assume that  $T_q\I_q(\Gamma_1)$ is given the
pre-stabilized orientation of $\Gamma_1$. 
Further assume that $T_q\F_q(\Gamma_2)$ is given the  capping
orientation of $\Gamma_2$.

By rigidity and transversality reasons we have that
\begin{equation*}
 T_q\I_q(\Gamma_1)\cap T_q\F_q(\Gamma_2)=0, \qquad
 T_q\I_q(\Gamma_1)\oplus T_q\F_q(\Gamma_2)=T_qM.
\end{equation*}
Let  $\sigma_1$ satisfy
\begin{equation}\label{eq:fel}
 \Or( T_qM) = (-1)^{\sigma_1}  \Or(T_q\I_q(\Gamma_1)) \wedge  \Or(T_q\F_q(\Gamma_2)).
\end{equation}
By arguments similar to the case \ref{f11} we get that the capping sign of $\Gamma$ is given by $(-1)^{\sigma_2}$
\begin{align*}
  \sigma_2 &= \sigma_1+ \dim \I_q(\Gamma_1) + \beta(\s, A_d)+j+1 + \mu(\Gamma)
 +\sigma.
\end{align*}
Here $j$ is the order of the boundary minimum of $\Delta(\Gamma)$ with smallest $\tau$-value. 
 This sign comes from the fact that we have inductively given
 $V_{\con}(\Gamma)$ the orientation \linebreak $(v_1,\dotsc,\hat{v}_j,\dotsc,v_{m-1})$, which by Lemma \ref{lma:confor} differs from the standard orientation of $V_{\con}(\Gamma)$ by a sign 
$(-1)^{j+1}$. The summand $\dim \I_q(\Gamma_1)$ comes from the fact that we should use \eqref {eq:glu10l} to give the orientation to the one point space $\Ker \Gamma_s$, instead of \eqref{eq:fel}. Also, the sign $\sigma$ is the gluing sign given in Proposition \ref{prp:sista}, with $\kd_{\Gamma_{1,s}}$ replaced by $\kd_{\Gamma_{1,ps}}$, and equals
\begin{align*}
 \sigma &\equiv [|\mu(a)| +k+\dim W^u(a) + k + n+1] \cdot \dim
 \kd_{\Gamma_1}\\
 &\quad{}+ [\bm(\Gamma_1) + 1 + \dim \vk(\Gamma_1)_{ps} +n]\cdot[\dim W^u(a) + n + |\mu(a)|],
\end{align*}
modulo 2.
Here we have used that 
\begin{equation*}
 \dim\cd_{\Gamma_1} \equiv \bm(\Gamma_1) + 1 + \dim \vk(\Gamma_1)_{ps}, \pmod{2}.
\end{equation*}
We can simplify the expression for $\sigma$ to get 
\begin{align*}
 \sigma &\equiv[ \kd_{\Gamma_1} + \dim \vk(\Gamma_1)_{ps}] \cdot  [|\mu(a)| +\dim W^u(a) + n] + \dim
 \kd_{\Gamma_1}\\
 &\quad{}+ [\bm(\Gamma_1) + 1  +n]\cdot[\dim W^u(a) + n + |\mu(a)|]\\
 &\equiv[ \dim W^u(a) + k + n] \cdot  [|\mu(a)|+\dim W^u(a) + n] + \dim
 \kd_{\Gamma_1}\\
 &\quad{}+ [\bm(\Gamma_1) + 1  +n]\cdot[\dim W^u(a) + n + |\mu(a)|]\\
  &\equiv k \cdot  [|\mu(a)|+\dim W^u(a) + n] + \dim W^u(a)\cdot [1 + |\mu(a)| + \bm(\Gamma_1) + 1  +n]\\
 &\quad{}+ [n+ |\mu(a)|] \cdot[\bm(\Gamma_1) + 1]  + \dim
 \kd_{\Gamma_1}, \pmod{2}.
\end{align*}

\begin{figure}[ht]
\labellist
\small\hair 2pt
\pinlabel ${a}$ [Br] at  35 385 
\pinlabel ${p}$ [Br] at 310 380
\pinlabel $s$ [Br] at 170  360
\pinlabel $\Gamma_{2}$ [Br] at 100 280
\pinlabel $\Sigma$ [Br] at 275 310
\pinlabel $\Gamma_1$ [Br] at 450 280
\pinlabel $\Gamma$ [Br] at 1140 280
\pinlabel $q$ [Br] at 445 380
\pinlabel $a$ [Br] at 850 100
\pinlabel $p_3$ [Br] at 580 488
\pinlabel $p_1$ [Br] at 650 480
\pinlabel $p_2$ [Br] at 670 370
\pinlabel $p_5$ [Br] at 620 295
\pinlabel $p_4$ [Br] at 580 245
\pinlabel $p_1$ [Br] at 1333 10
\pinlabel $p_2$ [Br] at 1333 57
\pinlabel $p_3$ [Br] at 1333 100
\pinlabel $p_4$ [Br] at 1333 143
\pinlabel $p_5$ [Br] at 1333 189
\endlabellist
\centering
\includegraphics[height=5cm]{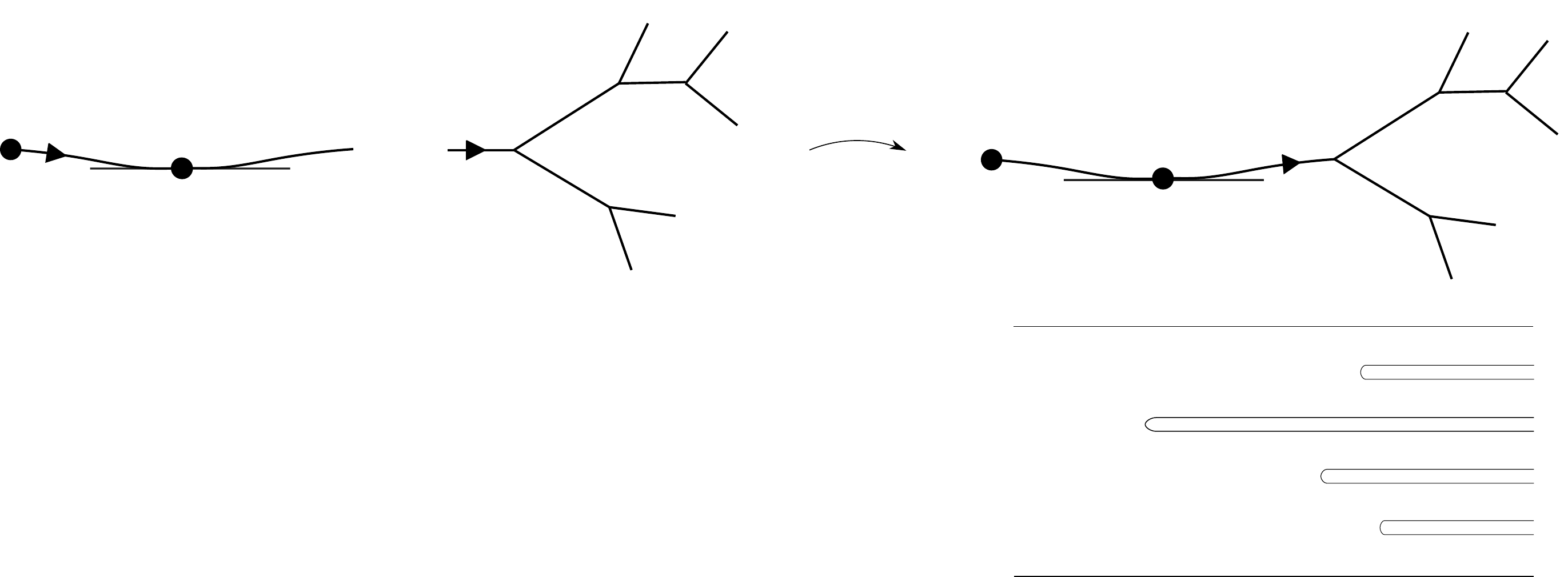}
\caption{The final gluing, giving the rigid flow tree. In this case $j= 3$.}
\label{fig:finish1val}
\end{figure}

If we now recall Remark \ref{rmk:sigmastab} we see that we get the sign from \cite{korta}.

\item[\ref{f2} -- \it Trees with trivial automorphism group, the positive puncture is 2-valent]
This is similar to the case of $1$-valent punctures, but now we have to perform two gluings to get the stabilized orientation of $\Gamma$.

Denote the $2$-valent positive puncture by $a$. Let $p_1$, $p_2$ be points on each of the edges of $\Gamma$ that contains $a$. By cutting $\Gamma$ at these points we get the following 3 partial flow trees.
\begin{itemize}
 \item An elementary tree $\Gamma_a$ that contains the positive 2-valent puncture $a$, and has $p_1$ and $p_2$ as negative special punctures.
 \item Sub flow trees $\Gamma_i$ with $p_i$ as a positive special puncture, $i=1,2$. 
\end{itemize}
See Figure \ref{fig:finish2val}. 
 Here we let $\Gamma_1$ denote the sub flow tree that corresponds to the lower part of the standard domain of $\Gamma$.
Let $\nu_i \in \{0,1\}$ so that 
\begin{equation*}
  \Or(T_{p_i}M) = (-1)^{\nu_i}\Or(T_{p_i} \F_{p_i}(\Gamma_i)), \qquad i=1,2,
 \end{equation*}
 where $T_{p_i}M$ is given the orientation from $M$, and  $T_{p_i} \F_{p_i}(\Gamma_i)$ is given the inductively defined stabilized capping orientation of $\Gamma_i$, $i=1,2$.

\begin{figure}[ht]
\labellist
\small\hair 2pt
\pinlabel ${a}$ [Br] at  285 470 
\pinlabel ${a}$ [Br] at 1005 470
\pinlabel $\Gamma_{2}$ [Br] at 115 480
\pinlabel $\Gamma_a$ [Br] at 305 390
\pinlabel $\Gamma_1$ [Br] at 485 480
\pinlabel $\Gamma$ [Br] at 1135 480
\pinlabel $p_1$ [Br] at 395 470
\pinlabel $p_2$ [Br] at 195 470
\pinlabel $q_6$ [Br] at 40 515
\pinlabel $q_7$ [Br] at 40 370
\pinlabel $q_5$ [Br] at 610 570
\pinlabel $q_4$ [Br] at 690 555
\pinlabel $q_3$ [Br] at 695 460
\pinlabel $q_2$ [Br] at 645 390
\pinlabel $q_1$ [Br] at 605 335
\pinlabel $a$ [Br] at 830 130
\pinlabel $q_1$ [Br] at 1310 7
\pinlabel $q_2$ [Br] at 1310 53
\pinlabel $q_3$ [Br] at 1310 98
\pinlabel $q_4$ [Br] at 1310 142
\pinlabel $q_5$ [Br] at 1310 185
\pinlabel $q_6$ [Br] at 1310 228
\pinlabel $q_7$ [Br] at 1310 273
\endlabellist
\centering
\includegraphics[height=5.5cm]{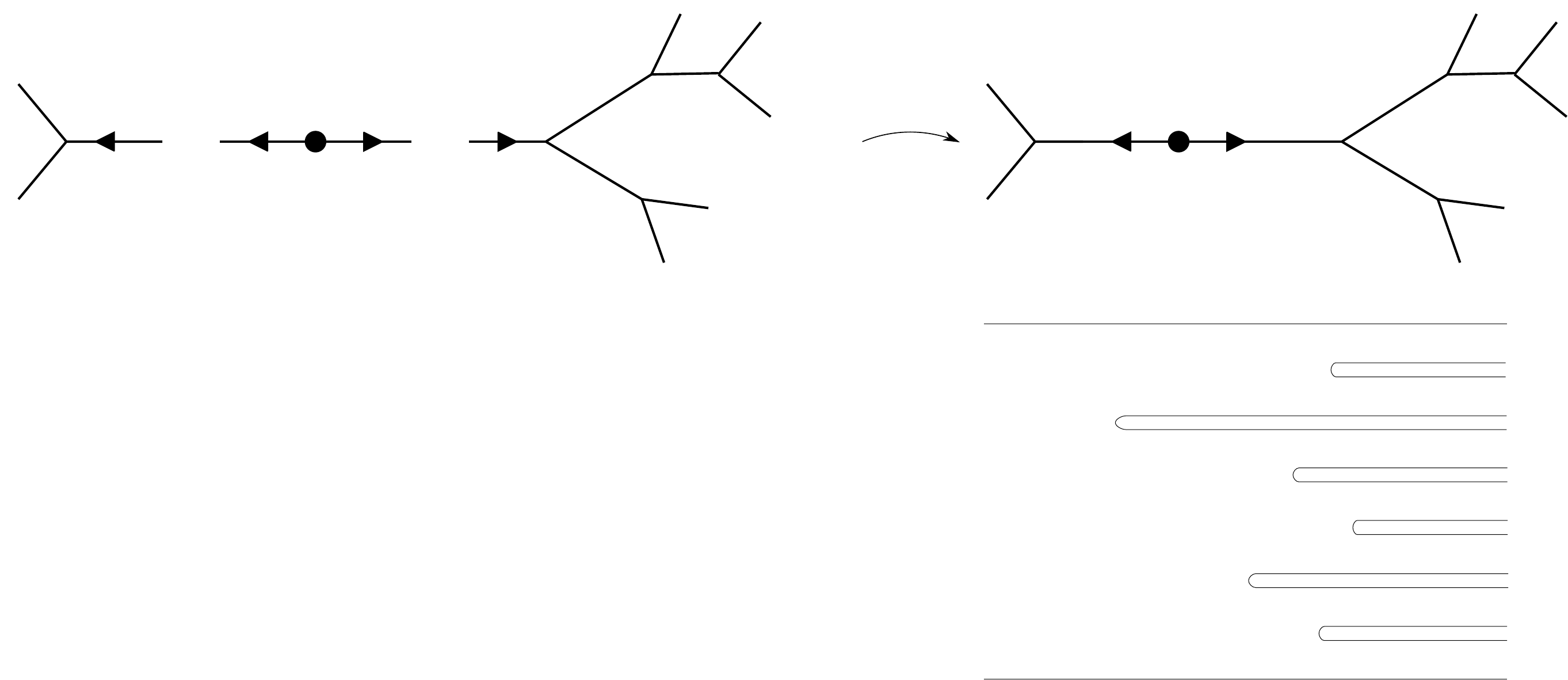}
\caption{The final gluing when the positive puncture is $2$-valent. In this example $j= 5$.}
\label{fig:finish2val}
\end{figure}

\noindent

\begin{rmk}
Notice that in this case we cannot have a switch adjacent to the positive puncture, because of transversality reasons. 
\end{rmk}

Now we glue $\Gamma_2$ to $\Gamma_a$, to get an intermediate tree $\Gamma_{2a}$, and then we glue $\Gamma_1$ to this tree. 

We get two gluing signs from these gluings, calculated using similar techniques as in Proposition \ref{prp:megalemma4paux}. That is, from the first gluing we get a sign 
\begin{align*}
 \sigma_1& \equiv|\mu(p_1)|\cdot  \dim \Ker \capp (\Gamma_2)  + |\mu(a)|\cdot \dim \kd_{\Gamma_2}   \\
 &\quad {}+ [\dim \cd_{\Gamma_2} +n]\cdot(n+|\mu(a)|)+ \dim \vk(\Gamma_2) \cdot \dim \kd_{\Gamma_2}, \pmod{2},
\end{align*}
and from the second gluing we get a sign 
\begin{align*}
 \sigma_2& \equiv \dim \kd_{\Gamma_1}\cdot[\dim \Ker
   \capp (\Gamma_2) +|\mu(a)|] + [\dim \cd_{\Gamma_1}+n]\cdot[n+1 +|\mu(p_1)|] \\
 &\quad {} + e(\Gamma_1)\cdot [\dim \Ker \capp(\Gamma_2)+ |\mu(p_2)| + e(\Gamma_2)] + \dim \cd_{\Gamma_{2a}} \cdot n \\
 &\quad{}+ \dim \vk(\Gamma_1) \cdot[ \dim \cd_{\Gamma_{2a}} + \dim \kd_{\Gamma_1}], \pmod{2}.
\end{align*}
Adding together these signs in the same time as we remember  \eqref{eq:kerdim} and \eqref{eq:sistacoker} we get 
\begin{align*}
 \sigma_1 + \sigma_2& \equiv |\mu(p_1)|\cdot  [\bm(\Gamma_2) + 1 + |\mu(p_2)|] + |\mu(a)|\cdot \dim \kd_{\Gamma_2}  \\
&\quad {} + [\bm(\Gamma_2)+\dim \vk(\Gamma_2) +n]\cdot[n+|\mu(a)|] 
 + [n+1]\cdot \dim \kd_{\Gamma_2} \\
 &\quad{} + \dim \kd_{\Gamma_1}\cdot[\bm(\Gamma_2) + 1 + |\mu(p_2)| +|\mu(a)|] \\
 &\quad{}+ [\bm(\Gamma_1)+\dim \vk(\Gamma_1)+n]\cdot[n+1 +|\mu(p_1)|] \\
 &\quad {} + e(\Gamma_1)\cdot [\bm(\Gamma_2) + 1  + e(\Gamma_2)] + \bm(\Gamma_2) \cdot n \\
 &\quad{}+ \dim \vk(\Gamma_1) \cdot \bm(\Gamma_2) + \dim \kd_{\Gamma_1}\cdot[n +1]\\
 &\equiv |\mu(p_1)|\cdot  [\bm(\Gamma_1) + \bm(\Gamma_2) + 1 + |\mu(p_2)|+n] + \bm(\Gamma_1)\cdot [n+1]\\
 &\quad{} +  [\bm(\Gamma_2) + n]\cdot|\mu(a)| + n +  e(\Gamma_1)\cdot [\bm(\Gamma_2) + 1  + e(\Gamma_2)]\\
 &\quad{} + \dim \kd_{\Gamma_1}\cdot[\bm(\Gamma_2) + 1 + |\mu(p_1)|+1 +n] \\
 &\quad{} + \dim \vk(\Gamma_1)\cdot[\bm(\Gamma_2)+|\mu(p_1)| +1 +n]\\
 &\quad{} + \dim \kd_{\Gamma_2}\cdot[n+1+|\mu(a)|] + \dim \vk(\Gamma_2)\cdot [n+|\mu(a)|]\\
  &\equiv |\mu(p_1)|\cdot  [\bm(\Gamma_1) + \bm(\Gamma_2) + 1 + |\mu(p_2)|] + \bm(\Gamma_1)\cdot [n+1]\\
 &\quad{} +  \bm(\Gamma_2) \cdot[|\mu(a)| + n] +  e(\Gamma_1)\cdot [\bm(\Gamma_2) + 1  + e(\Gamma_2)]\\
 &\quad{} + \dim \kd_{\Gamma_1}
+ \dim \kd_{\Gamma_2}, \pmod{2}.
\end{align*}
The remaining steps are completely similar to the case \ref{f1}.

\end{description}

 \section{An example}\label{sec:examples}
In this section we calculate the signs of some of the flow trees occurring in [\cite{trees}, Section 7]. First we recall the set up given there, and we modify it slightly to get rigid flow trees.

 We consider two fronts $F_1$, $F_2$ in $J^1(\R^2)$, where $F_1$ is given by the $0$-section, and $F_2$ is given by the $1$-jet lift of $f(x_1,x_2) = K - x_1$, $K \gg 0$. We consider the restriction of these fronts to the square $-R \leq |x_j| \leq R$, $j=1,2$. This gives rise to a 1-parameter family of Morse flow lines in the square.

We perturb the front $F_2$ slightly, so that we get a Reeb chord $a$ at $(-R,0)$ of index $1$, and a Reeb chord $b$ at $(R,0)$ of index $0$, and so that we have exactly one gradient flow line connecting $a$ to $b$, given by the $x_1$-axis restricted to the square. Thus we get a rigid flow tree $\Gamma$ with positive puncture $a$ and negative puncture $b$.

Now we perform the Legendrian isotopy of the $0$-section described in
[\cite{trees}, Section 7], and sketched in Figure \ref{fig:isot}. By
moving this construction in the $x_2$-direction, we get the flow tree
in Figure \ref{fig:tree5a}, and it
will be rigid. We denote this tree $\Gamma_A$,
, and we will show how to compute its capping orientation.

\begin{figure}[ht]
\centering
\includegraphics[height=3cm]{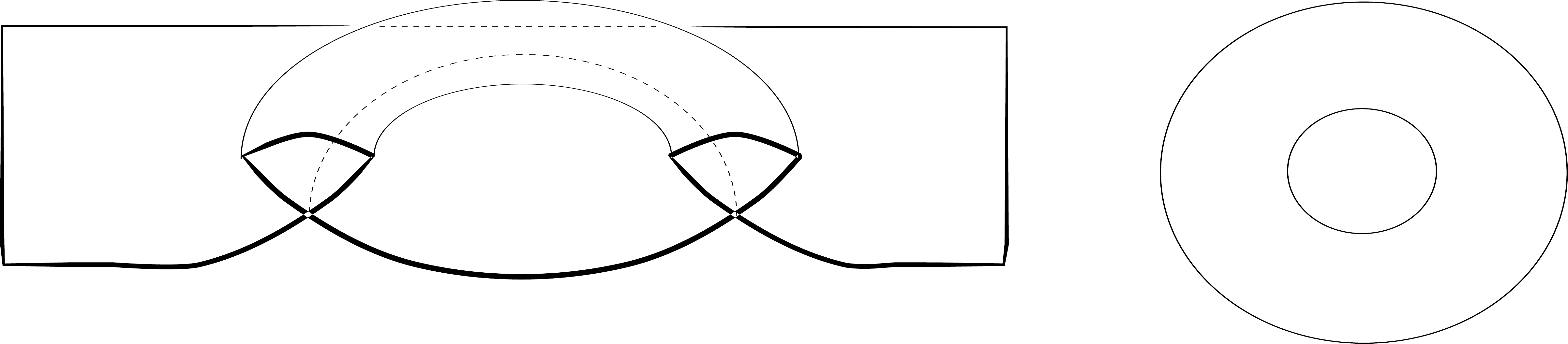}
\caption{The $0$-section after the Legendrian isotopy, to the left a cross-section and to the right the top-view.}
\label{fig:isot}
\end{figure}

\begin{figure}[ht]
\centering
\includegraphics[height=2.2cm]{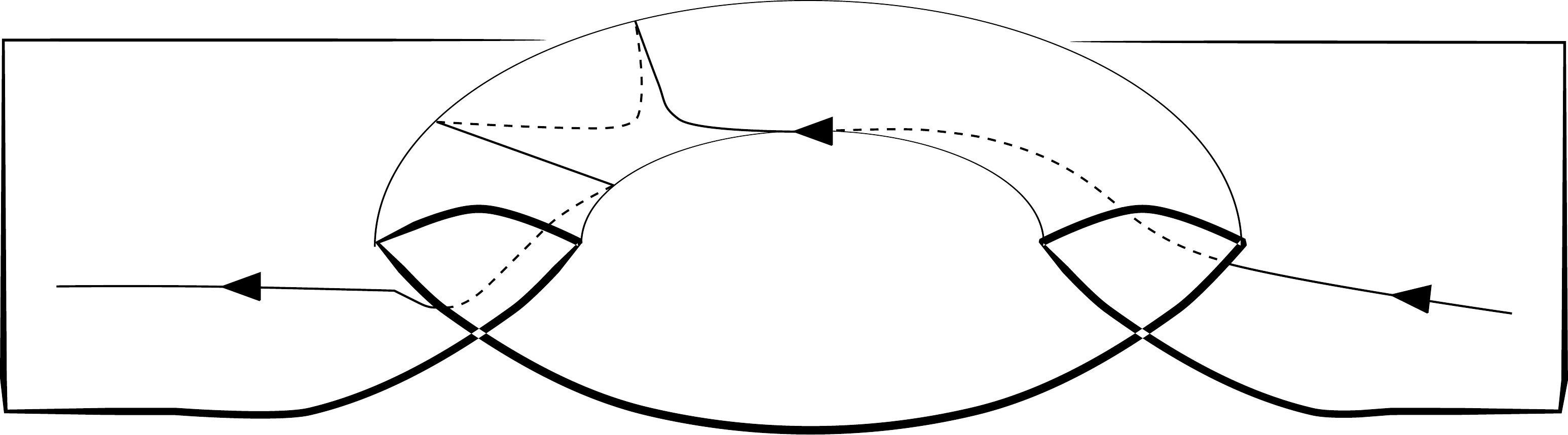}\\
\vspace{.5cm}
\labellist
\small\hair 2pt
\pinlabel $a$ [Br] at 27 70
\pinlabel $v_1$ [Br] at  135 40 
\pinlabel ${e_2}$ [Br] at 317 10
\pinlabel $s$ [Br] at  360 105
\pinlabel $b$ [Br] at  842 62 
\pinlabel $w_1$ [Br] at 349 38
\pinlabel $w_2$ [Br] at 410 29
\pinlabel $u_2$ [Br] at 57 105
\pinlabel $u_1$ [Br] at 130 155
\pinlabel ${e_1}$ [Br] at 280 255
\pinlabel ${v_0}$ [Br] at 280 115
\endlabellist
\centering
\includegraphics[height=2.7cm, width=8.5cm]{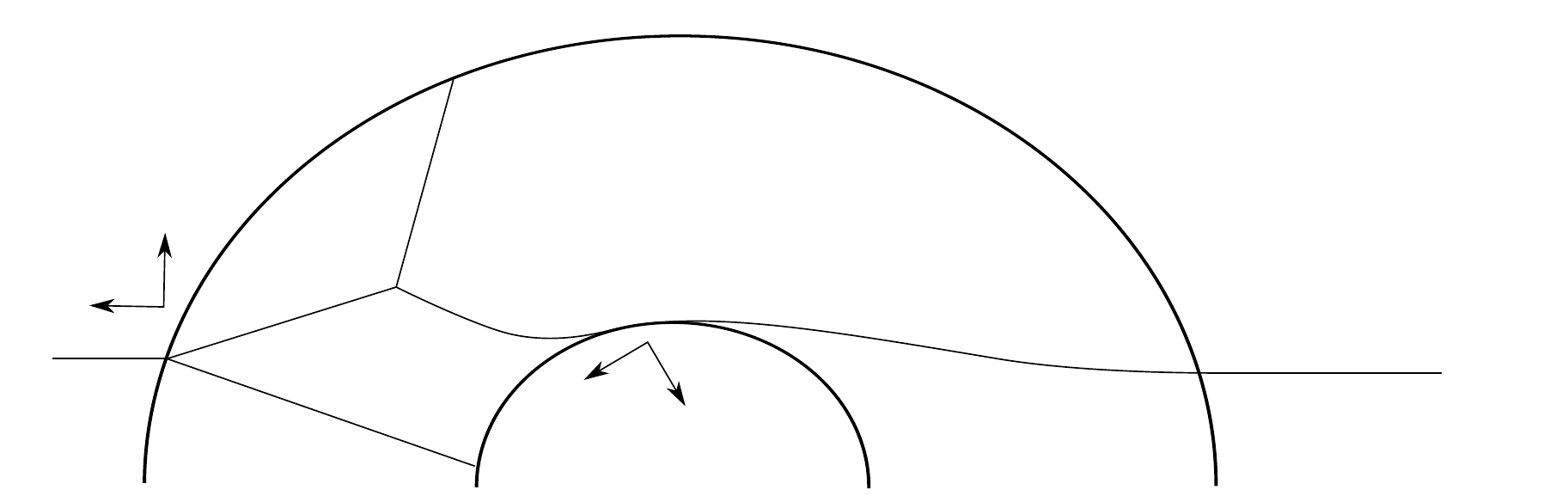}
\hspace{1.5cm}
\labellist
\small\hair 2pt
\pinlabel $a$ [Br] at 10  60
\pinlabel $v_1$ [Br] at  145 95 
\pinlabel $v_0$ [Br] at  240 40 
\pinlabel $s$ [Br] at 395 8
\pinlabel $b$ [Br] at 535 2
\pinlabel $e_1$ [Br] at 555 68
\pinlabel $e_2$ [Br] at 555 123
\endlabellist
\centering
\includegraphics[height=1.3cm]{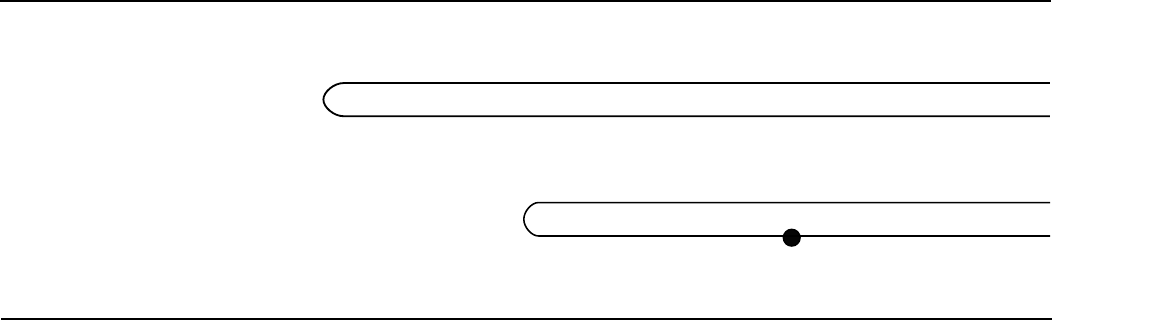}
%
\caption{In the top picture we find the lift of the flow lines of $\Gamma_A$ to the isotoped $0$-section. To the left we find the tree $\Gamma_A$ drawn in $\R^2$, together with the projection of the cusps.  To the right we see the corresponding standard domain.}
\label{fig:tree5a}
\end{figure}



\subsection{Initial choices}
We assume that  $|\mu(a)| = |\mu(b)|=0$. Then we can lift the trivialization of $T\R^2$ along $\Gamma$ to both the Legendrian sheets, to get a trivialization that we might assume to coincide with the one induced by the spin structure.

 To choose capping orientations of the stable and unstable manifolds associated to $a$ and $b$, we 
note that 
\begin{align*}
 &T_a W^u(a) = \spn (\partial_{x_1}) &  T_a W^s(a) &= \spn (\partial_{x_2}) \\
 &T_b W^u(b) = 0 &  T_b W^s(b)& = \spn (\partial_{x_1},\partial_{x_2}) .
\end{align*}
We choose the following orientations of the unstable manifolds
\begin{align*}
 &\Or(T_a W^u(a)) = +\partial_{x_1}   &\Or(T_b W^u(b)) = +0, &
\end{align*}
and hence we get the induced orientations 
\begin{align*}
 &\Or(T_a W^s(a)) = +\partial_{x_2}
& \Or(T_b W^s(b)) = \partial_{x_1} \wedge \partial_{x_2}. &
\end{align*}
\subsection{Orientation of  \texorpdfstring{$\Gamma$}{g}}\label{sec:orientshorttree}
We divide $\Gamma$ into two sub flow trees, at $p=(0,0)$, say. Let $\Gamma_a$ denote the tree with puncture $a$, and $\Gamma_b$ the tree with puncture $b$. 

From Section \ref{sec:orp} it follows that the capping orientation of $\Gamma_a$ is given by
\begin{equation*}
 \Or(T_p\F_p(\Gamma_a)) = \partial_{x_1}
\end{equation*}
and that the capping orientation of $\Gamma_b$ is given by
\begin{equation*}
  \Or(T_p\F_p(\Gamma_b)) = (-1)^{\sigma^{(0,0)}_{\negg,1}(0)} \partial_{x_1} \wedge \partial_{x_2}.
\end{equation*}

Now we glue $\Gamma_a$ and $\Gamma_b$ back together at $p$, to recover the capping orientation of $\Gamma$.
We see that 
\begin{equation*}
 T_p\F_p(\Gamma_b) \cap T_p\F_p(\Gamma_a) = \spn(\partial_{x_1}),
\end{equation*}
and that the sequence \eqref{eq:glu10l} induces the orientation
\begin{equation*}
  (-1)^{\sigma^{(0,0)}_{\negg,1}(0)}\partial_{x_1}
\end{equation*}
on this space, given the capping orientation of $\Gamma_a$ and
$\Gamma_b$. Thus, from the formula in Section \ref{sec:finish}, Case
\ref{f11}, we get that the capping orientation of $\Gamma$ is
given by the sign $(-1)^\sigma$,
\begin{equation*}
  \sigma =  \sigma^{(0,0)}_{\negg,1}(0)+1+0+1\cdot [n+0]+0+n \equiv
   \sigma^{(0,0)}_{\negg,1}(0) +1\pmod{2}.
\end{equation*}
Notice that $\beta(\s,A_d)= \mu(\Gamma) = 0$ in this case.

\subsection{Orientation of  \texorpdfstring{$\Gamma_A$}{G}}\label{sec:gammaasign}
In this case we have a tree with vertices given by a positive puncture $a$, a negative puncture $b$, a switch $s$, a $Y_0$-vertex $v_0$, a $Y_1$-vertex $v_1$, and two ends $e_1$, $e_2$. See Figure \ref{fig:tree5a}, where we also find the corresponding standard domain.

Again we divide the tree into the corresponding elementary pieces $\Gamma_a$, $\Gamma_b$, $\Gamma_s$, $\Gamma_{v_0}$, $\Gamma_{v_1}$, $\Gamma_{e_1}$ and $\Gamma_{e_2}$, where
the index indicates the vertex of the piece. We will orient these
pieces, and then glue them together, using the inductively described
scheme at the beginning of Section \ref{sec:stabstab}.

From Section \ref{sec:elementor} we see that the pieces should be given the following orientation.
\begin{itemize}
 \item The pieces $\Gamma_a$ and $\Gamma_b$ are oriented in the same way as when we calculated the orientation of $\Gamma$.
 \item The pieces $\Gamma_{e_1}$  and $\Gamma_{e_2}$ are given the orientation $\partial_{x_1}
   \wedge \partial_{x_2}$.
   \item The piece $\Gamma_{v_0}$ is given the orientation
     $(-1)^{\sigma^{(0,0)}_{Y_0}(1,1)}\partial_{x_1} \wedge \partial_{x_2}$.
 \item To orient $\Gamma_s$, consider Figure \ref{fig:tree5a}, where
   the arrow $w_2$  gives the outward normal of
   $\Pi(\Lambda_s)$, and $w_1=-(\partial_{x_1} +\partial_{x_2})$
   indicates the orientation of $T_s\Pi(\Sigma)$, oriented as the
   boundary of $\Pi(\Lambda_s)$. Notice that the negative puncture of
   $\Gamma_s$ has even Maslov
   index, and that the lifted tree passes the cusp in the lower sheet. Thus the capping
   orientation is given by $(-1)^{\sigma^{(0,1)}_{\swi}(l)}w_1=(-1)^{\sigma^{(0,1)}_{\swi}(l)+1}(\partial_{x_1} +\partial_{x_2})$.
 \item To orient $\Gamma_{v_1}$, consider Figure
\ref{fig:tree5a}. It follows that the capping orientation is given by
$(-1)^{\sigma^{(0,0)}_{Y_1}(0,1)} \partial_{x_2}$.
\end{itemize}

Now we are ready to glue. The scheme from Section \ref{sec:stabstab}
tells us that we should start by gluing $\Gamma_b$ to $\Gamma_s$, to get a sub flow tree $\Gamma_{bs}$. First we calculate the orientation of 
\begin{equation*}
  T_s\F_s(\Gamma_b) \cap T_s\Pi(\Sigma) = \spn(w_1),
\end{equation*}
induced by the sequence \eqref{eq:glu10l}, where $\Gamma_1 = \Gamma_b$, $\Gamma_2 = \Gamma_s$ and $T_s\Pi(\Sigma)$ is given the capping orientation of $\Gamma_s$. We get
that this orientation is given by
\begin{equation*}
  (-1)^{\sigma^{(0,0)}_{\negg,1}(0) + \sigma^{(0,1)}_{\swi}(l)}w_1,
\end{equation*}
and by Section \ref{sec:switchglu} it follows
that this coincides with the capping orientation of $\Gamma_{bs}$. This since $\kd_{\Gamma_1} = \spn(\partial_{x_1}, \partial_{x_2}) =  \kd_{\Gamma_{1,s}}$, $\kd_{\Gamma} = \spn(w_1)$.

Next we glue $\Gamma_{bs}$ to $\Gamma_{e_1}$ via $\Gamma_{v_0}$, to obtain the sub flow tree $\Gamma_{bsv_0e_1}$. In this case we have $\Gamma_1 = \Gamma_{bs}$, $\Gamma_2 = \Gamma_{e_1}$. If we let $w$ be the vector tangent to the edge between $v_0$ and $s$ at $v_0$, chosen so that it points against the flow direction, we see that $T_{v_0}\F_{v_0}(\Gamma_{bs}) = w$. Hence we get that the orientation of 
\begin{equation*}
  T_v\F_v(\Gamma_{bs}) \cap T_v\F_v(\Gamma_e) = \spn(w)
\end{equation*}
induced by the sequence \eqref{eq:glu10l} is given by
\begin{equation*}
  (-1)^{\sigma^{(0,0)}_{\negg,1}(0) + \sigma^{(0,1)}_{\swi}(l)+\sigma^{(0,0)}_{Y_0}(1,1)}w,
\end{equation*}
and the gluing sign from Section
\ref{sec:stabstabstab} is given by 
\begin{equation*}
  (-1)^{0+ 0 + [n+1+1]\cdot[0+1+1] + 1+2+1+n\cdot n} = 1.
\end{equation*}

Next we have to stabilize the tree
$\Gamma_{bsv_0e_1}$ after the gluing, by adding the vector field $t$
as described in Lemma \ref{lma:prestab}. Recall that this vector is nothing but the vector tangent to the edge between $v_1$ and $v_0$ oriented against the flow direction. From Figure \ref{fig:tree5a} we see that $w \wedge t = \partial_{x_1} \wedge \partial_{x_2}$, and thus we get that the flow-out of $\Gamma_{bsv_0e_1}$ is given the orientation 
\begin{equation*}
 (-1)^{\sigma^{(0,0)}_{\negg,1}(0) + \sigma^{(0,1)}_{\swi}(l) + \sigma^{(0,0)}_{Y_0}(1,1)+0} \partial_{x_1} \wedge \partial_{x_2}. 
\end{equation*}

Now we glue this tree to $\Gamma_{e_2}$ via $\Gamma_{v_1}$, to obtain
a sub flow tree $\Gamma_{B'}$. To do this, we first glue
$\Gamma_{e_2}$ to $\Gamma_{v_1}$. We get that the orientation of 
\begin{equation}\label{eq:avi1}
   T_{v_1}\F_{v_1}(\Gamma_{e_2}) \cap T_{v_1}\Pi(\Sigma) = \spn(\partial_{x_2})
 \end{equation}
 induced by the sequence \eqref{eq:glu10l} 
 is given by $(-1)^{\sigma^{(0,0)}_{Y_1}(0,1)} \partial_{x_2}$, where $T_{v_1} \Pi(\Sigma)$ is given the orientation from capping orientation of $\Gamma_{v_1}$. 
By intersecting $T_{v_1}\F_{v_1}(\Gamma_{bsv_0e_1})$ with this
oriented space and then calculating the sign from Section \ref{sec:y1glu} we get that the stabilized capping orientation of $\Gamma_{B'}$ is
given by
\begin{equation*}
  (-1)^{\sigma^{(0,0)}_{\negg,1}(0) + \sigma^{(0,1)}_{\swi}(l) +
      \sigma^{(0,0)}_{Y_0}(1,1)+\sigma^{(0,0)}_{Y_1}(0,1)}
    \partial_{x_2}.
  \end{equation*}
Here we have that $|\mu(p_1)| =0$, $|\mu(p_2)| = 1$, $\dim \kd_{\Gamma_1} = 1$, $\dim \kd_{\Gamma_2} = \dim \kd_{\Gamma_{2,s}} = \dim \kd_{\Gamma_{1,s}} =2$, $\dim \kd_\Gamma = 0$. 

Now it is time for the final gluing. We get that
 \begin{equation*}
 \Or(T_{v_1}\I_{v_1}(\Gamma_{B'})) \wedge \Or(T_{v_1}
    \F_{v_1}(\Gamma_a)) =
  (-1)^{ \sigma^{(0,0)}_{\negg,1}(0)
    + \sigma^{(0,1)}_{\swi}(l) + \sigma^{(0,0)}_{Y_0}(1,1)+\sigma^{(0,0)}_{Y_1}(0,1)+1}\partial_{x_1}\wedge \partial_{x_2}.
\end{equation*}
Moreover,  $\beta(s,A_d) = 0$, $\mu(\Gamma) = 0$, and $j=2$. So
by the formula in Section \ref{sec:finish}, Case \ref{f1}, we get that
the capping orientation of $\Gamma_B$ is given by $(-1)^\sigma$, where
\begin{align*}
  \sigma =&  \sigma^{(0,0)}_{\negg,1}(0)
    + \sigma^{(0,1)}_{\swi}(l) + \sigma^{(0,0)}_{Y_0}(1,1)+\sigma^{(0,0)}_{Y_1}(0,1)+1+0+1+0+2+1+0\\
  &\quad{} +0+ 1\cdot[1+0+2+1+n]+ [n+0]\cdot[2+1] +0\\
  =&  \sigma^{(0,0)}_{\negg,1}(0)
    + \sigma^{(0,1)}_{\swi}(l) + \sigma^{(0,0)}_{Y_0}(1,1)+\sigma^{(0,0)}_{Y_1}(0,1)+1.
\end{align*}

\begin{rmk}
  Since rigid flow trees can be used to define the differential in
  Legendrian contact homology, and since this homology theory gives
  an invariant for Legendrian submanifolds, it should be expected that the
  signs of $\Gamma$ and $\Gamma_A$ coincide. Thus we
  must have that
  \begin{align*}
    &\sigma^{(0,1)}_{\swi}(l) + \sigma^{(0,0)}_{Y_0}(1,1)+\sigma^{(0,0)}_{Y_1}(0,1) \equiv 0, \pmod{2}
  \end{align*}
  for $n=2$. In Subsection \ref{sec:longexgen} we consider a generalization of this example to derive more explicit relations, also in higher dimensions. 
\end{rmk}

\begin{rmk}
 It is often tricky to keep track of the dimensions of the true kernels $\kd_{\Gamma_1}, \kd_{\Gamma_2}, \kd_{\Gamma}$ since they are not geometric. Luckily, all these tend to cancel each other when summing up the total sign of a rigid tree.
\end{rmk}


 %

\section{Geometric derivation of sign functions}\label{sec:der}
In this section we will use Legendrian isotopies, similar to the one in Section \ref{sec:examples}, to find some of the values of the sign functions $\sigma_{Y_0}, \sigma_{Y_1}, \sigma_{\swi}$, and also some relations for the sign function $\sigma_{\negg,2}$. In Section \ref{sec:analder} we will use these results together with analytical computations to derive the remaining values of the sign functions to be able to explicitly compute the signs of rigid flow trees.

\subsection{Relations for \texorpdfstring{$\sigma_{Y_1}$}{Y}}\label{sec:y1rel}
To derive the expressions for $\sigma_{Y_1}$, consider first the case $n=1$, and let $\Gamma$ be a rigid flow tree with one positive puncture $a$, one negative puncture $b$, and no other vertices. The lift of $\Gamma$ then consists of two sheets, let $\sigma(u), \sigma(l) \in \Z_2$ be the sign of the upper respectively lower sheet, as defined in \eqref{eq:sheetsign}, but where we now use the additive structure. That is, we replace $-1$ with $1$, $1$ with $0$. Using similar notation as in Section \ref{sec:orientshorttree}, assume that 
\begin{equation*}
 \Or(T_p\F_p(\Gamma_a)) = \partial_{x_1}, \qquad
  \Or(T_p\F_p(\Gamma_b)) = (-1)^{\sigma^{(\sigma(u),\sigma(l))}_{\negg,1}(0)} \partial_{x_1},
\end{equation*}
where $\partial_{x_1}$ represents the positive orientation of $\R$.

Similar to how it is done in Section \ref{sec:examples} one computes that the capping orientation of $\Gamma$ is given by 
\begin{equation*}
 (-1)^{\sigma^{(\sigma(u),\sigma(l))}_{\negg,1}(0) + 1}.
\end{equation*}
Here we assume that the lifted trivialization, given in Section \ref{sec:treetriv}, coincides with the one induced by the spin structure (so that $\beta(\s,A_d) = 0$) and that the flow direction is given by $\partial_{x_1}$.

Now consider the two isotoped sheets in Figures \ref{fig:Y1eA} and \ref{fig:Y1eB} , which transform $\Gamma$ into rigid flow trees $\Gamma_1$ and $\Gamma_2$, respectively, both having one $Y_1$-vertex $v$, one end $e$ and true punctures $a,b$. Moreover, both these trees should have the same sign as $\Gamma$. To compute these signs, we use similar techniques as in Section \ref{sec:examples}. That is, we divide these trees into elementary trees $\Gamma_a^i, \Gamma_b^i, \Gamma_v^i, \Gamma_e^i$, $i=1,2$. The trees $  \Gamma_a^i, \Gamma_b^i$ are oriented as $\Gamma_a$ and $\Gamma_b$ above, $i = 1,2$, and the other trees have capping orientation
\begin{align*}
 \Or(\Gamma_v^1) &= (-1)^{\sigma_{Y_1}^{(\sigma(u),\sigma(l))}(1,\sigma(l) + \sigma(u)) +1} \\
  \Or(\Gamma_v^2) &= (-1)^{\sigma_{Y_1}^{(\sigma(u),\sigma(l))}(\sigma(l) + \sigma(u),1) +1} \\
   \Or(\Gamma_e^i) &= -\partial_{x_1}, \qquad i =1,2.
\end{align*}
Notice that the outward normal of the projection of the sheets meeting at the cusp at $v$ is given by $-\partial_{x_1}$, giving rise to the extra minus in the capping orientation of the elementary trees containing the $Y_1$-vertex. 

\begin{figure}[ht]
\centering
\labellist
\small\hair 2pt
\pinlabel $\partial_{x_1}$ [Br] at  1150 127 
\endlabellist
\centering
\includegraphics[height=2cm]{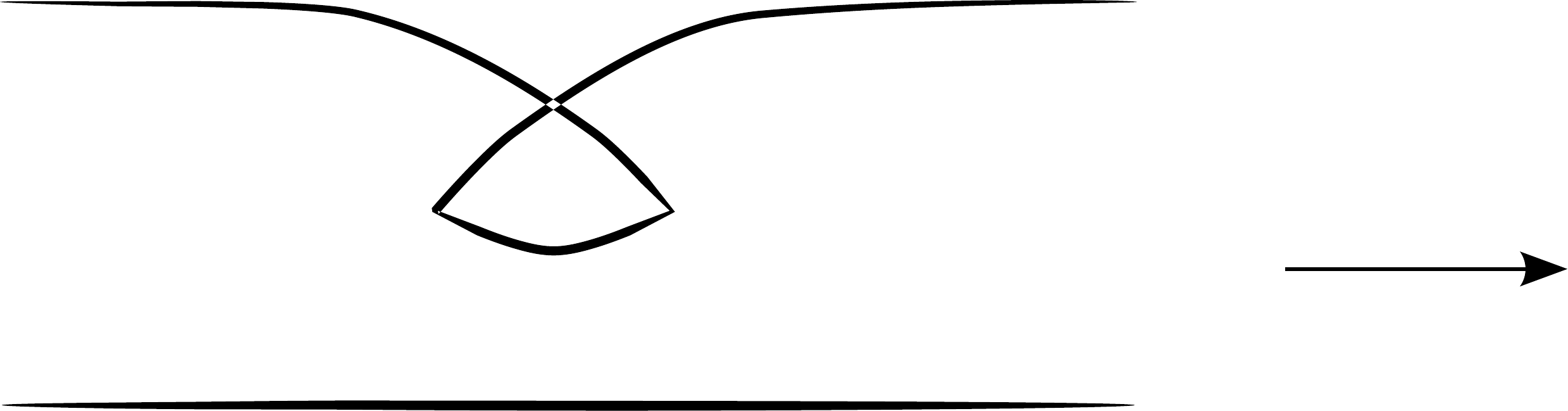}\\
\vspace{1.5cm}
\labellist
\small\hair 2pt
\pinlabel $a$ [Br] at -27 30
\pinlabel ${e}$ [Br] at 200 -5
\pinlabel $b$ [Br] at  200 72 
\endlabellist
\centering
\includegraphics[height=1.5cm, width=3.5cm]{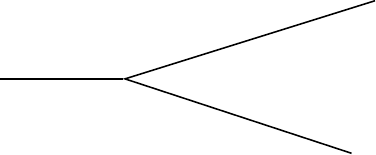}
\hspace{1.5cm}
\labellist
\small\hair 2pt
\pinlabel $a$ [Br] at 10  40
\pinlabel $b$ [Br] at 500 62
\pinlabel $e$ [Br] at 500 8
\pinlabel $v$ [Br] at 230 40
\endlabellist
\centering
\includegraphics[height=1.5cm]{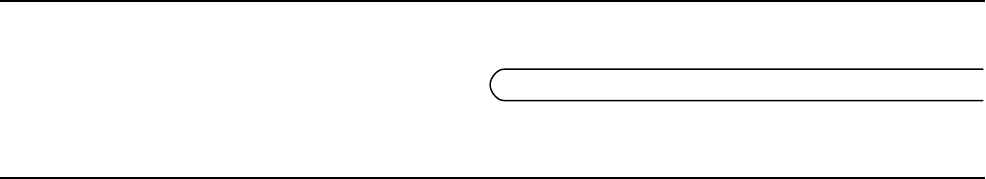}

\caption{The tree $\Gamma_1$.}
\label{fig:Y1eA}
\end{figure}

\begin{figure}[ht]
\centering
\labellist
\small\hair 2pt
\pinlabel $\partial_{x_1}$ [Br] at  1150 122 
\endlabellist
\centering
\includegraphics[height=2cm]{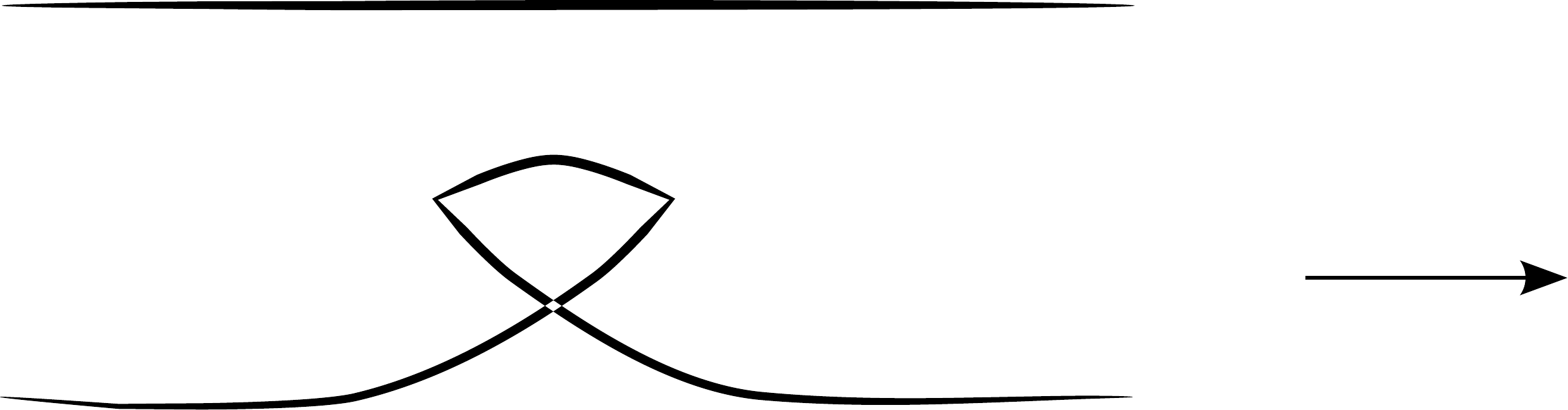}\\
\vspace{1.5cm}
\labellist
\small\hair 2pt
\pinlabel $a$ [Br] at -27 30
\pinlabel ${b}$ [Br] at 200 -5
\pinlabel $e$ [Br] at  200 72 
\endlabellist
\centering
\includegraphics[height=1.5cm, width=3.5cm]{endY1A}
\hspace{1.5cm}
\labellist
\small\hair 2pt
\pinlabel $a$ [Br] at 10  40
\pinlabel $e$ [Br] at 500 62
\pinlabel $b$ [Br] at 500 8
\pinlabel $v$ [Br] at 230 40
\endlabellist
\centering
\includegraphics[height=1.5cm]{domainY1eA}

\caption{The tree $\Gamma_2$.}
\label{fig:Y1eB}
\end{figure}

Now one computes that the capping orientation of $\Gamma_1$ is given by 
$(-1)^{\nu_1}$
where 
\begin{equation*}
 \nu_1 = \sigma^{(\sigma(u),\sigma(l))}_{\negg,1}(0)+\sigma_{Y_1}^{(\sigma(u),\sigma(l))}(1,\sigma(l) + \sigma(u)) 
\end{equation*} 
and that the  capping orientation of $\Gamma_2$ is given by $(-1)^{\nu_2}$ where
\begin{equation*}
 \nu_2 = \sigma^{(\sigma(u),\sigma(l))}_{\negg,1}(0)+\sigma_{Y_1}^{(\sigma(u),\sigma(l))}(\sigma(l) + \sigma(u),1) + \sigma(u) + \sigma(l) +1.  
\end{equation*} 
In both cases, we have $\beta(\s,A_d) = 1$. Indeed, for $\Gamma_1$ we get that $A_d$ performs  a $-2 \pi$--rotation in the $(\partial_{x_1}, \partial_{x_3})$--plane in the case when $\sigma(u)=0$ while we might assume that $A_\s$ is homotopic to a trivialization which only performs the geometric induced rotations, and that $A_s$ and $A_d^{-1}$ are forced to be concatenated to a $\pm 2\pi$ rotation in the $(\partial_{x_1}, \partial_{x_3})$--plane along one of the connected components of $\partial \Delta(\Gamma_1)$ in the case when $\sigma(u)=1$. For the $\Gamma_2$--tree we get a similar situation, but in this case we get a $-2\pi$-rotation in $A_d$ when $\sigma(l)=0$ and that $A_s$ and $A_d^{-1}$ are forced to be concatenated to a $\pm 2\pi$ rotation when $\sigma(l)=1$.

\begin{rmk}
 Using the notation from \cite{korta} we get that 
 \begin{align*}
  \nu_\en(\Gamma_1) &= (-1)^1, \\
  \nu_{\inter}(\Gamma_1) &= (-1)^0, \\
  \nu_{\stab}(\Gamma_1) &= (-1)^{\sigma^{(\sigma(u),\sigma(l))}_{\negg,1}(0)+\sigma_{Y_1}^{(\sigma(u),\sigma(l))}(1,\sigma(l) + \sigma(u)) }\\
  \nu_{\triv}(\Gamma_1) &= (-1)^1 \\
  \nu_\en(\Gamma_2) &= (-1)^0, \\ 
  \nu_{\inter}(\Gamma_2) &= (-1)^{0}, \\ 
  \nu_{\stab}(\Gamma_2) &= (-1)^{\sigma^{(\sigma(u),\sigma(l))}_{\negg,1}(0)+\sigma_{Y_1}^{(\sigma(u),\sigma(l))}(\sigma(l) + \sigma(u),1) + \sigma(u) + \sigma(l)} \\
  \nu_{\triv}(\Gamma_2) &= (-1)^1.
  \end{align*}
\end{rmk}

Since the sign of $\Gamma$ should coincide with the sign of both $\Gamma_1$ and $\Gamma_2$ we get that 
\begin{equation*}
 \nu_1 \equiv \nu_2 \equiv \sigma^{(\sigma(u),\sigma(l))}_{\negg,1}(0)+1, \pmod{2}
\end{equation*}
and hence that 
\begin{equation*}
 \sigma_{Y_1}^{(\sigma(u),\sigma(l))}(1,\sigma(l)+\sigma(u)) \equiv 1, \pmod{2}
\end{equation*}
for all choices of $\sigma(l),\sigma(u) \in \Z_2$, while 
\begin{equation*}
 \sigma_{Y_1}^{(0,0)}(0,1) \equiv \sigma_{Y_1}^{(1,1)}(0,1) \equiv 0 \pmod{2}.
\end{equation*}

Now we derive similar expressions when $n > 1$. To that end, let $W_1, W_2 \subset \R^n$ be oriented subspaces so that $\spn(\partial_{x_1}) \oplus W_1 \oplus W_2 = \R^n$ (unoriented) and so that 
\begin{equation*}
\Or(W_1) \wedge \Or( W_2) \wedge \partial_{x_1} = \Or(\R^n).
\end{equation*}
Assume that $T_aW^u(a) = \spn(\partial_{x_1}) \oplus W_2 $, $T_bW^s(b) = \spn(\partial_{x_1}) \oplus W_1 $ and that the capping orientation of $\Gamma_a$ and $\Gamma_b$ is given by  
\begin{equation*}
 \Or(\Gamma_a) = \Or(W_2) \wedge \partial_{x_1}
\end{equation*}
and
\begin{equation*}
\Or(\Gamma_b) = (-1)^{\sigma^{(\sigma(u),\sigma(l))}_{\negg,1}(\dim W_2)}\Or(W_1) \wedge \partial_{x_1},
\end{equation*}
respectively. 

Notice that 
\begin{align*}
 \Or(\Gamma_v^1) &= (-1)^{\sigma_{Y_1}^{(\sigma(u),\sigma(l))}(1,\sigma(l)+\sigma(u)) + 1} \Or(W_1) \wedge \Or(W_2), \\ 
  \Or(\Gamma_v^2) &= (-1)^{\sigma_{Y_1}^{(\sigma(u),\sigma(l))}(\sigma(l)+\sigma(u),1) + 1} \Or(W_1) \wedge \Or(W_2), \\
 \Or(\Gamma_e^i) &= \Or(W_1) \wedge \Or(W_2) \wedge \partial_{x_1}, \qquad i = 1,2.
\end{align*}

By computing the capping signs of these trees, similar to how it is done above, we get that the capping orientation of $\Gamma$ is given by $ (-1)^{\sigma^{(\sigma(u),\sigma(l))}_{\negg,1}(\dim W_2)+\nu}$, where 
\begin{equation*}
\nu = [\sigma(l) +\sigma(u)]\cdot[\dim W^u(a) + 1] + n\cdot[\dim W^u(a) + 1] + 1.
\end{equation*}
Here we assume that the lifted trivialization coincides with the one induced by the spin structure. In addition, we get that the capping orientation of $\Gamma_i$ is given by $ (-1)^{\sigma^{(\sigma(u),\sigma(l))}_{\negg,1}(\dim W_2)+\nu_i}$, $i=1,2$, where 
\begin{equation*}
 \nu_1 =   {[\sigma(l) +\sigma(u)]\cdot[\dim W^u(a) + 1] + n\cdot\dim W^u(a)  +  \sigma_{Y_1}^{(\sigma(u),\sigma(l))}(1,\sigma(l)+\sigma(u))},
\end{equation*}
\begin{equation*}
 \nu_2 = [\sigma(l) +\sigma(u)]\cdot\dim W^u(a) + n\cdot\dim W^u(a) +1 + \sigma_{Y_1}^{(\sigma(u),\sigma(l))}(\sigma(l)+\sigma(u),1).
\end{equation*}

 Since we should have $\nu \equiv \nu_1 \equiv \nu_2$ we see that 
 \begin{equation*}
  \sigma_{Y_1}^{(\sigma(u),\sigma(l))}(1,\sigma(l)+\sigma(u)) \equiv n +1\pmod{2}
 \end{equation*}

for all $\sigma(l), \sigma(u) \in \Z_2$, and that 
\begin{equation*}
 \sigma_{Y_1}^{(\sigma(u),\sigma(l))}(\sigma(l)+\sigma(u),1) \equiv \sigma(l) +\sigma(u) + n, \pmod{2}.
\end{equation*}

That is, we get 
\begin{align}\label{eq:y11}
&\sigma_{Y_1}^{(0,0)}(1,0) \equiv  \sigma_{Y_1}^{(1,1)}(1,0) \equiv  \sigma_{Y_1}^{(1,0)}(1,1) \equiv  \sigma_{Y_1}^{(0,1)}(1,1)   \equiv n + 1,  \\ \label{eq:y12}
&\sigma_{Y_1}^{(0,0)}(0,1) \equiv  \sigma_{Y_1}^{(1,1)}(0,1)  \equiv n, \pmod{2}.
\end{align}

\begin{rmk}
 Using the notation from \cite{korta} we get that 
 \begin{align*}
  \nu_{\en}(\Gamma) &= (-1)^0, \\
  \nu_{\inter}(\Gamma) &=(-1)^0, \\
  \nu_{\stab}(\Gamma) &= (-1)^{[\sigma(l) +\sigma(u)]\cdot[\dim W^u(a) + 1] + n\cdot[\dim W^u(a) + 1] + 1 + \sigma^{(\sigma(u),\sigma(l))}_{\negg,1}(\dim W_2)} \\
   \nu_{\triv}(\Gamma) &= (-1)^0 \\
  \nu_{\en}(\Gamma_1) &= (-1)^1, \\
  \nu_{\inter}(\Gamma_1) &=(-1)^{1 + n\cdot[1 + \dim W^u(a)]}, \\
  \nu_{\stab}(\Gamma_1) &= (-1)^{[\sigma(l) +\sigma(u)]\cdot[\dim W^u(a) + 1] + n + 1 + \sigma^{(\sigma(u),\sigma(l))}_{\negg,1}(\dim W_2)+ \sigma_{Y_1}^{(\sigma(u),\sigma(l))}(1,\sigma(l)+\sigma(u))} \\
   \nu_{\triv}(\Gamma_1) &= (-1)^1 \\
    \nu_{\en}(\Gamma_2) &= (-1)^0, \\
  \nu_{\inter}(\Gamma_2) &=(-1)^1, \\
  \nu_{\stab}(\Gamma_2) &= (-1)^{[\sigma(l) +\sigma(u)]\cdot\dim W^u(a) + n\cdot\dim W^u(a) + 1 + \sigma^{(\sigma(u),\sigma(l))}_{\negg,1}(\dim W_2) + \sigma_{Y_1}^{(\sigma(u),\sigma(l))}(\sigma(l)+\sigma(u),1)} \\ 
   \nu_{\triv}(\Gamma_2) &= (-1)^1.
 \end{align*}
\end{rmk}

\subsection{Relations for \texorpdfstring{$\sigma_{Y_0}$}{Y0} and \texorpdfstring{$\sigma_{\negg,2}$}{P2}}\label{sec:y0rel}
Here we compute some relations for $\sigma_{Y_0}$, and also state similar relations for $\sigma_{\negg,2}$, which can be computed in an analogous fashion.

 First consider the case when $n = 2$, and let $U \subset M= \R^2$ be a connected region so that we have four sheets of $\Lambda$ in $\Pi^{-1}(U)$ and a Morse flow tree $\Gamma_1$ as shown in Figure \ref{fig:reltree}, consisting of two $Y_0$-vertices, three negative punctures $b_1,b_2,b_3$ and one positive puncture $a$. Assume that the tree has a lift so that after a small Legendrian perturbation of the sheets in $\Pi^{-1}(U)$, supported away from the lifts of the punctures, the tree $\Gamma_1$ transforms to the tree $\Gamma_2$ in Figure \ref{fig:reltree2}. Then the sign of $\Gamma_1$ should coincide with the sign of $\Gamma_2$.

 \begin{figure}[ht]
\labellist
\small\hair 2pt
\pinlabel $\R$ [Br] at 65 280
\pinlabel ${1}$ [Br] at  290 240 
\pinlabel ${2}$ [Br] at 280 190
\pinlabel $3$ [Br] at 270 140
\pinlabel $4$ [Br] at 260 90
\pinlabel $b_1$ [Br] at 150 25
\pinlabel $b_2$ [Br] at 235 30
\pinlabel $b_3$ [Br] at 240 55
\pinlabel $\Gamma_1$ [Br] at 185 25
\pinlabel $M$ [Br] at 47 25
\pinlabel $b_1$ [Br] at 440 85
\pinlabel $b_2$ [Br] at 440 105
\pinlabel $b_3$ [Br] at 440 125
\pinlabel $v_1$ [Br] at 380 115
\pinlabel $v_2$ [Br] at 370 95
\endlabellist
\centering
\includegraphics[height=8cm, width=12cm]{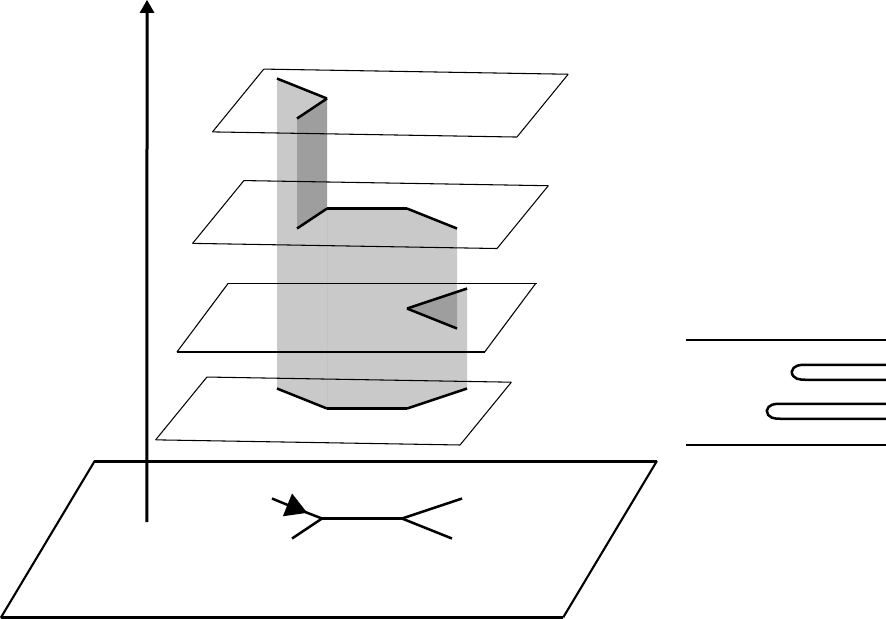}
\caption{The tree $\Gamma_1$ together with its lift and standard domain.}
\label{fig:reltree}
\end{figure}

 \begin{figure}[ht]
\labellist
\small\hair 2pt
\pinlabel $\R$ [Br] at 65 290
\pinlabel ${1}$ [Br] at  290 270 
\pinlabel ${2}$ [Br] at 280 210
\pinlabel $3$ [Br] at 270 160
\pinlabel $4$ [Br] at 260 110
\pinlabel $b_1$ [Br] at 150 25
\pinlabel $b_2$ [Br] at 205 30
\pinlabel $b_3$ [Br] at 210 75
\pinlabel $\Gamma_2$ [Br] at 185 55
\pinlabel $M$ [Br] at 47 25
\pinlabel $b_1$ [Br] at 480 100
\pinlabel $b_2$ [Br] at 480 120
\pinlabel $b_3$ [Br] at 480 140
\pinlabel $v_2$ [Br] at 400 130
\pinlabel $v_1$ [Br] at 410 110
\endlabellist
\centering
\includegraphics[height=8cm, width=12cm]{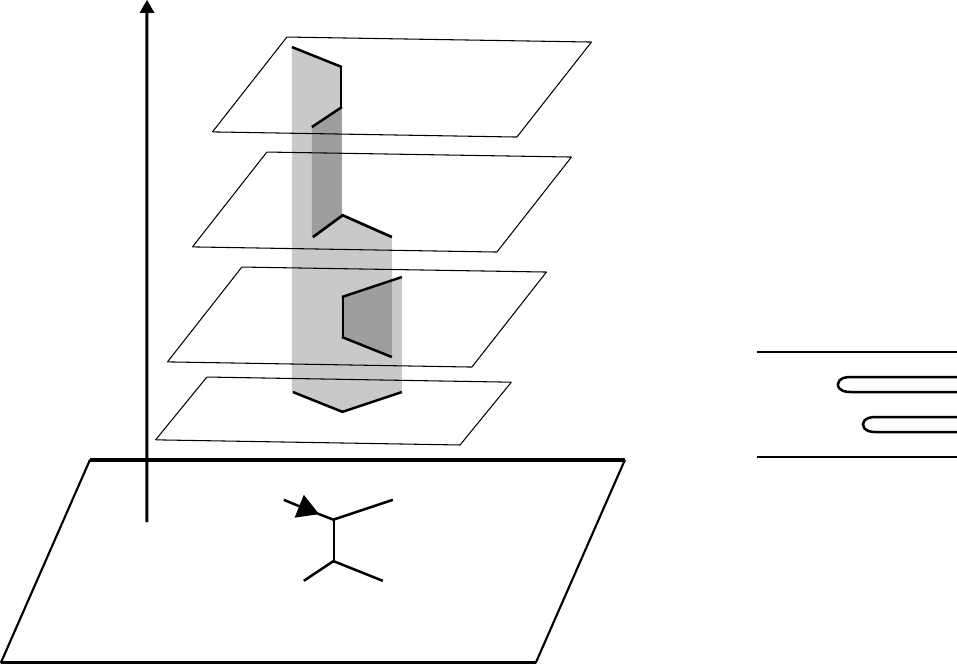}
\caption{The tree $\Gamma_1$ together with its lift and standard domain.}
\label{fig:reltree2}
\end{figure}
 
  Number the sheets as in Figure \ref{fig:reltree} and let $\sigma(i) \in \Z_2$, $i=1, \dotsc, 4$, be the signs of the sheets as defined in \eqref{eq:sheetsign}, using additive notation. Pick an orientation of $\R^2$.  To compute the capping orientation of $\Gamma_1$ and $\Gamma_2$, notice first that there are four different combinations of the dimensions of the unstable manifolds of the punctures $a, b_1, b_2, b_3$ which can occur making $\Gamma_1$ and $\Gamma_2$ rigid. Namely, exactly one of these can have even dimension, and the other three must have dimension 1. Moreover, if $W^u(a)$ is of even dimension, then it has dimension 2, and if $W^u(b_i)$ has even dimension, $i=1,2$ or $3$, then it has dimension 0. 
  
  In each case, we divide the tree $\Gamma_i$ into elementary pieces $\Gamma_a^i, \Gamma_{b_1}^i, \Gamma_{b_2}^i, \Gamma_{b_3}^i, \Gamma_{v_1}^i, \Gamma_{v_2}^i$, $i=1,2$, using notation similar to the one in Section \ref{sec:examples}, and where $v_1, v_2$ represents the $Y_0$-vertices. 
  
  The capping orientation of $\Gamma_{v_1}^i, \Gamma_{v_2}^i$, $i=1,2$ is then given by
  \begin{align*}
     \Or(\Gamma_{v_1}^1) &=  (-1)^{\sigma_{Y_0}^{(\sigma(2), \sigma(4))}(|\mu(b_2)|,|\mu(b_3)|)} \Or(\R^2) \\
   \Or(\Gamma_{v_2}^1) &=  (-1)^{\sigma_{Y_0}^{(\sigma(1), \sigma(4))}(|\mu(b_1)|,|\mu(b_2)| +| \mu(b_3)|)} \Or(\R^2) \\
      \Or(\Gamma_{v_1}^2) &=  (-1)^{\sigma_{Y_0}^{(\sigma(1), \sigma(3))}(|\mu(b_1)|,|\mu(b_2)|)} \Or(\R^2) \\
   \Or(\Gamma_{v_2}^2) &=  (-1)^{\sigma_{Y_0}^{(\sigma(1), \sigma(4))}(|\mu(b_1)| + |\mu(b_2)|, | \mu(b_3)|)} \Or(\R^2). 
  \end{align*}
The capping orientation of the other elementary trees we assume to be given as follows. Choose orientations 
 \begin{equation*}
  \Or_{\capp}(T_aW^u(a)), \quad \Or_{\capp}(T_{b_i}W^s(b_i)), \quad i = 1,2,3,
 \end{equation*}
of the stable and unstable manifolds related to the punctures as in Table \ref{table:0cases}, for each configuration. 
  \begin{table}[hhh]
 \begin{center}
   \caption{{\bf Different cases of configurations for the inputs and choices of orientations of flow manifolds.} 
   }\label{table:0cases}
 \tabulinesep=1.15mm
  \begin{tabu}{|>{\arraybackslash}m{1.2in} |>{\arraybackslash}m{.9in} | >{\arraybackslash}m{1in} |>{\arraybackslash}m{1in} |>{\arraybackslash}m{1in} |}
  \hline &$\Or(T_aW^u(a))$  & $\Or(T_{b_1}W^s(b_1))$ & $\Or(T_{b_2}W^s(b_2))$ & $\Or(T_{b_3}W^s(b_3))$ \\
   \hline 
   $\dim W^u(a) =2$&$\Or(\R^2) $ & $\partial_{x_1} + \partial_{x_2}$ & $-\partial_{x_1} + \partial_{x_2}$ & $-\partial_{x_1} - \partial_{x_2}$ \\
      \hline
     $\dim W^u(b_1) =0$&$\partial_{x_1} - \partial_{x_2}$ & $\Or(\R^2) $ &$-\partial_{x_1} + \partial_{x_2}$ & $-\partial_{x_1} - \partial_{x_2}$   \\
   \hline 
   $\dim W^u(b_2) =0$ &$\partial_{x_1} - \partial_{x_2}$ & $\partial_{x_1} + \partial_{x_2}$ & $\Or(\R^2) $  & $-\partial_{x_1} - \partial_{x_2}$   \\
       \hline 
   $\dim W^u(b_3) =0$ &$\partial_{x_1} - \partial_{x_2}$ & $\partial_{x_1} + \partial_{x_2}$ & $-\partial_{x_1} + \partial_{x_2}$ & $\Or(\R^2) $   \\
    \hline
\end{tabu}
 \end{center}
\end{table}
Then we get that the capping orientations of the elementary pieces containing true punctures are given as follows:
\begin{align*}
 \Or(\Gamma_a^1) &= \Or(\Gamma_a^2) = \Or(T_aW^u(a)) \\
  \Or(\Gamma_{b_1}^1) &= \Or(\Gamma_{b_1}^2) = (-1)^{\sigma^{(\sigma(1),\sigma(2))}_{\negg,1}( \dim(W^u(b_1))}\Or(T_{b_1}W^s(b_1)) \\
    \Or(\Gamma_{b_2}^1) &= \Or(\Gamma_{b_2}^2) = (-1)^{\sigma^{(\sigma(2),\sigma(3))}_{\negg,1}( \dim(W^u(b_2))}\Or(T_{b_2}W^s(b_2)) \\
      \Or(\Gamma_{b_3}^1) &= \Or(\Gamma_{b_3}^2) = (-1)^{\sigma^{(\sigma(3),\sigma(4))}_{\negg,1}( \dim(W^u(b_3))}\Or(T_{b_3}W^s(b_3)).
\end{align*}

Let 
\begin{align*}
 \nu_0 &= \sigma^{(\sigma(1),\sigma(2))}_{\negg,1}( \dim(W^u(b_1)) +\sigma^{(\sigma(2),\sigma(3))}_{\negg,1}( \dim(W^u(b_2)) \\
 &\quad{}+  \sigma^{(\sigma(3),\sigma(4))}_{\negg,1}( \dim(W^u(b_3)).
\end{align*}
Now, similarly to how it is done in Section 7, one computes that the capping orientation of $\Gamma_i$ is given by $(-1)^{\nu_0 + \nu_i}$, $i = 1,2,$ where 
\begin{align*}
 \nu_1 \equiv & \quad{} \sigma_{Y_0}^{(\sigma(2), \sigma(4))}(|\mu(b_2)|,|\mu(b_3)|) + \sigma_{Y_0}^{(\sigma(1), \sigma(4))}(|\mu(b_1)|,|\mu(b_2)|+|\mu(b_3)|) \\
 &\quad{} +\dim W^s(b_1) \cdot[|\mu(b_2)| + |\mu(b_3)| + 1]  + \dim W^s(b_2) \cdot[|\mu(b_3)| + n]  \\
&\quad{} + |\mu(b_1)|\cdot[|\mu(b_2)| + |\mu(b_3)|] + |\mu(b_2)|\cdot[n+1 + |\mu(b_3)|] \\
&\quad{} + \dim W^s(b_3) + \dim W^u(a)\cdot[n + |\mu(a)| + 1] +  |\mu(a)| + n \\
&\quad{} + (n+1) \cdot(\dim W^s(b_1) + \dim W^s(b_3)) + n \\
 \nu_2 \equiv & \quad{} \sigma_{Y_0}^{(\sigma(1), \sigma(3))}(|\mu(b_1)|,|\mu(b_2)|)+ 
 \sigma_{Y_0}^{(\sigma(1), \sigma(4))}(|\mu(b_1)| + |\mu(b_2)|, | \mu(b_3)|) \\
 &\quad{} + \dim W^s(b_1) \cdot[|\mu(b_2)| + |\mu(b_3)| + 1]  + \dim W^s(b_2) \cdot[|\mu(b_3)| + n]  \\
&\quad{} + |\mu(b_1)|\cdot[|\mu(b_2)| + |\mu(b_3)|] + |\mu(b_2)|\cdot[n+1 + |\mu(b_3)|] \\
&\quad{} + \dim W^s(b_3)\cdot n + \dim W^u(a)\cdot[n + |\mu(a)| + 1] +  |\mu(a)| + 1 \\
&\quad{} + (n+1)\cdot \dim W^s(b_1)+ 1, \\
\end{align*}
modulo 2. Here we assume that $\beta(\s,A_d) = 0$ for both trees. Since we should have $\nu_1 \equiv \nu_2$ modulo 2, we get that the following should be satisfied:
\begin{align}\label{eq:sigmayrelnotsimpl}
 &\sigma_{Y_0}^{(\sigma(2), \sigma(4))}(|\mu(b_2)|,|\mu(b_3)|) + \sigma_{Y_0}^{(\sigma(1), \sigma(4))}(|\mu(b_1)|,|\mu(b_2)| +| \mu(b_3)|) \\ \nonumber
 &\equiv 
 \sigma_{Y_0}^{(\sigma(1), \sigma(3))}(|\mu(b_1)|,|\mu(b_2)|)+ 
 \sigma_{Y_0}^{(\sigma(1), \sigma(4))}(|\mu(b_1)| + |\mu(b_2)|, | \mu(b_3)|), \pmod{2}.
\end{align}
Since 
\begin{align*}
 |\mu(b_1)| &\equiv  \sigma(1) + \sigma(2) \\
  |\mu(b_2)| &\equiv  \sigma(2) + \sigma(3) \\
   |\mu(b_3)| &\equiv  \sigma(3) + \sigma(4), \pmod{2}
\end{align*}
we get the following relations (after simplifying)
\begin{align}\label{eq:y01}
&\sigma_{Y_0}^{(0, 0)}(0, 0) \equiv \sigma_{Y_0}^{(0, 1)}(0, 1) \equiv \sigma_{Y_0}^{( 1, 0)}(1, 0) \\ \label{eq:y02}
&\sigma_{Y_0}^{(0, 1)}(1, 0)  \equiv \sigma_{Y_0}^{( 1, 0)}(0, 1)  \equiv \sigma_{Y_0}^{(1, 1)}(0, 0) \\ \label{eq:y03}
&\sigma_{Y_0}^{(0, 0)}(1, 1) + \sigma_{Y_0}^{( 1, 1)}(1, 1) \equiv \sigma_{Y_0}^{(0, 1)}(1, 0) + \sigma_{Y_0}^{(0, 1)}(0, 1) \equiv \sigma_{Y_0}^{(1, 0)}(1, 0) + \sigma_{Y_0}^{(1, 0)}(0, 1),
\end{align}
modulo 2.

In fact, these relations hold in all dimensions $\geq 2$. To see this, one can consider the same pair of trees in higher dimensions. That is, consider $\tilde \Gamma_i = \Gamma_i \times \{0\} \subset \R^2 \oplus \R^{n-2} = \R^n $, $i=1,2$. We assume that the unstable manifolds of the punctures  of $\tilde \Gamma_i$ coincide with the unstable manifolds of the punctures of $\Gamma_i$ in $\R^2$ as described above, and that the stable manifolds of the punctures of $\tilde \Gamma_i$ are given by the stable manifolds of $\Gamma_i$ in $\R^2$ as above times $\R^{n-2}$, $i=1,2$. Then one computes that the difference in capping orientation between $\tilde \Gamma_1$ and $\tilde \Gamma_2$ again is given by \eqref{eq:sigmayrelnotsimpl}.

\begin{rmk}
 Using the notation from \cite{korta} we get that 
 \begin{align*}
  \nu_{\en}(\tilde \Gamma_i) &= (-1)^0,\qquad i = 1,2, \\
  \nu_{\inter}(\tilde \Gamma_1) &=(-1)^{[n+1] \cdot\dim W^s(b_1) +1}, \\
   \nu_{\inter}(\tilde \Gamma_2) &=(-1)^{[n+1] \cdot[\dim W^s(b_1)+\dim W^s(b_3)]+n}\nu_{\inter}(\tilde \Gamma_2), \\
  \nu_{\stab}(\tilde \Gamma_1) &= (-1)^{\nu_0 + \nu_1 + [n+1] \cdot\dim W^s(b_1) +1} \\
  \nu_{\stab}(\tilde \Gamma_2) &= (-1)^{\nu_0 + \nu_2 + [n+1] \cdot[\dim W^s(b_1)+\dim W^s(b_3)]+n }
 \end{align*}
and without loss of generality we might assume that $\nu_{\triv} = (-1)^0$ for all trees. 
\end{rmk}

\begin{rmk}
 If $n = 1$ we cannot have rigid trees looking like $\Gamma_1$ and $\Gamma_2$.
\end{rmk}

We also perform similar calculations for 2-valent negative punctures. That is, consider the variations of Figures \ref{fig:reltree} and \ref{fig:reltree2} given by Figures \ref{fig:reltree3} -- \ref{fig:reltree10}.

 \begin{figure}[ht]
\labellist
\small\hair 2pt
\pinlabel $\R$ [Br] at 65 280
\pinlabel ${1}$ [Br] at  290 240 
\pinlabel ${2}$ [Br] at 280 190
\pinlabel $3$ [Br] at 270 140
\pinlabel $4$ [Br] at 260 90
\pinlabel $b_1$ [Br] at 175 15
\pinlabel $b_2$ [Br] at 235 40
\pinlabel $b_3$ [Br] at 200 30
\pinlabel $\Gamma_3$ [Br] at 185 55
\pinlabel $M$ [Br] at 47 25
\pinlabel $b_1$ [Br] at 440 85
\pinlabel $b_2$ [Br] at 440 105
\pinlabel $b_3$ [Br] at 440 125
\endlabellist
\centering
\includegraphics[height=8cm, width=12cm]{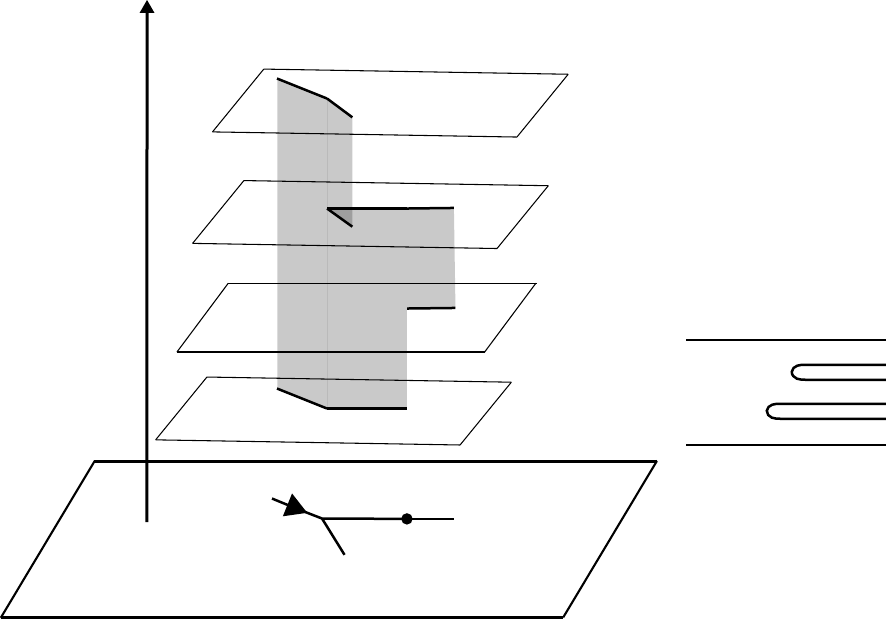}
\caption{The tree $\Gamma_3$ together with its lift and standard domain.}
\label{fig:reltree3}
\end{figure}

 \begin{figure}[ht]
\labellist
\small\hair 2pt
\pinlabel $\R$ [Br] at 65 280
\pinlabel ${1}$ [Br] at  290 240 
\pinlabel ${2}$ [Br] at 280 190
\pinlabel $3$ [Br] at 270 140
\pinlabel $4$ [Br] at 260 90
\pinlabel $b_3$ [Br] at 150 35
\pinlabel $b_1$ [Br] at 180 25
\pinlabel $b_2$ [Br] at 190 50
\pinlabel $\Gamma_3'$ [Br] at 170 60
\pinlabel $M$ [Br] at 47 25
\pinlabel $b_1$ [Br] at 440 85
\pinlabel $b_2$ [Br] at 440 105
\pinlabel $b_3$ [Br] at 440 125
\endlabellist
\centering
\includegraphics[height=8cm, width=12cm]{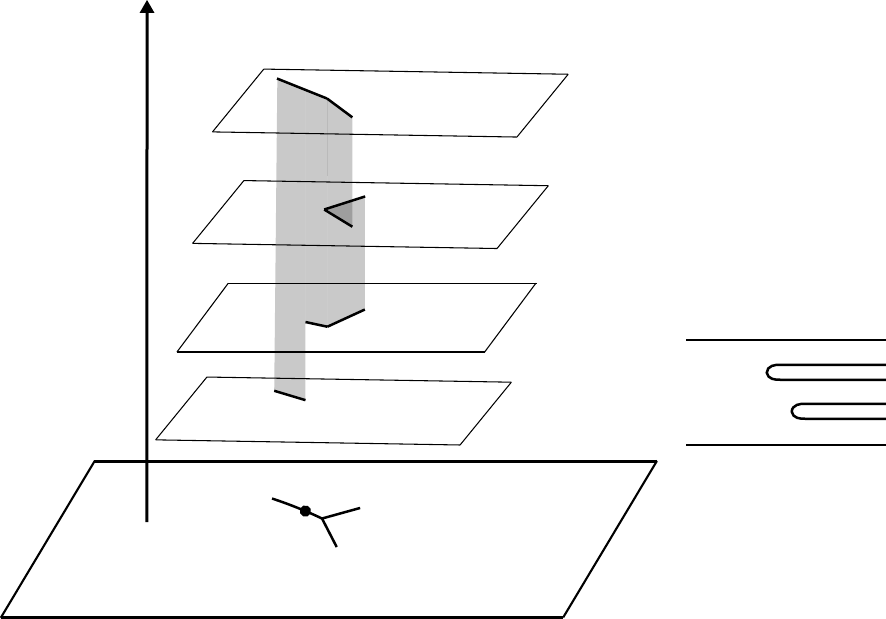}
\caption{The tree $\Gamma_3'$ together with its lift and standard domain.}
\label{fig:reltree4}
\end{figure}

 \begin{figure}[ht]
\labellist
\small\hair 2pt
\pinlabel $\R$ [Br] at 65 280
\pinlabel ${1}$ [Br] at  290 240 
\pinlabel ${2}$ [Br] at 280 190
\pinlabel $3$ [Br] at 270 140
\pinlabel $4$ [Br] at 260 90
\pinlabel $b_1$ [Br] at 165 35
\pinlabel $b_2$ [Br] at 235 10
\pinlabel $b_3$ [Br] at 240 45
\pinlabel $\Gamma_4$ [Br] at 205 55
\pinlabel $M$ [Br] at 47 25
\pinlabel $b_1$ [Br] at 440 85
\pinlabel $b_2$ [Br] at 440 105
\pinlabel $b_3$ [Br] at 440 125
\endlabellist
\centering
\includegraphics[height=8cm, width=12cm]{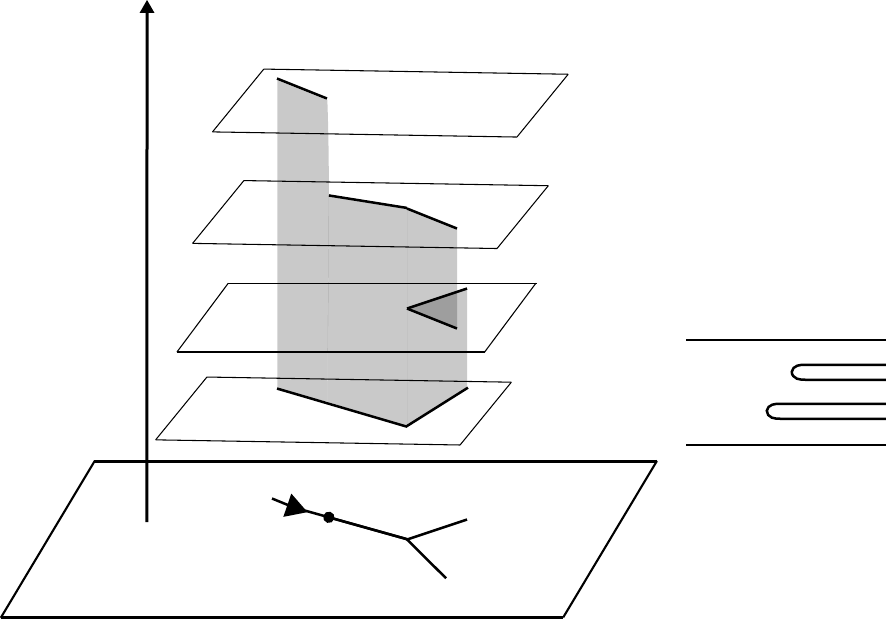}
\caption{The tree $\Gamma_4$ together with its lift and standard domain.}
\label{fig:reltree5}
\end{figure}

 \begin{figure}[ht]
\labellist
\small\hair 2pt
\pinlabel $\R$ [Br] at 65 280
\pinlabel ${1}$ [Br] at  290 240 
\pinlabel ${2}$ [Br] at 280 190
\pinlabel $3$ [Br] at 270 140
\pinlabel $4$ [Br] at 260 90
\pinlabel $b_1$ [Br] at 190 25
\pinlabel $b_2$ [Br] at 225 10
\pinlabel $b_3$ [Br] at 230 55
\pinlabel $\Gamma_4'$ [Br] at 185 55
\pinlabel $M$ [Br] at 47 25
\pinlabel $b_1$ [Br] at 440 85
\pinlabel $b_2$ [Br] at 440 105
\pinlabel $b_3$ [Br] at 440 125
\endlabellist
\centering
\includegraphics[height=8cm, width=12cm]{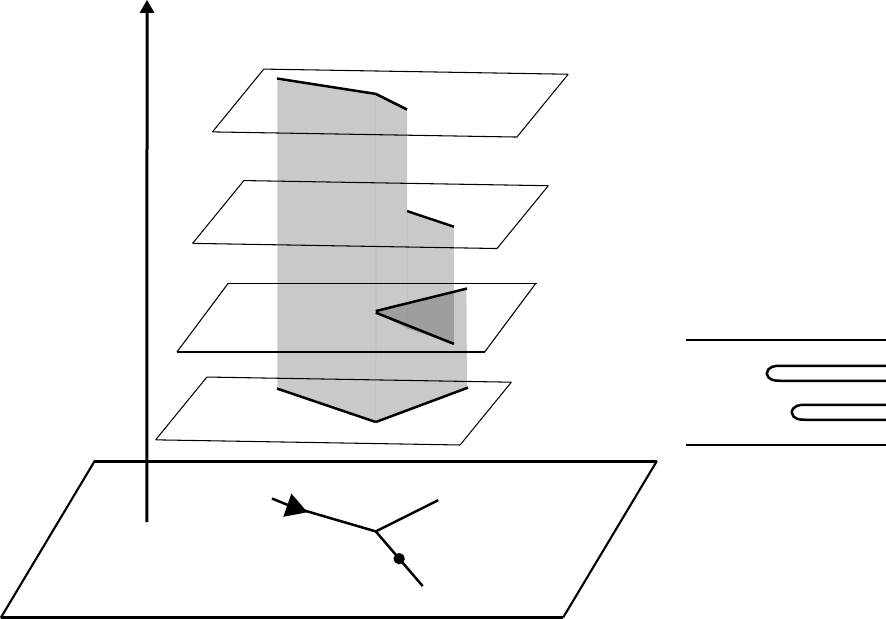}
\caption{The tree $\Gamma_4'$ together with its lift and standard domain.}
\label{fig:reltree6}
\end{figure}

 \begin{figure}[ht]
\labellist
\small\hair 2pt
\pinlabel $\R$ [Br] at 65 280
\pinlabel ${1}$ [Br] at  290 240 
\pinlabel ${2}$ [Br] at 280 190
\pinlabel $3$ [Br] at 270 140
\pinlabel $4$ [Br] at 260 90
\pinlabel $b_1$ [Br] at 180 20
\pinlabel $b_2$ [Br] at 205 30
\pinlabel $b_3$ [Br] at 240 40
\pinlabel $\Gamma_5$ [Br] at 135 25
\pinlabel $M$ [Br] at 47 25
\pinlabel $b_1$ [Br] at 440 85
\pinlabel $b_2$ [Br] at 440 105
\pinlabel $b_3$ [Br] at 440 125
\endlabellist
\centering
\includegraphics[height=8cm, width=12cm]{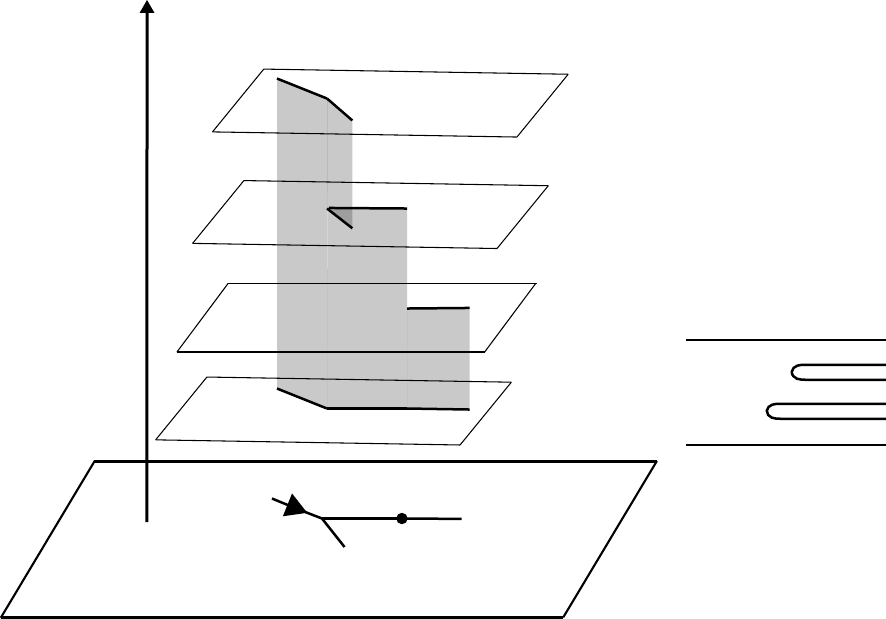}
\caption{The tree $\Gamma_5$ together with its lift and standard domain.}
\label{fig:reltree7}
\end{figure}

 \begin{figure}[ht]
\labellist
\small\hair 2pt
\pinlabel $\R$ [Br] at 65 280
\pinlabel ${1}$ [Br] at  290 240 
\pinlabel ${2}$ [Br] at 280 190
\pinlabel $3$ [Br] at 270 140
\pinlabel $4$ [Br] at 260 90
\pinlabel $b_1$ [Br] at 200 15
\pinlabel $b_2$ [Br] at 165 25
\pinlabel $b_3$ [Br] at 200 55
\pinlabel $\Gamma_5'$ [Br] at 125 25
\pinlabel $M$ [Br] at 47 25
\pinlabel $b_1$ [Br] at 440 85
\pinlabel $b_2$ [Br] at 440 105
\pinlabel $b_3$ [Br] at 440 125
\endlabellist
\centering
\includegraphics[height=8cm, width=12cm]{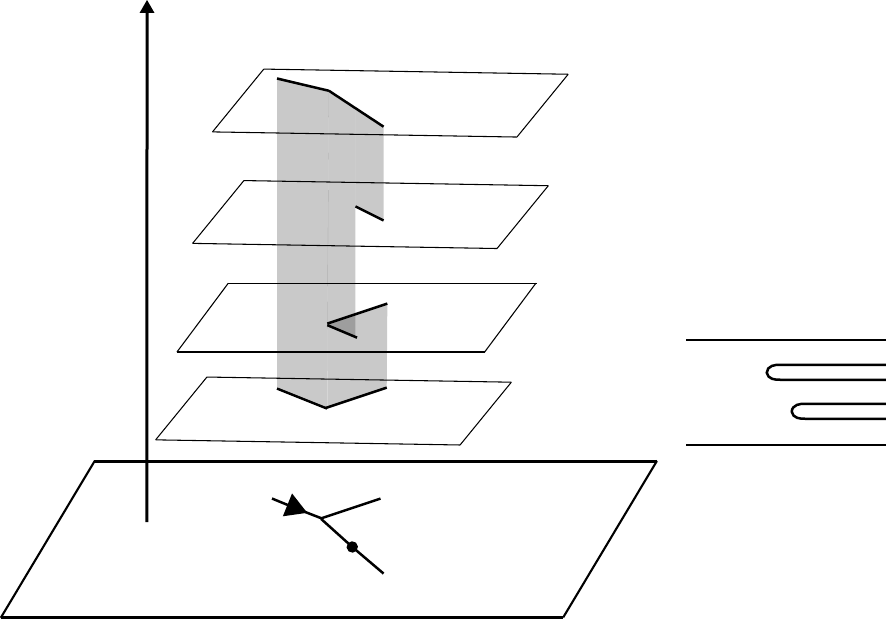}
\caption{The tree $\Gamma_5'$ together with its lift and standard domain.}
\label{fig:reltree8}
\end{figure}

 \begin{figure}[ht]
\labellist
\small\hair 2pt
\pinlabel $\R$ [Br] at 65 280
\pinlabel ${1}$ [Br] at  290 240 
\pinlabel ${2}$ [Br] at 280 190
\pinlabel $3$ [Br] at 270 140
\pinlabel $4$ [Br] at 260 90
\pinlabel $b_3$ [Br] at 200 30
\pinlabel $b_2$ [Br] at 245 45
\pinlabel $b_1$ [Br] at 165 30
\pinlabel $\Gamma_6$ [Br] at 185 60
\pinlabel $M$ [Br] at 47 25
\pinlabel $b_1$ [Br] at 440 85
\pinlabel $b_2$ [Br] at 440 105
\pinlabel $b_3$ [Br] at 440 125
\endlabellist
\centering
\includegraphics[height=8cm, width=12cm]{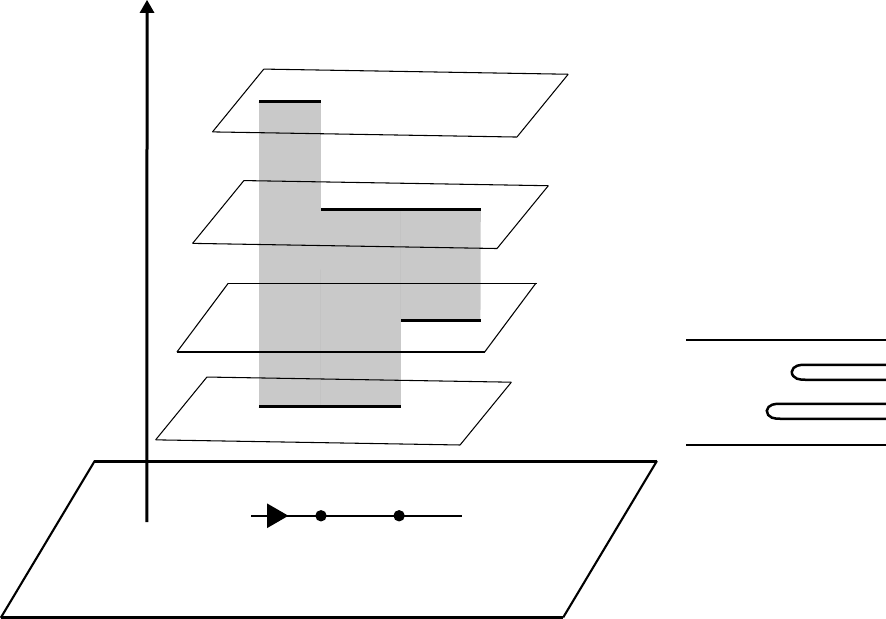}
\caption{The tree $\Gamma_6$ together with its lift and standard domain.}
\label{fig:reltree9}
\end{figure}

 \begin{figure}[ht]
\labellist
\small\hair 2pt
\pinlabel $\R$ [Br] at 65 280
\pinlabel ${1}$ [Br] at  290 250 
\pinlabel ${2}$ [Br] at 280 190
\pinlabel $3$ [Br] at 270 140
\pinlabel $4$ [Br] at 260 90
\pinlabel $b_1$ [Br] at 200 30
\pinlabel $b_2$ [Br] at 245 45
\pinlabel $b_3$ [Br] at 165 30
\pinlabel $\Gamma_6'$ [Br] at 185 60
\pinlabel $M$ [Br] at 47 25
\pinlabel $b_1$ [Br] at 440 85
\pinlabel $b_2$ [Br] at 440 105
\pinlabel $b_3$ [Br] at 440 125
\endlabellist
\centering
\includegraphics[height=8cm, width=12cm]{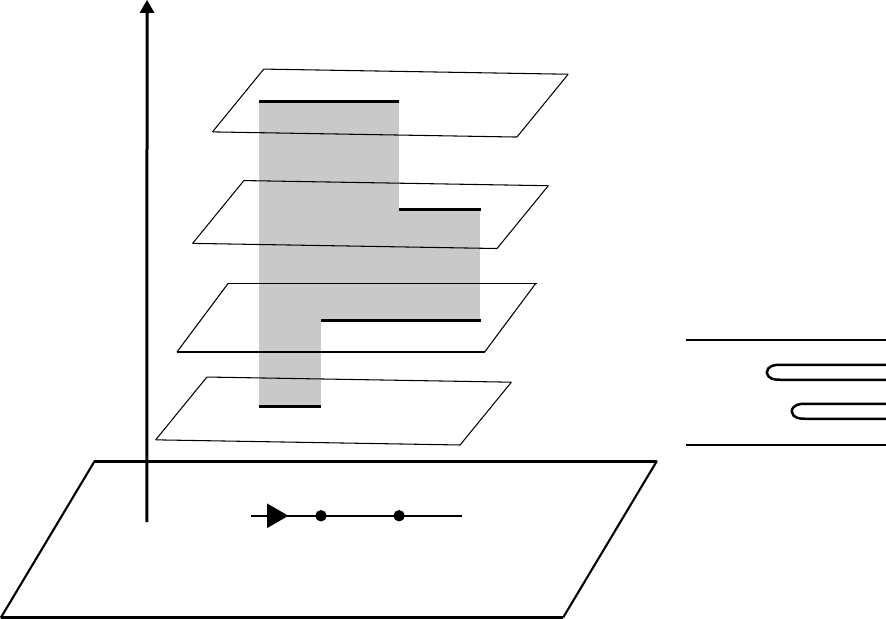}
\caption{The tree $\Gamma_6'$ together with its lift and standard domain.}
\label{fig:reltree10}
\end{figure}

The trees $\Gamma_i$ and $\Gamma_i'$ are related by a Legendrian isotopy, $i =3,4,5,6$, so they should have the same sign. Notice that $\Gamma_6$ and $\Gamma_6'$ only exist for $n=1$, and that  $\Gamma_i$ and $\Gamma_i'$ only exist for $n>1$, $i =3,4,5$.

To compute the difference in the expressions of the capping orientations of $\Gamma_i$ and $\Gamma_i'$, $i =3,4,5,6$, assume that we have chosen orientations of $W^u(a), W^u(b_i)$, $i=1,2,3,4$ and let this induce orientations of the elementary pieces as described in Section \ref{sec:elementor}. 

Using the similar notation as above, we get that the difference between the capping orientation of $\Gamma_i$ and $\Gamma_i'$ is given by $(-1)^{\nu_i}$, $i = 3,4,5,6$,
\begin{align*}
 \nu_3 &= \sigma_{\negg,2}^{(\sigma(2), \sigma(4))}(\sigma(3) + \sigma(4), \sigma(2) + \sigma(3), 2) + \sigma_{\negg,2}^{(\sigma(1), \sigma(4))}(\sigma(3) + \sigma(4), \sigma(1) + \sigma(3), 2) \\
 &{} \quad + \sigma_{Y_0}^{(\sigma(1), \sigma(4))}(\sigma(1) + \sigma(2), \sigma(2) + \sigma(4)) + \sigma_{Y_0}^{(\sigma(1), \sigma(3))}(\sigma(1) + \sigma(2), \sigma(2) + \sigma(3)) \\
 &{} \quad + (\sigma(1) + \sigma(2))\cdot(\sigma(3) + \sigma(4) + 1),
\end{align*}
\begin{align*}
 \nu_4 &= \sigma_{\negg,2}^{(\sigma(1), \sigma(4))}(\sigma(1) + \sigma(2), \sigma(2) + \sigma(4), 1) + \sigma_{\negg,2}^{(\sigma(1), \sigma(3))}(\sigma(1) + \sigma(2), \sigma(2) + \sigma(3), 1) \\
 &{} \quad + \sigma_{Y_0}^{(\sigma(2), \sigma(4))}(\sigma(2) + \sigma(3), \sigma(3) + \sigma(4)) + \sigma_{Y_0}^{(\sigma(1), \sigma(4))}(\sigma(1) + \sigma(3), \sigma(3) + \sigma(4)),
\end{align*}
\begin{align*}
 \nu_5 &= \sigma_{\negg,2}^{(\sigma(2), \sigma(4))}(\sigma(2) + \sigma(3), \sigma(3) + \sigma(4), 1) + \sigma_{\negg,2}^{(\sigma(1), \sigma(3))}(\sigma(2) + \sigma(3), \sigma(1) + \sigma(2), 2) \\
 &{} \quad + \sigma_{Y_0}^{(\sigma(1), \sigma(4))}(\sigma(1) + \sigma(2), \sigma(2) + \sigma(4)) + \sigma_{Y_0}^{(\sigma(1), \sigma(4))}(\sigma(1) + \sigma(3), \sigma(3) + \sigma(4)) \\
 &{} \quad + (\sigma(1) + \sigma(2))\cdot(\sigma(2) + \sigma(3) + 1),
\end{align*}
\begin{align*}
 \nu_6 &= \sigma_{\negg,2}^{(\sigma(1), \sigma(4))}(\sigma(1) + \sigma(2), \sigma(2) + \sigma(4), 1) + \sigma_{\negg,2}^{(\sigma(2), \sigma(4))}(\sigma(3) + \sigma(4), \sigma(2) + \sigma(3), 2) \\
 &{} \quad +\sigma_{\negg,2}^{(\sigma(1), \sigma(3))}(\sigma(1) + \sigma(2), \sigma(2) + \sigma(3), 1) + \sigma_{\negg,2}^{(\sigma(1), \sigma(4))}(\sigma(3) + \sigma(4), \sigma(1) + \sigma(3), 2) \\
 &{} \quad + (\sigma(1) + \sigma(2))\cdot(\sigma(3) + \sigma(4) + 1).
\end{align*}

Simplifying, we get the following relations for $n>1$
\begin{align*}
 \sigma^{(0,0)}_{\negg,2}(0, 0, 1) &\equiv  \sigma^{(0,0)}_{\negg,2}(0, 0, 2) \equiv   \sigma^{(0,1)}_{\negg,2}(0, 1, 1) \\
 &\equiv   \sigma^{(1,0)}_{\negg,2}(0, 1, 2) + 1 \\
  \sigma^{(1,1)}_{\negg,2}(0, 0, 1) &\equiv   \sigma^{(1,1)}_{\negg,2}(0, 0, 2) \equiv   \sigma^{(1,0)}_{\negg,2}(0, 1, 1)  \\
  &\equiv   \sigma^{(0,1)}_{\negg,2}(0, 1, 2)+ 1 
\end{align*}
\begin{align*}
  \sigma^{(0,1)}_{\negg,2}(1, 0, 1) &\equiv  \sigma^{(0,1)}_{\negg,2}(1, 0, 2) + \sigma_{Y_0}^{(0,1)}(0,1) + \sigma_{Y_0}^{(0,1)}(1,0) \\
  \sigma^{(0,0)}_{\negg,2}(1, 1, 1) &\equiv   \sigma^{(1,1)}_{\negg,2}(1, 1, 2) + \sigma_{Y_0}^{(1,0)}(1,0) + \sigma_{Y_0}^{(1,0)}(0,1) \\
    \sigma^{(1,0)}_{\negg,2}(1, 0, 1) &\equiv  \sigma^{(1,0)}_{\negg,2}(1, 0, 2) + \sigma_{Y_0}^{(1,0)}(0,1) + \sigma_{Y_0}^{(1,0)}(1,0)  \\
 \sigma^{(1,1)}_{\negg,2}(1, 1, 1) &\equiv  \sigma^{(0,0)}_{\negg,2}(1, 1, 2) + \sigma_{Y_0}^{(0,1)}(1,0) + \sigma_{Y_0}^{(0,1)}(0,1) \\
  \sigma^{(0,1)}_{\negg,2}(1, 0, 1) &\equiv  \sigma^{(1,1)}_{\negg,2}(1, 1, 2) + \sigma_{Y_0}^{(1,1)}(1,1) + \sigma_{Y_0}^{(1,1)}(0,0) \\  
    \sigma^{(1,0)}_{\negg,2}(1, 0, 2) &\equiv  \sigma^{(1,1)}_{\negg,2}(1, 1, 1) + \sigma_{Y_0}^{(1,1)}(1,1) + \sigma_{Y_0}^{(1,1)}(0,0), 
  \end{align*}
  modulo 2. Also, for $n=1$ we get 
  \begin{align*}
 \sigma^{(0,0)}_{\negg,2}(0, 0, 1) &\equiv  \sigma^{(0,1)}_{\negg,2}(0, 1, 1), \\
  \sigma^{(1,0)}_{\negg,2}(0, 1, 2) &\equiv   \sigma^{(0,0)}_{\negg,2}(0, 0, 2)+ 1, \\
  \sigma^{(1,1)}_{\negg,2}(0, 0, 2) &\equiv  \sigma^{(0,1)}_{\negg,2}(0, 1, 2) +1, \\
    \sigma^{(1,0)}_{\negg,2}(0, 1, 1) &\equiv  \sigma^{(1,1)}_{\negg,2}(0, 0, 1), 
  \end{align*}
  \begin{align*}
\sigma^{(0,0)}_{\negg,2}(1, 1, 1) +  \sigma^{(1,1)}_{\negg,2}(1, 1, 2) & \equiv \sigma^{(0,1)}_{\negg,2}(1, 0, 1) +  \sigma^{(0,1)}_{\negg,2}(1, 0, 2), \\
\sigma^{(0,0)}_{\negg,2}(1, 1, 1) +  \sigma^{(0,0)}_{\negg,2}(1, 1, 2) & \equiv
\sigma^{(0,1)}_{\negg,2}(1, 0, 1) +  \sigma^{(1,0)}_{\negg,2}(1, 0, 2)  \\
 \sigma^{(1,1)}_{\negg,2}(1, 1, 1) +  \sigma^{(1,1)}_{\negg,2}(1, 1, 2)& \equiv \sigma^{(1,0)}_{\negg,2}(1, 0, 1) +  \sigma^{(0,1)}_{\negg,2}(1, 0, 2) , \\
\sigma^{(1,1)}_{\negg,2}(1, 1, 1) +  \sigma^{(0,0)}_{\negg,2}(1, 1, 2) & \equiv \sigma^{(1,0)}_{\negg,2}(1, 0, 1) +  \sigma^{(1,0)}_{\negg,2}(1, 0, 2), \\
  \end{align*}
modulo 2.

%
%
%
%

\subsection{More relations for \texorpdfstring{$\sigma_{Y_0}$}{Y} and relations for \texorpdfstring{$\sigma_{\swi}$}{s}}\label{sec:longexgen}

Using the results in the former subsections and a generalization of the example in Section \ref{sec:examples}, we can derive some more relations for $\sigma_{Y_0}$ and also for $\sigma_{\swi}$.

To that end, assume that the example in Section \ref{sec:examples} now takes place in $J^1(\R^n)$, and let $\sigma(u), \sigma(l) \in \Z_2$ be the sign of the upper respectively lower sheet of the lift of $\Gamma$, as defined in \eqref{eq:sheetsign}, using additive notation. Let $W_1$, $W_2 \subset \R^n$ be oriented subspaces  so that 
 $\spn(\partial_{x_1},\partial_{x_2}) \oplus W_1 \oplus W_2 = \R^n$ and so that 
\begin{equation*}
\Or(W_1) \wedge \Or( W_2) \wedge \partial_{x_1} \wedge \partial_{x_2}= \Or(\R^n).
\end{equation*}
Let $T_aW^u(a) = \spn(\partial_{x_1}) \oplus W_2 $, $T_bW^s(b) = \spn(\partial_{x_1},\partial_{x_2}) \oplus W_1 $ and choose capping orientation $\Or(\Gamma_a)$ and $\Or(\Gamma_b)$ of $\Gamma_a$ and $\Gamma_b$ so that 
\begin{equation*}
 \Or(\Gamma_a) = \partial_{x_1} \wedge \Or(W_2),
\end{equation*}
\begin{equation*}
\Or(\Gamma_b) = (-1)^{\sigma^{(\sigma(u),\sigma(l))}_{\negg,1}(\dim W_2)}\partial_{x_1}\wedge \partial_{x_2} \wedge \Or(W_1).
\end{equation*}

The capping orientation of $\Gamma$ is then given by 
$ (-1)^{\sigma^{(\sigma(u),\sigma(l))}_{\negg,1}(\dim W_2)+\nu}$, where 
\begin{equation*}
\nu = [\sigma(l) +\sigma(u)]\cdot[\dim W^u(a) + 1] + n\cdot[\dim W^u(a) + 1] + 1.
\end{equation*}

If one computes the capping orientation of the tree $\Gamma_A$ from Section \ref{sec:gammaasign} in this higher-dimensional setting one gets that the capping sign is given by $ (-1)^{\sigma^{(\sigma(u),\sigma(l))}_{\negg,1}(\dim W_2)+\nu_1}$, where 
\begin{align*}
 \nu_1 &= [\sigma(l) + \sigma(u)]\cdot[\dim W^u(a)+1] + n\cdot\dim W^u(a) +n + 1 \\
 &+   \sigma_{Y_0}^{(\sigma(u),\sigma(l))}(1+\sigma(l)+\sigma(u),1) + \sigma_{\swi}^{(\sigma(u),\sigma(l)+1)}(l) + \sigma_{Y_1}^{(\sigma(u),\sigma(l))}(\sigma(l)+\sigma(u),1).
\end{align*}
Since we should have $\nu \equiv \nu_1$ modulo 2, it follows that 
\begin{align}\label{eq:firstrel}
 \sigma_{Y_0}^{(\sigma(u),\sigma(l))}(1+\sigma(l)+\sigma(u),1) + \sigma_{\swi}^{(\sigma(u),\sigma(l)+1)}(l) + \sigma_{Y_1}^{(\sigma(u),\sigma(l))}(\sigma(l)+\sigma(u),1) \equiv 0, 
\end{align}
modulo 2, for all $\sigma(l), \sigma(u) \in \Z_2$. 
Moreover, since by assumption we have $\sigma_{\swi}^{(0, 0)}(l) =0 = \sigma_{\swi}^{(1, 1)}(l)$ and by \eqref{eq:y11} that $ \sigma_{Y_1}^{(0, 1)}(1,1) \equiv \sigma_{Y_1}^{(1,0)}(1,1) \equiv n+1$, it follows that 
\begin{align*}
 \sigma_{Y_0}^{(0, 1)}(0,1)  \equiv \sigma_{Y_0}^{(1, 0)}(0,1)  \equiv n+1. 
\end{align*}
From the relations \eqref{eq:y01}--\eqref{eq:y03} for $ \sigma_{Y_0}$ it then follows that 
\begin{align*}
& \sigma_{Y_0}^{(0, 1)}(1,0)  \equiv \sigma_{Y_0}^{(1, 0)}(1,0)   \equiv \sigma_{Y_0}^{(1, 1)}(0,0)  \equiv \sigma_{Y_0}^{(0, 0)}(0,0)   \equiv n+1, \\
& \sigma_{Y_0}^{(0, 0)}(1,1)  \equiv \sigma_{Y_0}^{(1, 1)}(1,1), \pmod{2}. 
\end{align*}
Returning to \eqref{eq:firstrel} with $\sigma(l) + \sigma(u) = 0$  and using \eqref{eq:y12} it follows that 
\begin{equation*}
 \sigma_{\swi}^{(1, 0)}(l) \equiv  \sigma_{\swi}^{(0, 1)}(l) \equiv \sigma_{Y_0}^{(0,0)}(1,1) + n, \pmod{2}.
\end{equation*}

Now consider a slight variation of the isotopy in Section \ref{sec:examples}, where it is the upper sheet that is isotoped, and where the isotopy is similar to the one in Figure \ref{fig:isot} but turned upside down. Then we get another rigid flow tree $\Gamma_B$, see Figure \ref{fig:tree7a}, which again should have the same capping sign as $\Gamma$. The main difference between $\Gamma_A$ and $\Gamma_B$ is the order of the vertices as they occur along the boundary of the standard domain, and that the turn of the switch is now in the upper sheet instead of the  lower. Similar to above, we compute that the capping orientation of $\Gamma_B$ is given by 
 $ (-1)^{\sigma^{(\sigma(u),\sigma(l))}_{\negg,1}(\dim W_2)+\nu_2}$, where 
\begin{align*}
 \nu_2 &= [\sigma(l) + \sigma(u)]\cdot[\dim W^u(a) + 1]+ n\cdot[\dim W^u(a) + 1]\\
  &+ \sigma_{Y_0}^{(\sigma(u),\sigma(l))}(1,\sigma(l)+\sigma(u)+1) + \sigma_{\swi}^{(\sigma(u)+1,\sigma(l))} (u)+ \sigma_{Y_1}^{(\sigma(u),\sigma(l))}(1,\sigma(l)+\sigma(u)).
\end{align*}

\begin{figure}[ht]
\vspace{.5cm}
\labellist
\small\hair 2pt
\pinlabel $a$ [Br] at 27 70
\pinlabel $v_1$ [Br] at  135 40 
\pinlabel ${e_2}$ [Br] at 317 10
\pinlabel $s$ [Br] at  360 105
\pinlabel $b$ [Br] at  842 62 
\pinlabel $w_1$ [Br] at 349 38
\pinlabel $w_2$ [Br] at 410 29
\pinlabel $u_2$ [Br] at 57 105
\pinlabel $u_1$ [Br] at 130 155
\pinlabel ${e_1}$ [Br] at 280 255
\pinlabel ${v_0}$ [Br] at 280 115
\endlabellist
\centering
\includegraphics[height=2.7cm, width=8.5cm]{extreeB}
\hspace{1.5cm}
\labellist
\small\hair 2pt
\pinlabel $a$ [Br] at 10  60
\pinlabel $v_0$ [Br] at  240 95 
\pinlabel $v_1$ [Br] at  140 40 
\pinlabel $s$ [Br] at 395 115
\pinlabel $e_2$ [Br] at 535 2
\pinlabel $e_1$ [Br] at 535 68
\pinlabel $b$ [Br] at 525 123
\endlabellist
\centering
\includegraphics[height=1.3cm]{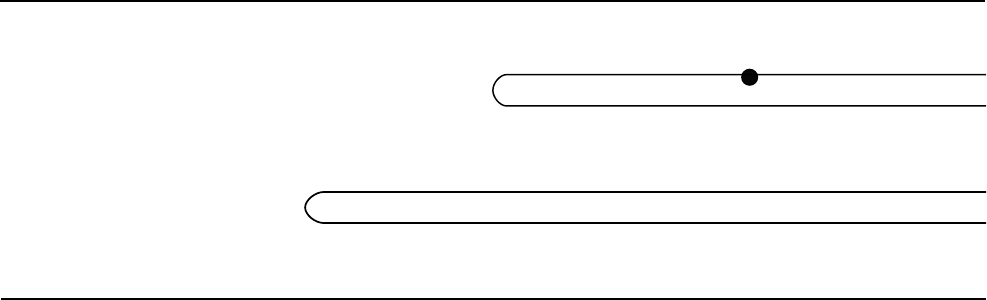}
\caption{The tree $\Gamma_B$.}
\label{fig:tree7a}
\end{figure}

Since we again should have $\nu \equiv \nu_2$ modulo 2 it follows that 
\begin{equation*}
 \sigma_{Y_0}^{(\sigma(u),\sigma(l))}(1,\sigma(l)+\sigma(u)+1) + \sigma_{\swi}^{(\sigma(u)+1,\sigma(l))} (u)+ \sigma_{Y_1}^{(\sigma(u),\sigma(l))}(1,\sigma(l)+\sigma(u)) \equiv 1 \pmod{2}.
\end{equation*}
Similar to above we get that 
\begin{equation*}
 \sigma_{\swi}^{(0, 0)}(u) \equiv  \sigma_{\swi}^{(1, 1)}(u) \equiv 1, \pmod{2}.
\end{equation*}
and that 
\begin{equation*}
 \sigma_{\swi}^{(1, 0)}(u) \equiv  \sigma_{\swi}^{(0, 1)}(u) \equiv \sigma_{\swi}^{(1, 0)}(l) \equiv \sigma_{Y_0}^{(0,0)}(1,1) + n , \pmod{2}.
\end{equation*}

\begin{rmk}
 Using the notation from \cite{korta} we get that 
 \begin{align*}
  \nu_{\en}(\Gamma) &= (-1)^0, \\
  \nu_{\inter}(\Gamma) &=(-1)^0, \\
  \nu_{\stab}(\Gamma) &= (-1)^{[\sigma(l) +\sigma(u)]\cdot[\dim W^u(a) + 1] + n\cdot[\dim W^u(a) + 1] + 1 + \sigma^{(\sigma(u),\sigma(l))}_{\negg,1}(\dim W_2)} \\
  \nu_{\en}(\Gamma_A) &= (-1)^0, \\
  \nu_{\inter}(\Gamma_A) &=(-1)^{n+1}, \\
  \nu_{\stab}(\Gamma_A) &= (-1)^{[\sigma(l) + \sigma(u)]\cdot[\dim W^u(a)+1] + n\cdot\dim W^u(a) +\sigma^{(\sigma(u),\sigma(l))}_{\negg,1}(\dim W_2)} \cdot \\
  & (-1)^{ \sigma_{Y_0}^{(\sigma(u),\sigma(l))}(1+\sigma(l)+\sigma(u),1) + \sigma_{\swi}^{(\sigma(u),\sigma(l)+1)}(l) + \sigma_{Y_1}^{(\sigma(u),\sigma(l))}(\sigma(l)+\sigma(u),1)} \\
    \nu_{\en}(\Gamma_B) &= (-1)^0, \\
  \nu_{\inter}(\Gamma_B) &=(-1)^{1}, \\
  \nu_{\stab}(\Gamma_B) &= (-1)^{[\sigma(l) + \sigma(u)]\cdot[\dim W^u(a) + 1]+ n\cdot[\dim W^u(a) + 1]+1+ \sigma^{(\sigma(u),\sigma(l))}_{\negg,1}(\dim W_2)} \cdot \\
  & (-1)^{\sigma_{Y_0}^{(\sigma(u),\sigma(l))}(1,\sigma(l)+\sigma(u)+1) + \sigma_{\swi}^{(\sigma(u)+1,\sigma(l))} (u)+ \sigma_{Y_1}^{(\sigma(u),\sigma(l))}(1,\sigma(l)+\sigma(u))} 
 \end{align*}
and without loss of generality we might assume that $\nu_{\triv} = (-1)^0$ for all trees. 
\end{rmk}

\subsection{Results}\label{sec:results}
Now we give a summary of some of the results derived in this section.

\subsubsection{Values and relations for \texorpdfstring{$\sigma_{Y_0}$}{Y}, modulo 2}

\begin{align*}
&\sigma_{Y_0}^{(1, 1)}(0,0)  \equiv \sigma_{Y_0}^{(0, 0)}(0,0)   \equiv n+1, \\
& \sigma_{Y_0}^{(0, 1)}(0,1)  \equiv \sigma_{Y_0}^{(1, 0)}(0,1)  \equiv n+1 \\ 
& \sigma_{Y_0}^{(0, 1)}(1,0)  \equiv \sigma_{Y_0}^{(1, 0)}(1,0)    \equiv n+1, \\
& \sigma_{Y_0}^{(0, 0)}(1,1)  \equiv \sigma_{Y_0}^{(1, 1)}(1,1). 
\end{align*}



\subsubsection{Values and relations for \texorpdfstring{$\sigma_{Y_1}$}{Y}, modulo 2}
$ $ \\
$n=1$:
\begin{align*}
&\sigma_{Y_1}^{(0,0)}(1,0) \equiv  \sigma_{Y_1}^{(1,1)}(1,0) \equiv  \sigma_{Y_1}^{(1,0)}(1,1) \equiv  \sigma_{Y_1}^{(0,1)}(1,1)   \equiv 1,  \\ 
&\sigma_{Y_1}^{(0,0)}(0,1) \equiv  \sigma_{Y_1}^{(1,1)}(0,1)  \equiv 0,
\end{align*}
$n>1$:
\begin{align*}
&\sigma_{Y_1}^{(0,0)}(1,0) \equiv  \sigma_{Y_1}^{(1,1)}(1,0) \equiv  \sigma_{Y_1}^{(1,0)}(1,1) \equiv  \sigma_{Y_1}^{(0,1)}(1,1)   \equiv n + 1,  \\ 
&\sigma_{Y_1}^{(0,0)}(0,1) \equiv  \sigma_{Y_1}^{(1,1)}(0,1)  \equiv n.
\end{align*}

\subsubsection{Values and relations for \texorpdfstring{$\sigma_{{\swi}}$}{sw}, modulo 2}
\begin{align*}
 &\sigma_{\swi}^{(0, 0)}(l) \equiv \sigma_{\swi}^{(1, 1)}(l) \equiv 0, \\
 &\sigma_{\swi}^{(0, 0)}(u) \equiv  \sigma_{\swi}^{(1, 1)}(u) \equiv 1, \\
 &\sigma_{\swi}^{(1, 0)}(u) \equiv  \sigma_{\swi}^{(0, 1)}(u) \equiv \sigma_{\swi}^{(1, 0)}(l) \equiv  \sigma_{\swi}^{(0, 1)}(l) \equiv \sigma_{Y_0}^{(0,0)}(1,1) + n.
\end{align*}

\subsubsection{Values and relations for \texorpdfstring{$\sigma_{{\negg,2}}$}{neg2}, modulo 2}
$ $ \\

$n>1:$
\begin{align*}
 \sigma^{(0,0)}_{\negg,2}(0, 0, 1) &\equiv  \sigma^{(0,0)}_{\negg,2}(0, 0, 2) \equiv   \sigma^{(0,1)}_{\negg,2}(0, 1, 1) \\
 &\equiv   \sigma^{(1,0)}_{\negg,2}(0, 1, 2) + 1 \\
  \sigma^{(1,1)}_{\negg,2}(0, 0, 1) &\equiv   \sigma^{(1,1)}_{\negg,2}(0, 0, 2) \equiv   \sigma^{(1,0)}_{\negg,2}(0, 1, 1)  \\
  &\equiv   \sigma^{(0,1)}_{\negg,2}(0, 1, 2)+ 1 
\end{align*}
\begin{align*}
  \sigma^{(0,1)}_{\negg,2}(1, 0, 1) &\equiv  \sigma^{(0,1)}_{\negg,2}(1, 0, 2)  \\
  \sigma^{(0,0)}_{\negg,2}(1, 1, 1) &\equiv   \sigma^{(1,1)}_{\negg,2}(1, 1, 2)  \\
    \sigma^{(1,0)}_{\negg,2}(1, 0, 1) &\equiv  \sigma^{(1,0)}_{\negg,2}(1, 0, 2)  \\
 \sigma^{(1,1)}_{\negg,2}(1, 1, 1) &\equiv  \sigma^{(0,0)}_{\negg,2}(1, 1, 2) \\
  \sigma^{(0,1)}_{\negg,2}(1, 0, 1) &\equiv  \sigma^{(1,1)}_{\negg,2}(1, 1, 2) + \sigma_{Y_0}^{(1,1)}(1,1) + n+1 \\  
    \sigma^{(1,0)}_{\negg,2}(1, 0, 2) &\equiv  \sigma^{(1,1)}_{\negg,2}(1, 1, 1) + \sigma_{Y_0}^{(1,1)}(1,1) + n+1
  \end{align*}
  $n=1:$
  \begin{align*}
 \sigma^{(0,0)}_{\negg,2}(0, 0, 1) &\equiv  \sigma^{(0,1)}_{\negg,2}(0, 1, 1), \\
  \sigma^{(1,0)}_{\negg,2}(0, 1, 2) &\equiv   \sigma^{(0,0)}_{\negg,2}(0, 0, 2)+ 1 \\
  \sigma^{(1,1)}_{\negg,2}(0, 0, 2) &\equiv  \sigma^{(0,1)}_{\negg,2}(0, 1, 2) +1 \\
    \sigma^{(1,0)}_{\negg,2}(0, 1, 1) &\equiv  \sigma^{(1,1)}_{\negg,2}(0, 0, 1) 
  \end{align*}
  \begin{align*}
\sigma^{(0,0)}_{\negg,2}(1, 1, 1) +  \sigma^{(1,1)}_{\negg,2}(1, 1, 2) & \equiv \sigma^{(0,1)}_{\negg,2}(1, 0, 1) +  \sigma^{(0,1)}_{\negg,2}(1, 0, 2) \\
\sigma^{(0,0)}_{\negg,2}(1, 1, 1) +  \sigma^{(0,0)}_{\negg,2}(1, 1, 2) & \equiv
\sigma^{(0,1)}_{\negg,2}(1, 0, 1) +  \sigma^{(1,0)}_{\negg,2}(1, 0, 2)  \\
 \sigma^{(1,1)}_{\negg,2}(1, 1, 1) +  \sigma^{(1,1)}_{\negg,2}(1, 1, 2)& \equiv \sigma^{(1,0)}_{\negg,2}(1, 0, 1) +  \sigma^{(0,1)}_{\negg,2}(1, 0, 2)  \\
\sigma^{(1,1)}_{\negg,2}(1, 1, 1) +  \sigma^{(0,0)}_{\negg,2}(1, 1, 2) & \equiv \sigma^{(1,0)}_{\negg,2}(1, 0, 1) +  \sigma^{(1,0)}_{\negg,2}(1, 0, 2)
  \end{align*}

\section{Analytic derivation of the remaining sign functions}\label{sec:analder}
\subsection{Derivation of sign functions}\label{sec:signfunctions}
In this section finish the derivation of the explicit expressions for the sign functions for $n>1$, given the choices from Section \ref{sec:elementor}. 

From Section \ref{sec:elementor} it follows that the sign functions are independent of $\lambda$ for $\lambda$ small enough, and we will throughout work in the limit $\lambda = 0$.

\subsubsection{Assumptions}\label{sec:assumptions}
Recall the capping operators of a special puncture $p$ with odd Maslov index. Both are surjective with one-dimensional kernel in the direction corresponding to the second auxiliary coordinate, which we represent with the vector $\partial_{x_{n+2}}$. Moreover, the kernel corresponding to the positive capping operator can be seen as given by constant functions, while the kernel corresponding to the negative capping operator is spanned (over $\R$) by the function $i(z-1)$. Let $\sigma_{1,+}^{(0,1)}, \sigma_{1,+}^{(1,0)} \in \{0,1\}$ so that the orientation of the determinant line of the positive capping operator is given by
 \begin{equation*}
  \det \dbar_{p,+} = (-1)^{\sigma_{1,+}^{(0,1)}} \partial_{x_{n+2}}
 \end{equation*}
when $p$ is of type $(0,1)$, and by  
  \begin{equation*}
  \det \dbar_{p,+} = (-1)^{\sigma_{1,+}^{(1,0)}} \partial_{x_{n+2}}
 \end{equation*}
  when $p$ is of type $(1,0)$. Similarly, let $\sigma_{1,-}^{(0,1)}, \sigma_{1,-}^{(1,0)} \in \{0,1\}$ so that the orientation of the determinant line of the negative capping operator is given by
 \begin{equation*}
  \det \dbar_{p,-} = (-1)^{\sigma_{1,-}^{(0,1)}} i(z-1)\partial_{x_{n+2}}
 \end{equation*}
 when $p$ is of type $(0,1)$, and by  
  \begin{equation*}
  \det \dbar_{p,-} = (-1)^{\sigma_{1,-}^{(1,0)}} i(z-1)\partial_{x_{n+2}}
 \end{equation*}
   when $p$ is of type $(1,0)$.
 Recall that the two former are induced by our choice of orientation of the end-piece, and that the two latter then get their orientation from the gluing sequence \eqref{eq:capseqspec} together with the canonical orientation of $\det \dbar_p$ times $(-1)^n$. We notice the following.
 
 \begin{lma}\label{lma:oddsigmarel}
  Let $\Gamma_e$ be an end-piece with $p$ as positive special puncture and choose orientation of $\det \dbar_{p,+}$ so that the capping orientation of $\Gamma_e$ is given by $\R^n$ both in the case when $p$ is of type $(0,1)$ and $(1,0)$. Then 
  \begin{equation}
   \sigma_{1,+}^{(0,1)} \equiv \sigma_{1,+}^{(1,0)} + 1, \pmod{2}.
  \end{equation} 
 \end{lma}
 \begin{proof}
  We have two different model end-pieces, depending on the type of $p$. Consider first the fully capped problem $\dbar_{\hat \Gamma_e}$ in the case when $p$ is of type $(1,0)$. The trivialized Lagrangian boundary condition for this problem is given by 
  \begin{equation}\label{eq:trivend}
    \diag(\underbrace{1,\dotsc,1}_{n}, e^{-i\pi s},1) \hat * R_{-\pi}(n+1,n+2) \hat * R_{-\pi}(n,n+2) * \diag(\underbrace{1,\dotsc,1}_{n-1}, -e^{i\pi s},1,1).
  \end{equation} 
  Similarly, in the case when $p$ is of type $(0,1)$ the trivialized Lagrangian boundary condition for the $\dbar_{\hat \Gamma_e}$--problem is given by 
    \begin{align}\label{eq:trivend10}
    &\left( R_{\pi}(n,n+2)\cdot\diag(\underbrace{1,\dotsc,1}_{n+1},-1) \right) *\diag(\underbrace{1,\dotsc,1}_{n-1},-1, e^{-i\pi s},1) \hat * R_{-\pi}(n+1,n+2) \hat *\\ \nonumber
    &  \qquad{} R_\pi(n,n+2) * \diag(\underbrace{1,\dotsc,1}_{n-1}, e^{i\pi s},1,1) \hat * R_{-\pi}(n,n+2).
  \end{align} 
  We see that the latter boundary condition is homotopic to 
  \begin{align}\label{eq:trivend10homo}
   &\diag(\underbrace{1,\dotsc,1}_{n-1},-1, e^{-i\pi s},1) \hat * R_{-\pi}(n+1,n+2) \hat * R_\pi(n,n+2) * \diag(\underbrace{1,\dotsc,1}_{n-1}, e^{i\pi s},1,1) 
  \end{align} 
  which is nothing but the trivialized Lagrangian boundary condition for the $(0,1)$--case after applying the change of coordinates $\partial_{x_n} \leftrightarrow -\partial_{x_n}$. Hence these 2 problems have opposite canonical orientation. 
  
  If we now consider the corresponding capping sequences we get 
    \begin{align}\label{eq:cappingend}
 0 \to \kd_{\hat \Gamma_e}\xrightarrow{\alpha} 
\begin{bmatrix}
\kd_e\\
 \kd_{p,+}\\
\kd_{\Gamma_e}
\end{bmatrix}
\to 
\begin{bmatrix}
 0 \\
0
\end{bmatrix}
\to 
0 \to 0.
\end{align}
Since we assume that $\kd_{\Gamma_e} = \R^n$ oriented in both cases, and since the orientation chosen on $\kd_e$, see Section \ref{sec:endmp}, is independent of trivialization, it follows that we must have 
 \begin{equation*}
   \sigma_{1,+}^{(0,1)} \equiv \sigma_{1,+}^{(1,0)} + 1, \pmod{2}
  \end{equation*} 
  to get opposite orientation induced on $\kd_{\hat \Gamma_e}$ from the gluing sequence in the two cases. 
 \end{proof}

Next consider the $(n+2)$--dimensional Maslov problem $\dbar_{-1,1}$ with trivialized boundary condition given by 
 \begin{equation}\label{eq:bc-+}
 \diag(\underbrace{1,\dotsc,1}_{n}, e^{-i\pi s}, 1)\hat * R_{-\pi}(n+1,n+2)  *\diag(\underbrace{1,\dotsc,1}_{n}, 1, -e^{i\pi s}).
 \end{equation} 
This is a surjective problem with $(n+2)$--dimensional kernel, spanned by \linebreak 
$\partial_{x_1}, \dotsc, \partial_{x_n}, v_1\partial_{x_{n+2}}, v_2\partial_{x_{n+2}}$, where $v_1$ is close to the cutoff constant solution on the part of the disk corresponding to the constant boundary condition in the $\partial_{x_{n+2}}$--direction, and $v_2$ is close to the cutoff of the solution $i(z-1)$ on the part of the disk corresponding to the boundary condition $-e^{i\pi s}$ in the $\partial_{x_{n+2}}$--direction. Let $\mu(0) \in \{0,1\}$ such that $\det \dbar_{-1,1}$ has canonical orientation given by 
\begin{equation}\label{eq:orientmu0}
 (-1)^{\mu(0)} \partial_{x_1} \wedge \dotsm \wedge \partial_{x_n} \wedge v_1\partial_{x_{n+2}} \wedge v_2\partial_{x_{n+2}}.
\end{equation} 

\begin{lma}\label{lma:mu0}
 We have that 
 \begin{equation}
\mu(0) \equiv \sigma_{1,+}^{(1,0)} + \sigma_{1,-}^{(1,0)} 
\pmod{2}.              
\end{equation} 
\end{lma}
\begin{proof}
 Consider the exact sequence \eqref{eq:capseqspec} in the case $p$ is a special puncture of type $(1,0)$. In this case we get that the trivialized boundary conditions for the glued $\dbar_p$--problem is given by \eqref{eq:bc-+}. Hence $\det \dbar_p$ has canonical orientation given by \eqref{eq:orientmu0}. 
 
 To simplify the notation, let $\langle v\rangle$ be the span of the vector $v$.
 Then the sequence \eqref{eq:capseqspec} is given by 
  \begin{align*}
 0 \to 
 \begin{bmatrix}
 \R^n\\
\langle v_1 \partial_{x_{n+2}} \rangle \\
\langle v_2 \partial_{x_{n+2}} \rangle
 \end{bmatrix}
 \xrightarrow{\alpha} 
\begin{bmatrix}
 \langle (-1)^{\sigma_{1,+}^{(1,0)}}\partial_{x_{n+2}} \rangle\\
 \R^n\\
\langle (-1)^{\sigma_{1,-}^{(1,0)}}i(z-1) \partial_{x_{n+2}} \rangle
\end{bmatrix}
\to 
\begin{bmatrix}
 0 \\
0
\end{bmatrix}
\to 
0 \to 0,
\end{align*} 
which by assumption induces the canonical orientation \eqref{eq:orientmu0} on $\det \dbar_p$, times $(-1)^n$. 

It remains to understand the map $\alpha$. Clearly, $\alpha|_{\R^n}$ is given by the identity, and by the definition of $v_1$ and $v_2$ we have that $\alpha$ maps $v_1 \partial_{x_{n+2}}$ to a positive multiple of $\partial_{x_{n+2}}$ and  maps $v_2 \partial_{x_{n+2}}$ to a positive multiple of $i(z-1) \partial_{x_{n+2}}$. 

Since moving the $(-1)^{\sigma_{1,+}^{(1,0)}}\partial_{x_{n+2}}$--term over $\R^n$ in the second non-trivial column costs $(-1)^n$, it follows that this sequence induces the orientation 
\begin{equation*}
 (-1)^{\sigma_{1,+}^{(1,0)} + \sigma_{1,-}^{(1,0)} +n} \partial_{x_1} \wedge \dotsm \wedge \partial_{x_n} \wedge v_1\partial_{x_{n+2}} \wedge v_2\partial_{x_{n_2}}
\end{equation*}
on $\det \dbar_p$. Hence it follows that $\mu(0) \equiv \sigma_{1,+}^{(1,0)}+ \sigma_{1,-}^{(1,0)}, \pmod{2} $.
\end{proof}

 Now we turn to the capping operators of a special puncture $p$ of even Maslov index. Recall that the corresponding $\dbar$--problems are isomorphisms. Let $\sigma_{0,\pm}^{(0,0)}, \sigma_{0,\pm}^{(1,1)} \in \{0,1\}$ so that the orientation of the determinant line of the positive capping operator is given by
 \begin{equation*}
  \det \dbar_{p,+} = (-1)^{\sigma_{0,+}^{(0,0)}} \mathbbm{1}  \qquad (  \det \dbar_{p,+} = (-1)^{\sigma_{0,+}^{(1,1)}} \mathbbm{1})
 \end{equation*}
 when $p$ is of type $(0,0)$ ($(1,1)$), 
 and the orientation of the determinant line of the negative capping operator is given by
 \begin{equation*}
  \det \dbar_{p,-} = (-1)^{\sigma_{0,-}^{(0,0)}} \mathbbm{1} \qquad (  \det \dbar_{p,-} = (-1)^{\sigma_{0,-}^{(1,1)}} \mathbbm{1}) 
 \end{equation*}
   when $p$ is of type $(0,0)$ ($(1,1)$).
 Recall that the two former are induced by our choice that $\sigma_s^{(1,1)}(l) = \sigma_s^{(0,0)}(l) = 0$, and that the latter two then get their orientation from the gluing sequence \eqref{eq:capseqspec} together with the canonical orientation of $\det \dbar_p$.

\subsubsection{Choice of orientation of $\C$}\label{sec:choiceofC}
Now it is time to fix an orientation of $\C$. Recall that this is part of our initial choices, as discussed in Section \ref{sec:orientconv}. First we notice the following.

\begin{lma}\label{lma:orientC}
 The canonical orientation of Maslov problems of Maslov index $4j+2$ changes if we change the orientation of $\C$, while the canonical orientation of Maslov problems of Maslov index $4j$ remains unchanged.
\end{lma}
\begin{proof}
 This follows from the proof of [\cite{fooo}, Proposition 8.1.4] together with the discussion in [\cite{orientbok}, Section 4.5.6].
\end{proof}

Consider the  $n+2$--dimensional Maslov problem with boundary conditions 
\begin{equation}\label{eq:bcevencap}
 \diag(\underbrace{1,\dotsc,1}_{n}, e^{-i\pi s}, e^{-i\pi s}) \hat * R_\pi(n+1,n+2).
\end{equation}
This problem has Maslov index $-2$, is surjective and has kernel spanned by $\partial_{x_1}, \dotsc, \partial_{x_n}$. 
By Lemma \ref{lma:orientC} we can choose an orientation of $\C$ so that this problem gets canonical orientation represented by $\partial_{x_1} \wedge \dotsm \wedge \partial_{x_n}$. 
{\bf From now on, assume that we have fixed this orientation on $\C$.}

Now consider the 1-dimensional Maslov problem    $\dbar_{1,-2}$ with boundary conditions given by $\diag(e^{-2\pi i s})$. This is injective with 1-dimensional cokernel. We notice the following.

\begin{lma}\label{lma:maslovnegturn}
 There is a canonical identification 
 \begin{equation}\label{eq:cokerident}
  \gamma: \Coker(\dbar_{1,-2}) \simeq \R.
 \end{equation} 
\end{lma}
\begin{proof}
 Consider the 1-dimensional $\dbar$-problems $\dbar_{-1,\pm}$ defined on the once punctured disk with a small negative weight at the puncture and with trivialized boundary conditions given by $\diag(e^{-i\pi s})$ and $\diag(-e^{-i\pi s})$, respectively. These problems give isomorphisms, and if we glue them to each other we obtain $\dbar_{1,-2}$ together with the gluing sequence
  \begin{align*}
 0 \to  0 \to
\begin{bmatrix}
0 \\
0
\end{bmatrix}
\to 
\begin{bmatrix}
 0 \\ 
 \R \\
0
\end{bmatrix} 
 \xrightarrow{\gamma} 
\Coker \dbar_{1,-2} \to 0,
\end{align*}
where $\gamma$ is given as in Remark \ref{rmk:maps}.
\end{proof}

Using this lemma, we will study the two $n+2$--dimensional Maslov problems on the closed disk with boundary conditions given by 
\begin{equation*}
 I_0 = \diag(\underbrace{1,\dotsc,1}_{n+1}, e^{-i\pi s}) * \diag(\underbrace{1,\dotsc,1}_{n+1}, -e^{i\pi s})
\end{equation*}
 and
\begin{equation*}
 J_1 = \diag(\underbrace{1,\dotsc,1}_{n+1}, e^{-2\pi i s}),
\end{equation*}
respectively.
With our choice of orientation of $\C$, we get $\nu, \tilde \nu \in \{0,1\}$ so that the canonical orientation of $\det \dbar_{I_0}$ is given by 
\begin{equation*}
 (-1)^\nu \partial_{x_1} \wedge \dotsm \wedge \partial_{x_{n+1}} \wedge v\partial_{x_{n+2}} 
\end{equation*}
and the canonical orientation of $\det \dbar_{J_1}$ is given by 
\begin{equation*}
 (-1)^{\tilde \nu} \partial_{x_1} \wedge \dotsm \wedge \partial_{x_{n+1}} \otimes \partial_{x_{n+2}}. 
\end{equation*}
Here the cokernel elements are identified with constants via the map $\gamma$ from \eqref{eq:cokerident}, and $ v = v(z)$ is the solution of norm 1 that satisfies $\alpha(v) = \kappa i (z-1)$ for some $\kappa >0$, where $\alpha $ is the gluing map 
\begin{equation*}
\spn( v) \xrightarrow{\alpha} \spn(i(z-1)) \oplus 0 
\end{equation*}
induced by the gluing of two once-punctured disks, where the first one has trivialized Lagrangian boundary condition given by $e^{-i \pi s}$ and has a small negative weight at the puncture, and the second disk has trivialized Lagrangian boundary condition given by $-e^{i \pi s}$ and has a small positive weight at the puncture.

\begin{lma}\label{lma:IJsign}
Let 
\begin{equation*}
 \sigma(k) = 
 \begin{cases}
  k / 2, & k \text{ even,}\\
  \frac{k-1}{2} + \tilde v, & k \text{ odd}.
 \end{cases}
\end{equation*}
Then the $(n+2)$--dimensional Maslov problem with boundary conditions given by 
 \begin{equation*}
  J_k = \diag(\underbrace{1,\dotsc,1}_{n+2-k}, \underbrace{e^{-2\pi i s}, \dotsc, e^{-2\pi i s}}_{k})
 \end{equation*}
has canonical orientation 
\begin{equation*}
 (-1)^{\sigma(k)} \partial_{x_1} \wedge \dotsm \wedge \partial_{x_{n-k+2}} \otimes \partial_{x_{n-k+3}} \wedge \dotsm \wedge \partial_{x_{n+2}}, 
\end{equation*}
and the  $(n+2)$--dimensional Maslov problem with boundary conditions given by 
 \begin{equation*}
  I_k = \diag(\underbrace{1,\dotsc,1}_{n+1-k}, \underbrace{e^{-2\pi i s}, \dotsc, e^{-2\pi i s}}_{k}, e^{-i\pi s}) * \diag(\underbrace{1,\dotsc,1}_{n+1}, -e^{i\pi s})
 \end{equation*}
has canonical orientation 
\begin{equation*}
 (-1)^{\sigma(k) + \nu + k} \partial_{x_1} \wedge \dotsm \wedge \partial_{x_{n-k+1}} \wedge v \partial_{x_{n+2}} \otimes \partial_{x_{n-k+2}} \wedge \dotsm \wedge \partial_{x_{n+1}}.
\end{equation*}
\end{lma}
\begin{proof}
 We start with proving the first assertion. It clearly holds for $k=0$ and $k=1$. Assume, by induction, that it holds for $k-1$ and glue the Maslov problem with boundary conditions
  \begin{equation}\label{eq:bcweglue}
  C_{k-1} = \diag(\underbrace{1,\dotsc,1}_{n+2-k}, \underbrace{e^{-2\pi i s}, \dotsc, e^{-2\pi i s}}_{k-1}, 1) 
 \end{equation}
 to the Maslov problem $\partial_{J_1}$. The former problem has by assumption the canonical orientation 
 \begin{equation*}
 (-1)^{\sigma(k-1) + k-1 } \partial_{x_1} \wedge \dotsm \wedge \partial_{x_{n-k+2}} \wedge \partial_{x_{n+2}}\otimes \partial_{x_{n-k+3}} \wedge \dotsm \wedge \partial_{x_{n+1}}. 
\end{equation*}

The gluing above induces the exact sequence 
  \begin{align}
 0 \to 
 \begin{bmatrix}
  \langle\partial_{x_1}\rangle \\
  \vdots \\
   \langle\partial_{x_{n-k+2}}\rangle
 \end{bmatrix}
 \xrightarrow{\alpha} 
 \begin{bmatrix}
 \begin{bmatrix}
  \langle \partial_{x_1}\rangle \\
  \vdots \\
   \langle \partial_{x_{n-k+2}}\rangle\\
      \langle \partial_{x_{n+2}}\rangle
 \end{bmatrix}\\
\R^{n+1} 
\end{bmatrix}
\xrightarrow{\beta}  
\begin{bmatrix}
\begin{bmatrix}
 \langle \partial_{x_{n-k+3}}\rangle \\
 \vdots \\
 \langle \partial_{x_{n+1}}\rangle
 \end{bmatrix} \\
\R^{n+2} \\
\begin{bmatrix}
\langle \partial_{x_{n+2}}\rangle
\end{bmatrix}
\end{bmatrix} 
\xrightarrow{\gamma}  
\begin{bmatrix}
 \langle \partial_{x_{n-k+3}}\rangle \\
 \vdots \\
\langle\partial_{x_{n+2}}\rangle
 \end{bmatrix} 
 \to 0.
\end{align}
Moving $\langle\partial_{x_{n+2}} \rangle$ in the second non-trivial column to the bottom costs $(-1)^{n+1}$, and moving the bottommost $\langle\partial_{x_{n+2}} \rangle$ in the third column over $\R^{n+2}$ costs $(-1)^n$. Since $\beta(\partial_{x_1}, \dotsc, \partial_{x_{n+2}}) = (\partial_{x_1}, \dotsc, \partial_{x_{n+1}}, -\partial_{x_{n+2}} )$ we see that this gluing sequence induces the orientation 
\begin{equation*}
 (-1)^{\sigma(k-1) + k - 1 + \tilde \nu} \partial_{x_1} \wedge \dotsm \wedge \partial_{x_{n-k+2}} \otimes \partial_{x_{n-k+3}} \wedge \dotsm \wedge \partial_{x_{n+2}}
\end{equation*}
on $\det \dbar_{J_k}$, and since this is induced by the canonical orientation of the two glued problems, this gives the canonical orientation. Now, if $k$ is odd we have by assumption that $\sigma(k-1) = \frac{k-1}{2}$ and hence 
\begin{equation*}
 \sigma(k-1) + k - 1 + \tilde \nu \equiv \frac{k-1}{2} + \tilde \nu, \pmod{2},
\end{equation*}
and if $k$ is even we have by assumption that $\sigma(k-1) = \frac{k-2}{2} + \tilde \nu$ and hence 
\begin{equation*}
 \sigma(k-1) + k - 1 + \tilde \nu \equiv \frac{k-2}{2} + 1 \equiv \frac{k}{2}, \pmod{2}.
\end{equation*}
Hence the first assertion holds.

To prove the second assertion, we glue the Maslov problem with boundary conditions $C_k$ to the $\dbar_{I_0}$-problem, which gives a Maslov problem with boundary condition homotopic to $I_k$. We get a gluing sequence similar to the one above, and from this we see that the glued problem gets canonical orientation
\begin{equation*}
 (-1)^{\sigma(k) + k + \nu} \partial_{x_1} \wedge \dotsm \wedge \partial_{x_{n-k+1}} \wedge v  \partial_{x_{n+2}} \otimes \partial_{x_{n-k+2}} \wedge \dotsm \wedge \partial_{x_{n+1}}.
\end{equation*}
Hence the second assertion also holds. 
\end{proof}

 We will need the following.
 \begin{lma}\label{lma:sigmarelations}
  Let $n,k \in \Z_{\geq 0}$ and let 
  \begin{equation*}
   \sigma(n,k) = \sigma(n+1) + \sigma(k) + \sigma(n-k+1).
  \end{equation*}
  Then $\sigma(n,k) \equiv n\cdot k, \pmod{2}$. 
 \end{lma}
\begin{proof}
 Follows by straightforward calculations. Namely,
 \begin{align*}
  \sigma(2m,2l) \equiv & \frac{2m}{2} + \tilde v + \frac{2l}{2} + \frac{2m-2l}{2} + \tilde \nu \equiv 2m \\
  \sigma(2m,2l+1) \equiv & \frac{2m}{2} + \tilde v + \frac{2l}{2} + \tilde v + \frac{2m-2l-1+1}{2} \equiv 2m \\  
  \sigma(2m + 1,2l) \equiv & \frac{2m+2}{2} + \frac{2l}{2} + \frac{2m+1-2l+1}{2}\equiv 2m +2 \\
  \sigma(2m + 1,2l+1) \equiv & \frac{2m+2}{2} + \frac{2l}{2} + \tilde \nu + \frac{2m+1-2l-1}{2}+ \tilde \nu\equiv 2m +1, \\
 \end{align*}
 everything modulo 2. 
\end{proof}

\subsection{Explicit description of the orientation of the capping operators}
Now we are ready to give explicit values for the chosen orientations of the capping operators, using the results from Section \ref{sec:der}.

First we notice the following.

\begin{lma}\label{lma:sigma0}
 With the chosen orientation of $\C$ we get that 
 \begin{align}\label{eq:mu0}
 \sigma_{0,+}^{(0,0)} \equiv &\sigma_{0,-}^{(0,0)}, \\
   \sigma_{0,+}^{(1,1)} \equiv &\sigma_{0,-}^{(1,1)}, \pmod{2}.
 \end{align} 
\end{lma}
\begin{proof}
 Consider the exact sequence \eqref{eq:capseqspec} in the case when $p$ is a special puncture of type $(0,0)$. We get that the trivialized boundary condition of the glued $\dbar_p$--problem is homotopic to the  $n+2$--dimensional Maslov problem with boundary condition \eqref{eq:bcevencap}.  
By our choice of orientation of $\C$ it then follows that $\det \dbar_p$ gets canonical orientation $\partial_{x_1} \wedge \dotsm \wedge \partial_{x_n}$. Moreover, since \eqref{eq:capseqspec} in this case is given by 
  \begin{align}\label{eq:glglgl1}
 0 \to \R^n \xrightarrow{\alpha} 
\begin{bmatrix}
 (-1)^{\sigma_{0,+}^{(0,0)}}\mathbbm{1} \\
 \R^n\\
(-1)^{\sigma_{0,-}^{(0,0)}}\mathbbm{1}
\end{bmatrix}
\to 
\begin{bmatrix}
 0 \\
0
\end{bmatrix}
\to 
0 \to 0,
\end{align}
where $\alpha$ is given by the identity, it follows that $\sigma_{0,+}^{(0,0)} \equiv \sigma_{0,-}^{(0,0)} $ modulo 2. 

 In the case when $p$ is of type $(1,1)$, we get the same gluing sequence (with $\sigma_{0,\pm}^{(0,0)}$ replaced by $\sigma_{0,\pm}^{(1,1)}$), but now the trivialized boundary condition for the $\dbar_p$--problem is given by 
 \begin{equation*}
  \diag(\underbrace{1,\dotsc,1}_{n}, e^{-i\pi s}, -e^{-i\pi s}) * R_{-\pi}(n+1,n+2),
 \end{equation*}
 which differs from the one above by the coordinate change $\partial_{x_{n+2}} \leftrightarrow -\partial_{x_{n+2}}$, and hence this problem has canonical orientation  given by $-\partial_{x_1} \wedge \dotsm \wedge \partial_{x_n}$. Thus we get that $\sigma_{0,+}^{(1,1)} \equiv \sigma_{0,-}^{(1,1)}, \pmod{2}$.
\end{proof}

\begin{lma}\label{lma:sigmay0}
Assume that $n>1$. Then we have that 
 \begin{equation}
  \sigma_{0,+}^{(0,0)} \equiv \sigma_{Y_0}^{(0,0)} (0,0) \equiv \sigma_{0,+}^{(1,1)} + 1 \pmod{2}.
 \end{equation} 
\end{lma}
\begin{proof}
 Consider the capping sequence for the $Y_0$-piece with all special punctures being of even Maslov index:  
 \begin{align*}
 0 \to \R^n \xrightarrow{\alpha} 
\begin{bmatrix}
(-1)^{\sigma_{0,-}^{(k,k)} }\mathbbm{1} \\
(-1)^{\sigma_{0,-}^{(k,k)} }\mathbbm{1} \\
 (-1)^{\sigma_{0,+}^{(k,k)} }\mathbbm{1} \\
 (-1)^{\sigma_{Y_0}^{(k,k)} (0,0)}\R^n
\end{bmatrix}
\to 
\begin{bmatrix}
 0 \\
 0 \\
 0 \\
0
\end{bmatrix}
\to 
0 \to 0,
\end{align*}
where $k=0$ or $k=1$, depending on the type of vertex, and $\alpha$ is given by the identity. The induced $\hat \Gamma_0$-problems have boundary conditions given by 
\begin{align*}
 &\diag(\underbrace{1,\dotsc,1}_{n}, e^{-i\pi s}, e^{-i\pi s}) \hat * R_\pi(n+1,n+2), \quad k = 0, \\
 &\diag(\underbrace{1,\dotsc,1}_{n}, e^{-i\pi s}, -e^{-i\pi s}) \hat * R_{-\pi}(n+1,n+2), \quad k = 1. 
\end{align*}
Since the boundary condition for $k=1$ differs from the boundary condition for $k=0$ by the changes of coordinates $\partial_{x_{n+2}} \leftrightarrow -\partial_{x_{n+2}}$, it follows that $\det \dbar_{\hat \Gamma_0}$ has canonical orientation $(-1)^{k}\partial_{x_1} \wedge \dotsm \wedge \partial_{x_n}$. On the other hand, these problems get induced orientation 
\begin{equation*}
 (-1)^{\sigma_{Y_0}^{(k,k)}(0,0) + \sigma_{0,-}^{(k,k)}+ \sigma_{0,-}^{(k,k)} + \sigma_{0,+}^{(k,k)}} \partial_{x_1} \wedge \dotsm \wedge \partial_{x_n}
\end{equation*}
from the gluing sequence. Hence $\sigma_{Y_0}^{(k,k)}(0,0) + \sigma_{0,+}^{(k,k)} \equiv k$ modulo 2. Now the result follows from Section \ref{sec:results}, where we have $\sigma_{Y_0}^{(1, 1)}(0,0)  \equiv \sigma_{Y_0}^{(0, 0)}(0,0)$.
\end{proof}

We also have the following.
\begin{lma}\label{lma:sigmay0cor}
Assume that $n>1$. Then we have that 
 \begin{equation}
  \sigma_{0,-}^{(0,0)} \equiv \sigma_{0,-}^{(1,1)} \equiv \sigma_{Y_1}^{(1,0)} (0,0) + \sigma_{{\swi}}^{(1,0)}(l) \pmod{2}.
 \end{equation} 
\end{lma}
\begin{proof}
Consider the trivialized Lagrangian boundary condition $A$ for the fully capped $\sigma_{Y_1}^{(1,0)}(0,0)$-- problem. Notice that it coincides with the trivialized Lagrangian boundary condition for the fully capped  $ \sigma_{\swi}^{(1,0)}(l)$--problem. Then notice that the orientation induced on $\det \dbar_A$ by the two different capping sequences corresponding to these problems differ by a factor $(-1)^{\sigma_{0,-}^{(0,0)}+ \sigma_{Y_1}^{(1,0)} (0,0) + \sigma_{{\swi}}^{(1,0)}(l)}$. 

On the other hand, $A$ also coincides with the trivialized Lagrangian boundary condition for the fully capped  $ \sigma_{\swi}^{(1,0)}(u)$--problem, but now the orientation induced on $\det \dbar_A$ by the two different capping sequences corresponding to these problems differ by a factor $(-1)^{\sigma_{0,-}^{(1,1)}+ \sigma_{Y_1}^{(1,0)} (0,0) + \sigma_{{\swi}}^{(1,0)}(u)}$. Since we from Section \ref{sec:results} have that $\sigma_{{\swi}}^{(1,0)}(l) \equiv \sigma_{{\swi}}^{(1,0)}(u)$ modulo 2, the result follows.
\end{proof}

\begin{cor}\label{lma:signposcapev}
Assume that $n>1$. Then we have that  
 \begin{align*}
  \sigma_{0,+}^{(0,0)} &\equiv n+1, & \sigma_{0,-}^{(0,0)} &\equiv n+1, \\
  \sigma_{0,+}^{(1,1)} &\equiv n,  & \sigma_{0,-}^{(1,1)} &\equiv n+1, \pmod{2}.
 \end{align*} 
\end{cor}
\begin{proof}
 This follows from Section \ref{sec:results}, Lemma \ref{lma:sigma0}, Lemma \ref{lma:sigmay0} and Lemma \ref{lma:sigmay0cor}.
\end{proof}

To derive similar results for the capping operators at special punctures of odd Maslov index, we need the following.

\begin{lma}\label{lma:rpiI}
 The loop 
 \begin{equation*}
  R_{\pi}(n,n+1) \hat * R_{\pi}(n,n+2) \hat * R_{\pi}(n+1,n+2) 
 \end{equation*}
 represents the identity of $\pi_1(SO(n+2))$.
\end{lma}
\begin{proof}
 This can be seen by considering the corresponding loop in $SO(3)$ in terms of the \emph{plate trick}, see \cite{plate_trick}.
\end{proof}

\begin{lma}\label{lma:oddend}
 With our orientation choices we have that 
 \begin{equation*}
  \sigma_{1,-}^{(1,0)} \equiv 1 \pmod{2}.
 \end{equation*}
\end{lma}
\begin{proof}
 Consider the fully capped problem for an end-piece $\Gamma$ with positive puncture of type $(1,0)$. It has trivialized Lagrangian boundary conditions given by  \eqref{eq:trivend}, and after the change of coordinates $\partial_{x_{n}} \leftrightarrow \partial_{x_{n+2}}$ this turns into 
   \begin{equation}\label{eq:trivend2}
    \diag(\underbrace{1,\dotsc,1}_{n}, e^{-i\pi s},1) \hat * R_{\pi}(n,n+1) \hat * R_{\pi}(n,n+2) * \diag(\underbrace{1,\dotsc,1}_{n+1}, -e^{i\pi s}).
  \end{equation}

  Using Lemma \ref{lma:rpiI} we get that 
  \begin{equation*}
R_{\pi}(n,n+1) \hat * R_{\pi}(n,n+2) \hat *R_{\pi}(n+1,n+2)
  \end{equation*}
represents a trivial homotopy class, which in turn implies 
that \eqref{eq:trivend2} represents the same homotopy class as \eqref{eq:bc-+} does. By our assumptions it then follows that the Maslov problem with boundary condition given by \eqref{eq:trivend2} has canonical orientation 
\begin{equation*}
 (-1)^{\mu(0)} \partial_{x_1} \wedge \dotsm \wedge \partial_{x_n} \wedge v_1\partial_{x_{n+2}} \wedge v_2\partial_{x_{n+2}},
\end{equation*} 
and hence our fully capped problem $\dbar_{\hat \Gamma}$ has canonical orientation given by 
\begin{equation*}
 (-1)^{\mu(0)+1} \partial_{x_1} \wedge \dotsm \wedge \partial_{x_{n-1}} \wedge v_1\partial_{x_{n}} \wedge v_2\partial_{x_{n}} \wedge \partial_{x_{n+2}}.
\end{equation*} 

On the other hand, the corresponding capping sequence is given by 
  \begin{align*}
 0 \to 
 \begin{bmatrix}
  \langle\partial_{x_1}\rangle \\
  \vdots \\
   \langle\partial_{x_{n-1}}\rangle \\
   \langle v_1\partial_{x_n}\rangle \\
   \langle v_2\partial_{x_n}\rangle \\
   \langle \partial_{x_{n+2}}\rangle
 \end{bmatrix}
 \xrightarrow{\alpha} 
 \begin{bmatrix}
  \langle i(z-1)\partial_{x_{n}}\rangle \\
    \langle (-1)^{\sigma_{1,+}^{(1,0)}}\partial_{x_{n+2}}\rangle\\
 \begin{bmatrix}
  \langle \partial_{x_1}\rangle \\
  \vdots \\
   \langle \partial_{x_{n}}\rangle
 \end{bmatrix}
\end{bmatrix}
 \to 0,
\end{align*}
which induces the orientation 
\begin{equation*}
 (-1)^{\sigma_{1,+}^{(1,0)}} \partial_{x_1} \wedge \dotsm \wedge \partial_{x_{n-1}} \wedge v_1\partial_{x_{n}} \wedge v_2\partial_{x_{n}} \wedge \partial_{x_{n+2}}
\end{equation*} 
on $\det \dbar_{\hat \Gamma}$. By assumption, this should give the canonical orientation. Using this together with Lemma \ref{lma:mu0} it follows that 
\begin{equation*}
 \sigma_{1,-}^{(1,0)} \equiv 1 \pmod{2}.
\end{equation*}
\end{proof}

\begin{lma}\label{lma:oddend2}
  Assume that $n>1$. Then we have that 
 \begin{equation*}
  \sigma_{1,-}^{(1,0)} \equiv  \sigma_{1,-}^{(0,1)} \pmod{2}.
 \end{equation*}
\end{lma}
\begin{proof}
Consider the fully capped problems corresponding to the $Y_0$--vertices of type   $\sigma_{Y_0}^{(1,0)}(0,1)$ and $\sigma_{Y_0}^{(0,1)}(1,0)$, respectively. The first has trivialized Lagrangian boundary condition given by \eqref{eq:bc-+}, and the second has trivialized Lagrangian boundary condition given by  
\begin{align*}
 &\left(R_{\pi}(n,n+2)\cdot\diag(\underbrace{1,\dotsc,1}_{n-+1},-1)  \right) *\diag(\underbrace{1,\dotsc,1}_{n-1},-1, e^{-i\pi s}, 1)\hat * R_{-\pi}(n+1,n+2) \hat* \\
 & \qquad R_{\pi}(n,n+2) \hat* R_{-\pi}(n,n+2)*\diag(\underbrace{1,\dotsc,1}_{n-1},-1, 1, -e^{i\pi s})\hat * R_{-\pi}(n,n+2),
\end{align*}
which is homotopic to 
\begin{align}\label{eq:oddcapp01homo}
 &\diag(\underbrace{1,\dotsc,1}_{n-1},-1, e^{-i\pi s}, 1) \hat * R_{-\pi}(n+1,n+2)*\diag(\underbrace{1,\dotsc,1}_{n-1},-1, 1, -e^{i\pi s}).
\end{align}
But the latter is nothing but the boundary condition given by \eqref{eq:bc-+} after applying the change of coordinates $\partial_{x_n} \leftrightarrow -\partial_{x_n}$. Hence, the fully capped $\sigma_{Y_0}^{(1,0)}(0,1)$--problem has canonical orientation given by 
\begin{equation*}
 (-1)^{\mu(0)} \partial_{x_1} \wedge \dotsm \wedge \partial_{x_n} \wedge v_1\partial_{x_{n+2}} \wedge v_2\partial_{x_{n+2}},
\end{equation*} 
and the the fully capped $\sigma_{Y_0}^{(0,1)}(1,0)$--problem has canonical orientation given by 
\begin{equation*}
 (-1)^{\mu(0) +1 } \partial_{x_1} \wedge \dotsm \wedge \partial_{x_n} \wedge v_1\partial_{x_{n+2}} \wedge v_2\partial_{x_{n+2}}.
\end{equation*} 

Now we consider the capping sequences for these two problems, which by assumption should induce the canonical orientation of the fully capped problems. In the $\sigma_{Y_0}^{(1,0)}(0,1)$--case this sequence is given by 
   \begin{align*}
 0 \to 
 \begin{bmatrix}
  \langle\partial_{x_1}\rangle \\
  \vdots \\
   \langle\partial_{x_{n}}\rangle \\
   \langle v_1\partial_{x_{n+2}}\rangle \\
   \langle v_2\partial_{x_{n+2}}\rangle 
 \end{bmatrix}
 \xrightarrow{\alpha} 
  \begin{bmatrix}
   (-1)^{\sigma_{0,-}^{(1,1)}}\mathbbm{1} \\
  \langle (-1)^{\sigma_{1,-}^{(1,0)}}i(z-1)\partial_{x_{n+2}}\rangle \\
  \langle (-1)^{\sigma_{1,+}^{(1,0)}}\partial_{x_{n+2}}\rangle \\
\R^n
 \end{bmatrix}
  \to 0,
 \end{align*}
and in the $\sigma_{Y_0}^{(0,1)}(1,0)$--case this sequence is given by 
  \begin{align*}
 0 \to 
 \begin{bmatrix}
  \langle\partial_{x_1}\rangle \\
  \vdots \\
   \langle\partial_{x_{n}}\rangle \\
   \langle v_1\partial_{x_{n+2}}\rangle \\
   \langle v_2\partial_{x_{n+2}}\rangle 
 \end{bmatrix}
 \xrightarrow{\alpha} 
 \begin{bmatrix}
  \langle (-1)^{\sigma_{1,-}^{(0,1)}}i(z-1)\partial_{x_{n+2}}\rangle \\
  (-1)^{\sigma_{0,-}^{(1,1)}}\mathbbm{1} \\
  \langle (-1)^{\sigma_{1,+}^{(0,1)}}\partial_{x_{n+2}}\rangle \\
\R^n
\end{bmatrix}
 \to 0.
\end{align*}
We see that these sequences induce the orientation 
\begin{equation*}
 (-1)^{\sigma_{1,-}^{(1,0)} + \sigma_{0,-}^{(1,1)} + \sigma_{1,+}^{(1,0)} +  \sigma_{Y_0}^{(1,0)}(0,1)+1} \partial_{x_1} \wedge \dotsm \wedge \partial_{x_n} \wedge v_1\partial_{x_{n+2}} \wedge v_2\partial_{x_{n+2}}
\end{equation*} 
on the $\sigma_{Y_0}^{(1,0)}(0,1)$--problem and 
\begin{equation*}
 (-1)^{\sigma_{1,-}^{(0,1)} + \sigma_{0,-}^{(1,1)} + \sigma_{1,+}^{(0,1)} +  \sigma_{Y_0}^{(0,1)}(1,0)+1} \partial_{x_1} \wedge \dotsm \wedge \partial_{x_n} \wedge v_1\partial_{x_{n+2}} \wedge v_2\partial_{x_{n+2}}
\end{equation*} 
on the $\sigma_{Y_0}^{(0,1)}(1,0)$--problem. 
Hence we get that 
\begin{equation*}
 \sigma_{1,-}^{(0,1)} + \sigma_{1,+}^{(0,1)} +  \sigma_{Y_0}^{(0,1)}(1,0) \equiv \sigma_{1,-}^{(1,0)}  + \sigma_{1,+}^{(1,0)} +  \sigma_{Y_0}^{(1,0)}(0,1) + 1 \pmod{2}.
\end{equation*}
From Section \ref{sec:results} we have that $\sigma_{Y_0}^{(0,1)}(1,0) \equiv   \sigma_{Y_0}^{(1,0)}(0,1) \pmod{2}$, and by Lemma \ref{lma:oddsigmarel} we have that $\sigma_{1,+}^{(0,1)}  \equiv \sigma_{1,+}^{(1,0)} +1 \pmod{2}$. Hence the result follows.
\end{proof}

\begin{lma}\label{lma:megalemma21}
 Assume that $n>1$. Then we have that 
 \begin{equation*}
  \sigma_{1,+}^{(1,0)} \equiv n + 1 + \sigma_{0,+}^{(0,0)}.
 \end{equation*}
\end{lma}
\begin{proof}
 Consider the $\dbar$--problem corresponding to a switch-piece $\Gamma$ of type $\sigma^{(0,0)}_{\swi}(u)$. By our results from Section \ref{sec:results} we have that $\sigma_{\swi}^{(0,0)}(u) = 1$. The trivialized Lagrangian boundary condition for the fully capped problem is given by 
  \begin{align}\label{eq:switch1lbc}
   &\diag(\underbrace{1,\dotsc,1}_{n},e^{-i\pi s},e^{-i\pi s}) \hat * R_{\pi}(n+1,n+2) * \\ \nonumber
    &{}\qquad   \diag(\underbrace{1,\dotsc,1}_{n-1},e^{-i\pi s},1,1) \hat * R_{-\pi}(n,n+2)*\diag(\underbrace{1,\dotsc,1}_{n},1,-e^{i\pi s}). 
  \end{align} 
  Notice that this boundary condition can be given by gluing the trivialized Lagrangian boundary condition \eqref{eq:bcevencap} to the trivialized Lagrangian boundary condition
  \begin{equation*}
    \diag(\underbrace{1,\dotsc,1}_{n-1},e^{-i\pi s},1,1) \hat* R_{-\pi}(n,n+2)*\diag(\underbrace{1,\dotsc,1}_{n},1,-e^{i\pi s}).
  \end{equation*}
After a change of coordinates $\partial_{x_n} \leftrightarrow \partial_{x_{n+1}}$, the latter boundary condition turns into the one given by \eqref{eq:bc-+}, and hence the corresponding Maslov problem has canonical orientation given by 
\begin{equation*}
 (-1)^{\mu(0)+1} \partial_{x_1} \wedge \dotsm \wedge \partial_{x_{n-1}} \wedge \partial_{x_{n+1}}\wedge v_1\partial_{x_{n+2}} \wedge v_2\partial_{x_{n+2}}.
\end{equation*} 
Also recall that the Maslov problem with boundary conditions \eqref{eq:bcevencap} has canonical orientation 
\begin{equation*}
 \partial_{x_1} \wedge \dotsm \wedge \partial_{x_{n}}.
\end{equation*} 
In addition, gluing these two Maslov problems together induces the gluing sequence 
  \begin{align}
 0 \to 
 \begin{bmatrix}
  \langle\partial_{x_1}\rangle \\
  \vdots \\
   \langle\partial_{x_{n-1}}\rangle \\
   \langle v\partial_{x_{n+2}}\rangle \\
 \end{bmatrix}
 \xrightarrow{\alpha} 
 \begin{bmatrix}
 \begin{bmatrix}
  \langle \partial_{x_1}\rangle \\
  \vdots \\
   \langle \partial_{x_{n}}\rangle
 \end{bmatrix}\\
 \begin{bmatrix}
  \langle \partial_{x_1}\rangle \\
  \vdots \\
   \langle \partial_{x_{n-1}}\rangle \\
   \langle \partial_{x_{n+1}}\rangle \\
    \langle v_1\partial_{x_{n+2}}\rangle \\
    \langle v_2\partial_{x_{n+2}}\rangle 
 \end{bmatrix}
\end{bmatrix}
\xrightarrow{\beta}  
\R^{n+2} 
 \to 0.
\end{align}
Since $v_1$ by assumption is close to a cutoff constant function, and $v_2$ is close to the cutoff of the  solution $i(z-1)$, which $v$ is mapped to under $\alpha$ (possibly times a positive constant), we see that we might assume that 
\begin{equation*}
 \alpha(v\partial_{x_{n+2}}) =  v_2\partial_{x_{n+2}}, \qquad \beta(v_1\partial_{x_{n+2}}) = \partial_{x_{n+2}}.
\end{equation*}
Hence this sequence induces the orientation 
\begin{equation*}
 (-1)^{(\mu(0)+1) + 1 + (n-1) + (n+2)} \partial_{x_1} \wedge \dotsm \wedge \partial_{x_{n-1}} \wedge v\partial_{x_{n+2}}
\end{equation*} 
on $\det \dbar_{\hat \Gamma}$. Here one 1 comes from $\beta(\partial_{x_n})=-\partial_{x_n}$, the $(n-1)$--summand comes from moving $\partial_{x_n}$ in the second non-trivial column over $\partial_{x_1}, \dotsc, \partial_{x_{n-1}}$, and the $(n+2)$--summand comes from moving $v_2\partial_{x_{n+2}}$ over $\partial_{x_1}, \dotsc, \partial_{x_n},\partial_{x_{n+1}},v_2\partial_{x_{n+2}}$. It follows that the canonical orientation of $\det \dbar_{\hat \Gamma}$ is given by 
\begin{equation*}
 (-1)^{\mu(0)+1} \partial_{x_1} \wedge \dotsm \wedge \partial_{x_{n-1}} \wedge v\partial_{x_{n+2}}.
\end{equation*} 

On the other hand, if we consider the capping sequence for $\Gamma$, which is given by 
  \begin{align}
 0 \to 
 \begin{bmatrix}
  \langle\partial_{x_1}\rangle \\
  \vdots \\
   \langle\partial_{x_{n-1}}\rangle \\
   \langle v\partial_{x_{n+2}}\rangle \\
 \end{bmatrix}
 \xrightarrow{\alpha} 
 \begin{bmatrix}
  \langle (-1)^{\sigma_{1,-}^{(1,0)}}i(z-1)\partial_{x_{n+2}}\rangle \\
 (-1)^{\sigma_{0,+}^{(0,0)}}\mathbbm{1} \\
 \begin{bmatrix}
  \langle \partial_{x_1}\rangle \\
  \vdots \\
   \langle \partial_{x_{n-1}}\rangle
 \end{bmatrix}
\end{bmatrix}
 \to 0,
\end{align}
we see that this sequence induces the orientation 
\begin{equation*}
 (-1)^{\sigma_{1,-}^{(1,0)}+\sigma_{0,+}^{(0,0)} + \sigma_{{\swi}}^{(0,0)}(u) + (n -1)} \partial_{x_1} \wedge \dotsm \wedge \partial_{x_{n-1}} \wedge v\partial_{x_{n+2}}
\end{equation*} 
on $\det \dbar_{\hat \Gamma}$. Since this should give the canonical orientation it follows that 
\begin{equation*}
 \mu(0) +1\equiv \sigma_{1,-}^{(1,0)}+\sigma_{0,+}^{(0,0)} + n -1 +1\pmod{2}.
\end{equation*}
Using Lemma \ref{lma:mu0} we get that 
\begin{equation*}
 \sigma_{1,+}^{(1,0)} \equiv \sigma_{0,+}^{(0,0)} + n  + 1 \pmod{2}.
\end{equation*}
\end{proof}

\begin{cor}\label{lma:signallspeccap}
Assume that $n>1$. Then we have that  
 \begin{align*}
  \sigma_{0,+}^{(0,0)} &\equiv n+1, & \sigma_{0,-}^{(0,0)} &\equiv n+1, \\
  \sigma_{0,+}^{(1,1)} &\equiv n,  & \sigma_{0,-}^{(1,1)} &\equiv n+1, \\
  \sigma_{1,+}^{(1,0)} & \equiv 0, & \sigma_{1,-}^{(1,0)} & \equiv 1, \\
    \sigma_{1,+}^{(0,1)} & \equiv 1, & \sigma_{1,-}^{(0,1)} & \equiv 1,\pmod{2}
 \end{align*} 
\end{cor}
\begin{proof}
 This follows from Corollary \ref{lma:signposcapev}, Lemma \ref{lma:oddsigmarel}, Lemma \ref{lma:oddend}, Lemma \ref{lma:oddend2} and  Lemma \ref{lma:megalemma21}.
 
\end{proof}

\subsection{Explicit values of remaining switch, $Y_1$ and $Y_0$--signs}
Now we can compute the values of the switch, $Y_1$ and $Y_0$--signs that we were not able to compute in Section \ref{sec:der}.

\begin{lma}\label{lma:signswitch}
 With our orientation choices we get that 
 \begin{equation*}
 \sigma_{Y_1}^{(1,0)}(0,0) \equiv \sigma_{0,-}^{(1,1)} +\sigma_{0,-}^{(0,0)}+ \sigma_{1,+}^{(1,0)}  + n+1.
 \end{equation*}
\end{lma}
\begin{proof}
 Consider the fully capped trivialized boundary conditions for a $\sigma_{Y_1}^{(1,0)}(0,0)$-problem, given by 
  \begin{align}\label{eq:switch2lbc}
   &\diag(\underbrace{1,\dotsc,1}_{n},e^{-i\pi s},1) \hat * R_{-\pi}(n+1,n+2) \hat * R_{-\pi}(n,n+2) *   \diag(\underbrace{1,\dotsc,1}_{n-1},-e^{-i\pi s},1,1). 
  \end{align}  
 After applying the change of coordinates $\partial_{x_n} \leftrightarrow \partial_{x_{n+2}}$, this boundary condition turns into
   \begin{align*}
    &\diag(\underbrace{1,\dotsc,1}_{n},e^{-i\pi s},1) \hat * R_{\pi}(n,n+1) \hat * R_{\pi}(n,n+2) * 
    \diag(\underbrace{1,\dotsc,1}_{n},1,-e^{-i\pi s}),
  \end{align*}  
 which is homotopic to 
    \begin{align*}
   & \diag(\underbrace{1,\dotsc,1}_{n},e^{-i\pi s},-e^{-i\pi s}) \hat * R_{\pi}(n,n+1) \hat * R_{\pi}(n,n+2). 
  \end{align*} 
  After applying the change of coordinates $\partial_{x_{n+2}} \leftrightarrow -\partial_{x_{n+2}}$ we get the boundary condition 
   \begin{align}\label{eq:switch5lbc}
   & \diag(\underbrace{1,\dotsc,1}_{n},e^{-i\pi s},e^{-i\pi s}) \hat * R_{\pi}(n,n+1) \hat * R_{-\pi}(n,n+2). 
  \end{align} 
Since $ R_{\pi}(n,n+1) \hat * R_{-\pi}(n,n+2) \hat *  R_{-\pi}(n+1,n+2)$ represents the trivial homotopy class in $\pi_1(SO(n+2))$ it follows that the Maslov problem with trivialized boundary conditions \eqref{eq:switch5lbc} has the same canonical orientation as the Maslov problem with the trivialized boundary condition \eqref{eq:bcevencap}.
  
  Summing up, we get that the canonical orientation of the $\dbar_{\hat \Gamma}$--problem is given by 
$(-1)^{1+1}\partial_{x_1} \wedge \dotsm \wedge \partial_{x_{n-1}} \wedge \partial_{x_{n+2}}$.

On the other hand, the capping sequence for $\Gamma$ is given by 
  \begin{align*}
 0 \to 
 \begin{bmatrix}
  \langle\partial_{x_1}\rangle \\
  \vdots \\
   \langle\partial_{x_{n-1}}\rangle \\
   \langle \partial_{x_{n+2}}\rangle
 \end{bmatrix}
 \xrightarrow{\alpha} 
 \begin{bmatrix}
  (-1)^{\sigma_{0,-}^{(1,1)}} \mathbbm{1} \\
  (-1)^{\sigma_{0,-}^{(0,0)}} \mathbbm{1} \\
    \langle (-1)^{\sigma_{1,+}^{(1,0)}} \partial_{x_{n+2}}\rangle \\
 \begin{bmatrix}
  \langle \partial_{x_1}\rangle \\
  \vdots \\
   \langle \partial_{x_{n-1}}\rangle
 \end{bmatrix}
\end{bmatrix}
 \to 0,
\end{align*}
  and thus induces the orientation 
  \begin{equation*}
   (-1)^{\sigma_{0,-}^{(1,1)} +\sigma_{0,-}^{(0,0)}+ \sigma_{1,+}^{(1,0)} + \sigma_{Y_1}^{(1,0)}(0,0) + n + 1} \partial_{x_1} \wedge \dotsm \wedge \partial_{x_{n-1}} \wedge \partial_{x_{n+2}}
  \end{equation*}
on $\det \dbar_{\hat \Gamma}$. Since this orientation should coincide with the canonical one we obtain 
   \begin{equation*}
   \sigma_{Y_1}^{(1,0)}(0,0) \equiv \sigma_{0,-}^{(1,1)} +\sigma_{0,-}^{(0,0)}+ \sigma_{1,+}^{(1,0)}  + n+1  \pmod{2}.
   \end{equation*}
\end{proof}

\begin{lma}\label{lma:remainingY1}
Assume that $n>1$. With our orientation choices we have that 
 \begin{align*}
  \sigma_{Y_1}^{(1,0)}(0,0) & \equiv  n+1, \\
   \sigma_{\swi}^{(1,0)}(l) & \equiv 0 \\
   \sigma_{Y_1}^{(0,1)}(0,0) & \equiv  n+1, \pmod{2}.
 \end{align*}
\end{lma}
\begin{proof}

The first statement follows from Lemma \ref{lma:signswitch} and Corollary \ref{lma:signallspeccap}.

 Similarly, the second statement follows from the first statement together with Lemma \ref{lma:sigmay0cor} and Corollary \ref{lma:signallspeccap}.  

To prove the third statement we use the argument from the proof of Lemma \ref{lma:sigmay0cor}. That is, the trivialized Lagrangian boundary condition for the fully capped $\sigma_{Y_1}^{(0,1)}(0,0) $--problem coincides with the one for the fully capped $ \sigma_{\swi}^{(0,1)}(l)$--problem. But now the orientation induced on this problem by the two different capping sequences differs by a factor $(-1)^{ \sigma_{0,-}^{(1,1)}}$. The result now follows from 
the second statement together with Corollary \ref{lma:signallspeccap} and the fact that $\sigma_{\swi}^{(0,1)}(l)=\sigma_{\swi}^{(0,1)}(l)$ from Section \ref{sec:results}.
\end{proof}

\begin{lma}\label{cor:y1rem}
 If $n=1$ we have that 
  \begin{align*}
  \sigma_{Y_1}^{(0,1)}(0,0) & \equiv  
   \sigma_{Y_1}^{(1,0)}(0,0)  \pmod{2}.
 \end{align*}
\end{lma}
\begin{proof}
This follows from the results in Lemma \ref{lma:oddsigmarel} and Lemma \ref{lma:remainingY1}. Indeed, in all dimensions we get that the canonical orientation of the fully capped $\sigma_{Y_1}^{(1,0)}(0,0)$--problem differs from the canonical orientation of the fully capped $\sigma_{Y_1}^{(0,1)}(0,0)$--problem by a factor 
\begin{equation*}
(-1)^{  \sigma_{Y_1}^{(1,0)}(0,0) +   \sigma_{Y_1}^{(0,1)}(0,0) + \sigma_{1,+}^{(0,1)} + \sigma_{1,+}^{(1,0)}},                                                                                                                                                                                                                                                                                                                                                                                                                                                                                                                                                                                                                                                   \end{equation*}
which by Lemma \ref{lma:oddsigmarel} reduces to 
\begin{equation*}
(-1)^{  \sigma_{Y_1}^{(1,0)}(0,0) +   \sigma_{Y_1}^{(0,1)}(0,0) +1}.                                                                                                                                                                                                                                                                                                                                                                                                                                                                                                                                                                                                                                                   \end{equation*}
From Lemma \ref{lma:remainingY1} we get that these two problems thus have opposite canonical orientation if $n>1$, and thus their trivialized Lagrangian boundary conditions represent different homotopy classes. But since the boundary conditions are constant in the first $n-1$ directions, they also represent different homotopy classes for $n=1$. Thus we must have 
\begin{equation*}
 \sigma_{Y_1}^{(1,0)}(0,0) \equiv  \sigma_{Y_1}^{(0,1)}(0,0) \pmod 2
\end{equation*}
also in the case when $n=1$. 
\end{proof}

\begin{cor}
If $n=1$ we have that $\sigma_{0,-}^{(1,1)} +\sigma_{0,-}^{(0,0)}+ \sigma_{1,+}^{(1,0)} \equiv 0$, modulo 2.
\end{cor}
\begin{proof}
This follows from Lemma \ref{lma:signswitch}, Lemma \ref{cor:y1rem} and the choice of orientations of capping operators in Remark \ref{rmk:n1choice}.
\end{proof}

\begin{cor}\label{lma:y011sign}
 With our orientation choices we have that 
 \begin{equation*}
  \sigma_{Y_0}^{(0,0)}(1,1) \equiv n \pmod{2}.
 \end{equation*}
\end{cor}
\begin{proof}
This follows from Section \ref{sec:results} and Lemma \ref{lma:signswitch}. 
\end{proof}

\subsection{Derivation of the signs of elementary trees corresponding to true 1-valent negative punctures}\label{sec:signs1val}

Now we are ready to derive explicit formulas for the signs $\sigma_{\negg,1}$. 
First we need some auxiliarly results.

\begin{lma}\label{lma:tp1}
 Assume that $k>0$. With our orientation choices we get that 
\begin{align*}
 \sigma^{t,(l,l)}_{+}(k) &= \sigma_{0,-}^{(l,l)} + \sigma(n-k) + l, \qquad l \in \{0,1\},\\
 \sigma^{t,(1,0)}_{+}(k) &= \sigma_{1,-}^{(1,0)} + \sigma(n-k+1)+\nu + n+1, \\
  \sigma^{t,(0,1)}_{+}(k) &= \sigma_{1,-}^{(0,1)}  + \sigma(n-k+1)+\nu + n, \pmod{2}.
\end{align*}
\end{lma}
\begin{proof}
Let $\Gamma$ be an elementary tree corresponding to a true 1-valent positive puncture $p$ of index $k$. Then $\kd_\Gamma = \spn(\partial_{x_1}, \dotsc, \partial_{x_k})$ = $T_p(W^u(p))$. 

First assume that $|\mu(p)| =0$ and consider the fully capped problem $\dbar_{\hat \Gamma}$ corresponding to $\Gamma$. This satisfies $\kd_{\hat \Gamma} = \spn(\partial_{x_1}, \dotsc, \partial_{x_k})$ and the cokernel can be identified with $\spn(\partial_{x_{k+1}}, \dotsc, \partial_{x_n})$ via \eqref{eq:cokerident}. Moreover, the trivialized Lagrangian boundary condition for the $\dbar_{\hat \Gamma}$--problem is given by 
 \begin{equation}\label{eq:firstcase}
  \diag(\underbrace{1,\dotsc,1}_{k}, \underbrace{e^{-2\pi i s}, \dotsc, e^{-2\pi i s}}_{n-k}, e^{-i\pi s}, e^{-i\pi s}) \hat * R_\pi(n+1,n+2) 
 \end{equation}
 in the case $p$ is of type $(0,0)$, and is given by 
  \begin{equation*}
  \diag(\underbrace{1,\dotsc,1}_{k}, \underbrace{e^{-2\pi i s}, \dotsc, e^{-2\pi i s}}_{n-k}, e^{-i\pi s}, -e^{-i\pi s}) \hat * R_{-\pi}(n+1,n+2)  
 \end{equation*}
 in the case $p$ is of type $(1,1)$.
 
 In the $(0,0)$--case, notice that the boundary conditions in \eqref{eq:firstcase} can be given by gluing the trivialized boundary condition
 \begin{equation*}
  A =  \diag(\underbrace{1,\dotsc,1}_{k}, \underbrace{e^{-2\pi i s}, \dotsc, e^{-2\pi i s}}_{n-k}, 1,1) 
 \end{equation*}
to the trivialized boundary condition 
\begin{equation*}
 B = \diag(\underbrace{1,\dotsc,1}_{n}, e^{-i\pi s}, e^{-i\pi s}) \hat* R_\pi(n+1,n+2).
\end{equation*}
Recall that, by Lemma \ref{lma:IJsign}, the canonical orientation of $\det \dbar_A$ is given by 
\begin{equation*}
 (-1)^{\sigma(n-k)} \partial_{x_1} \wedge \dotsm \wedge \partial_{x_{k}} \wedge \partial_{x_{n+1}} \wedge \partial_{x_{n+2}}\otimes \partial_{x_{k+1}} \wedge \dotsm \wedge \partial_{x_{n}}, 
\end{equation*}
and, by our choice of orientation of $\C$, that the canonical orientation of $\det \dbar_B$ is given by 
\begin{equation*}
 \partial_{x_1} \wedge \dotsm \wedge \partial_{x_{n}}.
\end{equation*}
Moreover, the gluing of $\dbar_A$ to $\dbar_B$ induces the exact sequence 
  \begin{align}
 0 \to 
 \begin{bmatrix}
  \langle\partial_{x_1}\rangle \\
  \vdots \\
   \langle\partial_{x_{k}}\rangle
 \end{bmatrix}
 \xrightarrow{\alpha} 
 \begin{bmatrix}
 \begin{bmatrix}
  \langle \partial_{x_1}\rangle \\
  \vdots \\
   \langle \partial_{x_{k}}\rangle\\
   \langle \partial_{x_{n+1}}\rangle\\
      \langle \partial_{x_{n+2}}\rangle
 \end{bmatrix}\\
 \begin{bmatrix}
  \langle \partial_{x_1}\rangle \\
  \vdots \\
   \langle \partial_{x_{n}}\rangle
 \end{bmatrix}
\end{bmatrix}
\xrightarrow{\beta}  
\begin{bmatrix}
\begin{bmatrix}
 \langle \partial_{x_{k+1}}\rangle \\
 \vdots \\
 \langle \partial_{x_{n}}\rangle
 \end{bmatrix} \\
\R^{n+2} \\
\begin{bmatrix}
0
\end{bmatrix}
\end{bmatrix} 
\xrightarrow{\gamma}  
\begin{bmatrix}
 \langle \partial_{x_{k+1}}\rangle \\
 \vdots \\
 \langle \partial_{x_{n}}\rangle
 \end{bmatrix} 
 \to 0.
\end{align}
It follows that the canonical orientation of $\det \dbar_{\hat \Gamma}$ is given by 
\begin{equation*}
 (-1)^{\sigma(n-k)} \partial_{x_1} \wedge \dotsm \wedge \partial_{x_{k}}\otimes \partial_{x_{k+1}} \wedge \dotsm \wedge \partial_{x_{n}}
\end{equation*}
in the $(0,0)$-case. Also, since the trivialized Lagrangian boundary condition for the $\dbar_{\hat \Gamma}$--problem in the $(1,1)$--case coincides with \eqref{eq:bcevencap} under the change of coordinates $\partial_{x_{n+2}} \leftrightarrow - \partial_{x_{n+2}}$, it follows that the canonical orientation of $\det \dbar_{\hat \Gamma}$ is given by 
\begin{equation*}
 (-1)^{\sigma(n-k)+1} \partial_{x_1} \wedge \dotsm \wedge \partial_{x_{k}}\otimes \partial_{x_{k+1}} \wedge \dotsm \wedge \partial_{x_{n}}
\end{equation*}
in the $(1,1)$-case.

By comparing with the capping sequence for $\Gamma$, given by
  \begin{align}
 0 \to 
 \begin{bmatrix}
  \langle\partial_{x_1}\rangle \\
  \vdots \\
   \langle\partial_{x_{k}}\rangle
 \end{bmatrix}
 \xrightarrow{\alpha} 
 \begin{bmatrix}
  0 \\
  0\\
 \begin{bmatrix}
  \langle \partial_{x_1}\rangle \\
  \vdots \\
   \langle \partial_{x_{k}}\rangle
 \end{bmatrix}
\end{bmatrix}
\xrightarrow{\beta}  
\begin{bmatrix}
 0\\
\begin{bmatrix}
 \langle \partial_{x_{k+1}}\rangle \\
 \vdots \\
 \langle \partial_{x_{n}}\rangle
 \end{bmatrix} \\
  0\\
\end{bmatrix} 
\xrightarrow{\gamma}  
\begin{bmatrix}
 \langle \partial_{x_{k+1}}\rangle \\
 \vdots \\
 \langle \partial_{x_{n}}\rangle
 \end{bmatrix} 
 \to 0,
\end{align}
and recalling that $ \sigma^{t,(0,0)}_{+}(k)$ and $ \sigma^{t,(1,1)}_{+}(k)$ are choosen so that the capping orientation of $\Gamma$ is given by $ \partial_{x_1} \wedge \dotsm \wedge \partial_{x_{k}}$, we see that this sequence induces the orientation 
\begin{equation*}
 (-1)^{\sigma_{0,-}^{(l,l)} + \sigma^{t,(l,l)}_{+}(k)}  \partial_{x_1} \wedge \dotsm \wedge \partial_{x_{k}}\otimes \partial_{x_{k+1}} \wedge \dotsm \wedge \partial_{x_{n}}
\end{equation*}
of $\det \dbar_{\hat \Gamma}$, $l\in \{0,1\}$. Since this should coincide with the canonical orientation of $\det \dbar_{\hat \Gamma}$ it follows that 
\begin{equation*}
 \sigma^{t,(l,l)}_{+}(k) \equiv \sigma_{0,-}^{(l,l)} + \sigma(n-k) + l \pmod{2}, \qquad l \in \{0,1\}.
\end{equation*}

In the case when  $|\mu(p)| =1$ we have that the fully capped problem $\dbar_{\hat \Gamma}$ has $\kd_{\hat \Gamma} = \spn(\partial_{x_1}, \dotsc, \partial_{x_k}, v\partial_{x_{n+2}})$ and the cokernel can be identified with $\spn(\partial_{x_{k+1}}, \dotsc, \partial_{x_n}, \partial_{x_{n+1}})$ via \eqref{eq:cokerident}. Moreover, its trivialized Lagrangian boundary condition is given by 
 \begin{equation}\label{eq:secondcase}
  \diag(\underbrace{1,\dotsc,1}_{k}, \underbrace{e^{-2\pi i s}, \dotsc, e^{-2\pi i s}}_{n-k}, e^{-2\pi i s}, e^{-\pi i s}) * \diag(\underbrace{1,\dotsc,1}_{n}, 1, -e^{\pi i s})
 \end{equation}
 in the case $p$ is of type $(1,0)$, and is given by 
  \begin{align*}
&\left(R_{\pi}(n,n+2)\cdot\diag(\underbrace{1,\dotsc,1}_{n+1},-1)  \right)*\diag(\underbrace{1,\dotsc,1}_{k}, \underbrace{e^{-2\pi i s}, \dotsc, e^{-2\pi i s}}_{n-k-1},-e^{-2\pi i s}, e^{-2\pi i s}, e^{-\pi i s}) \hat *
\\
& \qquad R_\pi(n,n+2) \hat *R_{-\pi}(n,n+2) * \diag(\underbrace{1,\dotsc,1}_{n-1},-1, 1, -e^{\pi i s}) \hat * R_{-\pi}(n,n+2)
 \end{align*}
in the case $p$ is of type $(0,1)$. Note that the latter is homotopic to the boundary condition of the former after having applied the change of coordinates $\partial_x \leftrightarrow -\partial_x$.

From Lemma \ref{lma:IJsign} follows that the canonical orientation of $\det \dbar_{\hat \Gamma}$ is given by 
\begin{equation*}
 (-1)^{\sigma(n-k+1)+\nu + n-k+1} \partial_{x_1} \wedge \dotsm \wedge \partial_{x_{k}} \wedge v\partial_{x_{n+2}}\otimes \partial_{x_{k+1}} \wedge \dotsm \wedge \partial_{x_{n}}\wedge \partial_{x_{n+1}}
\end{equation*}
in the case when $p$ is of type $(1,0)$, and by 
\begin{equation*}
 (-1)^{\sigma(n-k+1)+\nu + n-k} \partial_{x_1} \wedge \dotsm \wedge \partial_{x_{k}} \wedge v\partial_{x_{n+2}}\otimes \partial_{x_{k+1}} \wedge \dotsm \wedge \partial_{x_{n}}\wedge \partial_{x_{n+1}}
\end{equation*}
in the case $p$ is of type $(0,1)$. 

Next consider the capping sequence for $\Gamma$, given by   
  \begin{align}
 0 \to 
 \begin{bmatrix}
  \langle\partial_{x_1}\rangle \\
  \vdots \\
   \langle\partial_{x_{k}}\rangle \\
   \langle v\partial_{x_{n+2}}\rangle
 \end{bmatrix}
 \xrightarrow{\alpha} 
 \begin{bmatrix}
  \langle i(z-1)\partial_{x_{n+2}}\rangle \\
  0\\
 \begin{bmatrix}
  \langle \partial_{x_1}\rangle \\
  \vdots \\
   \langle \partial_{x_{k}}\rangle
 \end{bmatrix}
\end{bmatrix}
\xrightarrow{\beta}  
\begin{bmatrix}
 0\\
\begin{bmatrix}
 \langle \partial_{x_{k+1}}\rangle \\
 \vdots \\
 \langle \partial_{x_{n}}\rangle \\
 \langle \partial_{x_{n+1}}\rangle
 \end{bmatrix} \\
  0
\end{bmatrix} 
\xrightarrow{\gamma}  
\begin{bmatrix}
 \langle \partial_{x_{k+1}}\rangle \\
 \vdots \\
 \langle \partial_{x_{n}}\rangle \\
  \langle \partial_{x_{n+1}}\rangle
 \end{bmatrix} 
 \to 0.
\end{align}
We see that this sequence induces the orientation 
\begin{equation*}
 (-1)^{\sigma_{0,-} + \sigma^t_{+}(k) + 0 + k}  \partial_{x_1} \wedge \dotsm \wedge \partial_{x_{k}}\wedge v\partial_{x_{n+2}}\otimes \partial_{x_{k+1}} \wedge \dotsm \wedge \partial_{x_{n}}
\end{equation*}
of $\det \dbar_{\hat \Gamma}$, where we should decorate $\sigma_{1,-}, \sigma^t_{+}(k)$ with $(0,1)$ or $(1,0)$, depending on which case we are considering. Hence it follows that 
\begin{align*}
 \sigma^{t,(1,0)}_{+}(k) &\equiv \sigma_{1,-}^{(1,0)} + k + \sigma(n-k+1)+\nu + n-k+1, \\
  \sigma^{t,(0,1)}_{+}(k) &\equiv \sigma_{1,-}^{(0,1)} + k + \sigma(n-k+1)+\nu + n-k,
\end{align*}
modulo 2. 
\end{proof}
%

We should also consider the case when $k=0$. 
\begin{lma}\label{lma:tp3}
With our orientation choices we get that 
\begin{align*}
\sigma^{t,(l,l)}_{+}(0) &=  \sigma(n) + l, \qquad l \in \{0,1\},\\
\sigma^{t,(1 ,0)}_{+}(0) & = \sigma_{1,-}^{(1,0)} + \sigma_{0,-}^{(0,0)} + \sigma(n+1)+\nu + n+1, \\
\sigma^{t,(0,1)}_{+}(0) & =  \sigma_{1,-}^{(0,1)} + \sigma_{0,-}^{(1,1)} + \sigma(n+1)+\nu + n.
\end{align*}
\end{lma}
\begin{proof}
By assumption, the sign $\sigma^{t,(0,0)}_{+}(0)$ ($\sigma^{t,(1,1)}_{+}(0)$) is choosen so that $\det \dbar_\Gamma$ gets orientation $(-1)^0 \mathbbm{1}$, where $\Gamma$ is an elementary tree with a 2-valent positive puncture $p$ of type $(0,0)$ ($(1,1)$), and 2 negative punctures of even Maslov index in the case when $|\mu(p)| = 0$. Similarly, $\sigma^{t,(0,1)}_{+}(0)$ ($\sigma^{t,(1,0)}_{+}(0)$) is choosen so that $\det \dbar_\Gamma$ gets orientation $(-1)^0 \mathbbm{1}$ where $\Gamma$ is an elementary tree with a 2-valent positive puncture $p$ of type $(0,1)$ ($(1,0)$) and 2 negative punctures $p_1, p_2$ such that $|\mu(p_1)| = 1, |\mu(p_1)| = 0$ in the case when $|\mu(p)| = 1$. 

We start with the case when $|\mu(p)| =0$. If $p$ is of type $(0,0)$, the fully capped problem $\dbar_{\hat \Gamma}$ has trivialized Lagrangian boundary condition 
 \begin{equation}\label{eq:fifthcase}
  \diag(\underbrace{e^{-2\pi i s}, \dotsc, e^{-2\pi i s}}_{n}, e^{-i\pi s}, e^{-i\pi s}) \hat * R_\pi(n+1,n+2),
 \end{equation}
 and if $p$ is of type $(1,1)$ it is given by 
  \begin{equation*}
  \diag(\underbrace{e^{-2\pi i s}, \dotsc, e^{-2\pi i s}}_{n}, e^{-i\pi s}, -e^{-i\pi s}) \hat * R_{-\pi}(n+1,n+2). 
 \end{equation*}

Similar to the arguments in the proof of Lemma \ref{lma:tp1}  it follows that the canonical orientation of $\det \dbar_{\hat \Gamma}$ is given by 
\begin{equation*}
 (-1)^{\sigma(n)+ l} \mathbbm{1} \otimes \partial_{x_1} \wedge \dotsm \wedge \partial_{x_{n}}
\end{equation*}
in the case when $p$ is of type $(l,l)$, $l \in \{0,1\}$.

On the other hand, the capping sequence for $\Gamma$ is now given by
  \begin{align}
 0 \to 
0
 \xrightarrow{\alpha} 
 \begin{bmatrix}
  0 \\
  0\\
  0\\
0\\
\end{bmatrix}
\xrightarrow{\beta}  
\begin{bmatrix}
 0\\
 0\\
\begin{bmatrix}
 \langle \partial_{x_{1}}\rangle \\
 \vdots \\
 \langle \partial_{x_{n}}\rangle
 \end{bmatrix} \\
 0
\end{bmatrix} 
\xrightarrow{\gamma}  
\begin{bmatrix}
 \langle \partial_{x_{1}}\rangle \\
 \vdots \\
 \langle \partial_{x_{n}}\rangle
 \end{bmatrix} 
 \to 0,
\end{align}
which induces the orientation 
\begin{equation*}
 (-1)^{\sigma_{0,-} + \sigma_{0,-} + \sigma^t_{+}(0) + 0}  \mathbbm{1} \otimes \partial_{x_{1}} \wedge \dotsm \wedge \partial_{x_{n}}.
\end{equation*}
on $\det \dbar_{\hat \Gamma}$, where $\sigma_{0,-},\sigma_{0,-}, \sigma^t_{+}(0)$ should be decorated with $(0,0)$ or $(1,1)$, depending on which case we are considering. Hence it follows that 
\begin{equation*}
\sigma^{t,(l,l)}_{+}(0) \equiv  \sigma_{0,-}^{(l,l)} + \sigma_{0,-}^{(l,l)} + \sigma(n)+ l \pmod{2}, \qquad l \in \{0,1\}.
\end{equation*}

In the case when  $|\mu(p)| =1$ the trivialized Lagrangian boundary condition for the $\dbar_{\hat \Gamma}$--problem is given by 
 \begin{equation}\label{eq:sixthcase}
  \diag(\underbrace{e^{-2\pi i s}, \dotsc, e^{-2\pi i s}}_{n}, e^{-2i\pi s}, e^{-i\pi s}) * \diag(\underbrace{1,\dotsc,1}_{n}, 1, -e^{i\pi s}),
 \end{equation}
 when $p$ is of type $(1,0)$, and is given by 
  \begin{align*}
&\left(R_{\pi}(n,n+2)\cdot\diag(\underbrace{1,\dotsc,1}_{n+1},-1)  \right)*\diag(\underbrace{e^{-2\pi i s}, \dotsc, e^{-2\pi i s}}_{n-1},-e^{-2\pi i s}, e^{-2\pi i s}, e^{-\pi i s}) \hat *
\\
& \qquad R_\pi(n,n+2) \hat *R_{-\pi}(n,n+2) * \diag(\underbrace{1,\dotsc,1}_{n-1},-1, 1, -e^{\pi i s}) \hat * R_{-\pi}(n,n+2)
\end{align*}
  when $p$ is of type $(0,1)$. We see that the latter boundary condition is homotopic to \eqref{eq:sixthcase}, after applying the change of coordinates $\partial_x \leftrightarrow -\partial_x$.

From Lemma \ref{lma:IJsign} it follows that the canonical orientation of $\det \dbar_{\hat \Gamma}$ is given by 
\begin{equation*}
 (-1)^{\sigma(n+1)+\nu + n+1}  v\partial_{x_{n+2}}\otimes \partial_{x_{1}} \wedge \dotsm \wedge \partial_{x_{n}}\wedge \partial_{x_{n+1}}
\end{equation*}
in the case when $p$ is of type $(1,0)$, and by 
\begin{equation*}
 (-1)^{\sigma(n+1)+\nu + n}  v\partial_{x_{n+2}}\otimes \partial_{x_{1}} \wedge \dotsm \wedge \partial_{x_{n}}\wedge \partial_{x_{n+1}}
\end{equation*}
when $p$ is of type $(0,1)$. 
On the other hand, the capping sequence for $\Gamma$ is given by   
  \begin{align}
 0 \to 
   \langle v\partial_{x_{n+2}}\rangle
 \xrightarrow{\alpha} 
 \begin{bmatrix}
  \langle i(z-1)\partial_{x_{n+2}}\rangle \\
  0\\
  0\\
0\\
\end{bmatrix}
\xrightarrow{\beta}  
\begin{bmatrix}
 0\\
 0\\
\begin{bmatrix}
 \langle \partial_{x_{1}}\rangle \\
 \vdots \\
 \langle \partial_{x_{n}}\rangle \\
 \langle \partial_{x_{n+1}}\rangle
 \end{bmatrix}\\
 0
\end{bmatrix} 
\xrightarrow{\gamma}  
\begin{bmatrix}
 \langle \partial_{x_{1}}\rangle \\
 \vdots \\
 \langle \partial_{x_{n}}\rangle \\
  \langle \partial_{x_{n+1}}\rangle
 \end{bmatrix} 
 \to 0,
\end{align}
in both cases. 
We see that this sequence induces the orientation 
\begin{equation*}
 (-1)^{\sigma_{1,-} + \sigma_{0,-} + \sigma^t_{+}(0) }  v\partial_{x_{n+2}}\otimes \partial_{x_{1}} \wedge \dotsm \wedge \partial_{x_{n}}.
\end{equation*}
of $\det \dbar_{\hat \Gamma}$, where $\sigma_{1,-}, \sigma^t_{+}(0)$ should be decorated with $(1,0)$ or $(0,1)$ and $ \sigma_{0,-}$ with $(0,0)$ or $(1,1)$, depending on which case we are considering. Hence it follows that 
\begin{align*}
\sigma^{t,(1,0)}_{+}(0) & \equiv  \sigma_{1,-}^{(1,0)} + \sigma_{0,-}^{(0,0)} + \sigma(n+1)+\nu + n+1 \\
\sigma^{t,(0,1)}_{+}(0) & \equiv  \sigma_{1,-}^{(0,1)} + \sigma_{0,-}^{(1,1)} + \sigma(n+1)+\nu + n
\end{align*}
modulo 2.

\end{proof}

Now we derive similar results for $\sigma_{-}^t$ from the gluing sequences of the capping operators at a true negative puncture.

\begin{lma}\label{lma:tp2}
 With our orientation choices we get that 
\begin{align*}
 \sigma^{t,(l,l)}_{-}(k) &= \sigma(n+1) + \nu + n+1 + \sigma^{t,(l,l)}_{+}(k)  + k \cdot(n-1) + l , \qquad l \in \{0,1\},\\
  \sigma^{t,(1,0)}_{-}(k) &= \sigma(n+1) + \nu + n+1 + \sigma^{t,(1,0)}_{+}(k)  + k \cdot n, \\
  \sigma^{t,(0,1)}_{-}(k) &= \sigma(n+1) + \nu + n+ \sigma^{t,(0,1)}_{+}(k)  + k \cdot n.
\end{align*} 
\end{lma}

\begin{proof}
 Start with the case when $p$ is a true puncture of index $k$ and with $|\mu(p)|=0$, and consider the $\dbar_p$-problem associated to $p$, obtained by gluing the positive capping operator at $p$ to  the negative capping operator. In the case when $p$ is of type $(0,0)$, the trivialized Lagrangian boundary condition of $\dbar_p$ is given by 
   \begin{align*}
  &\diag(\underbrace{1,\dotsc,1}_{k}, \underbrace{e^{-2\pi i s}, \dotsc, e^{-2\pi i s}}_{n-k}, e^{-i\pi s}, e^{-i\pi s}) \hat* R_\pi(n+1,n+2) \hat *\\
  & {} \qquad R_{-\pi}(n+1,n+2) * \diag(\underbrace {e^{-2\pi i s}, \dotsc, e^{-2\pi i s}}_{k}, \underbrace{1, \dotsc, 1}_{n-k}, -e^{-i\pi s}, -e^{i\pi s}) 
 \end{align*}
 where the first $k$ directions correspond to $T_p(W^u(p))$. This boundary condition is homotopic to 
   \begin{equation*}
  \diag(\underbrace{e^{-2\pi i s}, \dotsc, e^{-2\pi i s}}_{n}, e^{-2\pi i s}, e^{-i\pi s})* \diag(\underbrace{1, \dotsc, 1}_{n}, 1, -e^{i\pi s}), 
 \end{equation*}
 so by Lemma \ref{lma:IJsign} we have  that the canonical orientation of $\det \dbar_p$ is given by 
\begin{equation*}
 (-1)^{\sigma(n+1) + \nu + n+1} v \partial_{x_{n+2}} \otimes \partial_{x_1} \wedge \dotsm \wedge \partial_{x_{n+1}}.
\end{equation*}

Similarly, when $p$ is of type $(1,1)$, the trivialized Lagrangian boundary condition of $\dbar_p$ is given by    \begin{align*}
  &\diag(\underbrace{1,\dotsc,1}_{k}, \underbrace{e^{-2\pi i s}, \dotsc, e^{-2\pi i s}}_{n-k}, e^{-i\pi s}, -e^{-i\pi s}) \hat * R_{-\pi}(n+1,n+2)\hat * \\
  &{} \qquad R_{\pi}(n+1,n+2) * \diag(\underbrace {e^{-2\pi i s}, \dotsc, e^{-2\pi i s}}_{k}, \underbrace{1, \dotsc, 1}_{n-k}, -e^{-i\pi s}, e^{i\pi s}) 
 \end{align*}
which is homotopic to 
   \begin{equation*}
  \diag(\underbrace{e^{-2\pi i s}, \dotsc, e^{-2\pi i s}}_{n}, e^{-2\pi i s},-e^{-i\pi s})* \diag(\underbrace{1, \dotsc, 1}_{n}, 1, e^{i\pi s}), 
 \end{equation*}
and hence we have  that the canonical orientation of $\det \dbar_p$ is given by 
\begin{equation*}
 (-1)^{\sigma(n+1) + \nu + n} v \partial_{x_{n+2}} \otimes \partial_{x_1} \wedge \dotsm \wedge \partial_{x_{n+1}}.
\end{equation*}

Now, if we consider the gluing sequence induced by gluing $\dbar_{p+}$ to $\dbar_{p-}$, this is given by 
  \begin{align}
 0 \to 
   \langle v\partial_{x_{n+2}}\rangle
 \xrightarrow{\alpha} 
 \begin{bmatrix}
  0\\
  \langle i(z-1)\partial_{x_{n+2}}\rangle \\
\end{bmatrix}
\xrightarrow{\beta}  
\begin{bmatrix}
\begin{bmatrix}
 \langle \partial_{x_{k+1}}\rangle \\
 \vdots \\
 \langle \partial_{x_{n}}\rangle 
 \end{bmatrix} \\
\begin{bmatrix}
 \langle \partial_{x_{1}}\rangle \\
 \vdots \\
 \langle \partial_{x_{k}}\rangle \\
 \langle \partial_{x_{n+1}}\rangle
 \end{bmatrix} 
\end{bmatrix} 
\xrightarrow{\gamma}  
\begin{bmatrix}
 \langle \partial_{x_{1}}\rangle \\
 \vdots \\
  \langle \partial_{x_{n+1}}\rangle
 \end{bmatrix} 
 \to 0,
\end{align}
and gives that the canonical orientation of $\det \dbar_p$ is given by 
\begin{equation*}
 (-1)^{\sigma^t_{+}(k) + \sigma^t_{-}(k) + k \cdot(n-k)} v \partial_{x_{n+2}} \otimes \partial_{x_1} \wedge \dotsm \wedge \partial_{x_{n+1}},
\end{equation*}
where $\sigma^t_{+}(k),\sigma^t_{-}(k) $ should be decorated with $(0,0)$ or $(1,1)$, depending on which case we are considering. 
Hence it follows that 
\begin{equation*}
 \sigma^{t,(l,l)}_{-}(k) \equiv \sigma(n+1) + \nu + n+1 + \sigma^{t,(l,l)}_{+}(k)  + k \cdot(n-1) + l \pmod{2}, \qquad l \in \{0,1\}.
\end{equation*}

Now consider the case when $|\mu(p)| =1$. In the case of a $(1,0)$--puncture we get that the trivialized Lagrangian boundary condition of $\dbar_p$ is homotopic to $I_{n+1}$ and hence $\det \dbar_p$ has canonical orientation given by 
\begin{equation*}
 (-1)^{\sigma(n+1) + \nu + n+1} v \partial_{x_{n+2}} \otimes \partial_{x_1} \wedge \dotsm \wedge \partial_{x_{n+1}}.
\end{equation*}

Similarly, in the case of a $(0,1)$--puncture we get that the trivialized Lagrangian boundary condition for $\dbar_p$ is homotopic to $I_{n+1}$ after the change of coordinates  $\partial_x \leftrightarrow -\partial_x$.
%
It follows that
$\det \dbar_p$ has canonical orientation given by 
\begin{equation*}
 (-1)^{\sigma(n+1) + \nu + n} v \partial_{x_{n+2}} \otimes \partial_{x_1} \wedge \dotsm \wedge \partial_{x_{n+1}}.
\end{equation*}

In both cases the gluing sequence induced by gluing $\dbar_{p+}$ to $\dbar_{p-}$ is given by 
  \begin{align}
 0 \to 
   \langle v\partial_{x_{n+2}}\rangle
 \xrightarrow{\alpha} 
 \begin{bmatrix}
  0\\
  \langle i(z-1)\partial_{x_{n+2}}\rangle \\
\end{bmatrix}
\xrightarrow{\beta}  
\begin{bmatrix}
\begin{bmatrix}
 \langle \partial_{x_{k+1}}\rangle \\
 \vdots \\
 \langle \partial_{x_{n+1}}\rangle 
 \end{bmatrix} \\
\begin{bmatrix}
 \langle \partial_{x_{1}}\rangle \\
 \vdots \\
 \langle \partial_{x_{k}}\rangle 
 \end{bmatrix} 
\end{bmatrix} 
\xrightarrow{\gamma}  
\begin{bmatrix}
 \langle \partial_{x_{1}}\rangle \\
 \vdots \\
  \langle \partial_{x_{n+1}}\rangle
 \end{bmatrix} 
 \to 0.
\end{align}
Thus we get that 
\begin{align*}
 \sigma^{t,(1,0)}_{-}(k) \equiv \sigma(n+1) + \nu + n+1 + \sigma^{t,(1,0)}_{+}(k)  + k \cdot n, \pmod{2} \\
  \sigma^{t,(0,1)}_{-}(k) \equiv \sigma(n+1) + \nu + n+ \sigma^{t,(0,1)}_{+}(k)  + k \cdot n, \pmod{2} 
\end{align*}
\end{proof}

Now we are ready to derive formulas for the signs $\sigma_{\negg,1}$.

\begin{lma}\label{lma:1valsign}
With our orientation choices we get that, for $k > 0$,
\begin{align*}
 \sigma_{\negg,1}^{(l,l)}(k) &=  \sigma_{0,-}^{(l,l)} +\sigma_{0,+}^{(l,l)}  +l + n\cdot (k + 1), \qquad l \in \{0,1\},\\
 \sigma_{\negg,1}^{(1,0)}(k) &= n\cdot k +k + 1, \\
 \sigma_{\negg,1}^{(0,1)}(k) &=  n\cdot k +k + \sigma_{1,-}^{(0,1)} \qquad \pmod{2} 
\end{align*} 
 and 
\begin{align*}
 \sigma_{\negg,1}^{(l,l)}(0) &=  n + l + \sigma_{0,+}^{(l,l)} ,\qquad l \in \{0,1\},\\
   \sigma_{\negg,1}^{(1,0)}(0) &= \sigma_{0,-}^{(0,0)} + 1, \\
\sigma_{\negg,1}^{(0,1)}(0) &= \sigma_{0,-}^{(1,1)} + \sigma_{1,-}^{(0,1)} \qquad \pmod{2}.
\end{align*} 
\end{lma}
\begin{proof}
Let $\Gamma$ be an elementary tree corresponding to a true negative puncture $p$ of index $k$. Then $\kd_\Gamma = \spn(\partial_{x_{k+1}}, \dotsc, \partial_{x_n})$ = $T_pW^s(p)$. 

First assume that $|\mu(p)| =0$ and consider the fully capped problem $\dbar_{\hat \Gamma}$ corresponding to $\Gamma$. This satisfies $\kd_{\hat \Gamma} = \spn(\partial_{x_{k+1}}, \dotsc, \partial_{x_n},v\partial_{x_{n+2}})$ and the cokernel can be identified with $\spn(\partial_{x_{1}}, \dotsc, \partial_{x_k},\partial_{x_{n+1}})$ via \eqref{eq:cokerident}. Moreover, the trivialized Lagrangian boundary condition for the $\dbar_{\hat \Gamma}$--problem is given by 
 \begin{align}\label{eq:tirdcase}
   &\diag(\underbrace{1,\dotsc,1}_{n},  e^{-i\pi s}, e^{-i\pi s}) \hat * R_\pi(n+1,n+2) \hat *\\ \nonumber
   & \qquad R_{-\pi}(n+1,n+2) *\diag( \underbrace{e^{-2\pi i s}, \dotsc, e^{-2\pi i s}}_{k}, \underbrace{1,\dotsc,1}_{n-k}, -e^{-i\pi s}, -e^{i\pi s}) 
 \end{align}
 in the case when $p$ is of type $(0,0)$, and is given by 
  \begin{align*}
 &\diag(\underbrace{1,\dotsc,1}_{n},  e^{-i\pi s}, -e^{-i\pi s}) \hat * R_{-\pi}(n+1,n+2) \hat* \\
   &\qquad R_{\pi}(n+1,n+2) *\diag( \underbrace{e^{-2\pi i s}, \dotsc, e^{-2\pi i s}}_{k}, \underbrace{1,\dotsc,1}_{n-k}, -e^{-i\pi s}, e^{i\pi s}) 
 \end{align*}
 in the case when $p$ is of type $(1,1)$. 
 
The boundary condition in \eqref{eq:tirdcase} is homotopic to the boundary condition $I_{k+1}$ after a change of coordinates, and from Lemma \ref{lma:IJsign} it then follows that the canonical orientation of $\det \dbar_{\hat \Gamma}$ is given by 
\begin{equation*}
 (-1)^{k\cdot(n-k) + \sigma(k+1)+\nu + k+1} \partial_{x_{k+1}} \wedge \dotsm \wedge \partial_{x_{n}}\wedge v\partial_{x_{n+2}}\otimes \partial_{x_1} \wedge \dotsm \wedge \partial_{x_{k}}\wedge \partial_{x_{n+1}}
\end{equation*} 
in the case when $p$ is of type $(0,0)$.
Here the $(-1)^{k \cdot (n-k)}$--factor comes from the change of coordinates.

Similarly we get that the canonical orientation of $\det \dbar_{\hat \Gamma}$ is given by 
\begin{equation*}
 (-1)^{k\cdot(n-k) + \sigma(k+1)+\nu + k} \partial_{x_{k+1}} \wedge \dotsm \wedge \partial_{x_{n}}\wedge v\partial_{x_{n+2}}\otimes \partial_{x_1} \wedge \dotsm \wedge \partial_{x_{k}}\wedge \partial_{x_{n+1}}
\end{equation*} 
in the case when $p$ is of type $(1,1)$.

By comparing with the capping sequence for $\Gamma$, given by
  \begin{align}
 0 \to 
 \begin{bmatrix}
  \langle\partial_{x_{k+1}}\rangle \\
  \vdots \\
   \langle\partial_{x_{n}}\rangle\\
   \langle v\partial_{x_{n+2}} \rangle
 \end{bmatrix}
 \xrightarrow{\alpha} 
 \begin{bmatrix}
  \langle i(z-1) \partial_{x_{n+2}} \rangle\\
  0\\
 \begin{bmatrix}
  \langle\partial_{x_{k+1}}\rangle \\
  \vdots \\
   \langle\partial_{x_{n}}\rangle
 \end{bmatrix}
\end{bmatrix}
\xrightarrow{\beta}  
\begin{bmatrix}
\begin{bmatrix}
 \langle \partial_{x_{1}}\rangle \\
 \vdots \\
 \langle \partial_{x_{k}}\rangle \\
 \langle \partial_{x_{n+1}}\rangle
 \end{bmatrix} \\
 0\\
 0
\end{bmatrix} 
\xrightarrow{\gamma}  
\begin{bmatrix}
 \langle \partial_{x_{1}}\rangle \\
 \vdots \\
 \langle \partial_{x_{k}}\rangle \\
 \langle \partial_{x_{n+1}}\rangle
 \end{bmatrix}
 \to 0,
\end{align}
 we see that this sequence induces the orientation 
\begin{equation*}
 (-1)^{\sigma^{t}_{-}(k) + \sigma_{0,+} +  \sigma_{\negg,1}(k) + n-k}  \partial_{x_{k+1}} \wedge \dotsm \wedge \partial_{x_{n}}\wedge v\partial_{x_{n+2}}\otimes \partial_{x_1} \wedge \dotsm \wedge \partial_{x_{k}}\wedge \partial_{x_{n+1}}.
\end{equation*}
of $\det \dbar_{\hat \Gamma}$. Here $ \sigma^t_{-}(k),\sigma_{0,+},  \sigma_{\negg,1}(k) $ should be decorated with $(0,0)$ or $(1,1)$ depending on which case we are considering. Hence we get that 
\begin{align*}
\sigma_{\negg,1}^{(l,l)}(k) \equiv & \sigma^{t,(l,l)}_{-}(k) + \sigma_{0,+}^{(l,l)} +  n-k + k\cdot(n-k) + \sigma(k+1)+\nu + k+1+l \\
\equiv & \sigma^{t,(l,l)}_{-}(k) + \sigma_{0,+}^{(l,l)} + \sigma(k+1)+\nu + n+ k\cdot(n-1) + 1+l,
\end{align*}
modulo 2, where $l \in \{0,1\}$.

Now, if we use the results from Lemma \ref{lma:tp2} and \ref{lma:tp1}, we get, if $k >0 $, that  
\begin{align*}
\sigma_{\negg,1}^{(l,l)}(k) \equiv &  \sigma(n+1) + \nu + n+1 + \sigma^{t,(l,l)}_{+}(k)  + k \cdot(n-1)+l\\
&\quad{} + \sigma_{0,+}^{(l,l)} + \sigma(k+1)+\nu + n+ k\cdot(n-1) + 1 + l\\
\equiv & \sigma^{t,(l,l)}_{+}(k)+ \sigma_{0,+}^{(l,l)} +  \sigma(n+1) + \sigma(k+1)\\
\equiv &  \sigma_{0,-}^{(l,l)} + \sigma(n-k)  + l +\sigma_{0,+}^{(l,l)} + \sigma(n+1) + \sigma(k+1) \pmod{2}.
\end{align*}
Using Lemma \ref{lma:sigmarelations} with the change of variables $l = n-k$ it follows that $\sigma(n-k)  + \sigma(n+1) + \sigma(k+1) \equiv n\cdot (k+1), \pmod 2$. 
Hence the result follows.

If $k=0$ we get that 
\begin{align*}
\sigma_{\negg,1}^{(l,l)}(0) \equiv &  \sigma(n+1) + \nu + n+1 + \sigma^{t,(l,l)}_{+}(0)+l  \\
&\quad{} + \sigma_{0,+}^{(l,l)} + \sigma(1)+\nu + n + 1 +l\\
\equiv &  \sigma(n+1) + \sigma^{t,(l,l)}_{+}(0)  + \sigma_{0,+}^{(l,l)} + \sigma(1) \\
\equiv & \sigma(n+1) + \sigma(1) + \sigma(n) + l + \sigma_{0,+}^{(l,l)} \\
\equiv & n + l + \sigma_{0,+}^{(l,l)} \pmod{2}.
\end{align*}

In the case when  $|\mu(p)| =1$ we have that the fully capped problem $\dbar_{\hat \Gamma}$ has $\kd _{\hat \Gamma} = \spn(\partial_{x_{k+1}}, \dotsc, \partial_{x_n},v_1\partial_{x_{n+2}}, v_2\partial_{x_{n+2}})$ and the cokernel can be identified with \linebreak
$\spn(\partial_{x_{1}}, \dotsc, \partial_{x_k})$ via \eqref{eq:cokerident}. Recall that $v_1, v_2$ are defined in Section \ref{sec:choiceofC}. Moreover, the trivialized Lagrangian boundary condition for the $\dbar_{\hat \Gamma}$--problem is given by 
 \begin{align}\label{eq:forthcase}
   \diag(\underbrace{1,\dotsc,1}_{n},  e^{-i\pi s},1) \hat * R_{-\pi}(n+1,n+2) *\diag( \underbrace{e^{-2\pi i s}, \dotsc, e^{-2\pi i s}}_{k}, \underbrace{1,\dotsc,1}_{n-k}, 1, -e^{i\pi s}) 
 \end{align}
 in the case when $p$ is of type $(1,0)$, and is given by 
 \begin{align*}
   &\left( R_{\pi}(n,n+2)\cdot\diag(\underbrace{1,\dotsc,1}_{n+1},-1) \right)*\diag(\underbrace{1,\dotsc,1}_{n-1},-1, e^{-i\pi s},1) \hat* R_{-\pi}(n+1,n+2)  \hat *\\ 
   & {} \qquad R_\pi(n,n+2) \hat * R_{-\pi}(n,n+2) *\diag( \underbrace{e^{-2\pi i s}, \dotsc, e^{-2\pi i s}}_{k}, \underbrace{1,\dotsc,1}_{n-k-1},-1, 1, -e^{i\pi s})  \hat *\\ 
   & {} \qquad R_{-\pi}(n,n+2) 
 \end{align*}
in the case when $p$ is of type $(0,1)$. Notice that the latter is homotopic to
 \begin{equation*}
   \diag(\underbrace{1,\dotsc,1}_{n-1},-1, e^{-i\pi s},1) \hat * R_{-\pi}(n+1,n+2) *\diag( \underbrace{e^{-2\pi i s}, \dotsc, e^{-2\pi i s}}_{k}, \underbrace{1,\dotsc,1}_{n-k-1},-1, 1, -e^{i\pi s})  
 \end{equation*}
 and that this turns into the trivialized Lagrangian boundary condition for the $(1,0)$--case under the change of coordinates $\partial_{x_n} \leftrightarrow -\partial_{x_n}$. Notice that we have $\dim W^s(p)>0$ if $p$ is a 1-valent negative puncture.

 Notice that the boundary condition in \eqref{eq:forthcase} can be given by gluing the trivialized boundary condition
 \begin{equation*}
  A =  \diag(\underbrace{1,\dotsc,1}_{n}, e^{-i\pi s},  1) \hat * R_{-\pi}(n+1,n+2) *\diag( \underbrace{1, \dotsc, 1}_{n}, 1, -e^{i\pi s}) 
 \end{equation*}
to the trivialized boundary condition 
\begin{equation*}
 B = \diag( \underbrace{e^{-2\pi i s}, \dotsc, e^{-2\pi i s}}_{k}, \underbrace{1,\dotsc,1}_{n-k}, 1, 1) .
\end{equation*}
Notice that $A$ coincides with the trivialized boundary condition in \eqref{eq:bc-+}, and that the canonical orientation of $\det \dbar_A$ hence is given by \eqref{eq:orientmu0}. Moreover, $B$ is the trivialized boundary condition $J_k$ after a change of coordinates, and hence has canonical orientation 
\begin{equation*}
 (-1)^{\sigma(k) + k \cdot(n-k)} \partial_{x_{k+1}} \wedge \dotsm \wedge \partial_{x_{n+2}}\otimes  \partial_{x_1} \wedge \dotsm \wedge \partial_{x_{k}}.
\end{equation*}


Moreover, the gluing of $\dbar_A$ to $\dbar_B$ induces the exact sequence 
  \begin{align}
 0 \to 
 \begin{bmatrix}
  \langle\partial_{x_{k+1}}\rangle \\
  \vdots \\
   \langle\partial_{x_{n}}\rangle \\
   \langle v_1\partial_{x_{n+2}}\rangle \\
   \langle v_2\partial_{x_{n+2}}\rangle 
 \end{bmatrix}
 \xrightarrow{\alpha} 
 \begin{bmatrix}
 \begin{bmatrix}
  \langle \partial_{x_1}\rangle \\
  \vdots \\
   \langle \partial_{x_{n}}\rangle\\
   \langle v_1\partial_{x_{n+2}}\rangle \\
   \langle v_2\partial_{x_{n+2}}\rangle 
 \end{bmatrix}\\
 \begin{bmatrix}
 \langle\partial_{x_{k+1}}\rangle \\
  \vdots \\
   \langle\partial_{x_{n+2}}\rangle \\
 \end{bmatrix}
\end{bmatrix}
\xrightarrow{\beta}  
\begin{bmatrix}
 0 \\
 \R^{n+2}\\
\begin{bmatrix}
 \langle \partial_{x_{1}}\rangle \\
 \vdots \\
 \langle \partial_{x_{k}}\rangle
 \end{bmatrix} \\
\end{bmatrix} 
\xrightarrow{\gamma}  
\begin{bmatrix}
 \langle \partial_{x_{1}}\rangle \\
 \vdots \\
 \langle \partial_{x_{k}}\rangle
 \end{bmatrix} 
 \to 0.
\end{align}
It follows that the canonical orientation of $\det \dbar_{\hat \Gamma}$ is given by 
\begin{equation*}
 (-1)^{\mu(0) + \sigma(k) + k \cdot(n-k) + k + k\cdot(n-k) + n\cdot k}\partial_{x_{k+1}} \wedge \dotsm \wedge \partial_{x_{n}} \wedge v_1 \partial_{x_{n+2}} \wedge v_2 \partial_{x_{n+2}} \otimes \partial_{x_1} \wedge \dotsm \wedge \partial_{x_{k}}
\end{equation*}
in the case when $p$ is of type $(1,0)$.
Here the last $k\cdot (n-k)$ comes from rearrangements in the second non-trivial column, and the $k$ before that from the determinant of the map $\beta$. The $n\cdot k$--summand comes from rearrangements in the third non-trivial column.

Similarly, it follows that the canonical orientation of $\det \dbar_{\hat \Gamma}$ is given by 
\begin{equation*}
 (-1)^{\mu(0) + \sigma(k) + k \cdot(n-k) + k + k\cdot(n-k) + n\cdot k+1} \partial_{x_{k+1}} \wedge \dotsm \wedge \partial_{x_{n}} \wedge v_1 \partial_{x_{n+2}} \wedge v_2 \partial_{x_{n+2}} \otimes \partial_{x_1} \wedge \dotsm \wedge \partial_{x_{k}}
\end{equation*}
in the case when $p$ is of type $(0,1)$.

Now consider the capping sequence for $\Gamma$, given by   
  \begin{align}
 0 \to 
 \begin{bmatrix}
  \langle\partial_{x_{k+1}}\rangle \\
  \vdots \\
   \langle\partial_{x_{n}}\rangle \\
   \langle v_1\partial_{x_{n+2}}\rangle \\
   \langle v_2\partial_{x_{n+2}}\rangle 
 \end{bmatrix}
 \xrightarrow{\alpha} 
 \begin{bmatrix}
  \langle(i(z-1)\partial_{x_{n+2}} \rangle \\
  \langle \partial_{x_{n+2}} \rangle \\
  \begin{bmatrix}
  \langle \partial_{x_{k+1}}\rangle \\
  \vdots \\
   \langle \partial_{x_{n}}\rangle
 \end{bmatrix}
\end{bmatrix}
\xrightarrow{\beta}  
\begin{bmatrix}
   \begin{bmatrix}
  \langle \partial_{x_1}\rangle \\
  \vdots \\
   \langle \partial_{x_{k}}\rangle
 \end{bmatrix}\\
 0
\end{bmatrix} 
\xrightarrow{\gamma}  
\begin{bmatrix}
 \langle \partial_{x_{1}}\rangle \\
 \vdots \\
 \langle \partial_{x_{k}}\rangle
 \end{bmatrix} 
 \to 0.
\end{align}
We see that this sequence induces the orientation 
\begin{equation*}
 (-1)^{\sigma_{-}^t(k) + \sigma_{1,+} + 1 + \sigma_{\negg,1}(k)}  \partial_{x_{k+1}} \wedge \dotsm \wedge \partial_{x_{n}} \wedge v_1 \partial_{x_{n+2}} \wedge v_2 \partial_{x_{n+2}} \otimes \partial_{x_1} \wedge \dotsm \wedge \partial_{x_{k}}.
\end{equation*}
of $\det \dbar_{\hat \Gamma}$, where $\sigma_{-}^t(k), \sigma_{1,+}, \sigma_{\negg,1}(k)$ should be decorated with $(1,0)$ or $(0,1)$ depending on which case we are considering. Hence it follows that 
\begin{align*}
  \sigma_{\negg,1}^{(1,0)}(k) &\equiv \sigma_{-}^{t,(1,0)}(k) + \sigma_{1,+}^{(1,0)} + 1 + \mu(0) + \sigma(k)  + k\cdot(n-k), \pmod{2}, \\  
  \sigma_{\negg,1}^{(0,1)}(k) &\equiv \sigma_{-}^{t,(0,1)}(k) + \sigma_{1,+}^{(0,1)} + \mu(0) + \sigma(k)  + k\cdot(n-k), \pmod{2}.
  \end{align*}

 Now, if we use the results from Lemma \ref{lma:tp2} and \ref{lma:tp1}, we get that, in the case $k>0$, 
 \begin{align*}
 \sigma_{\negg,1}^{(1,0)}(k) & \equiv  \sigma(n+1) + \nu + n+1 + \sigma^{t,(1,0)}_{+}(k)  + k \cdot n\\
 &\quad{}+ \sigma_{1,+}^{(1,0)} + 1  +\mu(0)+  \sigma(k)  + k\cdot(n-k) \\
 & \equiv   \sigma(n+1) + \nu + n+1 + \sigma(n-k+1)+\nu + n+1+ \sigma_{1,-}^{(1,0)}  + k \cdot n\\
 &\quad{}+ \sigma_{1,+}^{(1,0)} + 1  +\mu(0) + \sigma(k) + k \cdot (n-k)\\
 &\equiv   \sigma(n+1)  + \sigma(n-k+1) +  \sigma(k) + \sigma_{1,-}^{(1,0)} + \sigma_{1,+}^{(1,0)} + \mu(0) + k +1 \\
 &\equiv  n\cdot k +k + 1, \pmod{2},
 \end{align*}
where the last step follows from Lemma \ref{lma:sigmarelations} and Lemma \ref{lma:mu0}. Similarly we get that 
\begin{equation*}
 \sigma_{\negg,1}^{(0,1)}(k) \equiv \sigma(n+1)  + \sigma(n-k+1) +  \sigma(k) + \sigma_{1,-}^{(0,1)} + \sigma_{1,+}^{(0,1)} + \mu(0) + k , \pmod{2},
\end{equation*}
and by using Lemma \ref{lma:sigmarelations}  we hence get that 
\begin{equation*}
 \sigma_{\negg,1}^{(0,1)}(k) \equiv n\cdot k + \sigma_{1,-}^{(0,1)} + \sigma_{1,+}^{(0,1)} + \mu(0) + k , \pmod{2}.
\end{equation*}
Now it follows from Lemma \ref{lma:mu0}, Lemma \ref{lma:oddsigmarel} and Lemma \ref{lma:oddend} that 
\begin{equation}\label{eq:good}
  \sigma_{1,-}^{(0,1)} + \sigma_{1,+}^{(0,1)} + \mu(0) \equiv  \sigma_{1,-}^{(0,1)},
\end{equation} 
and hence 
\begin{equation*}
 \sigma_{\negg,1}^{(0,1)}(k) \equiv n\cdot k + \sigma_{1,-}^{(0,1)}  + k , \pmod{2}.
\end{equation*}
 
 In the case $k=0$ we get that 
  \begin{align*}
 \sigma_{\negg,1}^{(1,0)}(0) & \equiv  \sigma(n+1) + \nu + n+1 + \sigma_{1,-}^{(1,0)} + \sigma_{0,-}^{(0,0)}  + \sigma(n+1)  + \nu + n+1\\
 &\quad{}+\sigma_{1,+}^{(1,0)} + 1  +\mu(0) + \sigma(0) \\
 &\equiv   \sigma_{1,-}^{(1,0)} + \sigma_{0,-}^{(0,0)}+\sigma_{1,+}^{(1,0)} + 1  +\mu(0) \\
 &\equiv \sigma_{0,-}^{(0,0)} + 1, \pmod{2},
 \end{align*}
 and similarly that 
 \begin{equation*}
  \sigma_{\negg,1}^{(0,1)}(0) \equiv \sigma_{0,-}^{(1,1)}+\sigma_{1,-}^{(0,1)} + \sigma_{1,+}^{(0,1)} + \mu(0), \pmod{2},
 \end{equation*}
 so in this case the result again follows from \eqref{eq:good}.
\end{proof}

\begin{cor}
Assume that $n>1$. Then
\begin{align*}
 \sigma_{\negg,1}^{(l,l)}(k) &= n\cdot (k + 1), \qquad l \in \{0,1\},\\
 \sigma_{\negg,1}^{(0,1)}(k) &= n\cdot k +k + 1, \\
 \sigma_{\negg,1}^{(1,0)}(k) &=  n\cdot k +k + 1,
\end{align*} 
and 
\begin{align*}
 \sigma_{\negg,1}^{(l,l)}(0) &=  1,\qquad l \in \{0,1\},\\
   \sigma_{\negg,1}^{(0,1)}(0) &= n, \\
\sigma_{\negg,1}^{(1,0)}(0) &= n.
\end{align*} 
\end{cor}
\begin{proof}
This follows from Lemma \ref{lma:1valsign} and Corollary \ref{lma:signallspeccap}.
\end{proof}

\subsection{Derivation of the signs of elementary trees corresponding to 2--valent punctures}\label{sec:2valentsigns}
Now we derive the signs for the elementary trees with 2--valent punctures in the case $n>1$. Recall that we have choosen signs of the capping operators corresponding to positive punctures of index $0$ so that we get
\begin{equation}\label{eq:sigma2posrel}
 \sigma_{\pos,2}^{(0,0)}(0,0) = \sigma_{\pos,2}^{(1,1)}(0,0)=\sigma_{\pos,2}^{(0,1)}(1,0)=\sigma_{\pos,2}^{(1,0)}(1,0)=0.                                                                                                                                                                                                                                                                                                               \end{equation}
 Hence it remains to compute the signs of $\sigma_{\pos,2}^{(1,1)}(0,0), \sigma_{\pos,2}^{(1,1)}(1,1), \sigma_{\pos,2}^{(1,0)}(0,1)$ and $\sigma_{\pos,2}^{(0,1)}(0,1)$.
 Also, from  Section \ref{sec:results} and Lemma \ref{lma:y011sign} we see that it is enough to compute the signs  $\sigma_{\negg,2}^{(0,0)}(0,0,1),\sigma_{\negg,2}^{(1,1)}(0,0,1),\sigma_{\negg,2}^{(0,1)}(1,0,1),\sigma_{\negg,2}^{(1,0)}(1,0,1)$ for $n>1$. We also find expressions of some of the $\sigma_{\pos,2}$- and $\sigma_{\negg,2}$-signs in the case when $n=1$. 

 All this is computed similarly to how we computed the sign functions in Section \ref{sec:signs1val}, that is, we compare the expression for the canonical orientation of the corresponding fully capped problem derived from the capping sequence with a known expression for this canonical orientation.
 
 \begin{lma}\label{lma:p2+1}
  We have that 
  \begin{align*}
   \sigma_{\pos,2}^{(0,1)}(0,1) &\equiv \sigma_{0,-}^{(0,0)} + \sigma_{0,-}^{(1,1)}, \\
   \sigma_{\pos,2}^{(1,0)}(0,1) &\equiv \sigma_{0,-}^{(0,0)} + \sigma_{0,-}^{(1,1)}.
  \end{align*}
 \end{lma}

 \begin{proof}
  The trivialized Lagrangian boundary condition for the fully capped $\sigma_{\pos,2}^{(0,1)}(0,1)$--problem equals that of the fully capped $\sigma_{\pos,2}^{(0,1)}(1,0)$--problem. By comparing the capping sequences for these two problems we see that 
  \begin{equation*}
   \sigma_{\pos,2}^{(0,1)}(0,1) \equiv \sigma_{\pos,2}^{(0,1)}(1,0) + \sigma_{0,-}^{(0,0)} + \sigma_{0,-}^{(1,1)}.
  \end{equation*}
The result now follows from the assumptions \eqref{eq:sigma2posrel} . 

The argument in the $\sigma_{\pos,2}^{(1,0)}(0,1)$--case is completely analogous.
 \end{proof}
 \begin{cor}
  Assume that $n>1$. Then we have that 
  \begin{align*}
   \sigma_{\pos,2}^{(0,1)}(0,1) &\equiv 0, \\
   \sigma_{\pos,2}^{(1,0)}(0,1) &\equiv 0.
  \end{align*}
 \end{cor}
 \begin{proof}
  This follows from Lemma \ref{lma:p2+1} and Corollary \ref{lma:signposcapev}.
 \end{proof}

 \begin{lma}
  Assume that $n>1$. Then we have that
 \begin{align*}
    \sigma_{\pos,2}^{(0,0)}(1,1) & \equiv 1,\\
    \sigma_{\pos,2}^{(1,1)}(1,1) & \equiv 1, \pmod{2}.
 \end{align*}
 \end{lma}
\begin{proof}
Consider first the  $\sigma_{\pos,2}^{(0,0)}(1,1)$--problem. 
 The corresponding capping sequence is given by
    \begin{align}\label{eq:cappingneed}
 0 \to 
 \begin{bmatrix}
   \langle w_1\partial_{x_{n+2}}\rangle \\
    \langle w_2\partial_{x_{n+2}}\rangle 
 \end{bmatrix}
 \xrightarrow{\alpha} 
 \begin{bmatrix}
   \langle i(z-1)\partial_{x_{n+2}} \rangle \\
  \langle i(z-1)\partial_{x_{n+2}} \rangle \\
0\\
0
\end{bmatrix}
\xrightarrow{\beta}  
\begin{bmatrix}
 0\\
 0\\
   \begin{bmatrix}
  \langle \partial_{x_1}\rangle \\
  \vdots \\
   \langle \partial_{x_{n}}\rangle
   \\
 0
 \end{bmatrix}
\end{bmatrix} 
\xrightarrow{\gamma}  
\begin{bmatrix}
 \langle \partial_{x_{1}}\rangle \\
 \vdots \\
 \langle \partial_{x_{n}}\rangle
 \end{bmatrix} 
 \to 0,
\end{align}
where $w_1,w_2$ is a choice of basis of the kernel of the Maslov problem with trivialized Lagrangian boundary condition $e^{2\pi i s}$ so that $w_1\partial_{x_{n+2}}$ is mapped to a positive multiple of the $(0,1)$--capping kernel and $w_2\partial_{x_{n+2}}$ is mapped to a positive multiple of the $(1,0)$--capping kernel under the gluing map $\alpha$.
Hence the capping sequence induces the orientation
  \begin{equation}\label{eq:lat1}
   (-1)^{\sigma^{(1,0)}_{1,-} + \sigma^{(0,1)}_{1,-} + \sigma^{t,(0,0)}_{+}(0) + \sigma_{\pos,2}^{(0,0)}(1,1)} w_1 \partial_{x_{n+2}} \wedge  w_2 \partial_{x_{n+2}} \otimes \partial_{x_1} \wedge \dotsm \wedge \partial_{x_{n}}
  \end{equation}
  on the fully capped problem. 
  
 Next we note that the fully capped problem has trivialized Lagrangian boundary condition given by
  \begin{align}\label{eq:n1ter}
   &\diag(\underbrace{e^{-2 \pi i s}, \dotsc,e^{-2 \pi i s}}_{n},e^{-\pi i s},e^{-\pi i s})  \hat * R_{\pi}(n+1,n+2) \hat *\\ 
   &{}\quad R_{-\pi}(n,n+2) * \diag(\underbrace{1,\dotsc,1}_{n-1},-1,1, -e^{\pi i s})\hat *R_{-\pi}(n,n+2) * \diag(\underbrace{1,\dotsc,1}_{n+1}, -e^{\pi i s}). \nonumber
  \end{align}
  This boundary condition we get by gluing the boundary condition $J_n$ (after a change of coordinates) to the fully capped boundary condition for the $\sigma_{Y_0}^{(0,0)}(1,1)$--problem. The former has canonical orientation 
  \begin{equation*}
   (-1)^{\sigma(n)} \partial_{x_{n+1}} \wedge \partial_{x_{n+2}} \otimes \partial_{x_1} \wedge \dotsm \wedge \partial_{x_n}
  \end{equation*}
  and the latter has canonical orientation 
  \begin{equation*}
   (-1)^{\sigma_{Y_0}^{(0,0)}(1,1)+\sigma^{(1,0)}_{1,-} + \sigma^{(0,1)}_{1,-} + \sigma^{(0,0)}_{0,+} }\partial_{x_1} \wedge \dotsm \wedge \partial_{x_n}\wedge w_1 \partial_{x_{n+2}} \wedge w_2 \partial_{x_{n+2}}.
  \end{equation*}
Gluing them together gives the canonical orientation 
  \begin{equation}\label{eq:lat2}
   (-1)^{\sigma_{Y_0}^{(0,0)}(1,1)+\sigma^{(1,0)}_{1,-} + \sigma^{(0,1)}_{1,-} + \sigma^{(0,0)}_{0,+}+ \sigma(n)} w_1 \partial_{x_{n+2}} \wedge w_2 \partial_{x_{n+2}}\otimes \partial_{x_1} \wedge \dotsm \wedge \partial_{x_n}
  \end{equation}
to the fully capped problem of our positive 2--piece, 
  
 The result now follows by comparing \eqref{eq:lat1} and \eqref{eq:lat2} and using Lemma \ref{lma:tp3}, Corollary \ref{lma:y011sign} and Corollary \ref{lma:signposcapev}.

  The argument for the $\sigma_{\pos,2}^{(1,1)}(1,1)$--problem is similar. That is, the corresponding fully capped problem has trivialized Lagrangian boundary condition given by 
  \begin{align*}
   &\diag(\underbrace{e^{-2 \pi i s}, \dotsc,e^{-2 \pi i s}}_{n},e^{-\pi i s},-e^{-\pi i s})  \hat * R_{-\pi}(n+1,n+2)*\\
   &{} \diag(\underbrace{1,\dotsc,1}_{n+1}, -e^{\pi i s})\hat*R_{-\pi}(n,n+2) * \diag(\underbrace{1,\dotsc,1}_{n-1},-1,1, -e^{\pi i s}) \hat * R_{-\pi}(n,n+2) .
  \end{align*}
  This boundary condition we get by gluing the boundary condition $J_n$ (after a change of coordinates) to the fully capped boundary condition for the $\sigma_{Y_0}^{(1,1)}(1,1)$--problem. The former has canonical orientation 
  \begin{equation*}
   (-1)^{\sigma(n)+1} \partial_{x_{n+1}} \wedge \partial_{x_{n+2}} \otimes \partial_{x_1} \wedge \dotsm \wedge \partial_{x_n}
  \end{equation*}
  and the latter has canonical orientation 
  \begin{equation*}
   (-1)^{\sigma_{Y_0}^{(1,1)}(1,1)+\sigma^{(0,1)}_{1,-} + \sigma^{(1,0)}_{1,-} + \sigma^{(1,1)}_{0,+} }\partial_{x_1} \wedge \dotsm \wedge \partial_{x_n}\wedge w_1 \partial_{x_{n+2}} \wedge w_2 \partial_{x_{n+2}}.
  \end{equation*}
Gluing them together gives the canonical orientation 
  \begin{equation*}
   (-1)^{\sigma_{Y_0}^{(1,1)}(1,1)+\sigma^{(0,1)}_{1,-} + \sigma^{(1,0)}_{1,-} + \sigma^{(1,1)}_{0,+}+ (\sigma(n)+1) +1} w_1 \partial_{x_{n+2}} \wedge w_2 \partial_{x_{n+2}}\otimes \partial_{x_1} \wedge \dotsm \wedge \partial_{x_n}
  \end{equation*}
to the fully capped problem of our positive 2--piece. Here we get one factor $(-1)^1$ from the gluing sequence, since we are gluing in the trivialization given by $-\R^{n+2}$. 
 
  The capping sequence is given by \eqref{eq:cappingneed} and
induces the orientation
  \begin{equation*}
   (-1)^{\sigma^{(0,1)}_{1,-} + \sigma^{(1,0)}_{1,-} + \sigma^{t,(1,1)}_{+}(0) + \sigma_{\pos,2}^{(1,1)}(1,1)} w_1 \partial_{x_{n+2}} \wedge  w_2 \partial_{x_{n+2}} \otimes \partial_{x_1} \wedge \dotsm \wedge \partial_{x_{n}}
  \end{equation*}
  
  Using that this should give the canonical orientation together with Lemma \ref{lma:tp3}  we get that 
\begin{equation*}
  \sigma_{\pos,2}^{(1,1)}(1,1) \equiv \sigma_{Y_0}^{(1,1)}(1,1) + \sigma_{0,+}^{(1,1)} + 1.
\end{equation*}

Now the result follows Section \ref{sec:results}, which says that $\sigma_{Y_0}^{(0, 0)}(1,1)  \equiv \sigma_{Y_0}^{(1, 1)}(1,1)$ modulo 2, together with Corollary \ref{lma:y011sign} and Corollary \ref{lma:signposcapev}.

\end{proof}

\begin{lma}
 We have that 
   \begin{align*}
  \sigma_{\negg,2}^{(0,0)}(0,0,1) &\equiv  \sigma_{\negg,2}^{(0,0)}(0,0,2) \equiv \sigma_{0,+}^{(0,0)}, \\
  \sigma_{\negg,2}^{(1,1)}(0,0,1) &\equiv  \sigma_{\negg,2}^{(1,1)}(0,0,2) \equiv \sigma_{0,+}^{(1,1)} + 1, \pmod{2}.  
 \end{align*}
 In particular, if $n>1$ we get 
 \begin{align*}
  \sigma_{\negg,2}^{(0,0)}(0,0,1) &\equiv  \sigma_{\negg,2}^{(0,0)}(0,0,2) \equiv n+1, \\
  \sigma_{\negg,2}^{(1,1)}(0,0,1) &\equiv  \sigma_{\negg,2}^{(1,1)}(0,0,2) \equiv n +1, \pmod{2}.  
 \end{align*}

\end{lma}

\begin{proof}
We focus on the $\sigma_{\negg,2}^{(0,0)}(0,0,1)$, $\sigma_{\negg,2}^{(1,1)}(0,0,1)$--problems. The other two are completely similar. 

 First we have that the trivialized Lagrangian boundary condition for the fully capped $\sigma_{\negg,2}^{(0,0)}(0,0,1)$--problem is given by
  \begin{align*}
   &\diag(\underbrace{1,\dotsc,1}_{n},,e^{-\pi i s},e^{-\pi i s}) \hat* R_\pi(n+1,n+2) \hat * \\
   &{}\qquad R_{-\pi}(n+1,n+2)*\diag(\underbrace{e^{-2 \pi i s}, \dotsc,e^{-2 \pi i s}}_{n},-e^{- \pi i s},-e^{\pi i s}),
  \end{align*}
  and the trivialized Lagrangian boundary condition for the fully capped $\sigma_{\negg,2}^{(1,1)}(0,0,1)$--problem is given by
    \begin{align*}
   &\diag(\underbrace{1,\dotsc,1}_{n},e^{-\pi i s},-e^{-\pi i s}) \hat* R_{-\pi}(n+1,n+2) \hat * \\
   &{}\qquad R_{\pi}(n+1,n+2)*\diag(\underbrace{e^{-2 \pi i s}, \dotsc,e^{-2 \pi i s}}_{n},-e^{- \pi i s},e^{\pi i s}).
  \end{align*}
  
  We see that the first one is homotopic to the boundary condition $I_{n+1}$, and the second one is also homotopic to $I_{n+1}$ after the change of coordinates $\partial_{x_{n+2}} \leftrightarrow -\partial_{x_{n+2}}$. Hence the fully capped  $\sigma_{\negg,2}^{(0,0)}(0,0,1)$--problem has canonical orientation 
    \begin{equation*}
   (-1)^{\sigma(n+1) + \nu + n+1} v \partial_{x_{n+2}} \otimes \partial_{x_1} \wedge \dotsm \wedge \partial_{x_{n+1}},
  \end{equation*}
  and the fully capped  $\sigma_{\negg,2}^{(1,1)}(0,0,1)$--problem has canonical orientation 
    \begin{equation*}
   (-1)^{\sigma(n+1) + \nu + n} v \partial_{x_{n+2}} \otimes \partial_{x_1} \wedge \dotsm \wedge \partial_{x_{n+1}}.
  \end{equation*}
  Both problems have corresponding capping sequence
    \begin{align*}
 0 \to 
   \langle v\partial_{x_{n+2}}\rangle 
 \xrightarrow{\alpha} 
 \begin{bmatrix}
  \langle i(z-1)\partial_{x_{n+2}} \rangle \\
  0 \\
0\\
0
\end{bmatrix}
\xrightarrow{\beta}  
\begin{bmatrix}
 0\\
 0\\
   \begin{bmatrix}
  \langle \partial_{x_1}\rangle \\
  \vdots \\
   \langle \partial_{x_{n+1}}\rangle
 \end{bmatrix}
\end{bmatrix} 
\xrightarrow{\gamma}  
\begin{bmatrix}
 \langle \partial_{x_{1}}\rangle \\
 \vdots \\
 \langle \partial_{x_{n+1}}\rangle
 \end{bmatrix} 
 \to 0,
\end{align*}
which induces the orientation
  \begin{equation*}
   (-1)^{\sigma_{0,+} + \sigma_{0,-} + \sigma^{t}_{-}(n) + \sigma_{\negg,2}(0,0,1)} v \partial_{x_{n+2}} \otimes \partial_{x_1} \wedge \dotsm \wedge \partial_{x_{n+1}}
  \end{equation*}
  on the fully capped problems. Here $\sigma_{0,+},\sigma_{0,-},\sigma^{t}_{-}(n), \sigma_{\negg,2}(0,0,1)$ should be decorated with $(0,0)$ or $(1,1)$, depending on which case we are considering.
  
 The result now follows by using Lemma \ref{lma:tp2}, Lemma \ref{lma:tp1} for $n=1$, and also Corollary \ref{lma:signposcapev} for $n>1$. 
 \end{proof}

\begin{lma}
 We have that
  \begin{align*}
      \sigma_{\negg,2}^{(1,0)}(1,0,1) &\equiv \sigma_{0,-}^{(0,0)} + 1, \\
      \sigma_{\negg,2}^{(1,0)}(1,0,2) &\equiv \sigma_{0,-}^{(1,1)} + 1 ,\\
     \sigma_{\negg,2}^{(0,1)}(1,0,1) &\equiv  \sigma_{0,-}^{(1,1)} + \sigma^{(0,1)}_{1,-},\\
     \sigma_{\negg,2}^{(0,1)}(1,0,2) &\equiv \sigma_{0,-}^{(0,0)} + \sigma^{(0,1)}_{1,-} ,\pmod{2}.
 \end{align*}
 
 In particular, if $n>1$ we get that 
 \begin{align*}
      \sigma_{\negg,2}^{(0,1)}(1,0,1) &\equiv  \sigma_{\negg,2}^{(0,1)}(1,0,2) \equiv n ,\\
     \sigma_{\negg,2}^{(1,0)}(1,0,1) &\equiv  \sigma_{\negg,2}^{(1,0)}(1,0,2) \equiv n ,\pmod{2}.
 \end{align*}
\end{lma}

 \begin{proof}
 We first consider the $\sigma_{\negg,2}^{(0,1)}(1,0,1)$-- and $\sigma_{\negg,2}^{(1,0)}(1,0,1)$--problems, and then we point out the adjustments for the  $\sigma_{\negg,2}^{(0,1)}(1,0,2)$-- and $\sigma_{\negg,2}^{(1,0)}(1,0,2)$--problems.
 
  The fully capped $\sigma_{\negg,2}^{(1,0)}(1,0,1)$--problem has trivialized Lagrangian boundary condition given by
  \begin{align}\label{eq:bcnegpos2}
   &\diag(\underbrace{1,\dotsc,1}_{n},e^{-\pi i s},1) \hat * R_{-\pi}(n+1,n+2)*\diag(\underbrace{e^{-2 \pi i s}, \dotsc,e^{-2 \pi i s}}_{n},1,-e^{\pi i s}),
  \end{align}
  and the fully capped $\sigma_{\negg,2}^{(0,1)}(1,0,1)$--problem has trivialized Lagrangian boundary condition given by
 \begin{align*}
   &\left( R_{\pi}(n,n+2)\cdot\diag(\underbrace{1,\dotsc,1}_{n+1},-1)  \right)*\diag(\underbrace{1,\dotsc,1}_{n-1},-1,e^{-\pi i s},1) \hat * R_{-\pi}(n+1,n+2) \hat * \\
   &{}\quad R_\pi(n,n+2) \hat *R_{-\pi}(n,n+2)* \diag(\underbrace{e^{-2 \pi i s}, \dotsc,e^{-2 \pi i s}}_{n-1},-e^{-2 \pi i s},1,-e^{\pi i s}) \hat * R_{-\pi}(n,n+2),
  \end{align*}
  which is homotopic to the Lagrangian boundary condition \eqref{eq:bcnegpos2} after the change of coordinates $\partial_{x_n} \leftrightarrow -\partial_{x_n}$. Hence these two problems have opposite canonical orientation. To compute this, notice that the boundary condition \eqref{eq:bcnegpos2} is homotopic to the boundary condition we get by gluing the boundary condition \eqref{eq:bc-+} to the boundary condition $J_{n}$ (after a change of coordinates for the latter). The former has canonical orientation
      \begin{equation*}
   (-1)^{\mu(0)}  \partial_{x_1} \wedge \dotsm \wedge \partial_{x_{n}} \wedge v_1 \partial_{x_{n+2}} \wedge v_2 \partial_{x_{n+2}}, 
  \end{equation*}
  and the latter has canonical orientation 
    \begin{equation*}
   (-1)^{\sigma(n)}  \partial_{x_{n+1}} \wedge \partial_{x_{n+2}} \otimes  \partial_{x_1} \wedge \dotsm \wedge \partial_{x_{n}}. 
  \end{equation*}
  Gluing these two problems gives that the Maslov problem with boundary conditions given by \eqref{eq:bcnegpos2} has canonical orientation
    \begin{equation}\label{eq:10010}
   (-1)^{\sigma(n) + \mu(0) + n+n} v_1 \partial_{x_{n+2}} \wedge v_2 \partial_{x_{n+2}}\otimes \partial_{x_1} \wedge \dotsm \wedge \partial_{x_{n}}.
  \end{equation}
  Here one of the $n$--summands comes from the determinant of the gluing map $\beta$ in the corresponding gluing sequence, and the other one from rearrangements in the third non-trivial column of the gluing sequence. Hence the fully capped   $\sigma_{\negg,2}^{(0,1)}(1,0,1)$--problem has canonical orientation given by \eqref{eq:10010}, and the  fully capped $\sigma_{\negg,2}^{(1,0)}(1,0,1)$--problem has canonical orientation given by 
     \begin{equation*}
   (-1)^{\sigma(n) + \mu(0) + 1} v_1 \partial_{x_{n+2}} \wedge v_2 \partial_{x_{n+2}}\otimes \partial_{x_1} \wedge \dotsm \wedge \partial_{x_{n}}.
  \end{equation*}
  
In both cases we get that the capping sequence is given by
    \begin{align*}
 0 \to 
 \begin{bmatrix}
   \langle v_1\partial_{x_{n+2}}\rangle \\
    \langle v_2\partial_{x_{n+2}}\rangle
 \end{bmatrix}
    \xrightarrow{\alpha} 
 \begin{bmatrix}
  \langle i(z-1)\partial_{x_{n+2}} \rangle \\
  0 \\
  \langle \partial_{x_{n+2}} \rangle \\\\
0
\end{bmatrix}
\xrightarrow{\beta}  
\begin{bmatrix}
   \begin{bmatrix}
  \langle \partial_{x_1}\rangle \\
  \vdots \\
   \langle \partial_{x_{n}}\rangle
 \end{bmatrix}\\
 0\\
 0\\
 0
\end{bmatrix} 
\xrightarrow{\gamma}  
\begin{bmatrix}
 \langle \partial_{x_{1}}\rangle \\
 \vdots \\
 \langle \partial_{x_{n}}\rangle
 \end{bmatrix} 
 \to 0,
\end{align*}
which induces the orientation
  \begin{equation*}
   (-1)^{1+\sigma^{(1,0)}_{1,+} + \sigma^{(0,0)}_{0,-} + \sigma^{t,(1,0)}_{-}(n) + \sigma_{\negg,2}^{(1,0)}(1,0,1)} v_1 \partial_{x_{n+2}} \wedge v_2 \partial_{x_{n+2}}\otimes \partial_{x_1} \wedge \dotsm \wedge \partial_{x_{n}}
  \end{equation*}
  on the fully capped  $\sigma_{\negg,2}^{(1,0)}(1,0,1)$--problem, and the orientation
    \begin{equation*}
   (-1)^{1+\sigma^{(0,1)}_{1,+} + \sigma^{(1,1)}_{0,-} + \sigma^{t,(0,1)}_{-}(n) + \sigma_{\negg,2}^{(0,1)}(1,0,1)} v_1 \partial_{x_{n+2}} \wedge v_2 \partial_{x_{n+2}}\otimes \partial_{x_1} \wedge \dotsm \wedge \partial_{x_{n}}
  \end{equation*}
   on the fully capped  $\sigma_{\negg,2}^{(0,1)}(1,0,1)$--problem.
  
 The result now follows by using Lemma \ref{lma:tp2}, Lemma \ref{lma:tp1}, Lemma \ref{lma:mu0} and Lemma \ref{lma:sigmarelations}  in the $(1,0)$--case, and also \eqref{eq:good} in the $(0,1)$--case.

 The $\sigma_{\negg,2}^{(1,0)}(1,0,2)$--case is similar to $\sigma_{\negg,2}^{(1,0)}(1,0,1)$, except that the sign $\sigma^{(0,0)}_{0,-}$ from the capping sequence should be replaced by $\sigma^{(1,1)}_{0,-}$. Similar arguments apply to the $\sigma_{\negg,2}^{(0,1)}(1,0,2)$--case by comparing with $\sigma_{\negg,2}^{(0,1)}(1,0,1)$.

 For the statements when $n>1$ we use Corollary \ref{lma:signallspeccap}.
 \end{proof}

\subsection{Values of all sign functions}\label{sec:allresults}
  Now we give a summary of the values of all sign functions.
  
 \subsubsection{Values and relations for \texorpdfstring{$\sigma_{Y_0}$}{Y}, modulo 2} 
  \begin{align*}
&\sigma_{Y_0}^{(1, 1)}(0,0)  \equiv \sigma_{Y_0}^{(0, 0)}(0,0)   \equiv n+1, \\
& \sigma_{Y_0}^{(0, 1)}(0,1)  \equiv \sigma_{Y_0}^{(1, 0)}(0,1)  \equiv n+1, \\ 
& \sigma_{Y_0}^{(0, 1)}(1,0)  \equiv \sigma_{Y_0}^{(1, 0)}(1,0)    \equiv n+1, \\
& \sigma_{Y_0}^{(0, 0)}(1,1)  \equiv \sigma_{Y_0}^{(1, 1)}(1,1)  \equiv n. 
\end{align*}
 
\subsubsection{Values and relations for \texorpdfstring{$\sigma_{Y_1}$}{Y}, modulo 2}
$ $ \\
$n=1$:
\begin{align*}
&\sigma_{Y_1}^{(0,0)}(1,0) \equiv  \sigma_{Y_1}^{(1,1)}(1,0) \equiv  \sigma_{Y_1}^{(1,0)}(1,1) \equiv  \sigma_{Y_1}^{(0,1)}(1,1)   \equiv 1,  \\ 
&\sigma_{Y_1}^{(0,0)}(0,1) \equiv  \sigma_{Y_1}^{(1,1)}(0,1)  \equiv 0,\\
 &\sigma_{Y_1}^{(0,1)}(0,0)  \equiv  \sigma_{Y_1}^{(1,0)}(0,0)  \equiv  0.
 \end{align*}
$n>1$:
\begin{align*}
&\sigma_{Y_1}^{(0,0)}(1,0) \equiv  \sigma_{Y_1}^{(1,1)}(1,0) \equiv  \sigma_{Y_1}^{(1,0)}(1,1) \equiv  \sigma_{Y_1}^{(0,1)}(1,1)   \equiv n + 1,  \\ 
&\sigma_{Y_1}^{(0,0)}(0,1) \equiv  \sigma_{Y_1}^{(1,1)}(0,1)  \equiv n, \\
 &\sigma_{Y_1}^{(0,1)}(0,0)  \equiv  \sigma_{Y_1}^{(1,0)}(0,0)  \equiv  n+1.
\end{align*}

\subsubsection{Values and relations for \texorpdfstring{$\sigma_{{\swi}}$}{sw}, modulo 2}
\begin{align*}
 &\sigma_{\swi}^{(0, 0)}(l) \equiv \sigma_{\swi}^{(1, 1)}(l) \equiv 0, \\
 &\sigma_{\swi}^{(0, 0)}(u) \equiv  \sigma_{\swi}^{(1, 1)}(u) \equiv 1, \\
 &\sigma_{\swi}^{(1, 0)}(u) \equiv  \sigma_{\swi}^{(0, 1)}(u) \equiv \sigma_{\swi}^{(1, 0)}(l) \equiv  \sigma_{\swi}^{(0, 1)}(l)  \equiv 0.
\end{align*}

\subsubsection{Values and relations for \texorpdfstring{$\sigma_{{\negg,1}}$}{P1-}, modulo 2} 
$ $\\

$n=1$:
\begin{align*}
 \sigma_{\negg,1}^{(0,0)}(1) &=  0, \\
  \sigma_{\negg,1}^{(1,1)}(1) &= 1,\\
 \sigma_{\negg,1}^{(0,1)}(1) &= \sigma_{1,-}^{(0,1)}, \\
 \sigma_{\negg,1}^{(1,0)}(1) &=  1  \\
 \sigma_{\negg,1}^{(l,l)}(0) &=  1 + l + \sigma_{0,+}^{(l,l)} ,\qquad l \in \{0,1\},\\
   \sigma_{\negg,1}^{(1,0)}(0) &= \sigma_{0,-}^{(0,0)} + 1, \\
\sigma_{\negg,1}^{(0,1)}(0) &= \sigma_{0,-}^{(1,1)}+\sigma_{1,-}^{(0,1)}.
\end{align*} 

$n>1$,$k>0$:
\begin{align*}
 \sigma_{\negg,1}^{(l,l)}(k) &= n\cdot (k + 1), \qquad l \in \{0,1\},\\
 \sigma_{\negg,1}^{(0,1)}(k) &= n\cdot k +k + 1, \\
 \sigma_{\negg,1}^{(1,0)}(k) &=  n\cdot k +k + 1, 
\end{align*} 
and 
\begin{align*}
 \sigma_{\negg,1}^{(l,l)}(0) &=  1,\qquad l \in \{0,1\},\\
   \sigma_{\negg,1}^{(0,1)}(0) &= n, \\
\sigma_{\negg,1}^{(1,0)}(0) &= n.
\end{align*} 

\subsubsection{Values and relations for \texorpdfstring{$\sigma_{{\pos,2}}$}{P2+}, modulo 2} 
\begin{equation*}
 \sigma_{\pos,2}^{(0,0)}(0,0) = \sigma_{\pos,2}^{(1,1)}(0,0)=\sigma_{\pos,2}^{(0,1)}(1,0)=\sigma_{\pos,2}^{(1,0)}(1,0)=0                                                                                                                                                                                                                                                                                                                                   \end{equation*}
 for all $n$. 
 
 For $n=1$:
  \begin{align*}
   &\sigma_{\pos,2}^{(0,1)}(0,1) \equiv \sigma_{\pos,2}^{(1,0)}(0,1) \equiv \sigma_{0,-}^{(0,0)} + \sigma_{0,-}^{(1,1)}. \\
  \end{align*}

   For $n>1$: 
  \begin{align*}
   &\sigma_{\pos,2}^{(0,1)}(0,1) \equiv \sigma_{\pos,2}^{(1,0)}(0,1) \equiv 0, \\
   &\sigma_{\pos,2}^{(0,0)}(1,1)  \equiv \sigma_{\pos,2}^{(1,1)}(1,1)  \equiv 1. 
 \end{align*}

\subsubsection{Values and relations for \texorpdfstring{$\sigma_{{\negg,2}}$}{P2-}, modulo 2} 
$ $ \\
 For $n>1$:
\begin{align*}
 \sigma^{(0,0)}_{\negg,2}(0, 0, 1) &\equiv  \sigma^{(0,0)}_{\negg,2}(0, 0, 2) \equiv   \sigma^{(0,1)}_{\negg,2}(0, 1, 1) \equiv n+1\\
 &\equiv   \sigma^{(1,0)}_{\negg,2}(0, 1, 2) + 1, \\
  \sigma^{(1,1)}_{\negg,2}(0, 0, 1) &\equiv   \sigma^{(1,1)}_{\negg,2}(0, 0, 2) \equiv   \sigma^{(1,0)}_{\negg,2}(0, 1, 1)  \equiv n+1\\
  &\equiv   \sigma^{(0,1)}_{\negg,2}(0, 1, 2)+ 1, \\
  \sigma^{(0,1)}_{\negg,2}(1, 0, 1) &\equiv  \sigma^{(0,1)}_{\negg,2}(1, 0, 2) \equiv
  \sigma^{(0,0)}_{\negg,2}(1, 1, 1) \equiv   \sigma^{(1,1)}_{\negg,2}(1, 1, 2) \equiv n, \\
    \sigma^{(1,0)}_{\negg,2}(1, 0, 1) &\equiv  \sigma^{(1,0)}_{\negg,2}(1, 0, 2) \equiv
 \sigma^{(1,1)}_{\negg,2}(1, 1, 1) \equiv  \sigma^{(0,0)}_{\negg,2}(1, 1, 2) \equiv n.\\
  \end{align*}
 

%
 
  
  For $n=1$:
  \begin{align*}
 \sigma^{(0,0)}_{\negg,2}(0, 0, 1) &\equiv  \sigma^{(0,1)}_{\negg,2}(0, 1, 1) \equiv \sigma_{0,+}^{(0,0)}, \\
  \sigma^{(1,0)}_{\negg,2}(0, 1, 2) &\equiv   \sigma^{(0,0)}_{\negg,2}(0, 0, 2)+ 1 \equiv \sigma_{0,+}^{(0,0)}+1, \\
  \sigma^{(1,1)}_{\negg,2}(0, 0, 2) &\equiv  \sigma^{(0,1)}_{\negg,2}(0, 1, 2) +1 \equiv \sigma_{0,+}^{(1,1)} + 1, \\
    \sigma^{(1,0)}_{\negg,2}(0, 1, 1) &\equiv  \sigma^{(1,1)}_{\negg,2}(0, 0, 1) \equiv \sigma_{0,+}^{(1,1)} + 1, 
  \end{align*}
  
    \begin{align*}
      \sigma_{\negg,2}^{(1,0)}(1,0,1) &\equiv \sigma_{0,-}^{(0,0)} + 1, \\
      \sigma_{\negg,2}^{(1,0)}(1,0,2) &\equiv \sigma_{0,-}^{(1,1)} + 1 ,\\
     \sigma_{\negg,2}^{(0,1)}(1,0,1) &\equiv  \sigma_{0,-}^{(1,1)} + \sigma^{(0,1)}_{1,-} ,\\
     \sigma_{\negg,2}^{(0,1)}(1,0,2) &\equiv \sigma_{0,-}^{(0,0)} + \sigma^{(0,1)}_{1,-},\\
 \end{align*}

  \begin{align*}
\sigma^{(0,0)}_{\negg,2}(1, 1, 1) +  \sigma^{(1,1)}_{\negg,2}(1, 1, 2) & \equiv \sigma_{0,-}^{(1,1)} +  \sigma_{0,-}^{(0,0)},\\
\sigma^{(0,0)}_{\negg,2}(1, 1, 1) +  \sigma^{(0,0)}_{\negg,2}(1, 1, 2) & \equiv
\sigma_{0,-}^{(0,1)} +  1,  \\
 \sigma^{(1,1)}_{\negg,2}(1, 1, 1) +  \sigma^{(1,1)}_{\negg,2}(1, 1, 2)& \equiv \sigma_{0,-}^{(0,1)} +  1,  \\
\sigma^{(1,1)}_{\negg,2}(1, 1, 1) +  \sigma^{(0,0)}_{\negg,2}(1, 1, 2) & \equiv \sigma_{0,-}^{(1,1)} +  \sigma_{0,-}^{(0,0)}.
  \end{align*}

 \appendix
\section{Families of Fredholm operators related to trees}\label{app:mini}

Here we discuss some results about families of Fredholm operators associated to Morse flow trees.

Recall that the pair $(A,\w_\nu)$ of a trivialized Lagrangian boundary condition $A$ and a weight $\w_\nu$ defined on a standard domain $\Delta$ is called admissible if the associated $\dbar$-operator $\dbar_A:\Hi_{2,\nu}[A] \to \Hi_{1,\nu}[0]$ is Fredholm. 

\begin{defi}
 Let $\Delta = \Delta_{m+1}$ be a standard domain and let $\w_\lambda$
 be a continuous family of weights defined on $\Delta$. Then a family
 $A_\lambda: \partial \Delta \to U(n)$, $\lambda \in [0,1]$, of
 trivialized boundary conditions is \emph{$C^2$-admissible} if 
 \begin{enumerate}
  \item $(A_\lambda ,\w_\lambda)$ is admissible for all $\lambda \in [0,1]$,
  \item $A_\lambda: \partial\Delta  \to U(n)$ is smooth and depends
   $C^2$ on $\lambda\in[0,1]$,  with  $\|A_\lambda\|_{C^2}$,  $\|A^{-1}_\lambda\|_{C^2}$ bounded for each $\lambda$,
  \item $\dbar A_\lambda|_{\partial \Delta} = 0$.
  \end{enumerate}
\end{defi}

Here the $C^k$-norm of a matrix $A: X \to GL(n,\C)$ is defined by
\begin{equation*}
 \|A\|_{C^k} = \sum_{|\alpha_j|\leq k} \sup_{x \in X}\|D^{\alpha_j}A(x)\|_{\max}  
\end{equation*}
where the derivative is taken component-wise and
 $\|A\|_{\max} = \max |a_{ij}|$.

\begin{lma}\label{lma:goodchoice}
Let $\Gamma$ be a rigid flow tree, an elementary tree or a sub flow
tree. Let $\w_\lambda$ be the weight associated to $\Gamma$ as in
Section \ref{sec:analys}. Then there is a $\lambda_0 > 0$ so that if
$\lambda\leq \lambda_0$, then the family of trivialized boundary
conditions associated to the pre-glued disk $w_\lambda$ of $\Gamma$ as described in Section \ref{sec:cutstab}, is $C^2$-admissible.
\end{lma}
\begin{proof}
This follows by the constructions in \cite{trees} together with
results found in [\cite{legsub}, Section 6]. In particular, from Lemma
\ref{lma:wfred} it follows that this holds for all $\lambda$ so small
so that $\theta_\lambda < \delta$.   
\end{proof}

The following proposition shows that under certain circumstances, the kernels of the operators $\dbar_{A_\lambda}$ constitute a vector bundle, with topology induced by $[0,1] \times \Hi_{2,\nu}(\Delta,\C^n)$.

\begin{prp}\label{prp:keriso}
Let  $A_\lambda $, $\lambda \in[0,1]$ be  a $C^2$-admissible family of
boundary conditions on $\Delta_{m+1}$ for a given a weight $\w = \w_\nu$. Assume that $\dbar_{A_0}$ is surjective, and let $(v_1,\dotsc,v_k)$ be a basis for $\kd_{A_0}$. Then there is a $\lambda_0 >0$ such that for every $\lambda \in [0, \lambda_0]$ there is a basis $(v_1^\lambda,\dotsc,v_k^\lambda)$ for $\kd_{A_\lambda}$ satisfying
\begin{equation*}
 \|v_i^\lambda - v_i\|_{2,\nu} \leq C \| A_\lambda - A_0\|_{C^2} 
\end{equation*}
where $C = C(\{A_\lambda\}_{\lambda\in[0,1]})$.
\end{prp}

\begin{rmk}
 We have a similar statement for the cokernel elements in the case when $\dbar_{A_0}$ is injective, obtained by considering the adjoint operator $\partial$. 
\end{rmk}

First we recall the following standard result.
\begin{lma}
 Let $X$ be a manifold, and let $\w_\nu$ be a weight function defined on $X$. If $A: X \to GL(n,\C)$ and $v:X \to \C^n$ then
 \begin{equation*}
  \|Av\|_{i,\nu} \leq C_i n \|A\|_{C^i} \|v\|_{i,\nu}, \qquad  i=1,2,
 \end{equation*}
where $\| \cdot\|_{i,\nu}$ denotes the norm on $\Hi_{i,\nu}(X,\C^n)$
associated to the weight $\w = \w_\nu$, and $C_i>0$ is a constant that
only depends on the Sobolev regularity $i$.
\end{lma}
\begin{proof}
 For $B \in GL(n,\C)$ let $\|B\|_{2}$ denote the operator norm of $B$, viewed as a linear operator, i.e.
 \begin{equation*}
  \|B\|_{2} = \sup_{\|z\|_{\C^n} =1} \|Bz\|_{\C^n}.
 \end{equation*}
 
 Then we have the following inequality
 \begin{equation*}
  \|B\|_2 \leq n \|B\|_{\max}
 \end{equation*}
and  hence 
\begin{equation*}
\langle Av, Av\rangle (x) \leq \|A(x)\|_2^2\langle v,v\rangle (x) \leq n^2\|A(x)\|_{\max}^2\langle v,v\rangle (x) 
\end{equation*}
 point-wise.

This implies that 
 \begin{equation*}
  \|Av \w\|_{L^2}^2 = \int_X \langle Av\w, Av\w \rangle dm \leq n^2\|A\|_{C^0}^2 \int_X\langle v\w, v\w \rangle dm = n^2\|A\|^2_{C^0}  \|v\w\|^2_{L^2},
 \end{equation*}
 \begin{multline*}
  \|D(Av\w)\|_{L^2} =  \|D(A)v\w +AD(v\w)\|_{L^2} \leq \|D(A)v\w\|_{L^2} +  \|AD(v\w)\|_{L^2} \\
   \leq  n\|D A\|_{C^0}  \|v\w\|_{L^2} +  n\|A\|_{C^0}  \|D(v\w)\|_{L^2},
 \end{multline*}
so
\begin{align*}
 \MoveEqLeft[32]{ \|Av\|_{1,\nu}^2 = \|Av\w\|_{L^2}^2 +\|D(Av\w)\|_{L^2}^2}&\\
\leq n^2 \|A\|^2_{C^0}( \|v\w\|_{L^2}^2 +\|D(v\w)\|_{L^2}^2)+  n^2\|D A\|^2_{C^0}  \|v\w\|_{L^2}^2& \\
\quad {} + 2n^2\|D A\|_{C^0}\|A\|_{C^0} \|v\w\|_{L^2}\|D(v\w)\|_{L^2} &\\
 \leq 10 n^2 \|A\|^2_{C^1}\|v\|^2_{1,\nu}.&
\end{align*}

The proof for $i=2$ is completely similar, but here we also have to
make an estimate for  $\|D^2(Av\w)\|_{L^2}$.
\end{proof}

\begin{proof}[Proof of Proposition \ref{prp:keriso}]
 By assumptions the operator $\dbar_{A_\lambda} : \Hi_{2,\nu}[A_\lambda] \to \Hi_{1,\nu}[0]$ is Fredholm for all $\lambda \in [0,1]$. Now, for each $\lambda \in [0,1]$, consider the diagram  
  \begin{align*}
 \xymatrix{
\Hi_{2,\nu}[A_0] \ar[d]^{\dbar_{A_0}} \ar[r]^{A_\lambda A_0^{-1}} 
& \Hi_{2,\nu}[A_\lambda] \ar[d]^{\dbar_{A_\lambda}} \\
\Hi_{1,\nu}[0] \ar[r]^{A_\lambda A_0^{-1}} 
& \Hi_{1,\nu}[0],
}
\end{align*}
 where the multiplication by $A_\lambda A_0^{-1}$ induce isomorphisms. From this we get an operator 
 \begin{equation*}
  \check{A}_\lambda := \dbar + A_0A_\lambda^{-1}\dbar(A_\lambda A_0^{-1}): \Hi_{2,\nu}[A_0] \to \Hi_{1,\nu}[0]
 \end{equation*}
 which is conjugate to $\dbar_{A_\lambda}$. In particular,
 $\check{A}_\lambda$ gives a family of Fredholm operators. And since
 $\check{A}_{0}$ is assumed to be surjective it follows from standard
 arguments in Fredholm theory that  there is a $\tilde \lambda_0 > 0$
 so that $\check{A}_{\lambda}$ also is surjective for $\lambda
 \leq \tilde \lambda_0$, and that the kernels of these operators
 constitute a vector bundle over $[0, \tilde \lambda_0]$.

Moreover, we have the following estimate. 
 \begin{align*}
   &\|(\check{A}_{0}-\check{A}_{\lambda})v\|_{1,\nu} \\
   & \quad{} \leq n^4C_1^4\|A_0\|_{C^1} \|A^{-1}_\lambda\|_{C^1}(\|\dbar(A_\lambda
   -A_0) \|_{C^1}\|A_0^{-1}\|_{C_1} + \|A_\lambda
   -A_0 \|_{C^1}\|\dbar A_0^{-1}\|_{C^1})\|v\|_{2,\nu} \\
   & \quad{}\leq n^4C_1^4 E \|A_\lambda - A_0\|_{C^2}\|v\|_{2,\nu},
 \end{align*}
 where $E = E(\|A_0\|_{C^2},
 \{\|A_\lambda^{-1}\|_{C^2}\}_{\lambda \in [0,1]})$. This means that
 we have an estimate
 \begin{equation}\label{eq:est1}
   \|\check{A}_{0}-\check{A}_{\lambda}\|_{\op} \leq  n^4C_1^4 E \|A_\lambda - A_0\|_{C^2}
 \end{equation}
where the operator norm $\|\cdot \|_{\op}$ is with respect to
$(\check{A}_{0}-\check{A}_{\lambda})$ regarded as a map from
$\Hi_{2,\nu}[A_0]$ to $\Hi_{1,\nu}[0]$.

 Now let $T$ be a bounded right inverse for $\check A_0$ with norm
 $\|T\|_{\op}$ regarded as an operator $T: \Hi_{1,\nu}[0[ \to
 \Hi_{2,\nu}[A_0]$. Let $\lambda_0'$ be so small so that
 \begin{equation}\label{eq:est2}
   \|A_\lambda - A_0\|_{C^2}\leq \frac{1}{2 n^4C_1^4 E \|T\|_{\op}},
 \end{equation}
 for all $\lambda \in [0,\lambda_0]$, and let $\lambda_0 = \min
 \{\lambda_0', \tilde \lambda_0\}$.
 Let $\Ker \check A$ denote the vector bundle over $[0, \lambda_0]$ with
 fibers $\Ker \check A_\lambda$. Then $\Ker \check A \simeq [0,\lambda_0]
 \times \Ker \check A_0$ with trivialization given by
\begin{equation*}
 \id - T(\check A_0 - \check A_\lambda): \Ker \check A_\lambda  \to \Ker \check A_0.
\end{equation*}

It follows that if $(v_1,\dotsc,v_k)$ is a basis for $\Ker \check A_0$ then $(\tilde v_1^\lambda,\dotsc,\tilde v_k^\lambda)$ gives a basis for $\Ker \check A_\lambda$, where 
\begin{equation*}
 \tilde v^\lambda_i = (\id - T(\check A_0 - \check A_\lambda))^{-1} v_i, \qquad i=1,\dotsc,k.
\end{equation*}
Moreover, 
\begin{multline*}
 \| \tilde v^\lambda_i- v_i\|_{2,\nu} \leq \sum_{i=1}^\infty\|
 (T(\check A_0 - \check A_\lambda))^iv_i\|_{2,\nu}
 \leq \sum_{i=1}^\infty\|
 T\|_{\op}^i\|\check A_0 - \check A_\lambda\|_{\op}^i\|v_i\|_{2,\nu}\\
 \leq \tilde C \|A_0 -A_\lambda\|_{C^2}\|v_i\|_{2,\nu},
\end{multline*}
where $\tilde C = 2 \|T\|_{\op}n^4C_1^4E$.

Indeed, from \eqref{eq:est1} and \eqref{eq:est2} we get that
\begin{equation*}
  \|T\|_{\op}\|\check A_0 - \check A_\lambda\|_{\op} \leq
  \|T\|_{\op}n^4C_1^4E\|A_\lambda-A_0\|_{C^2} \leq \frac{1}{2},
\end{equation*}
so
\begin{equation*}
   \sum_{i=0}^\infty\|
 T\|_{\op}^i\|\check A_0 - \check A_\lambda\|_{\op}^i\leq 2
\end{equation*}
and hence
\begin{equation*}
   \sum_{i=1}^\infty\|
   T\|_{\op}^i\|\check A_0 - \check A_\lambda\|_{\op}^i
   \leq 2 \|T\|_{\op}\|\check A_0 - \check A_\lambda\|_{\op}
   \leq 2  \|T\|_{\op}n^4C_1^4E \|A_0 -A_\lambda\|_{C^2}.
\end{equation*}

Now if we multiply $\tilde v^\lambda_i$ by $A_\lambda A_0^{-1}$ for $i=1,\dotsc,k$, we get a basis for $\Ker \dbar_{A_\lambda}$ for $\lambda \leq \lambda_0$, and
\begin{align*}
 \|A_{\lambda} A_0^{-1}\tilde v^{\lambda}_i - v_i \|_{2,\nu}
 &
 \leq 
 \|(A_{\lambda} - A_{0})A_0^{-1}\tilde v^{\lambda}_i\|_{2,\nu} + \|\tilde  v^{\lambda}_i - v_i \|_{2,\nu} \\
 & \leq n^2C_2^2\|A_0^{-1}\|_{C^2} \|( A_\lambda -
   A_0)\|_{C^2}\|\tilde v^\lambda_i\|_{2,\nu} + \tilde C \|( A_0 -
   A_\lambda)\|_{C^2} \|v_i\|_{2,\nu}.
\end{align*}
The result follows.
\end{proof}

We summarize the implications of Proposition \ref{prp:keriso} by giving some results which are used in Section \ref{sec:glupf}.

\begin{lma}\label{lma:genkeriso}
 Suppose that $A_\lambda$, $\lambda \in [0,1]$, is a $C^2$-admissible family of trivialized boundary conditions on $D_0$ so that $A_0 =\id$. Then there is a $\lambda_0 >0$ so that for all $0 \leq \lambda \leq \lambda_0$ we have that $\Coker \dbar_{A_\lambda}$ is trivial, and there are functions $ v_{i,\lambda} \in \Hi_{2}[A_\lambda]$, $i=1,\dotsc,n$, satisfying
 \begin{equation*}
 \|v_{i,\lambda} -\partial_{x_i}\|_{2} \to 0 \text{ as } \lambda \to 0
\end{equation*}
and so that 
\begin{equation*}
 \kd_{A_\lambda} = \spn (v_{1,\lambda},\dotsc, v_{n,\lambda}).
\end{equation*}
\end{lma}

\begin{proof}
In the limit $\lambda =0$ the problem is surjective with $n$-dimensional kernel spanned by $(\partial_{x_1},\dotsc,\partial_{x_n})$. Since surjectivity is an open condition and the problem is Fredholm for all $\lambda$ the result then follows from Proposition \ref{prp:keriso}. 
\end{proof}

\begin{cor}\label{cor:canonorient}
The canonical orientation of $\det \dbar_{A_\lambda}$ is given by the ordered tuple $(v_{1,\lambda},\dotsc, v_{n,\lambda})$.
\end{cor}

\begin{cor}\label{cor:canonbundle}
We have an oriented vector bundle $\Ker\ul{\dbar_A} \simeq [0, \lambda_0] \times \kd_{A_0}$ where the trivializations are given by $v_{i,\lambda} \mapsto \partial_{x_i}$, $i=1,\dotsc,n$. Moreover, the orientation of each fiber is given by the canonical orientation.
\end{cor}

We have a similar result for $C^2$-admissible families defined on $\Delta_i$,  $i=1$ or $i=2$. Indeed, assume that 
 $A_\lambda$, $\lambda \in [0,1]$ is a family of trivialized boundary conditions on $\Delta_i$ with $A_0= \id$. Assume that  $\w=\w_\nu$ is a weight function so that $(A_\lambda, \w)$ is admissible for all $\lambda \in[0,1]$, and so that the weight vectors satisfy 
 \begin{align*}
 \text{If } i=1:& \qquad \nu_0 = (\underbrace{-\delta,\dotsc,-\delta}_k, \underbrace{\delta,\dotsc,\delta}_{n-k}) \\
 \text{If } i=2:&
  \qquad \nu_0 = (\underbrace{-\delta,\dotsc,-\delta}_k, \underbrace{\delta,\dotsc,\delta}_{n-k}), \quad \nu_1 = ({-\delta,\dotsc,-\delta}), \text { or } \\ 
   & \qquad \nu_0 =({-\delta,\dotsc,-\delta}), 
    \quad \nu_1 = (\underbrace{-\delta,\dotsc,-\delta}_k, \underbrace{\delta,\dotsc,\delta}_{n-k}). 
 \end{align*}
 Then we have the following.

\begin{lma}\label{lma:genkerisoweight}
There is a $\lambda_0 >0$ so that for $0 \leq \lambda \leq \lambda_0$ we have that $\Coker \dbar_{A_\lambda}$ is trivial, and there are functions $ v_{i,\lambda} \in \Hi_{2,\nu}[A_\lambda]$, $i=1,\dotsc,k,$ satisfying
 \begin{equation*}
 \|v_{i,\lambda} -\partial_{x_i}\|_{2,\nu} \to 0 \text{ as } \lambda \to 0
\end{equation*}
and so that 
\begin{equation*}
 \kd_{A_\lambda} = \spn (v_{1,\lambda},\dotsc, v_{k,\lambda}).
\end{equation*}
\end{lma}

\begin{proof}
From \cite{legsub} it follows that in the limit $\lambda =0$ the problem is surjective with $k$-dimensional kernel spanned by $(\partial_{x_1},\dotsc,\partial_{x_k})$. Since surjectivity is an open condition and the problem is Fredholm for all $\lambda$ the result then follows from Proposition \ref{prp:keriso}. 
\end{proof}

\begin{cor}\label{cor:canonbundleweight}
We have an oriented vector bundle $\Ker\ul{\dbar_A} \simeq [0, \lambda_0] \times \kd_{A_0}$ where the trivializations are given by $v_{i,\lambda} \mapsto \partial_{x_i}$, $i=1,\dotsc,k$.
\end{cor}

\begin{rmk}\label{rmk:megarmk}
All the results above also holds for the cokernel of injective problems, by considering the adjoint operator $\partial$.
\end{rmk}


\vskip.1in \noindent \textsc{Acknowledgments.} The first part of this paper is part of
the authors Ph.D.\ thesis at Uppsala University, supported by the Knut
and Alice Wallenberg Foundation.  This part  was finished while the
author was sponsored by the ERC project Geodycon, as a postdoc at Nantes University, and by the grant KAW
2015.0353 from the Knut and Alice Wallenberg Foundation, as a postdoc
at Stanford University. The second part was written during the authors period as a postdoc at the University of Oslo.

The author would like to
thank Tobias Ekholm and Michael Sullivan for sharing their insights in the
subject, and for pointing out mistakes in earlier versions of the paper.

\newpage
\bibliographystyle{hplain}
\bibliography{c_main_n} 
\end{document}